\documentclass[letterpaper,11pt]{amsart}
\usepackage[margin=1.2in]{geometry}
\usepackage{amsmath,amsthm,amssymb}
\usepackage{xspace,xcolor}
\usepackage[breaklinks,colorlinks,citecolor=teal,linkcolor=teal,urlcolor=teal,pagebackref,hyperindex]{hyperref}
\usepackage[alphabetic]{amsrefs}
\usepackage[all]{xy}
\usepackage[english]{babel}
\usepackage{enumitem}
\usepackage{pgfplots}
\usepackage{tikz,xcolor}
\usepackage{tikz-cd}
\usepackage{mathrsfs}
\usepackage{color}
\usepackage{cancel}

\usepackage{comment}

\newcommand\twedge{\widetilde{\wedge}}
\newcommand\totimes{\widetilde{\otimes}}

\newcounter{intro}

\newtheorem{intro-conjecture}[intro]{Conjecture}
\newtheorem{intro-corollary}[intro]{Corollary}
\newtheorem{intro-theorem}[intro]{Theorem}

\newcommand{\theoremref}[1]{\hyperref[#1]{Theorem~\ref*{#1}}}
\newcommand{\lemmaref}[1]{\hyperref[#1]{Lemma~\ref*{#1}}}
\newcommand{\definitionref}[1]{\hyperref[#1]{Definition~\ref*{#1}}}
\newcommand{\propositionref}[1]{\hyperref[#1]{Proposition~\ref*{#1}}}
\newcommand{\conjectureref}[1]{\hyperref[#1]{Conjecture~\ref*{#1}}}
\newcommand{\corollaryref}[1]{\hyperref[#1]{Corollary~\ref*{#1}}}
\newcommand{\exampleref}[1]{\hyperref[#1]{Example~\ref*{#1}}}

\theoremstyle{plain}
\newtheorem{thm}{Theorem}[section]

\newtheorem{lem}[thm]{Lemma}
\newtheorem{prop}[thm]{Proposition}
\newtheorem{cor}[thm]{Corollary}
\newtheorem{conj}[thm]{Conjecture}
\newtheorem{problem}[thm]{Problem}

\theoremstyle{definition}
\newtheorem{defi}[thm]{Definition}

\newtheorem{eg}[thm]{Example}

\theoremstyle{remark}
\newtheorem{rmk}[thm]{Remark}

\newtheorem{claim}[thm]{Claim}

\def\Mustata{Mus\-ta\-\c{t}\u{a}\xspace}

\def\Z{{\mathbf Z}}
\def\Q{{\mathbf Q}}
\def\R{{\mathbf R}}
\def\C{{\mathbf C}}

\def\P{{\mathbf P}}

\def\cA{\mathcal{A}}

\def\cC{\mathcal{C}}
\def\cD{\mathcal{D}}
\def\cE{\mathcal{E}}
\def\cF{\mathcal{F}}

\def\cH{\mathcal{H}}

\def\cK{\mathcal{K}}
\def\cL{\mathcal{L}}
\def\cM{\mathcal{M}}
\def\cN{\mathcal{N}}
\def\cO{\mathcal{O}}

\def\cV{\mathcal{V}}

\def\.{\cdot}
\def\^{\widehat}

\def\({\left(}
\def\){\right)}

\renewcommand{\and}{ \ \ \text{ and } \ \ }

\DeclareMathOperator{\Gr} {Gr}

\DeclareMathOperator{\Sym} {Sym}

\usepackage{xpatch}
\makeatletter   
\xpatchcmd{\@tocline}
{\hfil\hbox to\@pnumwidth{\@tocpagenum{#7}}\par}
{\ifnum#1<0\hfill\else\dotfill\fi\hbox to\@pnumwidth{\@tocpagenum{#7}}\par}
{}{}
\makeatother    
\makeatletter
\def\l@section{\@tocline{1}{0pt}{0pc}{0pc}{\bfseries}}
 \def\l@subsection{\@tocline{2}{0pt}{4pc}{6pc}{}}
\def\l@subsubsection{\@tocline{3}{0pt}{8pc}{8pc}{}}
 \makeatother

\begin{document}

\author[Q.~Chen]{Qianyu Chen}

\address{Department of Mathematics, University of Michigan, 530 Church Street, Ann Arbor, MI 48109, USA}

\email{qyc@umich.edu}

\author[B.~Dirks]{Bradley Dirks}

\address{Department of Mathematics, Stony Brook University, Stony Brook, NY 11794-3651, USA}

\email{bradley.dirks@stonybrook.edu}

\author[S.~Olano]{Sebasti\'{a}n Olano}

\address{Department of Mathematics, University of Toronto, 40 St. George St., Toronto, Ontario Canada, M5S 2E4}

\email{seolano@math.toronto.edu}

\author[D.~Raychaudhury]{Debaditya Raychaudhury}

\address{Department of Mathematics, University of Arizona, 617 N. Santa Rita Ave., Tucson, AZ 85721, USA}

\email{draychaudhury@math.arizona.edu}

\thanks{ Q.C. was partially supported by NSF Grant No. DMS-1952399, the Simons Collaboration grant Moduli of Varieties and AMS-Simons travel grant. B.D. was partially supported by the National Science Foundation under grant No. DMS-1926686 and NSF-MSPRF grant DMS-2303070. D.R. was partially supported by an AMS-Simons travel grant.}

\title[Hodge theory of Secant varieties]{Hodge theory of Secant varieties}

\vspace{-40pt}

\begin{abstract} %We study the local cohomology modules for the secant variety of lines of a smooth projective variety $Y$. We can fully describe these modules in terms of the primitive singular cohomology of $Y$. As applications, we compute their Hodge-Lyubeznik numbers, the mixed Hodge structure on their singular cohomology, the pure Hodge structure on their intersection cohomology, the generating level of the Hodge filtration on their local cohomology, and several other related singularity invariants. 

%We also study the local cohomology modules for the higher secant varieties of a smooth projective curve $C$. We show that the local cohomological defect is always zero and, moreover, we compute the weight filtration \textcolor{black}{on this module. We further compute the singular cohomology of these varieties.} 

%We apply our results to study the $\Q$-factoriality defect and partial Poincar\'{e} duality for the secant variety of lines of a smooth projective variety of any dimension, and for the higher secant variety of a smooth projective curve. 
We study the local cohomology modules for the secant variety of lines of a smooth projective variety $Y$ and for higher secant varieties of smooth projective curves. We show that the local cohomological defect in the first case is related to the primitive cohomology of $Y$, and in the second case it is $0$.

As applications, we compute their (intersection) Hodge-Lyubeznik numbers, the mixed Hodge structure on their singular cohomology, the pure Hodge structure on their intersection cohomology, the generating level of the Hodge filtration on their local cohomology modules and their $\Q$-factoriality defect. As byproducts, we recover and refine various results from the literature by removing restrictive positivity assumptions.
\end{abstract}

\maketitle

\vspace{-35pt}

\tableofcontents

\section{Introduction}

Given a smooth projective variety $Y\subseteq\P^N$, its $k$-secant variety is, by definition, the closure of the union of $k$-secant $(k-1)$-planes. Secant varieties, by which we mean the secant variety of lines (i.e., $2$-secant varieties) and higher secant varieties arise naturally when studying embeddings of smooth projective varieties into projective spaces and have long sparked the interest of geometers. Their study dates back to classical projective geometry, where questions about dimension, defectiveness, and defining equations were already central (see \cites{sev, Zak}). Later, work has focused on describing their defining equations and syzygies, as well as their singularities \cites{CC, CR, ENP, Rai, SV, Ver, V2, V1}.  

Building on this classical picture, more recent work has focused on the singularities of these varieties. For example, Ullery established the normality of secant varieties in \cite{Ullery}, and Ein--Niu--Park did so for higher secant varieties of curves in \cite{ENP}, while subsequent articles studied the singularities of secant varieties from a Hodge-theoretic perspective. In particular, these include results on Du Bois and rational singularities, as well as their higher counterparts \cites{ChouSong, ORS}. The local cohomology modules of secant varieties were further analyzed in \cite{OR}, and the present paper refines and extends several of those results.

In this setting, local cohomology modules are natural objects of study, since secant varieties come equipped with an embedding into the projective space. These $\cD$-modules carry the additional structure of mixed Hodge modules, and studying them is equivalent to studying their duals, namely the constant Hodge module on the variety. This is the main Hodge-theoretic tool, together with slight generalizations, that we employ to investigate the Hodge theory of
secant varieties of smooth projective varieties of arbitrary dimension, and
higher secant varieties of smooth projective curves.

We will assume a mild positivity condition on the embedded projective variety to carry out our study. We emphasize that our positivity assumptions are ``minimal'' in a precise sense described below, unless explicitly stated otherwise. In particular, this allows us to refine various results from the existing literature by substantially weakening their positivity hypotheses. \textcolor{black}{Moreover, as in \cite{CDOIsolated}, our main results remain valid in any theory of $A$-mixed sheaves as defined by Saito \cites{SaitoMixedSheaves,SaitoArithmetic} (see Subsection \ref{secmixed} for details on mixed sheaves and Subsection \ref{rmk-SecantVarietyOverk} for the definition of ``(higher) secant varieties over a subfield of $\C$''), though for the introduction we prefer to work over the field of complex numbers $\C$ using the language of mixed Hodge modules.}

\medskip

\noindent{\bf Secant variety of lines.} For secant varieties, the setting is as follows: let $Y$ be a smooth projective variety of positive dimension and let $L$ be a very ample line bundle on $Y$ giving rise to an embedding $\iota: Y \hookrightarrow \P^N$ by its complete linear series, and let $\Sigma$ be its secant variety. We assume that $L$ is 3-very ample, which means that every $0$-dimensional subscheme of length $4$ imposes independent conditions on the sections of $L$. This is the ``minimal'' positivity condition under which $\Sigma$ is not defective (i.e., has the expected dimension $2\dim Y+1$), and moreover in this case we have an explicit log-resolution of the singularities of $\Sigma$ denoted by $$t:\P\to\Sigma$$ whose construction we will recall in Subsection \ref{prelimsecant}.

We aim to calculate the {\it local cohomological defect} ${\rm lcdef}(\Sigma)$ (resp. their variants ${\rm lcdef}_{\rm gen}(\Sigma)$, ${\rm lcdef}_{\rm gen}^{>0}(\Sigma)$) introduced in \cite{PS} (resp. in \cites{DOR, CDOIsolated}), which is an invariant that measures how far $\Sigma$ is from  being a {\it cohomological complete intersection} (from now on, we abbreviate this condition as {\it CCI}), and completely describe the weight filtration $W_{\bullet}$ on all the local cohomology modules. We further aim to describe the {\it non-rationally smooth locus} $\Sigma_{\rm nRS}$ of $\Sigma$, which is precisely the locus where $\Sigma$ fails to be a {\it rational homology manifold}, a widely studied topological notion. It is known (see for example \cite{PPLefschetz}) that the precise measure of failure of $\Sigma$ to become a rational homology manifold is encoded in the so-called {\it RHM-defect object} $\mathcal{K}_{\Sigma}^{\bullet}$ (also called the ``multiple-point complex'' and the ``comparison complex''), and we will also describe this object. 

%Our description of the local cohomology modules uses $\mathbf D_{\Sigma}^{H}$, an object lying in the derived category of mixed Hodge modules $D^b({\rm MHM}(\Sigma))$, considered in \cite{CDOIsolated} and defined as:
%\[ \mathbf D_{\Sigma}^{H} = \mathbf D_{\Sigma}(\Q^{H}_{\Sigma}[\dim \Sigma])(-\dim \Sigma)\]
%where $\Q^{H}_{\Sigma}[\dim \Sigma]$ is the constant Hodge module (see Section \ref{sec-Prelim} for details). 

%It has been shown in {\it loc. cit.} that the cohomology of these objects coincide with the local cohomology modules. 

We introduce some notation. Let $\mathbf V_{\rm prim}^k$ (resp.  $\mathbf V^k$) denote the variation of Hodge structures on $Y$ such that the fiber over $y$ is $H^k_{\rm prim}({\rm Bl}_y(Y))$ (resp. $H^k({\rm Bl}_y(Y))$) where the primitive cohomology is taken with respect to an appropriate polarization described in Section \ref{sec-GeneralResult}. It is well known that the primitive cohomology of ${\rm Bl}_y(Y)$ is essentially the same as that of $Y$ (see \cites{GH, HuhWang, Voisin}), except at $H^2$, where there always exists a new non-zero primitive class (see \lemmaref{lem-BlowupPrim} for a precise statement). Let $i_{\Sigma} \colon \Sigma \hookrightarrow \P^N$ be the embedding determined by $\iota \colon Y \hookrightarrow \P^N$, with $q = N - \dim \Sigma = N  - (2\dim Y +1)$. In this setting, we have the following result that follows from the main theorem  of \cite{CDOIsolated} (see \corollaryref{cor-generalSit1}).

\begin{intro-theorem}[$=$ \theoremref{thmA}]\label{thm-Secants}  Assume $L$ is 3-very ample and $\Sigma\neq\P^N$. Then
\begin{enumerate}
    \item\label{thm-Secants1} ${\rm lcdef}(\Sigma) = {\rm lcdef}_{\rm gen}(\Sigma) = {\rm lcdef}_{\rm gen}^{>0}(\Sigma)=\begin{cases}
        \dim Y -1 & \dim Y\geq 2,\, H^1(Y,\cO_Y) \neq 0\\
        \dim Y-2 & \dim Y\geq 2,\,H^1(Y,\cO_Y)=0\\
        0 & \dim Y=1
    \end{cases}.$
    %\in \{\dim Y-2, \dim Y -1\}$.  We have ${\rm lcdef}(\Sigma) = \dim Y -1$ if and only if $H^1(Y,\cO_Y) \neq 0$.
    \item $\Sigma_{\rm nCCI}=Y$ if $\Sigma$ is not CCI. Moreover, $\Sigma_{\rm nRS} =Y$ if $\Sigma$ is not a rational homology manifold (i.e., $Y\ncong\P^1$, see (5) below). 
    \item For all $0<j\leq {\rm lcdef}(\Sigma)$, we have an isomorphism of pure Hodge modules of weight $\dim \Sigma +j+1$:
    \[ \cH^{q+j}_{\Sigma}(\cO_{\P^N})\cong %{i_{\Sigma}}_*\cH^j \mathbf D_\Sigma^H(-q) \cong 
    \iota_* \mathbf V_{\rm prim}^{\dim Y -j}(-q-j-1).\]
    \item ${\rm Gr}^W_k \cH^q_{\Sigma}(\cO_{\P^N})$ is potentially non-zero only when $k\in\left\{N+q,N+q+1\right\}$, in which case they are given by:
    \[{\rm Gr}^W_k \cH^q_{\Sigma}(\cO_{\P^N})\cong\begin{cases}
        %{i_{\Sigma}}_*{\rm Gr}^W_{\dim \Sigma +1} \cH^0 \mathbf D_\Sigma^H(-q) \cong 
        \iota_* \mathbf V_{\rm prim}^{\dim Y}(-q-1) & k=N+q+1\\ %, \dim Y\neq 1\\
         %{i_{\Sigma}}_*{\rm Gr}^W_{\dim \Sigma +1} \cH^0 \mathbf D_\Sigma^H(-q) \cong
         %\iota_* (H^1(Y)\boxtimes\Q_Y^H[1])(-1) & k=N+q+1, \dim Y= 1\\
        %{i_{\Sigma}}_*{\rm Gr}^W_{\dim \Sigma} \cH^0 \mathbf D_\Sigma^H(-q)\cong 
        {i_{\Sigma}}_*{\rm IC}_{\Sigma}^H & k=N+q
    \end{cases}.\]
    %We have ${\rm Gr}^W_i \cH^q_{\Sigma}(\cO_{\P^N}) \neq 0 \implies i \in [N+q, N+q+1]$, an isomorphism
    %\[ {\rm Gr}^W_{N+q+1} \cH^q_{\Sigma}(\cO_{\P^N}) \cong {\rm Gr}^W_{\dim \Sigma +1} \cH^0 \mathbf D_\Sigma^H(-q) \cong \iota_* \mathbf V_{\rm prim}^{\dim Y}(-q-1),\]
    %and the other non-zero weight piece is ${\rm IC}_{\Sigma}^H = {\rm Gr}^W_{n+q} \cH^q_{\Sigma}(\cO_{\P^N})$, lying in a short exact sequence
    Moreover, ${\rm IC}_{\Sigma}^H$ lies in a short exact sequence
    \[ 0 \to {\rm IC}_{\Sigma}^H \to \cH^0 t_* \Q_{\mathbf P}^H[\dim \Sigma] \to i_* \mathbf V^{\dim Y +1} \to 0\] %\textrm{ if }\dim Y\geq 2,\textrm{ and }\]
    %\[0 \to {\rm IC}_{\Sigma}^H \to \cH^0 t_* \Q_{\mathbf P}^H[\dim \Sigma] \to i_* \Q_Y^H[1](-1) \to 0\textrm{ if }\dim Y=1,\]
    where $t \colon \P \to \Sigma$ is the standard resolution of singularities of the secant variety.
    \item The RHM defect object $\cK_\Sigma^\bullet$ is a pure complex, and we have a decomposition in $D^b({\rm MHM}  (\Sigma))$ given by
    \[ \cK_{\Sigma}^\bullet \cong \bigoplus_{\ell = 0}^{\dim Y -1} i_* \mathbf V^{\dim Y-\ell}_{\rm prim}[\ell].\]
    In particular,
    \begin{itemize}
        \item $\Sigma$ is a rational homology manifold if and only if $Y\cong\P^1$. 
        \item If $Y\ncong\P^1$, then $\cK_\Sigma^\bullet$ is a pure complex of weight $\dim \Sigma -1$.
    \end{itemize}
\end{enumerate}
\end{intro-theorem}

\begin{rmk}[\theoremref{thmA}]
    If $\dim Y=1$, part (4) of the above simplifies to give us that 
    \[{\rm Gr}^W_k \cH^q_{\Sigma}(\cO_{\P^N})\cong\begin{cases}
        \iota_* (H^1(Y)\boxtimes\Q_Y^H[1])(-1) & k=N+q+1\\
        {i_{\Sigma}}_*{\rm IC}_{\Sigma}^H & k=N+q
    \end{cases}\]
    and moreover, ${\rm IC}_{\Sigma}^H$ lies in a short exact sequence
    \[0 \to {\rm IC}_{\Sigma}^H \to \cH^0 t_* \Q_{\mathbf P}^H[\dim \Sigma] \to i_* \Q_Y^H[1](-1) \to 0.\] See also \theoremref{thm-ConstantSheafHigherSecants} for analogous statements about higher secant varieties of curves.
\end{rmk}

\color{black}

\begin{rmk}\label{rmk-GICT} 
We also make a note of the following facts:
\begin{enumerate}[label=\alph*)]
    \item The value of ${\rm lcdef}({\Sigma})$ described in \theoremref{thm-Secants}(\ref{thm-Secants1}) was proven in \cite{OR}*{Thm. A} under a more restrictive positivity assumption. 
    \item Analogues of some of the statements above for the secant and higher secant varieties of curves appear in \cite{Brogan}*{Thm. 2.15}. As mentioned, we discuss the curve case in more detail in \theoremref{thm-ConstantSheafHigherSecants}.
\end{enumerate}
\end{rmk}

We now turn to the invariant $c(\Sigma)$ (resp. ${\rm HRH}(\Sigma)$) introduced in \cite{CDOIsolated} (resp. \cites{DOR, PPLefschetz}). This invariant gives a numerical measure of how close $\Sigma$ is to being CCI (resp. a rational homology manifold). It is  $\infty$ if and only if $\Sigma$ is CCI (resp. a rational homology manifold). As described in \cites{CDOIsolated,DOR}, knowledge of the filtrations on the local cohomology modules described above enables us to compute these invariants. These notions are the Hodge-theoretic weakening of the conditions that $\Q_{\Sigma}[\dim \Sigma]$ be perverse (resp. simple perverse), see Subsection \ref{secsi} for precise definitions and properties. Our result is as follows:

\begin{intro-corollary}[$=$ \corollaryref{cor-chrhn}]\label{cor-chrh} Assume $L$ is 3-very ample and $\Sigma\neq\P^N$. Then we have:
\begin{enumerate}
    \item $c(\Sigma)=\begin{cases}
        \infty & \dim Y=1; \textrm{ or $\dim Y=2$ and $H^1(Y,\cO_Y) = 0$,}\\
        0 & \dim Y\geq 3\textrm{ and }H^{i}(Y,\cO_Y) = 0 \text{ for all } 0 < i < \dim Y,\\
        -1 & \textrm{otherwise.}
    \end{cases}$

    \smallskip
    
    \noindent In particular, $c(\Sigma)\geq 0$ if and only if $H^{i}(Y,\cO_Y) = 0 \text{ for all } 0 < i < \dim Y$.

    \smallskip
    
    \item ${\rm HRH}(\Sigma)=\begin{cases}
        \infty & Y\cong\P^1,\\
        0 & \dim Y\geq 2\textrm{ and } H^{i}(Y,\cO_Y) = 0 \text{ for all } i>0,\\
        -1 & \textrm{otherwise}.
    \end{cases}$

    \smallskip
    
    \noindent In particular, ${\rm HRH}(\Sigma)\geq 0$ if and only if $H^{i}(Y,\cO_Y) = 0 \text{ for all } i>0$.
\end{enumerate}
%Assume $\Q_{\Sigma}[\dim \Sigma]$ is not perverse (either $\dim Y \geq 3$ or $\dim Y =2$ and $H^1(Y,\cO_Y) \neq 0$). Then $c(\Sigma) \leq 0$, and equality holds if and only if 
%\[ H^{i}(\cO_Y) = 0 \text{ for all } 0 < i < \dim Y.\]
%If $Y \not\cong \P^1$, then $\Sigma$ is not a rational homology manifold. The variety $\Sigma$ satisfies ${\rm HRH}(\Sigma) \leq 0$, and equality holds if and only if $c(\Sigma) = 0$ and 
%\[ H^{\dim Y}(\cO_Y) = 0.\]
\end{intro-corollary}

\begin{comment}
By \cite{CDOIsolated}, it turns out that the invariant $c(\Sigma)$ is intimately related to the depth of the associated graded of the Du Bois complex $$\underline{\Omega}_{\Sigma}^p:=\textrm{Gr}^p_F(\underline{\Omega}_{\Sigma}^{\bullet})[p]\in D^b(\text{Coh}(\Sigma))$$ where $\underline{\Omega}_{\Sigma}^{\bullet}$ is the filtered Du Bois complex constructed in \cite{DuBois}. For example, we have 
\begin{equation}\label{depth}
    \textrm{$c(\Sigma)\geq 0$ if and only if $\textrm{depth}(\underline{\Omega}_{\Sigma}^0)= \dim Y$.}
\end{equation}
\end{comment}

For an arbitrary variety $Z$, the above invariants $c(Z)$ and ${\rm HRH}(Z)$ are intimately related to other singularity properties of $Z$, such as Du Bois and rational singularity conditions. The main implications among these properties are summarized in the diagram below:

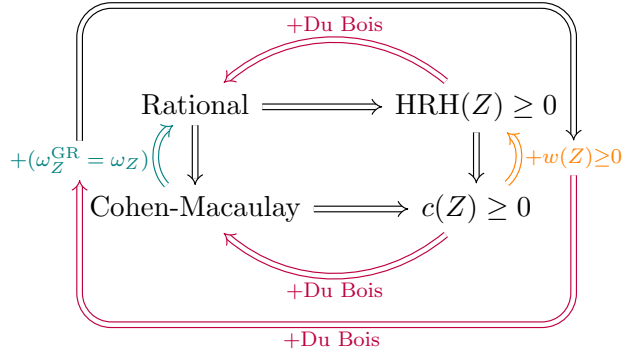
\[
\begin{tikzcd}
    \textrm{Rational}\arrow[Rightarrow]{r}\arrow[Rightarrow]{d} & {\rm HRH}(Z)\geq 0 \arrow[Rightarrow, purple, bend right=40,swap, "\textrm{$+$Du Bois}"]{l}\arrow[Rightarrow]{d}\\
    \textrm{Cohen-Macaulay}\arrow[Rightarrow]{r}\arrow[Rightarrow, teal, bend left=50,  "\textrm{$+(\omega_Z^{\rm GR}=\omega_Z)$}"{name=U}, pos=0.4]{u} & c(Z)\geq 0 \arrow[Rightarrow, purple, bend left=40, "\textrm{$+$Du Bois}"]{l}\arrow[Rightarrow, orange, bend right=50, swap, "+w(Z)\geq 0"{name=D}]{u} \arrow[Rightarrow, to path={ -- ([yshift=14ex]\tikztostart.south) -| (\tikztotarget)},
    rounded corners=12pt, from=U, to=D]
    \arrow[Rightarrow, purple, "\textrm{$+$Du Bois}", pos=0.245, to path={ -- ([yshift=-12ex]\tikztostart.south) -| (\tikztotarget)\tikztonodes},
    rounded corners=12pt, from=D, to=U]
\end{tikzcd}
\]

In the above, $\omega_Z^{\rm GR}$ is the Grauert-Riemenschneider sheaf (see Subsection \ref{secsi} for details). See  \definitionref{defw} for the definition of $w(Z)$, and \propositionref{wc} for its characterization.

In \cite{ChouSong}, the singularity invariants appearing on the left side of the diagram are analyzed for secant varieties under a restrictive positivity assumption that was shown to imply the Du Bois condition (\cite{ChouSong}*{Prop. 4.4} shows that just 3-very ampleness is not enough). This work clarifies the role of positivity in {\it loc. cit.} by showing that the analogous statements hold for the invariants on the right {\it without} any positivity assumption other than 3-very ampleness:

\begin{rmk} \corollaryref{cor-chrh} and Remark \ref{wSigma} give the following using the above diagram:

\smallskip

{\it Assume $L$ is 3-very ample and $\Sigma\neq\P^N$. If $\Sigma$ has Du Bois singularities, then:
\begin{itemize}
    \item $\Sigma$ is Cohen-Macaulay if and only if  $H^{i}(Y,\cO_Y) = 0 \text{ for all } 0 < i < \dim Y$.
    \item $\Sigma$ has rational singularities if and only if $H^{i}(Y,\cO_Y) = 0 \text{ for all } i>0$.
    \item $\omega_{\Sigma}^{\rm GR}=\omega_{\Sigma}$ if and only if $H^{\dim Y}(Y,\cO_Y)=0$.
\end{itemize}}
The above recovers the consequence on Cohen-Macaulayness derived in \cite{ChouSong}*{discussion after Thm. 1.3}, \cite{ChouSong}*{Cor. 1.5} and \cite{ChouSong}*{Thm. 1.4} respectively. Indeed, the above ``if and only if'' statements were proven in {\it loc. cit.} under a positivity assumption which was shown to imply the Du Bois condition.
\end{rmk}

\color{black}

\begin{rmk}
\corollaryref{cor-chrh}(2) was proven in \cite{DOR}*{Ex. 13.11} under a more restrictive positivity assumption. The fact that secant (and higher secant) varieties of rational normal curves are rational homology manifolds is also proven by Brogan \cite{Brogan}.
\end{rmk}

\textcolor{black}{We describe two other related invariants. Associated to $y\in \Sigma$, there are {\it Hodge-Lyubeznik numbers} $ \lambda_{ r,s}^{u,v}(\cO_{\Sigma,y})$ for $u,v,r,s\in\Z$ defined by Garc\'{i}a L\'{o}pez and Sabbah \cite{HodgeLyubeznik}. These numbers can be thought of as local measures of singularities at $y\in Y$ and are defined through filtrations on the local cohomology modules. Following \cite{CDOIsolated}, we provide an alternate equivalent definition of these numbers in Subsection \ref{secsi} that is intrinsic. There are also {\it intersection Hodge-Lyubeznik numbers} ${\rm I}\lambda_r^{u,v}(\cO_{\Sigma,y})$ introduced in {\it loc. cit.} The (intersection) Hodge-Lyubeznik numbers are only interesting if $y\in \Sigma_{\rm nRS}$, and we compute them below:} 

\begin{intro-corollary}[$=$ \corollaryref{hlnn}]\label{hln} \textcolor{black}{Assume $L$ is 3-very ample and $\Sigma\neq\P^N$. Further assume $\Sigma_{\rm nRS}\neq\emptyset$ (equivalently $Y\ncong\P^1$, whence $\Sigma_{\rm nRS}=Y$ by \theoremref{thm-Secants}).} Let $y\in Y \subseteq \Sigma$. Then:
\begin{enumerate}
    \item For $\dim Y + 2 \leq s < \dim \Sigma$, we have the equality
\[ \lambda_{\dim \Sigma,s}^{u,v}(\cO_{\Sigma,y}) = \begin{cases} h^{-v,-u}_{\rm prim}(Y) & u+v = \dim Y +1-s \text{ and } (u,v) \neq (-1,-1) \\ 1 + h^{1,1}_{\rm prim}(Y) & \, s = \dim Y +3 \text{ and } (u,v) = (-1,-1)\end{cases}\]
and all other numbers vanish for such $s$.
\item For $s= \dim \Sigma$, \textcolor{black}{the Hodge-Lyubeznik numbers $\lambda_{r,\dim \Sigma}^{u,v}(\cO_{\Sigma,y})$ are potentially non-zero only for $r\geq \dim Y +2$, in which case} we have the equality
\[ \lambda_{r,\dim \Sigma}^{u,v}(\cO_{\Sigma,y}) = \begin{cases} h^{-v,-u}_{\rm prim}(Y) & u+v = r-\dim \Sigma, (u,v) \neq (-1,-1), \textcolor{black}{\dim Y\geq 2} \\ 1+h^{1,1}_{\rm prim}(Y) & r = \dim \Sigma -2, (u,v) = (-1,-1), \textcolor{black}{\dim Y\geq 2}\\
\textcolor{black}{1} & \textcolor{black}{r=3, (u,v)=(0,0),\dim Y=1}\\
\textcolor{black}{0}& \textcolor{black}{\textrm{otherwise}}\end{cases}.\]
%and all other numbers vanish.
\item The intersection Hodge-Lyubeznik numbers are \textcolor{black}{potentially} non-zero only for $r\geq \dim Y +1$, in which case they are given by the following:
\[ {\rm I}\lambda_r^{u,v}(\cO_{\Sigma,y}) = \begin{cases} h^{-v,-u}_{\rm prim}(Y) & u+v = r-\dim \Sigma, (u,v) \neq (-1,-1), \textcolor{black}{\dim Y\geq 2} \\ 1+h^{1,1}_{\rm prim}(Y) & r = \dim \Sigma -2,(u,v) = (-1,-1), \textcolor{black}{\dim Y\geq 2}\\
\textcolor{black}{h^1(Y,\cO_Y)} & \textcolor{black}{(u,v)\in\left\{(-1,0),(0,-1)\right\}, r=2,\dim Y=1}\\
1 & r = 3, (u,v) = (0,0), \dim Y = 1 \\ 
\textcolor{black}{0} & \textcolor{black}{\textrm{otherwise}}\end{cases}.\]
\end{enumerate}
\end{intro-corollary}

\textcolor{black}{The Hodge-Lyubeznik and intersection Hodge-Lyubeznik numbers completely determine $c(\Sigma)$ and ${\rm HRH}(\Sigma)$ by \cite{CDOIsolated}*{Thm. A and Thm. B}, whence \corollaryref{hln} can be used to deduce \corollaryref{cor-chrh}.}

We also compute the generation level of the Hodge filtration on local cohomology modules for the embedding $i_{\Sigma}:\Sigma \hookrightarrow \P^N$. 
Recall that, for $(\cM,F)$ a filtered left $\cD_Y$-module on a smooth variety $Y$, we say that $F_\bullet \cM$ is \emph{generated at level $k$} if 
\begin{center}
    $F_j \cD_Y \cdot F_k \cM = F_{k+j}\cM$ for all $j\geq 0$.
\end{center} 
For the Hodge filtration on a mixed Hodge module, such a $k$ always exists. We let ${\rm gl}(\cM,F)$ denote the minimal such $k$. 

Using decreasing Hodge filtrations for Hodge structures (as is convention), define $$\mu^\ell_{\rm prim}(Y) = \inf\left\{\{p \mid {\rm Gr}_F^p H^{\ell}_{\rm prim}(Y) \neq 0\} \cup \{\infty\}\right\}$$ \textcolor{black}{(here we put ``$\cup\left\{\infty\right\}$'' to ensure that the generation level is $-\infty$ if the module is zero).} As mentioned before,  we set $q$ to be the codimension of the embedding $i_{\Sigma}:\Sigma \hookrightarrow \P^N$. Then we have the following: 

\begin{intro-corollary}[$=$ \corollaryref{cor-GenLevel2Secantsn}] \label{cor-GenLevel2Secants} 
Assume $L$ is 3-very ample and $\Sigma\neq\P^N$. Then:
\begin{enumerate}
    \item If %$\dim Y > 2$ and $j>0$
    $\dim Y\geq 2$ and $0<j\leq {\rm lcdef}(\Sigma)$, there are equalities
\[ {\rm gl}(\cH^{q+j}_{\Sigma}(\cO_{\P^N}),F) = \begin{cases} \dim Y - j - \mu^{\dim Y-j}_{\rm prim}(Y) & j \neq \dim Y -2 \\ 2 & j =\dim Y -2, H^2(Y,\cO_Y)\neq 0 \\ 1 & j = \dim Y -2, H^2(Y,\cO_Y) = 0 \end{cases}.\]
%and 
%\[ {\rm gl}({\rm Gr}^W_{N+q+1}\cH^{q}_{\Sigma}(\cO_{\P^N}),F) = \dim Y - \mu^{\dim Y}_{\rm prim}(Y).\]
%\item If $\dim Y =2$, then we have
%\[ {\rm gl}({\rm Gr}^W_{N+q+1}\cH^{q}_{\Sigma}(\cO_{\P^N}),F) = \begin{cases} 2 &  H^2(Y,\cO_Y)\neq 0 \\ 1 & H^2(Y,\cO_Y) = 0 \end{cases}.\]
%\item \textcolor{red}{If $\dim Y=2$ and $H^1(Y,\cO_Y)\neq 0$. Then ${\rm gl}(\cH^{q+1}_{\Sigma}(\cO_{\P^N}),F) =1$.}
\item ${\rm gl}({\rm Gr}^W_{N+q+1}\cH^{q}_{\Sigma}(\cO_{\P^N}),F) = \begin{cases}
        \dim Y- \mu^{\dim Y}_{\rm prim}(Y) & \dim Y\geq 3\\
        2 & \dim Y=2, H^2(Y,\cO_Y)\neq 0\\
        1 & \dim Y=2, H^2(Y,\cO_Y)=0\\
        1 & \dim Y=1, H^1(Y,\cO_Y)\neq 0\\
        -\infty & \textrm{otherwise $($i.e., $Y\cong\P^1)$}
        \end{cases}$. 
%In particular, if $\dim Y\leq 2$, then $${\rm gl}({\rm Gr}^W_{N+q+1}\cH^{q}_{\Sigma}(\cO_{\P^N}),F) =\begin{cases}
%\dim Y & H^{\dim Y}(Y,\cO_Y)\neq 0\\
%\dim Y-1 & \textrm{otherwise}
%\end{cases}.$$ }
\item\label{cor-GenLevel2Secants5} For $Y$ of any dimension, we have
$ {\rm gl}( \cH^q_{\Sigma}(\cO_{\P^N}),F) \leq \dim Y$.
\item Assume $L$ satisfies $(Q'_p)$-property (\definitionref{def-pos}) for some $0\leq p\leq \dim Y-1$. Then
\begin{enumerate}
    \item ${\rm gl}({\rm IC}_{\Sigma}^H(-q),F)\leq \dim Y-p-1$.
     \item Assume $L$ satisfies $(Q'_{\dim Y-1})$-property. Then ${\rm gl}({\rm IC}_{\Sigma}^H(-q),F)=0$ and $${\rm gl}(\cH^q_{\Sigma}(\cO_{\P^N}),F)=\begin{cases}
        \max\left\{0,\dim Y- \mu^{\dim Y}_{\rm prim}(Y)\right\} & \dim Y\geq 3\\
        2 & \dim Y=2, H^2(Y,\cO_Y)\neq 0\\
        1 & \dim Y=2, H^2(Y,\cO_Y)=0\\
        1 & \dim Y=1, H^1(Y,\cO_Y)\neq 0\\
        0 & \textrm{otherwise $($i.e., $Y\cong\P^1)$} 
    \end{cases}.$$
\end{enumerate}
\end{enumerate}
%In the above notation, we have for 
\end{intro-corollary}

To prove \corollaryref{cor-GenLevel2Secants}(3), we use a general bound on the generation level for intersection complex Hodge modules, see \lemmaref{lem-genLevelIC} below. %\textcolor{red}{The above mentioned $(Q'_p)$-property, which is a positivity assumption on $L$, is described in Section \ref{sec-Prelim}.}
\begin{rmk}
    \textcolor{black}{It is easy to see that \corollaryref{cor-GenLevel2Secants} recovers the result about ${\rm gl}(\cH^{{\rm lcd}(\P^N,\Sigma)}_{\Sigma}(\cO_{\P^N}),F)$ computed in \cite{OR}*{Thm. A} when ${\rm lcdef}(\Sigma)\geq 1$, and relaxes the positivity assumption on $L$ of {\it loc. cit.} (here ${\rm lcd}(\P^N,\Sigma)=q+{\rm lcdef}(\Sigma)$).}
%and \theoremref{thm-Secants}, we deduce the following: {\it Assume $L$ is 3-very ample and $\Sigma\neq\P^N$. If ${\rm lcdef}(\Sigma)\geq 1$ (whence $\dim Y\geq 2$), then  \[ {\rm gl}(\cH^{{\rm lcd}(\P^N,\Sigma)}_{\Sigma}(\cO_{\P^N}),F)=\begin{cases}
%1 & H^1(Y,\cO_Y)\neq 0,\\
%2 & H^1(Y,\cO_Y)= 0, H^2(Y,\cO_Y)\neq 0,\\
%1 & H^1(Y,\cO_Y)= 0, H^2(Y,\cO_Y)= 0,
%\end{cases}\]
\end{rmk}

We now turn to classical topological invariants, specifically singular cohomology and Goresky--MacPherson's intersection cohomology. While the former carries a mixed Hodge structure, the latter is pure. We start by computing the intersection cohomology of secant varieties:

\begin{intro-theorem}[$=$ \theoremref{thm-IHSecantn}] \label{thm-IHSecant} Assume $L$ is 3-very ample and $\Sigma\neq\P^N$. Then for all $j \leq 2\dim Y +1 = \dim \Sigma$, we have the following (non-canonical) isomorphisms of pure weight $j$ Hodge structures:
\[ {\rm IH}^j(\Sigma) \cong \bigoplus_{\ell=1}^{\dim Y -1} H^{j-2\ell}(Y)(-\ell) \oplus \bigoplus_{\ell = \lceil \frac{j}{2} \rceil +1}^{\min\{j,\dim Y\}}\left( H^{\ell}_{\rm prim}(Y) \otimes H^{j-\ell}(Y)\right)\oplus \mathcal Q_j,\]
where $\mathcal Q_j$ depends on the class of $j$ mod $4$ as follows:
\[\mathcal Q_j=
\begin{cases}
    H^{k}(Y) \otimes H^{k+1}(Y) & \textrm{if $j=2k+1$ is odd};\\
    \bigwedge^2 H^{2k-1}(Y)(-1) \oplus {\rm Sym}^2 H^{2k}(Y) & \textrm{if $j=4k$};\\
    {\rm Sym}^2 H^{k-1}(Y)(-1) \oplus \bigwedge^2 H^{k}(Y) & \textrm{if $j = 2k$ with $k$ odd}.
\end{cases}
\]
\end{intro-theorem}

\begin{rmk}
The above recovers \cite{Brogan}*{Thm. 2.12} when $Y$ is a curve  (see Remark \ref{brogrmk} for details), but we note that Brogan's results apply to higher secant varieties of curves as well.
\end{rmk}

We can also study the singular cohomology of secant varieties. In the formulae below, we use $P_\ell$ to denote $H^{\ell}_{\rm prim}(Y)$ for notational convenience. %generalizing the formula of Brogan in the case of curves \cite{Brogan}*{}. We let $P_\ell = H^{\ell}_{\rm prim}(Y)$ for convenience. In the following statement, the isomorphisms are non-canonical.

\begin{intro-theorem}[$=$ \theoremref{sing-coh2n}]\label{sing-coh2} Assume $L$ is 3-very ample and $\Sigma\neq\P^N$. Then, for all $j > 0$, we have ${\rm Gr}^W_w H^j(\Sigma) = 0$ for all $w \notin \{j-1,j\}$, and we have (non-canonical) isomorphisms of mixed Hodge structures
\[ \bigoplus_{j\in \Z} {\rm Gr}^W_{j-1} H^j(\Sigma) \cong \bigoplus_{\substack{\ell \text{ odd} \\ a \leq \dim Y - \ell}} {\rm Sym}^2 P_\ell(-a) \oplus \bigoplus_{\substack{0 < \ell \text{ even} \\ a \leq \dim Y - \ell}} \bigwedge^2 P_\ell(-a)\oplus\bigoplus_{\substack{0 < \ell_1 < \ell_2 \\ a \leq \dim Y - \ell_2}} P_{\ell_1} \otimes P_{\ell_2}(-a),\]
and
\begin{equation*}
    \begin{split}
        \bigoplus_{j\in \Z} {\rm Gr}^W_{j} H^j(\Sigma) \cong & \left(\bigoplus_{j}\bigoplus_{k=1}^{\dim Y -1} H^{j-2-2k}(Y)(-k-1) \right) \oplus \bigoplus_{  0 \leq  a\leq 2\dim Y+1} \Q^H(-a)\\
& \oplus \bigoplus_{\substack{\ell_1 < \ell_2 \\ a \leq \dim Y - \ell_2}} P_{\ell_1} \otimes P_{\ell_2}(\ell_1 - \dim Y-a-1)\\
& \oplus \bigoplus_{\substack{0 < \ell \text{ even} \\  \dim Y - \ell < a\leq 2(\dim Y - \ell)+1}} {\rm Sym}^2 P_\ell (-a) \oplus \bigoplus_{\substack{\ell \text{ odd} \\  \dim Y -\ell < a\leq 2(\dim Y - \ell)+1}} \bigwedge^2 P_\ell (-a).
    \end{split}
\end{equation*}
\end{intro-theorem}

\begin{rmk}
 The above generalizes the formulae of Brogan \cite{Brogan}*{} from secant varieties of curves to that of varieties of arbitrary dimension. 
\end{rmk}

As established by the recent work \cite{PPFactorial}, one can use the singular and intersection cohomology of a normal projective variety $X$ to study the $\Q$-factoriality of the variety, or more precisely its {\it $\Q$-factoriality defect} $\sigma(X)$ which is defined as follows:
$$\sigma(X):=\dim_{\Q}\frac{{\rm Div}_{\Q}(X)}{{\rm CDiv}_{\Q}(X)}.$$
In the above, we have 
\begin{center}
    ${\rm Div}_{\Q}(X):={\rm Div}(X)\otimes\Q$ and ${\rm CDiv}_{\Q}(X):={\rm CDiv}(X)\otimes\Q$
\end{center}
where ${\rm Div}(X)$ is the free abelian group of Weil divisors on $X$ and ${\rm CDiv}(X)$ is the subgroup generated by Cartier classes. We have the surjective cycle class morphism $${\rm cl}_{\Q}:{\rm Div}_{\Q}(\Sigma)\to {\rm IH}^2(\Sigma,\Q)\cap {\rm IH}^{1,1}(\Sigma)$$ whose kernel is ${\rm Alg}_{\Q}^1(X)$. The defect $\sigma(X)=0$ if and only if $X$ is $\Q$-factorial. Analogously, given $x\in X$, one can define {\it local analytic $\Q$-factoriality defect} $\sigma^{\rm an}(X;x)$ of $X$ at $x$ (\cites{Kaw,PPFactorial}). For secant varieties, we obtain the following using \theoremref{thm-Secants} and \theoremref{thm-IHSecant}: %to study the factoriality of $\Sigma$.

\begin{intro-corollary}[$=$ \corollaryref{corFactorialn}] \label{corFactorial} Assume $L$ is 3-very ample and $\Sigma\neq\P^N$. Further assume $\Sigma$ is normal. Then:
\begin{enumerate}
    \item $\sigma(\Sigma)<\infty$ if and only if $\mathrm{Alg}_{\Q}^1(\Sigma)\subseteq \mathrm{CDiv}_{\Q}(\Sigma)$ if and only if $H^1(Y,\mathcal{O}_Y)=0$.
     \item The value of $\sigma(\Sigma)$ is given by \[ \sigma(\Sigma) = \begin{cases} \infty & H^1(Y,\mathcal{O}_Y)\neq 0\\ h^{1,1}_{\Q}(Y) & \dim Y \geq 2 \, \textrm{ and }\, H^1(Y,\mathcal{O}_Y)=0\\ 0 & \dim Y=1 \, \textrm{ and }\, H^1(Y,\mathcal{O}_Y)=0\, (\textrm{i.e. $Y \cong \P^1$})  \end{cases}.\]
     \item The following are equivalent:
     \begin{itemize}
     \item[(a)] $\Sigma$ is locally analytically $\Q$-factorial,
         \item[(b)] $\Sigma$ is ${\Q}$-factorial,
         \item[(c)] the restriction of the cycle class map to Cartier classes $${\rm cl}_{\Q}^{\rm res}:{\rm CDiv}_{\Q}(\Sigma)\to {\rm IH}^2(\Sigma,\Q)\cap {\rm IH}^{1,1}(\Sigma)$$ is surjective,
         \item[(d)] $(Y,L)\cong({\P}^1,\mathcal{O}_{\P^1}(d))$ with $d\geq 4$.
     \end{itemize}
     Moreover, $\Sigma$ is never factorial.
     \item In addition, assume $\dim Y\geq 2$ and  $L$ satisfies $(U_2)$-property (\definitionref{def-pos}). Then $\sigma^{\rm an}(\Sigma;y)<\infty$ for any (equivalently for all) $y\in Y$ if and only if $H^1(Y,\cO_Y)=0$, in which case, for $y\in Y$ we have
    \[\sigma^{\rm an}(\Sigma;y)\leq h^2(Y)=2h^2(Y,\cO_Y)+h^{1,1}(Y)
    \]
     with equality if $H^2(Y,\cO_Y)=0$.
\end{enumerate}
\end{intro-corollary}

%\textcolor{red}{The above mentioned positivity property $(U_p)$ is explained in Section \ref{sec-Prelim}.}

\textcolor{black}{Combining \corollaryref{cor-chrh}, \corollaryref{cor-GenLevel2Secants} and \corollaryref{corFactorial}, one immediately obtains various equivalent characterizations of secant varieties of rational normal curves which we accumulate in \corollaryref{p1} below.}

%\textcolor{red}{Note that if $L$ satisfies $(Q_0)$-property, $\Sigma\neq\P^N$ and $\dim Y=1$, then $$\sigma^{\rm an}(\Sigma)<\infty\iff \sigma^{\rm an}(\Sigma)=0\iff Y\cong\P^1.$$}

\begin{rmk}
The fact that the secant varieties of rational normal curves of degree $\geq 4$ are $\Q$-factorial is also observed in \cite{KPW}*{Introduction}, the above proves the converse.
\end{rmk}

To exemplify the scope of our results, we explicitly describe the cohomology and singularity invariants of the secant varieties of $\P^2$ in \exampleref{secp2}.

\textcolor{black}{Several of the results above are based on the fact that the variations $\mathbf V^j$ and $\mathbf V^j_{\rm prim}$ have trivial monodromy, which we show in \lemmaref{lem-trivialMonodromy} below.} Let us now highlight another important consequence of this result. 

Any compact complex variety $X$ comes equipped with the {\it intersection cohomology Hirzebruch class} ${{\rm IT}_y}_*(X)$ and the {\it Goresky-MacPherson $L$ class} $L_*(X)$, details of which we will recall in Subsection \ref{secCS}. It has been conjectured in \cite{BSY} that a characteristic class version of the Hodge index theorem should hold true, which amounts to asserting that $${{\rm IT}_1}_*(X)=L_*(X).$$ In {\it loc. cit.}, this conjecture is attributed to Cappell-Shaneson, and we prove the following: 

\begin{intro-theorem}[$=$ \theoremref{cpcn}]\label{cpc} Secant varieties satisfy the Cappell-Shaneson Conjecture if $L$ is 3-very ample.
\end{intro-theorem}

The Cappell-Shaneson conjecture was shown to be true for any $\Q$-homology manifold in \cite{dBPS}*{Thm. 2}. Secant varieties, if not a $\Q$-homology manifold, are quite far from being so (by \corollaryref{cor-chrh} above). In \cite{AMS}*{Ex. 1.8}, Aligholi-Maxim-Sch\"urmann observe that the Cappell-Shaneson Conjecture also holds for higher secant varieties of curves (among many other examples), to which we turn now.

\medskip 

\noindent\textbf{Higher secant varieties of curves.} Here our assumption is as follows: let $C\hookrightarrow\P^N$ be the embedding of a smooth projective curve given by a $(2k-1)$-very ample line bundle $L$ for $k\geq 2$, which merely means that any $0$-dimensional subscheme of length $2k$ imposes independent conditions on the sections of $L$. We let $\sigma_k$ denote the $k$-th secant variety. As before, this is the weakest positivity assumption under which $\sigma_k$ is not defective, and moreover we have an explicit resolution of singularities of $\sigma_k$ given by Bertram's {\it Terracini recursion}  \cite{Bertram} that we recall in Subsection \ref{subhs}. We also let $q_k = N - \dim \sigma_k = N - (2k-1)$. 

Our first main result is an analogous description of local cohomology modules:
\begin{intro-theorem}[$=$ \theoremref{thm-ConstantSheafHigherSecantsn}]\label{thm-ConstantSheafHigherSecants} Let $L$ be $(2k-1)$-very ample and assume $\sigma_k\neq\P^N$ for some $k\geq 2$. Then:
\begin{enumerate}
    \item $\Q_{\sigma_k}[2k-1]$ is perverse (equivalently ${\rm lcdef}(\sigma_k)=0$).
    \item We have isomorphisms
\[ {\rm Gr}^W_{2k-1-\ell} \Q^H_{\sigma_k}[2k-1] = {\rm Sym}^{\ell}(H^1(C)) \boxtimes {\rm IC}_{\sigma_{k-\ell}}\]
and dually
\[ {\rm Gr}^W_{N+q_k+\ell} \cH^{q_k}_{\sigma_k}(\cO_{\P^N}) \cong {\rm Sym}^{\ell}(H^1(C)) \boxtimes {\rm IC}_{\sigma_{k-\ell}}(-q_k-\ell).\]
\end{enumerate}
\end{intro-theorem}

\begin{rmk}
Under the assumption of the above theorem, we see that if $\sigma_k$ has Du Bois singularities, then it is Cohen-Macaulay. This is proven in \cite{ENP}*{Thm. 1.2(1)} under a positivity assumption that was shown to imply the Du Bois condition. (They also prove {\it arithmetic} Cohen-Macaulayness.)
\end{rmk}

Using the description of local cohomology derived in \theoremref{thm-ConstantSheafHigherSecants}, we can bound the generation level for the Hodge filtration on $\cH^{q_k}_{\sigma_k}(\cO_{\P^N})$. 

\begin{intro-corollary}[$=$ \corollaryref{glhsn}]\label{glhs} Assume $L$ is $(2k-1)$-very ample and $\sigma_k\neq\P^N$ where $k\geq 2$. Then we have
\[ {\rm gl}(\cH^{q_k}_{\sigma_k}(\cO_{\P^N},F)) \leq k-1,\]
and hence, for any log resolution $f\colon (\widetilde{P},E) \to \P^N$ of the pair $(\P^N,\sigma_k)$, we have
\[R^{q_k-1+i} f_* \Omega^{N-i}_{\widetilde{P}}(\log E) = 0\text{ for all } i \geq k.\]
\end{intro-corollary}

The last statement in the above corollary is a consequence of \cite{MPLocCoh}*{Thm. 10.2}.

We are also able to compute the Hodge-Lyubeznik numbers for the higher secant varieties $\sigma_k$ of a curve $C$. As $\Q_{\sigma_k}[2k-1]$ is always perverse by \theoremref{thm-ConstantSheafHigherSecants}, the usual Hodge-Lyubeznik numbers are not interesting. If $C \cong \P^1$, then $\sigma_k$ is a rational homology manifold, and so even the intersection Hodge-Lyubeznik numbers are not interesting in that case. The following computes the intersection Hodge-Lyubeznik numbers when genus of the curve is positive:

\begin{intro-theorem}[$=$ \theoremref{thm-HLHigherSecantsn}]\label{thm-HLHigherSecants} Assume $C$ is a smooth projective curve of genus $g > 0$ and $L$ is a $(2k-1)$-very ample line bundle such that $\sigma_k\neq\P^N$ where $k\geq 2$. Let $x \in U_a = \sigma_a \setminus \sigma_{a-1}$ for some $a\leq k-1$. Then the intersection Hodge-Lyubeznik number ${\rm I}\lambda^r_{p,q}(\cO_{\sigma_k,x})$ is only possibly non-zero for $r \in [k+a-1,2k-1]$, in which case we have
\[ {\rm I}\lambda^r_{p,q}(\cO_{\sigma_k,x}) = \begin{cases} \binom{g}{-p} \binom{g}{-q} & -p-q = 2k-1-r, p \in [-(2k-1-r),0] \\ 0 & \text{otherwise}\end{cases}.\]
\end{intro-theorem}

We use the above results to deduce the following: 

\begin{intro-corollary}[$=$ \corollaryref{charp1n}]\label{charp1} Let $L$ be $(2k-1)$-very ample and assume $\sigma_k\neq\P^N$ for some $k\geq 2$. Then
the following are equivalent:
\begin{itemize} 
\item[(1)] $C \cong \P^1$,
\item[(2)] $\sigma_\ell$ is a rational homology manifold for some (equivalently for all) $2 \leq \ell \leq k$.
\end{itemize}
\textcolor{black}{In addition, assume $\sigma_k$ is normal. Then any of the above is equivalent to any of the following:}
\begin{itemize}
\item[(3)] $\sigma_\ell$ has finite $\Q$-factoriality defect for some (equivalently for all) $2\leq \ell \leq k$,
\item[(4)] $\sigma_\ell$ is $\Q$-factorial for some (equivalently for all) $2\leq \ell \leq k$,
\item[(5)] \textcolor{black}{$\sigma_\ell$ has finite local analytic $\Q$-factoriality defect at every point $y\in\sigma_l$ for some (equivalently for all) $2\leq \ell \leq k$,}
\item[(6)] \textcolor{black}{$\sigma_\ell$ is locally analytically $\Q$-factorial for some (equivalently for all) $2\leq \ell \leq k$.}
\end{itemize}
\end{intro-corollary}

We use theorem \theoremref{thm-ConstantSheafHigherSecants} and the weight spectral sequence to compute the singular cohomology of the higher secant varieties of curves.

\begin{intro-theorem}[$=$ \theoremref{thm-SingCohn}]\label{thm-SingCoh} Assume $L$ is $(2k-1)$-very ample where $k\geq 2$ and $\sigma_k\neq\P^N$. We have the following for $0\leq j\leq 4k-2$:
\begin{itemize}
    \item If $w\leq k$, then 
    \[
    {\rm Gr}_w^WH_{\cM}^j(\sigma_k)=\begin{cases}
        A^{\cM}(-\frac{w}{2}) & j=w\textrm{ is even}\\
        {\rm Sym}^k(H^1_{\cM}(C)) & (w,j)=(k,2k-1)\\
        0 & \textrm{otherwise}
    \end{cases}.
    \]
    \item If $w\geq k+1$, then
    \[
    {\rm Gr}_w^WH_{\cM}^j(\sigma_k)=\begin{cases}
    A^{\cM}(-\frac{w}{2}) & j=w\textrm{ is even},\, w\notin [3k-2,4k-3]\\
        A^{\cM}(-\frac{w}{2})\oplus{\bf S}_{1^{4k-w-2}}(2k-j-1) & j=w\textrm{ is even},\, 3k-2\leq w\leq 4k-3\\
        {\bf S}_{1^{4k-w-2}}(2k-j-1) & j=w\textrm{ is odd},\, 3k-2\leq w\leq 4k-3\\
        {\bf S}_{j-w+1,1^{2w+4k-3j-3}}(2k-j-1) & w\in [\frac{3}{2}j-2k+\frac{3}{2},2j-3k+2],\,j> w\\
        0 & \textrm{otherwise}
    \end{cases}.
    \]
\end{itemize}
\end{intro-theorem}

In the above, $\mathbf S_{\lambda}$ denotes the Schur functor applied to $H^1(C)$ associated to the partition $\lambda$. We also provide an alternate proof of the above when $k=3$, see Section \ref{sec-HigherSec} for details.

\color{black}

The following graph shows which weight pieces are non-zero in which cohomological degrees. A box around a point means that there is overlap with the cohomology of another secant variety (indicated by the color of the box). We omit the boxes for the trivial overlap of $A^{\cM}$ factors in the top weight of even cohomology.

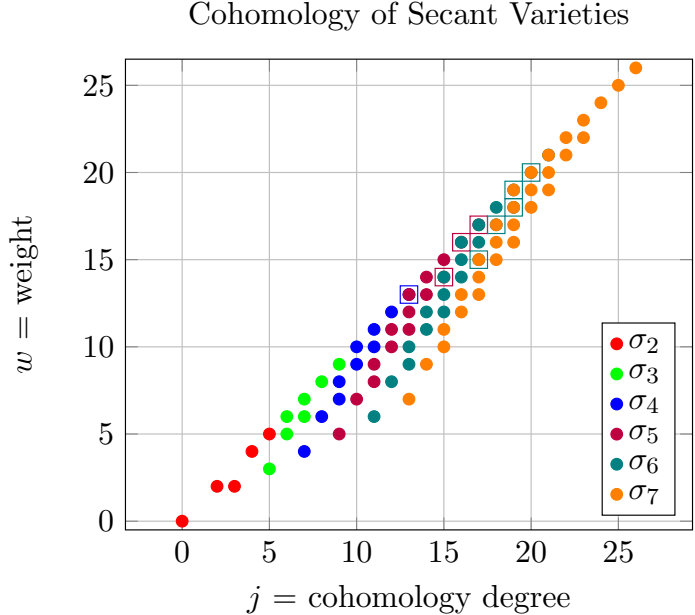
\begin{center}
\begin{tikzpicture}[scale=1.1] 
\begin{axis}[
    axis equal,
    title={Cohomology of Secant Varieties},
    xmin=-0.5, xmax=26.5,
    ymin=-0.5, ymax=26.5,
    xlabel={$j$ = cohomology degree},
    ylabel={$w$ = weight},
    grid = both,
    legend pos=south east 
]

\addplot[
	color = red,
    only marks,
    %mark=square*,
] coordinates {
    (0,0)
    (2,2)
    (3,2)
    (4,4)
    (5,5)
    
};\addlegendentry{$\sigma_2$}

\addplot[
	color = green,
    only marks,
    %mark=square*,
] coordinates {
    (5,3) 
    (6,5)
    (7,6)
    (7,7)
    (8,8)
    (9,9)
    (6,6)
};\addlegendentry{$\sigma_3$}

\addplot[
	color = blue,
    only marks,
    %mark=square*,
] coordinates {
    (7,4) 
    (8,6)
    (9,7)
    (9,8)
    (10,9)
    (11,10)
    (10,10)
    (11,11)
    (12,12)
    (13,13)
};\addlegendentry{$\sigma_4$}

\addplot[
	color = purple,
    only marks,
    %mark=square*,
] coordinates {
    (9,5)
    (10,7)
    (11,8)
    (11,9)
    (12,10)
    (12,11)
    (13,11)
    (13,12)
    (14,13)
    (15,14)
    (13,13)%overlap k=4, green
    (14,14)
    (15,15)
    (16,16)
    (17,17)
};\addlegendentry{$\sigma_5$}

\addplot[
	color = teal,
    only marks,
    %mark=square*,
] coordinates {
   (11,6)
   (12,8)
   (13,9)
   (13,10)
   (14,11)
   (14,12)
   (15,12)
   (15,13)
   (15,14) %only overlap with k =5, purple
   (16,14)
   (16,15)
   (17,15)
   (17,16)
   (18,17)
   (19,18)
   (16,16)%overlap k=5, purple
   (17,17)%overlap k=5, purple
   (18,18)
   (19,19)
   (20,20)
   (21,21)
};\addlegendentry{$\sigma_6$}

\addplot[
	color = orange,
    only marks,
    %mark=square*,
] coordinates {
   (13,7)
   (14,9)
   (15,10)
   (15,11)
   (16,12)
   (16,13)
   (17,13)
   (17,14)
   (17,15) %overlap k=6, teal
   (18,15)
   (18,16)
   (18,17) %overlap k=6, teal
   (19,16)
   (19,17)
   (19,18)%overlap k=6, teal
   (20,18)
   (20,19)
   (21,19)
   (21,20)
   (22,21)
   (23,22)
   (19,19)%overlap k=6, teal
   (20,20)%overlap k=6, teal
   (21,21)
   (22,22)
   (23,23)
   (24,24)
   (25,25)
   (26,26)
};\addlegendentry{$\sigma_7$}
\node[
    draw,
    rectangle,
    inner sep=3pt,     % controls how much bigger than the marker
    fill=none,         % no fill
    color = blue,
] at (axis cs:13,13) {};

\node[
    draw,
    rectangle,
    inner sep=3pt,     % controls how much bigger than the marker
    fill=none,         % no fill
    color = purple,
] at (axis cs:15,14) {};

\node[
    draw,
    rectangle,
    inner sep=3pt,     % controls how much bigger than the marker
    fill=none,         % no fill
    color = purple,
] at (axis cs:16,16) {};

\node[
    draw,
    rectangle,
    inner sep=3pt,     % controls how much bigger than the marker
    fill=none,         % no fill
    color = purple,
] at (axis cs:17,17) {};

\node[
    draw,
    rectangle,
    inner sep=3pt,     % controls how much bigger than the marker
    fill=none,         % no fill
    color = teal,
] at (axis cs:17,15) {};

\node[
    draw,
    rectangle,
    inner sep=3pt,     % controls how much bigger than the marker
    fill=none,         % no fill
    color = teal,
] at (axis cs:18,17) {};

\node[
    draw,
    rectangle,
    inner sep=3pt,     % controls how much bigger than the marker
    fill=none,         % no fill
    color = teal,
] at (axis cs:19,18) {};

\node[
    draw,
    rectangle,
    inner sep=3pt,     % controls how much bigger than the marker
    fill=none,         % no fill
    color = teal,
] at (axis cs:19,19) {};

\node[
    draw,
    rectangle,
    inner sep=3pt,     % controls how much bigger than the marker
    fill=none,         % no fill
    color = teal,
] at (axis cs:20,20) {};
\end{axis}

\end{tikzpicture}
\end{center}

\medskip

\noindent\textbf{Outline.} In Section \ref{sec-Prelim}, we review the important results and definitions from Saito's theory of mixed Hodge modules and mixed sheaves. Here we also review the main results and definitions of \cite{DOR} and \cite{CDOIsolated}.
In the following Section \ref{sec-GeneralResult}, we prove a general result and study the consequences for $c(X), {\rm HRH}(X)$, and Hodge Lyubeznik numbers.
Section \ref{sec-2Sec} applies this general result to secant varieties of lines for a smooth projective variety $Y$ that is suitably positively embedded into projective space. We prove \theoremref{thm-Secants} -- \theoremref{cpc} in this section.
%Using the general theorem, we can completely compute the local cohomology modules of the secant variety in terms of primitive cohomology of the blow-up ${\rm Bl}_y(Y)$. This recovers some results of \cite{OR} and \cite{Brogan}.
Finally, Section \ref{sec-HigherSec} concerns 
%the constant Hodge module for 
higher secant varieties of curves. We prove \theoremref{thm-ConstantSheafHigherSecants} -- \theoremref{thm-SingCoh} in this section. 
%Using an inductive strategy, we are able to give an explicit description of the weight filtration in terms of $H^1(C)$ and the ${\rm IC}$-modules on the smaller secant varieties. We use this and the weight spectral sequence to conjecture a formula for the singular cohomology of higher secant varieties of curves. We can verify this conjecture partially, which leads to the computation of the $\Q$-factoriality defect.
%We finish by discussing some open problems arising from this work in Section \ref{open}.

\medskip

\noindent{\bf Acknowledgments.} We thank Bhargav Bhatt, Daniel Brogan,
Hyunsuk Kim, Lauren\c{t}iu Maxim,  Mircea \Mustata, Jinhyung Park, Sung Gi Park, Mihnea Popa, Christian Schnell, Rosie Shen, Lei Song, Sridhar Venkatesh, Duc Vo and Jakub Witaszek for many conversations related to this work.

\section{Preliminaries} \label{sec-Prelim}
In this section, we review the relevant aspects of the theory of mixed Hodge modules, mixed sheaves, and the geometry of secant varieties. We will use, without review, the basic theory of algebraic $\cD$-modules (see \cite{HTT}).

\subsection{Mixed sheaves}\label{secmixed}
Morihiko Saito's theory of mixed Hodge modules \cites{SaitoMHP,SaitoMHM} gives a theory of weights, analogous to those on $\ell$-adic constructible sheaves (for varieties in characteristic $p$ using eigenvalues of the Frobenius operator \cites{DeligneWeilII,BBDG}), to certain constructible complexes on complex algebraic varieties. For example, any admissible $\Q$-variation of mixed Hodge structure defines an object in Saito's theory.

For any complex algebraic variety $X$, there is the abelian category  ${\rm MHM}(X)$ of mixed Hodge modules on $X$. The derived categories of these abelian categories satisfy a six functor formalism: we have functors $\mathbf D,\otimes$ and for any $f\colon X_1 \to X_2$ we have $f_*,f_!,f^*,f^!$ satisfying the familiar compatibility properties (like adjunction and base-change). There is a forgetful functor to constructible $\Q$-complexes, and these functors agree with the corresponding functors at the $\Q$-constructible level. By the six functors, we see that any object of geometric origin is given a mixed Hodge module structure via Saito's theory.

Soon after, Saito defined the general notion of a ``theory of mixed sheaves'' \cites{SaitoMixedSheaves,SaitoArithmetic}, an axiomatization of some of the desired properties for Beilinson's conjectural theory of \emph{mixed motivic sheaves} \cite{BeilinsonHeight}. Let $k$ be a field of characteristic $0$ which can be embedded into $\C$ (we fix some such embedding $k \hookrightarrow \C$) and let $A$ be a subfield of $\R$. Let $\cV$ denote the category of all $k$-varieties. A theory of $A$-mixed sheaves on $\cV$ is the association to each $X\in \cV$ an $A$-linear abelian category $\cM(X)$ with forgetful functors ${\rm For} \colon \cM(X) \to {\rm Perv}(X_{\C}^{\rm an},A)$, where $X_{\C} = X\times_{{\rm Spec}(k)} {\rm Spec}(\C)$ and $(-)^{\rm an}$ denotes the associated complex analytic space, with the following properties:
\begin{enumerate} \item The functor ${\rm For}$ is faithful and exact,

\item ${\rm For}(M)$ is $k$-constructible and quasi-unipotent,

\item Each $M \in \cM(X)$ admits a finite increasing \emph{weight filtration} $W_\bullet M$ with the property that every morphism in $\cM(X)$ is strictly compatible with respect to $W_\bullet$,

\item The associated graded pieces ${\rm Gr}^W_i M$ are semi-simple for all $i\in \Z$,
\end{enumerate}
and which satisfies a few basic compatibilities of a six functor formalism \cite{SaitoMixedSheaves}*{(1.2)-(1.6)}. The main content of \emph{loc. cit.} is to show that those compatibilities and the six functor formalism on perverse sheaves suffice to show that any theory of $A$-mixed sheaves satisfies the full six functor formalism, along with the expected behavior of weights in such a formalism.

We let $\cM(X)^{\rm go}$ denote the full subcategory of objects of \emph{geometric origin}: the smallest full-subcategory stable by the cohomological functors $\cH^i f_*(-), \cH^i f_!(-), \cH^ig^*(-), \cH^i g^!(-)$ for $f\colon Z \to X$ and $g\colon X \to Y$ morphisms of $k$-varieties, stable by sub-quotients (including isomorphisms), and such that $A^{\cM} \in \cM({\rm Spec}(k))^{\rm go}$. For any theory of $A$-mixed sheaves $\cM(-)$, we get a sub-theory of $A$-mixed sheaves $\cM(-)^{\rm go}$.

\begin{eg} \label{eg-SysReal} Saito provides many interesting examples \cite{SaitoMixedSheaves}*{Examples 1.8} beyond that of ${\rm MHM}(X,A)$. The example of primary interest to us is \cite{SaitoMixedSheaves}*{Ex. 1.8(iii)}: assume $k$ is a number field and let $\overline{k} \subseteq \C$ be the algebraic closure of $k$ in $\C$. Let $\overline{X} = X \times_{{\rm Spec}(k)} {\rm Spec}(\overline{k})$, and consider for any prime $\ell$ the category of $G$-equivariant $\Q_\ell$-perverse \'{e}tale sheaves ${\rm Perv}_G(\overline{X},\Q_\ell)$, where $G ={\rm Gal}(\overline{k}/k)$. This admits the canonical comparison functor
\[ {\rm Perv}_G(\overline{X},\Q_{\ell}) \to {\rm Perv}(X_{\C},\Q_{\ell}).\]

We have the theory of $\Q$-mixed sheaves given by the fiber product of ${\rm MHM}(X_{\C},\Q)$ with ${\rm Perv}_G(\overline{X},\Q_{\ell})$ over ${\rm Perv}(X_{\C},\Q_{\ell})$. The $G$-equivariant $\Q_{\ell}$-perverse sheaves should also admit finite filtrations $W_\bullet$ which are compatible with the $W_\bullet$ on the mixed Hodge module under these comparison maps.

One could also consider \cite{SaitoMixedSheaves}*{Ex. 1.8(v)} the theory obtained by varying all possible embeddings $k \hookrightarrow \C$, varying the embeddings of algebraic closures $\overline{k} \hookrightarrow \C$ and varying $\ell$. Finally, one could consider in any of these the sub-theory defined by objects of geometric origin, and this is the main example to keep in mind.
\end{eg}

The forgetful functors induce those on derived categories: 
\[ D^b(\cM(X)) \to D^b_{\rm cons}(X,A).\] They are no longer faithful, but they are \emph{conservative}: an object $M^\bullet \in D^b(\cM(X))$ is $0$ if and only if ${\rm For}(M^\bullet) = 0$. Moreover, we have
\[ {\rm For} \circ \cH^j = {}^p \cH^j \circ {\rm For} \text{ for all } j\in \Z.\]

We define the \emph{support} of an object $M^\bullet \in D^b(\cM(X))$ to be the support of ${\rm For}(M^\bullet)$. By assumption of $k$-constructibility, ${\rm Supp}(M^\bullet)$ is a $k$-subvariety of $X$.

We say that an object $M$ is \emph{pure of weight $w$} if ${\rm Gr}^W_i M \neq 0$ implies $i = w$, and so we see that pure objects are semi-simple. We let $\cM(X,w)$ and $\cM(X,w)^{\rm go}$ denote the full subcategories of objects that are pure of weight $w$.

\begin{eg} We have Tate twist functors in any theory of $A$-mixed sheaves as follows: let $A(-1) = \cH^2\kappa_* A^\cM_{\P^1}$, where $\kappa \colon \P^1_k \to {\rm Spec}(k)$ is the structure morphism. Define $A(1) = \mathbf D_{{\rm Spec(k)}}(A(-1))$, and inductively define $A(n)$ by taking tensor products of $A(1)$ if $n> 0$ or $A(-1)$ if $n< 0$.

Then via the identification $X \times_{{\rm Spec}(k)} {\rm Spec}(k) \cong X$, we define
\[ M(k) = M \boxtimes A^\cM(k) \text{ for any } M \in \cM(X), k\in \Z.\]

If $M$ is pure of weight $w$, then $M(k)$ is pure of weight $w-2k$.
\end{eg}

\begin{eg} On any $X\in \cM(X)$, we have the \emph{constant sheaf} $A^{\cM}_X = \kappa^*(A)$, where $\kappa \colon X \to {\rm Spec}(k)$ is the structure morphism. We have ${\rm For}(A^{\cM}_X) = A_X$.

If $X$ is smooth and equidimensional, then $A^{\cM}_X[\dim X] \in \cM(X)$ is pure of weight $\dim X$, and we have the isomorphism \cite{SaitoMixedSheaves}*{Prop. 4.3}:
\begin{equation} \label{eq-PolarizedConstant} \mathbf D_X(A^\cM_X[\dim X]) \cong A^{\cM}_X[\dim X](\dim X).\end{equation}
\end{eg}

\begin{rmk} As the notation suggests, a reader with a Hodge module inclination could read the arguments and results in this paper by replacing $A$ by $\Q$ and the symbol $\cM$ by $H$.
\end{rmk}

\begin{eg} As in the setting of mixed Hodge modules \cite{SaitoMHM}*{(4.5.9)}, we can identify the top weight of $\cH^{\dim X} A^\cM_X$ with the intersection complex: if $U$ is the maximal smooth open subset of $X$ of pure dimension $\dim X$ with $j\colon U \to X$, then we have
\[ {\rm Gr}^W_{\dim X} \cH^{\dim X} A^\cM_X = j_{!*}(A^\cM_{U}[\dim X]) = {\rm IC}^{\cM}_{X'},\]
where $X' \subseteq X$ is the closure of $U$, alternatively, it is the union of the maximal dimension irreducible components of $X$. Here $j_{!*}(-) = {\rm Im}(\cH^0 j_!(-) \to \cH^0 j_*(-))$ is defined in the standard way through six functors. Hence, the functor $j_{!*}(-)$ commutes with duality, and combining with the isomorphism \eqref{eq-PolarizedConstant} (applied to the smooth, equidimensional variety $U$), we conclude
\begin{equation} \label{eq-PolarizedIC} \mathbf D_X({\rm IC}_{X'}^\cM) \cong {\rm IC}_{X'}^{\cM}(\dim X). \end{equation}
\end{eg}

If $X$ is equidimensional, we consider the morphism $A_X^\cM[\dim X] \to {\rm IC}_X^{\cM}$. By dualizing this morphism (and Tate twisting), we obtain the canonical morphism
\[ \psi_X \colon A^\cM_X[\dim X] \to\mathbf D_X(A^{\cM}_X[\dim X])(-\dim X).\]

We give a name to the target of that morphism:
\[ \mathbf D_X^{\cM} = \mathbf D_X(A^{\cM}_X[\dim X])(-\dim X),\]
which satisfies the following pullback rule: for any morphism $f\colon X \to Y$, we have
\begin{equation} \label{eq-DXRule} f^!(\mathbf D_Y^H) = \mathbf D_X^H(\dim X - \dim Y)[\dim X - \dim Y].\end{equation}

Any theory of $A$-mixed sheaves satisfies the \emph{weight formalism} regarding the six functors. Specifically, we say an object $M^\bullet \in D^b(\cM(X))$ is \emph{mixed of weights $\geq w$} (resp. \emph{mixed of weights $\leq w$}) if ${\rm Gr}^W_{i+j} \cH^i M^\bullet \neq 0$ implies $j\geq w$ (resp. $j\leq w$). Such an object is \emph{pure of weight $w$} if it is mixed of weights $\geq w$ and mixed of weights $\leq w$. Pure objects always decompose (non-canonically) in the derived category: if $M^\bullet$ is pure of weight $w$, then there exists an isomorphism (\cite{SaitoMixedSheaves}*{Prop. 6.8})
\begin{equation} \label{eq-decompPure} M^\bullet \cong \bigoplus_{i\in \Z} \cH^i M^\bullet [-i]\text{ in } D^b(\cM(X)).\end{equation}

\begin{thm}[\cite{SaitoMixedSheaves}*{Thm. 6.7}]\label{thm-WeightFormalism} Let $f\colon X \to Y$ and $g\colon Z \to X$ be morphisms of $k$-varieties.

If $M^\bullet \in D^b(\cM(X))$ is mixed of weights $\geq w$ (resp. mixed of weights $\leq w$) then $f_* M^\bullet, g^! M^\bullet$ (resp. $f_! M^\bullet,g^* M^\bullet$) are, too.

In particular, if $f$ is proper (so that $f_! \cong f_*$) and $M^\bullet$ is pure of weight $w$, then so is $f_! M^\bullet = f_* M^\bullet$.
\end{thm}

\theoremref{thm-WeightFormalism} implies that, in any theory of $A$-mixed sheaves, an analogue of the Decomposition Theorem holds:
\begin{prop}[\cite{SaitoMixedSheaves}*{Prop. 6.9}] \label{thm-decompDirectImage} Let $f\colon X \to Y$ be a proper morphism of $k$-varieties and let $M^\bullet \in D^b(\cM(X))$ be pure of weight $w$. Then
\[ f_* M^\bullet \cong \bigoplus_{i\in \Z} \cH^i f_*M^\bullet[-i].\]
\end{prop}

In fact, if we restrict attention to objects of geometric origin (sufficient in our setting) and projective morphisms, then the relative Hard Lefschetz theorem holds.

\begin{thm}[\cite{SaitoMixedSheaves}*{Prop. 7.4}] \label{thm-HardLefschetz} Let $f\colon X \to Y$ be a projective morphism of $k$-varieties and let $\ell \in H^2(X_\C,A)$ be the first Chern class of an $f$-relatively ample line bundle. Then for $M \in \cM(X,w)^{\rm go}$ and $j>0$, we have an isomorphism
\[ \ell^j \colon \cH^{-j}f_* M \cong \cH^j f_* M(j) \text{ in } \cM(Y,w-j)\]
which, by applying ${\rm For}(-)$, agrees with the corresponding isomorphism of $A$-perverse sheaves.
\end{thm}

Formally, this implies that Lefschetz decompositions hold for projective pushforward of pure objects of geometric origin. In the notation of \theoremref{thm-HardLefschetz}, we define for $j\geq 0$:
\[ (\cH^{-j}f_*M)_{\rm prim} = \ker(\ell^{j+1} \colon \cH^{-j}f_* M \to \cH^{j+2} f_* M).\]

\begin{cor} Keeping the set-up of \theoremref{thm-HardLefschetz}, we have \emph{Lefschetz decompositions} for all $j\geq 0$:
\[ \cH^{-j} f_* M \cong \bigoplus_{a \geq 0} (\cH^{-j-2a} f_* M)_{\rm prim}(-a)\]
\[ \cH^{j} f_* M \cong \bigoplus_{a \geq 0} (\cH^{-j-2a} f_* M)_{\rm prim}(-j-a).\]
\end{cor}

Finally, let $X$ be an equidimensional $k$-variety with $\kappa \colon X \to {\rm pt}$ the structure morphism. We define the $\cM$-cohomology of $X$ by
\[ H^j_{\cM}(X) = \cH^j \kappa_* A^\cM_X, \, H^j_{\cM,c}(X) = \cH^j \kappa_! A^\cM_X,\]
and the intersection cohomology of $X$ as
\[ {\rm IH}^{j+\dim X}_{\cM}(X) = \cH^j \kappa_* {\rm IC}_X^{\cM}, \, {\rm IH}^{j+\dim X}_{\cM,c}(X) = \cH^j \kappa_! {\rm IC}_X^{\cM}.\]

The latter satisfies Poincar\'{e} duality by the isomorphism \eqref{eq-PolarizedIC}:
\[ {\rm IH}^{\dim X+j}_{\cM}(X)^{\vee} \cong {\rm IH}^{\dim X - j}_{\cM,c}(X)(-j).\]

If $X$ is projective, then the intersection cohomology also satisfies Hard Lefschetz: let $\ell \in H^2(X,A)$ be the first Chern class of an ample line bundle. Then for all $j>0$, we have isomorphisms
\[ \ell^j \colon {\rm IH}^{\dim X -j}_{\cM}(X) \cong {\rm IH}^{\dim X +j}_{\cM}(X)(j).\]

Moreover, define
\[ {\rm IH}^{\dim X-j}_{\cM,\rm prim}(X) = \ker(\ell^{j+1} \colon {\rm IH}^{\dim X-j}_{\cM}(X) \to {\rm IH}^{\dim X+ j +2}_{\cM}(X)),\]
then we have the Lefschetz decompositions for all $j \geq 0$:
\[ {\rm IH}^{\dim X -j}_{\cM}(X) = \bigoplus_{a\geq 0} {\rm IH}^{\dim X -j-2a}_{\cM}(X)(-a),\]\[ {\rm IH}^{\dim X + j}_{\cM}(X) = \bigoplus_{a\geq 0} {\rm IH}^{\dim X -j-2a}_{\cM}(X)(-j-a).\]

\subsection{Hodge filtration and Hodge theoretic singularity invariants}\label{secsi} For the application to singularities, we will take $k = \C, A = \Q$ and $\cM(X) = {\rm MHM}(X,\Q)$. We remind the reader about some important results in this setting.

On a smooth variety $Z$, a mixed Hodge module consists of the data $((\cM,F,W),(\cK,W),\alpha)$ where $(\cM,F,W)$ is a bi-filtered left $\cD_Z$-module ($F_\bullet \cM$ is a good ``Hodge'' filtration compatible with the order filtration of differential operators and $W_\bullet \cM$ a finite ``weight'' filtration by sub-$\cD_Z$-modules), $(\cK,W)$ is a finite filtered $\Q$-perverse sheaf on $Z^{\rm an}$ and 
\[\alpha \colon {\rm DR}_Z(\cM,W) \to (\cK,W)\otimes_{\Q} \C\]
is a filtered isomorphism of $\C$-perverse sheaves. These data are subject to many conditions which we do not explain here.

\begin{eg} Recall that, on a smooth variety $Z$, the \emph{de Rham complex} of a $\cD_Z$-module is defined as
\[ {\rm DR}_Z(\cM) = [ \cM \xrightarrow[]{\nabla} \Omega_Z^1 \otimes_{\cO} \cM \xrightarrow[]{\nabla} \dots \xrightarrow[]{\nabla} \omega_Z \otimes_{\cO} \cM],\]
in cohomological degrees $-\dim Z,\dots, 0$, where $\nabla$ is the connection defined by the $\cD_Z$-module structure. If $(\cM,F)$ is filtered, then we get a filtration
\[ F_p {\rm DR}_Z(\cM) = [F_p\cM \xrightarrow[]{\nabla} \Omega_Z^1\otimes_{\cO} F_{p+1}\cM \xrightarrow[]{\nabla} \dots \xrightarrow[]{\nabla} \omega_Z \otimes F_{p+\dim Z} \cM],\]
where, by the Leibniz rule, the differentials are not $\cO_Z$-linear. However, if we look at the associated graded complex ${\rm Gr}^F_p {\rm DR}_Z(\cM)$, then the differentials are $\cO_Z$-linear, and each term is $\cO_Z$-coherent, which means we have
\[ {\rm Gr}^F_p {\rm DR}_Z(\cM) \in D^b_{\rm coh}(\cO_Z).\]
\end{eg}

Using local embeddings in smooth varieties, we can define the category of mixed Hodge modules on an arbitrary complex variety $X$. For any $M \in {\rm MHM}(X)$, there are well-defined objects ${\rm Gr}^F_{p} {\rm DR}(M) \in D^b_{\rm coh}(\cO_X)$.

\begin{lem}[\cite{SaitoMHP}*{Section 2.3}] \label{lem-GrDRProperPushforward} Let $f\colon X \to Y$ be a proper morphism of complex algebraic varieties. Then for any $M \in D^b({\rm MHM}(X))$ and $p\in \Z$ there is a natural isomorphism
\[ Rf_* {\rm Gr}^F_p {\rm DR}_X(M) \cong {\rm Gr}^F_p {\rm DR}_Y(f_* M).\]
\end{lem}

\begin{rmk} Pure Hodge modules on algebraic varieties are, by definition, polarizable: if $M$ is pure of weight $w$, then there is an isomorphism
\[ \mathbf D(M) \cong M(w).\]

We saw behavior like this in isomorphism \eqref{eq-PolarizedConstant} and isomorphism \eqref{eq-PolarizedIC} above.
\end{rmk}

As mentioned above, admissible variations of mixed Hodge structures give examples of mixed Hodge modules on a smooth variety $Z$. In particular, any pure polarizable variation of Hodge structure (VHS) gives a pure Hodge module, though the weight is shifted: if $\mathbf V$ is a pure polarizable VHS of weight $w$, then, viewed as a Hodge module, it is a pure Hodge module of weight $w+\dim Z$.

Also, we mentioned above that the (derived) categories of mixed Hodge modules satisfy a six functor formalism. We will make use of the dual, pushforward and pullback functors.

\begin{eg} \label{eg-restrictVHS} We explain the pullback functors in an example where they are rather easy to understand. We will also use this result several times below.

Let $\mathbf V$ be a pure polarizable variation of Hodge structures of weight $w$ on a smooth variety $Z$ (determining a pure polarizable Hodge module of weight $w+\dim Z$, denoted by the same symbol).

If $x\in Z$ is a point with closed embedding $\iota_x \colon \{x\} \to Z$, then
\[ \iota_x^* \mathbf V \in D^b({\rm MHS})\]
actually only has one non-vanishing cohomology: $\cH^{-\dim Z} \iota_x^* \mathbf V \cong \mathbf V_x$, the fiber of the variation of Hodge structures at $x$.

To compute $\iota_x^! \mathbf V$, we write
\[ \iota_x^! \mathbf V =\mathbf D \iota_x^* \mathbf D(\mathbf V),\]
which by polarizability is isomorphic to
\[ \mathbf D(\iota_x^* \mathbf V (w+\dim Z)) = \mathbf D(\mathbf V_x[\dim Z])(-w-\dim Z),\]
and finally, dualizing once more (using polarizability of the Hodge structure $\mathbf V_x$), we get
\[ \iota_x^! \mathbf V \cong \mathbf V_x(-\dim Z)[-\dim Z].\]
\end{eg}

We will make use of the following standard result later. The shift by $+\dim Z$ in the index of the Hodge filtration comes from the indexing of the associated graded de Rham complex, which is due to our use of filtered left $\cD_Z$-modules.

\begin{lem} \label{lem-LowestHodge} Let $(\cM,F)$ be a filtered $\cD_Z$-module on a smooth variety $Z$. Then $F_{k+\dim Z}\cM = 0$ if and only if ${\rm Gr}^F_p {\rm DR}_Z(\cM) = 0$ for all $p\leq k$.
\end{lem}

\begin{eg} \label{eg-grDRComputation} As an example of how to compute ${\rm Gr}^F_p {\rm DR}(\cM)$, we take once more a variation of (possibly mixed) Hodge structures $\mathbf V$ on a smooth variety $Y$. We assume moreover that $\mathbf V$ has trivial monodromy, so that
\[ \mathbf V \cong H \boxtimes \Q_Y^H[\dim Y]\]
for some mixed Hodge structure $H$. This example is relevant in view of \lemmaref{lem-trivialMonodromy} below.

Essentially by definition, we have an isomorphism
\[ {\rm Gr}^F_i {\rm DR}(\mathbf V) \cong \bigoplus_{a+b = i} {\rm Gr}^F_a H \boxtimes {\rm Gr}^F_b {\rm DR}(\Q_Y^H[\dim Y]) \cong \bigoplus_b {\rm Gr}^F_{i-b} H \boxtimes \Omega_Y^{-b}[\dim Y+b],\]
and so, if we take $\cH^{-\ell}(-)$, we get
\begin{equation} \label{eq-grDRTrivialMonodromy} \cH^{-\ell} {\rm Gr}^F_b {\rm DR}(\Q_Y^H[\dim Y]) \cong {\rm Gr}^F_{i+\dim Y -\ell} H \boxtimes \Omega_Y^{\dim Y-\ell}.\end{equation}

In fact, for any object of the form $H\boxtimes M$ for some mixed Hodge module $M$ on $Y$, we have the formula (\cite{DOR}*{(6.15)})
\begin{equation} \label{eq-GrDRBoxtimes} {\rm Gr}^F_i {\rm DR}(H \boxtimes M) = \bigoplus_{a+b = i} {\rm Gr}^F_a H \boxtimes {\rm Gr}^F_b {\rm DR}(M).\end{equation}
\end{eg}

The object $\mathbf D_X^H$ is closely related to local cohomology, as the following standard lemma shows.

\begin{lem}[\cite{CDOIsolated}*{Lem. 2.3}]\label{lem-DXLocCoh} Let $i\colon X \to Y$ be a closed embedding of pure codimension $c$, where $Y$ is smooth. Then there is a natural isomorphism of bi-filtered $\cD_Y$-modules underlying mixed Hodge modules
\[ i_* \cH^{j} \mathbf D_X^H (-c) \cong (\cH^{c+j}_X(\cO_Y),F,W),\]
where on the right, the object is the local cohomology bi-filtered $\cD_Y$-module of the structure sheaf $\cO_Y$ along $X$.
\end{lem}

The Hodge filtration on a filtered (left) $\cD_Y$-module $(\cM,F)$ underlying a mixed Hodge module on a smooth variety $Y$ is \emph{good} with respect to the order filtration of differential operators, which implies that there exists some $k$ such that
\[ F_{k+j} \cM = F_j \cD_Y \cdot F_k \cM \text{ for all } j \geq 0.\]

\begin{defi} We define the \emph{generation level} of $(\cM,F)$ to be
\[ {\rm gl}(\cM,F) = \min\{k \mid F_{k+j} \cM = F_j \cD_Y \cdot F_k \cM \text{ for all } j\geq 0\},\]
which differs from the notion in \cite{SaitoOnTheHodge} by a shift of $p(\cM) = \min\{p \mid F_p \cM \neq 0\}$.
\end{defi}

The generation level is a measure of the complexity of the Hodge filtration on the module $(\cM,F)$. Giving bounds on the generation level allows for one to argue vanishing in certain situations, see \cite{MPLocCoh}*{Cor. 11.6}.

We have the following standard formula for the generation level:
\begin{lem}[\cite{MPLocCoh}*{Lem. 10.1}] \label{lem-preciseGL} Let $(\cM,F)$ be a filtered left $\cD_Y$-module. Then
\[ {\rm gl}(\cM,F) = \max \{ k \mid \cH^0{\rm Gr}^F_{k-\dim Y} {\rm DR}(\cM) \neq 0\}.\]
\end{lem}

The decomposition theorem gives the following easy bound on the generation level of ${\rm IC}_X$ in terms of a resolution of singularities. 

\begin{lem} \label{lem-genLevelIC} Let $i \colon X \to Y$ be a closed embedding of pure codimension $c$ with $Y$ smooth. Let $f\colon \widetilde{X} \to X$ be a resolution of singularities. Then
\[ \cH^0 {\rm Gr}^F_{q-\dim Y} {\rm DR}(i_* {\rm IC}_X) = 0 \text{ for all } q > \max\{\ell \mid R^\ell f_* \Omega^{\dim X -\ell}_{\widetilde{X}} \neq 0\}+c.\]

In particular, if $\eta = \max\{\dim f^{-1}(p) \mid p \in X\}$, then the vanishing holds for all $q > \eta+c$, and we get ${\rm gl}(i_* {\rm IC}_X(-c),F) \leq \eta$.
\end{lem}
\begin{proof} By the decomposition theorem, ${\rm IC}_X^H$ is a direct summand of $f_* \Q_{\widetilde{X}}^H[\dim X]$ in the derived category. Thus, the claimed vanishing follows from the same vanishing for $f_* \Q_{\widetilde{X}}^H[\dim X]$. We have by \lemmaref{lem-GrDRProperPushforward} the isomorphism
\[ \cH^0 {\rm Gr}^F_{q-\dim Y} {\rm DR}f_* \Q_{\widetilde{X}}^H[\dim X] = \cH^0 Rf_* {\rm Gr}^F_{q-\dim Y} {\rm DR}(\Q_{\widetilde{X}}^H[\dim X])\]
\[ = R^{q-c} f_* \Omega_{\widetilde{X}}^{\dim X - (q-c)},\]
which proves the first claim.

By projectiveness of $f$ and coherence of $\Omega_{\widetilde{X}}^{\dim X - (q-c)}$, we see by \cite{HartshorneBook}*{Cor. III.11.2} that this vanishing is automatic for $q-c > \eta$. The last claim is then an immediate result of this vanishing.
\end{proof}

We remind the reader of two singularity invariants associated to an equidimensional variety $Z$, defined in \cites{DOR,CDOIsolated}. These are Hodge theoretic weakenings of the properties that $\Q_Z[\dim Z]$ is a (semi-simple) perverse sheaf. Recall that $Z$ is called a \emph{rational homology manifold} (RHM) if the natural morphism
\[ \Q_Z[\dim Z] \to {\rm IC}_Z\]
is a quasi-isomorphism. In other words, if $K_Z^\bullet \in D^b_c(\Q_Z)$ is the shifted cone of that morphism (called the \emph{RHM defect object} in \cite{PPLefschetz} whose terminology we adopt), giving the exact triangle
\[ K_Z^\bullet \to \Q_Z[\dim Z] \to {\rm IC}_Z \xrightarrow[]{+1},\]
then $Z$ is an RHM if and only if $K_Z^\bullet = 0$. We denote also by $K_Z^\bullet$ the cone in the derived category of mixed Hodge modules $D^b({\rm MHM}(Z))$. By definition, $\Q_Z[\dim Z]$ is a perverse sheaf if and only if ${}^p \tau^{<0} \Q_Z[\dim Z] = 0$, in which case we say $Z$ is \emph{cohomologically complete intersection} (CCI).

Rather than requiring vanishing of these objects, the theory of Hodge modules allows us to consider varieties which only satisfy a partial vanishing.

\begin{defi} Let $Z$ be an equidimensional variety. Then let
\[c(Z) = \sup\{k \mid {\rm Gr}^F_{-p}{\rm DR}(\tau^{< 0} \Q_Z^H[\dim Z]) = 0\textrm{ for all }p\leq k\},\]
and
\[ {\rm HRH}(Z) = \sup\{k \mid {\rm Gr}^F_{-p}{\rm DR}(\cK_Z^\bullet) = 0\textrm{ for all } p\leq k\}.\]
\end{defi}

As $\tau^{<0} \Q_Z^H[\dim Z] = \tau^{<0} \cK_Z^\bullet$, it is not hard to see that ${\rm HRH}(Z) \leq c(Z)$. The condition ${\rm HRH}(Z) \geq k$ is equivalent to the condition $D_k$ of \cite{PPLefschetz}.

The partial vanishing is satisfied in many examples (\cite{DOR}*{Sec. E}, \cite{CDOIsolated}) and implies desirable properties, which we now summarize.

\begin{prop} Let $Z$ be an equidimensional variety. Then
\begin{enumerate}
    \item ${\rm HRH}(Z) \geq k$ if and only if for all $x\in Z$, the local cohomology at $x$ satisfies $F^{\dim Z -k} H^\bullet_{\{x\}}(Z) = 0$ for $\bullet \leq 2\dim Z$ and $\C$ for $\bullet = 2\dim Z$.
    \item If ${\rm HRH}(Z) \geq k$, then for all $j\in \Z$ the natural map
    \[F^{\dim Z -k} H^{\dim Z-j}(Z) \to F^k(H^{\dim Z+j}_c(Z)^{\vee})\]
    is an isomorphism.
    \item $c(Z) \geq k$ if and only if ${\rm depth} \, \underline{\Omega}_Z^p \geq \dim Z -p$ for all $p \leq k$.
\end{enumerate}
\end{prop}

Taking $k = +\infty$ in the above, we obtain Brion's characterization of rational homology manifolds \cite{Brion}, the fact that rational homology manifolds satisfy Poincar\'{e} duality, and \Mustata-Popa's result \cite{MPLocCoh}*{Cor. 12.6} for CCI varieties, respectively.

\begin{rmk} If $Z$ has LCI singularities, then it has $k$-rational singularities if and only if it has $k$-Du Bois singularities and ${\rm HRH}(Z) \geq k$. Of course, in this setting, $c(Z) = + \infty$.
\end{rmk}
\color{black}

\begin{defi}\label{defw}
    Let $i\colon Z \to Y$ be a closed embedding of pure codimension $c$, where $Y$ is smooth. We define the invariant $$w(Z):=\max\left\{k\geq 0\mid F_pW_{\dim Y+c}\cH^c_Z(\cO_Y)=F_p\cH^c_Z(\cO_Y)\textrm{ for all }p\leq k\right\}$$ with the convention that $w(Z)=-1$ if the above set is empty.
\end{defi}

By \cite{DOR}*{Thm. B} we have equality ${\rm HRH}(Z) =\min\left\{c(Z),w(Z)\right\}$.

\begin{rmk}
    The invariant $w$ is independent of the choice of a closed embedding. Indeed, by \lemmaref{lem-DXLocCoh}, we see that 
    \[ w(Z) = \sup\{k \mid {\rm Gr}^F_{-p}{\rm DR}(\cH^0\cK_Z^\bullet) = 0\textrm{ for all } p\leq k\}.\]
\end{rmk}

We say that a normal variety $Z$ is \textit{weakly rational} if the sheaf $\omega^{\rm GR}_Z$ is reflexive, or equivalently, it coincides with $\omega_Z = \cH^{-\dim{Z}} \omega^{\bullet}_Z$ where $\omega^{\bullet}_Z$ is the dualizing complex. Here $\omega^{\rm GR}_Z$ is the Grauert-Riemenschneider sheaf which is defined as $f_*\omega_{\Tilde{Z}}$ where $f:\Tilde{Z}\to Z$ is a resolution. The next result characterizes the property $\omega_Z^{GR}=\omega_Z$ through filtrations on local cohomology. Recall that the ``Ext'' filtration $E_\bullet \cH^q_Z(\cO_X)$ is defined using the Ext description:
\[ \cH^q_Z(\cO_X)= \varinjlim_{p}\mathcal E xt^q(\cO_X/\mathcal{J}_Z^{p+1},\cO_X),\]
where $\mathcal{J}_Z$ is the ideal sheaf of $Z$ in $X$, and the filtration is defined by
\[ E_\bullet \cH^q_Z(\cO_X) = {\rm Im}\left[\mathcal E xt^q(\cO_X/\mathcal{J}_Z^{\bullet+1},\cO_X)\to \cH^q_Z(\cO_X)\right].\]

\begin{prop}\label{wc} Let $Z$ be an equidimensional variety, and $i:Z \to Y$ an embedding into a smooth $n$-dimensional variety $Y$. We denote the codimension by $c$. Then, $\omega_Z^{GR}=\omega_Z$ if and only if 
\begin{equation}\label{mpr}
    F_0W_{n+c}\cH^c_Z(\cO_Y) = E_0\cH^c(\cO_Y).
\end{equation} %\cong \mathcal{E}xt^c(\cO_Z, \cO_Y)$.    
In particular, if $Z$ has Du Bois singularities, then $\omega_Z^{GR}=\omega_Z$ if and only if $w(Z)\geq 0$.
\end{prop}

\begin{proof} We have that ${\rm gr}^F_{-n}{\rm DR} (W_{n+c}\cH^c_Z(\cO_Y)) = F_0W_{n+c}\cH^c_Z(\cO_Y)\otimes \omega_Y \cong \omega^{\rm GR}_Z$ \cite{KS}*{Prop. 8.2}. Moreover, by Grothendieck duality we have that 
\[ \mathcal{RH}om( i_*\cO_Z, \omega_Y[n]) \cong \mathcal{RH}om(\cO_Z, \omega_Z^{\bullet})\cong \omega^{\bullet}_Z.\] Therefore, $E_0\cH^c(\cO_Y) \otimes \omega_Y = \mathcal{E}xt^c(\cO_Z, \omega_Y) \cong \cH^{-\dim{Z}}(\omega^{\bullet}_Z) = \omega_Z$ (see \cite{MPLocCoh}*{Rmk. 7.9} for the first equality). This means that %$Z$ is weakly rational if and only if 
$\omega^{\rm GR}_Z = \omega_Z$ if and only if $F_0W_{n+c}\cH^c_Z(\cO_Y) = E_0\cH^c(\cO_Y)$.

The last assertion follows immediately by \cite{MPLocCoh}*{Thm. C}.
\end{proof}

As mentioned in the introduction, $c(Z)$ and ${\rm HRH}(Z)$ can also be computed through {\it Hodge-Lyubeznik numbers} and {\it intersection Hodge-Lyubeznik numbers} whose definitions we recall. Given $z\in Z$ and $p,q,r,s\in\Z$, the former is defined as 
\[
\lambda_{r,s}^{p,q}(\cO_{Z,z}):=\dim_{\C}{\rm Gr}_{-p}^F{\rm Gr}_{p+q}^W\cH_z^r(\cH^{\dim Z-s}{\bf D}_Z^H(\dim Z))
\]
where $\cH^r_z(-):=\cH^ri_z^!(-)$ and $i_z:\left\{z\right\}\to Z$ is the inclusion. The latter is defined as
\[
{\rm I}\lambda_{r}^{p,q}(\cO_{Z,z}):=\dim_{\C}{\rm Gr}_{-p}^F{\rm Gr}_{p+q}^W\cH_z^r{\rm IC}_Z(\dim Z).
\]
We refer to \cites{HodgeLyubeznik,CDOIsolated} for details about these notions. In fact, via duality we can also compute these invariants using the constant sheaf as
\[
\lambda_{r,s}^{p,q}(\cO_{Z,z}):=\dim_{\C}{\rm Gr}_{p}^F{\rm Gr}_{-p-q}^W\cH_z^{-r}(\cH^{s-\dim Z}\Q^H_Z[\dim Z])
\]
and
\[
{\rm I}\lambda_{r}^{p,q}(\cO_{Z,z}):=\dim_{\C}{\rm Gr}_{p}^F{\rm Gr}_{-p-q}^W\cH_z^{-r}{\rm IC}_Z.
\]

These numbers also characterize $w(Z)$. In fact, the proof of the following is immediate from \cite{CDOIsolated}*{Pf. of Thm. B}:

\begin{thm}
    Let $Z$ be an equidimensional variety. Then 
    \[w(Z)=\sup\left\{k\mid \lambda^{p,q}_{r,\dim Z}(\cO_{Z,z})={\rm I}\lambda^{p,q}_r(\cO_{Z,z})\textrm{ for all }z\in Z,\, q,r\in\Z,\,p\geq -k\right\}\] with the convention $w(Z)=-1$ if the above set is empty.
\end{thm}
\color{black}

\subsection{Conjectural characteristic class version of Hodge index theorem}\label{secCS} Let $X$ be a compact complex pure-dimensional algebraic variety. We have the following diagram: 
\begin{equation}\label{maxim}
\begin{tikzcd}
    K_0({\rm Var}/X)\arrow[r,"\chi_{{\rm Hdg}}"]\arrow[dr,"{\rm sd}_*"] & K_0({\rm MHM}(X))\arrow[d,"{\rm Pol}"]\arrow[r,"{{\rm MHT}_y}_*"] & H_*(X)\otimes\Q[y^{\pm 1}]\arrow[d,"y=1"]\\
     & \Omega_{\R}(X)\arrow[r,"L_*"] & H_*(X)\otimes\Q
\end{tikzcd}
\end{equation}
%\arrow[bend left=20, "{T_y}_*"]{rr}
where $H_*(X)$ is the even degree homology $H_{2*}(X;\Z)$. In the above diagram, $K_0({\rm Var}/X)$ is the relative Grothendieck group of algebraic varieties over $X$, meaning the quotient of the free abelian group of isomorphism classes of morphisms $[Y\to X]$ by the ``scissor relation'' given by $$[Y\to X]=[Z\hookrightarrow Y\to X]+[Y\setminus Z\hookrightarrow Y\to X]$$ for closed subvarieties $Z\subset Y$. The group $K_0({\rm Var}/X)$ is generated by the classes $[Y\to X]$ where $Y$ is smooth, pure dimensional and proper over $X$. The map $\chi_{{\rm Hdg}}$ is defined by
$$\chi_{{\rm Hdg}}([f:Y\to X])=[f_!\Q_Y^H].$$
Similarly $K_0({\rm MHM}(X))$ is the Grothendieck group of mixed Hodge modules on $X$. We refer to \cite{AMS} for details about the other objects and the maps involved. Here we only recall the two key definitions:
\begin{itemize}
\item The {\it intersection cohomology
Hirzebruch class} of $X$ is defined by $${{\rm IT}_y}_*(X):={{\rm MHT}_y}_*({\rm IC}_X^H[-\dim X])\in H_*(X)\otimes\Q[y^{\pm 1}].$$
\item The {\it Goresky-MacPherson $L$ class} of $X$ is given by $$L_*(X):=L_*([{\rm IC}_X])=(L_*\circ{\rm Pol})({\rm IC}_X^H[-\dim X])\in H_*(X)\otimes\Q.$$
\end{itemize}
Now, the outer square of the diagram \eqref{maxim} commutes by \cite{BSY}, and the left triangle of the same commutes by \cite{FPS}. The commutativity of the inner square is expected to fail, but conjecturally the two ways of mapping ${\rm IC}_X^H[-\dim X]\in K_0({\rm MHM}(X))$ should produce the same outcome:
\begin{conj}[Characteristic class version of Hodge index theorem]\label{conjcs}
    Let $X$ be a compact
pure-dimensional complex algebraic variety. Then $L_*(X)={{\rm IT}_1}_*(X)$.
\end{conj}
The above conjecture is attributed to Cappell and Shaneson in \cite{BSY}*{Rmk. 5.4}. We note the following:
\begin{rmk}
    It is clear from the commutativity of the outer square and left triangle of diagram \eqref{maxim} that the above conjecture holds for $X$ if ${\rm IC}_X^H[-\dim X]$ lies in the image of $\chi_{{\rm Hdg}}$. 
\end{rmk}

We will use the above remark to prove Conjecture \ref{conjcs} for secant varieties.

\color{black}

\subsection{Geometry of secant varieties and higher secant varieties of curves} Until otherwise stated (i.e. until Subsection \ref{rmk-SecantVarietyOverk}), in this subsection we assume $k = \C$. 

Let $Y$ be a smooth projective variety and let $L$ be a very ample line bundle on it. We will make use of the following positivity notion:
\begin{defi} For an integer $k\geq 1$, a line bundle $L$ on a smooth projective variety $Y$ is \emph{$k$-very ample} if the evaluation map
\[ H^0(Y,L) \to H^0(Y,L\otimes \cO_{\xi})\]
is surjective for any $0$-dimensional subscheme $\xi \subseteq Y$ of length $k+1$.
\end{defi}

For $k=1$, this recovers the usual notion of very ampleness.

\subsubsection{Secant variety of lines}\label{prelimsecant} As $Y$ is smooth, the Hilbert scheme of length $2$ subschemes ${\rm Hilb}^2(Y)$ is smooth. We can consider the universal length two subscheme:
\[ \Phi = \{(\xi,y) \mid y \in {\rm supp}(\xi)\} \subseteq {\rm Hilb}^2(Y) \times Y,\]
through which we can define the \emph{secant bundle for $L$} on ${\rm Hilb}^2(Y)$ by the formula
\[ E_L = p_{1*}(p^*(L))\]
where $p_1 \colon \Phi \to {\rm Hilb}^2(Y)$ and $p \colon \Phi \to Y$ are the natural projections.

\begin{rmk} The variety $\Phi$ can be identified with the blow-up $$b: {\rm Bl}_{\Delta}(Y\times Y)\to Y\times Y$$ along the diagonal $\Delta \subseteq Y \times Y$. With this identification, the morphism $p_1 \colon \Phi \to {\rm Hilb}^2(Y)$ is the quotient by a natural $\mathfrak S_2$-action (compatibly with the identification of ${\rm Hilb}^2(Y)$ with the blow-up of ${\rm Sym}^2(Y)$ along the diagonal). Moreover, the map $p \colon \Phi \to Y$ factors as
\[  \pi_2 \circ b\colon \Phi \to Y\times Y \to Y.\]
\end{rmk}

As $L$ is very ample, $E_L$ is globally generated, and so the evaluation map induces a morphism
\[ \P = \P(E_L) \to \P(H^0(Y,L))\]
whose image is denoted $\Sigma$ and is called the \emph{secant variety of lines} for $Y$ associated to the embedding determined by $L$. From now on, we will always assume that $L$ is 3-very ample and $\Sigma \neq \P(H^0(Y,L))$ (it is well-known that if $L$ is 3-very ample, then $\Sigma=\P(H^0(L))$ if and only if $(X,L)\cong(\P^1,\cO_{\P^1}(3))$). Under this 3-very ampleness assumption on $L$, it turns out that $\Sigma$ is irreducible of dimension $2\dim Y+1$ with $\Sigma_{\rm sing}=Y$ (see e.g. \cite{ORS}*{Proposition 2.2}). 

 We have the following diagram \cite{Ver} (see \cites{Ullery,OR,ORS})
\begin{equation}
    \begin{tikzcd} F_y:={\rm Bl}_yY\ar[r]\ar[d] &\Phi \ar[r] \ar[d,"p"] & \P \ar[d,"t"]\ar[dr, "\beta"] & \\ \left\{y\right\}\ar[r] & Y \ar[r, "i"]\arrow[rr, bend right=20, swap, "\iota"] & \Sigma \ar[r,"i_{\Sigma}"] & \P^N\end{tikzcd}
\end{equation}
with Cartesian squares and surjective vertical maps (while the horizontal maps are embeddings), where $\Phi$ and $\P$ are smooth, the morphism $p$ is smooth, and the fiber $F_y:=p^{-1}(y)$ over $y\in Y$ is naturally identified with the blow-up ${\rm Bl}_y(Y)$ of $Y$ at $y\in Y$ for any $y$. In fact, $t$ can be identified with the blow-up of $\Sigma$ along $Y$ with exceptional divisor $\Phi$. In particular, $t$ is an isomorphism on $\Sigma\setminus Y$, or in other words, it is a strong log-resolution of $\Sigma$. 

Let $\ell = \cO(1)$ denote the relatively ample line bundle coming from viewing $t$ as a blow-up (not to be confused with the tautological bundle $\cO_{\P}(1)$ coming from projective bundle $\P(E_L)$ structure). Note that $\ell \vert_{\Phi} \cong \cO_{\Phi}(-\Phi)$ is the conormal bundle for the embedding $\Phi \to \P$. 
Take $\cO_{\mathbf P}(-\Phi)$ to be the relatively ample bundle on $\P$ over $\Sigma$ (relatively ample because it is identified, up to positive multiple, with $\cO(1)$ in the expression of $\P$ as some blow-up of a coherent ideal sheaf on $\Sigma$). By \cite{Ullery}*{Proof of Lemma 2.3} (see also \cite{OR}*{(2.6)}), we see then that
\[ \cO_{\P}(-\Phi)\vert_{p^{-1}(y)} = \cN^{\vee}_{\Phi/\P}\vert_{p^{-1}(y)} \cong b_y^*(L)(-2E_y)\]
where $E_y$ is the exceptional divisor of the blow-up $b_y\colon {\rm Bl}_y(Y) \to Y$. By relative ampleness, we have just argued the following\footnote{\lemmaref{Seshadri} shows that conditions \cite{OR}*{Def. 1.4(2)} and \cite{ORS}*{Def. 3.1$(Q2)$} are redundant.}:
\begin{lem}\label{Seshadri} Assume $L$ is $3$-very ample on a smooth projective variety $Y$. Then $b^*_y(L)(-2E_y)$ is an ample line bundle on ${\rm Bl}_y(Y)$ for all $y\in Y$. (In other words, the Seshadri constant of a $3$-very ample bundle satisfies
$\epsilon(Y,L,y) \geq 2$.)
\end{lem}

We define a stronger positivity property of $L$ that will be useful for some of our results:

\begin{defi}\label{def-pos}
Let $L$ be a 3-very ample line bundle on $Y$. Given $y\in Y$, let $\mathcal{J}_y$ be the ideal sheaf of $y\in Y$.
\begin{itemize}
    \item Given $0\leq p\leq \dim Y-1$, $L$ is said to satisfy the {\it $(Q'_p)$-property} if for all $y\in Y$, we have:
%\begin{itemize}
    %\item The natural map ${\rm Sym}^i(H^0(L\otimes\mathcal{J}_y^2))\to H^0(L^{\otimes i}\otimes\mathcal{J}_y^{2i})$ is surjective for all $i\geq 1$, and 
    %\item 
    $$H^{\dim Y-k}(\Omega_{F_y}^k\otimes (b_y^*(jL)(-2jE_y)))=0\textrm{ for all }j\geq 1, 0\leq k\leq p.$$
    \item Given $p\geq 1$, $L$ is said to satisfy the {\it $(U_p)$-property} if for all $y\in Y$, we have:
    $$H^{k}(b_y^*(jL)(-2jE_y))=0\textrm{ for all }j\geq 1, 1\leq k\leq p.$$
\end{itemize}
%\end{itemize}

\end{defi}

Note that the $(Q'_p)$-property (resp. the $(U_p)$-property) is significantly weaker than the $(Q_p)$-property (resp. the $(Q_0)$-property) used in \cites{ChouSong, OR,ORS,Ullery} through which the singularities of $\Sigma$ had been investigated in {\it loc. cit.} for various $p$.

\color{black}
We need the following vanishing result:
\begin{prop}\label{glv}
    If $L$ satisfies $(Q'_p)$-property for some $0\leq p\leq \dim Y-1$, then 
    \[R^{\dim Y-k}t_*\Omega_{\P}^{\dim Y+k+1}(\log\Phi)(-\Phi)=0\textrm{ for all }0\leq k\leq p.\]
\end{prop}
\begin{proof}
    We follow \cite{ORS}*{Pf. of Prop. 4.3}. According to \cite{ORS}*{Claim 4.5, 4.7}, it is enough to check that 
    $$H^{\dim Y-k}(\Omega_{\Phi}^q|_{F_y}\otimes (b_y^*(jL)(-2jE_y))=0\textrm{ for all }y\in Y, j\geq 1\textrm{ and }q\in\left\{\dim Y+k,\dim Y+k+1\right\}.$$ As in \cite{ORS}*{Pf. of Claim 4.8} we consider the filtration 
$$F^0=\Omega_{\Phi}^q|_{F_y}\supseteq F^1\supseteq\cdots\supseteq F^q\supseteq F^{q+1}=0\textrm{ with }F^l/F^{l+1}\cong\bigwedge\limits^l\mathcal{N}_{F_y/\Phi}^*\otimes\Omega_{F_y}^{q-l}.$$
Since $\mathcal{N}_{F_y/\Phi}^*\cong\cO_{F_y}^{\oplus\dim Y}$, we have $F^l/F^{l+1}=0\textrm{ for all }l>\dim Y.$ The rest of the proof is identical to that of \cite{ORS}*{Claim 4.8} once we use the above observation and Nakano vanishing (thanks to \lemmaref{Seshadri}).
\end{proof}
\color{black}
\subsubsection{Higher secant varieties of curves}\label{subhs} Let $C$ be a smooth projective curve. We consider the smooth varieties ${\rm Sym}^k C = C^{(k)}$ for all $k\geq 1$. We have the $\mathfrak S_k$-quotient map
\[ C^{\times k} \to C^{(k)},\]
so that, by pull-back, the cohomology satisfies
\[ H^\bullet (C^{(k)}) = H^\bullet(C^{\times k})^{\mathfrak S_k},\]
where the $\mathfrak S_k$ action on cohomology is signed. In simpler terms, we have
\[ H^j (C^{(k)}) = {\rm Gr}^W_j {\rm Sym}^k H^\bullet(C), \]
where ${\rm Sym}^k(-)$ is taken in the category of super vector spaces (with a sign on odd degree terms and no sign on even degree terms). See Subsection \ref{subsec-ClassicalCohomology} below for more details.

We have the \emph{universal divisor} $\cD_k = \{(p+Q,p) \mid Q \in C^{(k-1)}\} \subseteq C^{(k)} \times C$. It can be viewed as an incidence variety, similar to $\Phi$ above. We define the \emph{$k$-th secant bundle} associated to $L$ to be
\[ E_L^{(k)} = p_{1*}(p_2^*(L)),\]
where $p_1 \colon \cD_k \to C^{(k)}$ and $p_2 \colon \cD_k \to C$ are the projections (we suppress the dependence of the projections on $k$ to simplify notation). 

We define the space $B^k = \mathbf P(E_L^{(k)})$, a $\P^{k-1}$-bundle over $C^{(k)}$. Evaluation of sections determines a morphism $B^k \to \P(H^0(C,L))$, and the image of this morphism is the $k$-th secant variety of the curve $C$ (associated to the embedding $C\hookrightarrow \P(H^0(C,L))$).

For any integers $a,b$ we have the addition maps (whose dependence on these integers will be clear from the context)
\[ a \colon C^{(a)}\times C^{(b)} \to C^{(a+b)},\,\, (x,y) \mapsto x+y,\]
and these induce Cartesian diagrams of secant bundles
\[ \begin{tikzcd} C^{(a)} \times B^b \ar[r,"\alpha"] \ar[d] & B^{a+b} \ar[d] \\ C^{(a)} \times C^{(b)} \ar[r,"a"] & C^{(a+b)} \end{tikzcd} \]
satisfying various compatibility properties. For us, the most important are the following two properties \cites{Bertram,Brogan}: the diagrams
\[\begin{tikzcd} C^{(a)} \times B^b \ar[r] \ar[d,"\alpha"] & B^b \ar[d] \\ B^{a+b} \ar[r] & \P(H^0(C,L))\end{tikzcd} \quad \begin{tikzcd} C^{(a)} \times C^{(b)} \times B^c \ar[r,"a\times {\rm id}"] \ar[d,"{\rm id}\times \alpha"] & C^{(a+b)} \times B^c \ar[d,"\alpha"] \\ C^{(a)} \times B^{b+c} \ar[r,"\alpha"] & B^{a+b+c} \end{tikzcd}\]
are commutative.

Following \cite{Brogan}, we define for any $k \geq m \geq 1$ the variety $Z^k_m = \alpha(C^{(k-m)} \times B^m) \subseteq B^k$. Following \cite{Bertram}, these varieties are called \emph{relative secant varieties}. The geometry of secant varieties is incredibly structured. If $L$ is $(2k-1)$-very ample and $\sigma_k \neq \P(H^0(C,L))$, then $\sigma_{\ell-1} = \sigma_{\ell,{\rm sing}}$ for all $\ell \leq k$. Moreover, the map $p \colon B^\ell \to \sigma_\ell$ is a resolution of singularities which is an isomorphism over $U_\ell = \sigma_\ell \setminus \sigma_{\ell-1}$.

For all $1\leq m\leq \ell \leq k$, the defining morphism $C^{(\ell-m)} \times B^m \to Z^\ell_m$ is also a resolution of singularities, which is an isomorphism over the complement of $Z^\ell_{m-1} \subseteq Z^\ell_m$. The open subset lying over (and isomorphic to) the complement $Z^\ell_m \setminus Z^\ell_{m-1}$ is $C^{(\ell-m)} \times U_m$. For any such pair, we have the cartesian diagram
\begin{equation} \label{eq-DiagramInductive} \begin{tikzcd} Z^\ell_m \ar[r] \ar[d] & B^\ell \ar[d] \\ \sigma_m \ar[r] & \sigma_\ell \end{tikzcd}.\end{equation}

So we have stratifications $\sigma_\ell = \bigsqcup_{i \leq \ell} \sigma_i$ and $Z^\ell_m = \bigsqcup_{i\leq m} C^{(\ell-i)} \times U_i$ which stratify the left vertical morphism in the Cartesian diagram \eqref{eq-DiagramInductive}.

\subsubsection{(Higher) Secant varieties over subfields of $\C$} \label{rmk-SecantVarietyOverk} Let $k$ be a field which is embeddable into $\C$. In the definition of a theory of mixed sheaves, we are allowed to work with varieties defined over such a ground field. We remark here what we mean by ``(higher) secant variety'' in this setting.

Let $Y$ be a smooth projective variety over $k$, which is assumed \emph{geometrically irreducible} (equivalently, geometrically connected). The classical definition of secant varieties is not amenable to this situation, for example, if $Y$ has no $k$-points, there is no clear way to define the ``variety of $k$-lines through two $k$-points''. 

The point of view taken here is to define the secant variety through the secant bundle construction. The following compatibility statements for changing base fields are key to this construction: for a $k$-variety $Y$ with a vector bundle $\cE$  on $Y$, we have natural isomorphisms:
\[ \P(\cE \times_k \C) \cong \P(\cE)\times_k \C,\]
\[ {\rm Hilb}^2(Y \times_k \C) \cong {\rm Hilb}^2(Y) \times_k \C,\]
\[ {\rm Sym}^\ell(Y\times_k \C) \cong {\rm Sym}^{\ell}(Y)\times_k \C,\]
and similarly, for any morphism $f\colon X \to Y$ of $k$-varieties, we have
\[ {\rm Im}(f \times_k \C) \cong {\rm Im}(f) \times_k \C.\]

Thus, if we define the secant bundle via pull-back of $L$ (the embedding very ample line bundle) to the universal length two subscheme in $Y \times {\rm Hilb}^2(Y)$ (resp. the universal divisor in $C\times {\rm Sym}^\ell(C)$).

We then pushforward the resulting bundle to ${\rm Hilb}^2(Y)$ (resp. ${\rm Sym}^\ell(C)$), so that we can then take the projective bundle. This gives $\P$ (resp. $B^\ell$), defined over $k$, with a morphism $\P \to \P(H^0(Y,L))$ (resp. $B^\ell \to \P(H^0(C,L))$) defined over $k$.

We define the secant variety of $Y$ (resp. $\ell$th secant variety of $C$) to be the image of this morphism. If we base-change from $k$ to $\C$, the resulting variety is the secant (resp. $\ell$th secant) variety of $Y\times_k \C$ (resp. $C\times_k \C$).

\subsection{Classical cohomology computations} \label{subsec-ClassicalCohomology} We remind the reader of some well-known cohomology computations concerning geometric operations on complex algebraic varieties. We explain their extension to any theory category of mixed sheaves.

\subsubsection{Cohomology of projective bundles} Let $Y$ be a smooth projective variety over $\C$ with a vector bundle $\cE$ of rank $r$. We have the following well-known formula regarding the cohomology of the projective bundle $\P(\cE)$ given by
\[ H^j(\P(\cE)) = \bigoplus_{i=0}^{r-1} H^{j-2i}(Y)\xi^i,\]
where $\xi = c_1(\cO_{\P}(1))$ is the relative hyperplane class. This gives an isomorphism of pure Hodge structures of weight $j$
\[ H^j(\P(\cE)) \cong \bigoplus_{i=0}^{r-1} H^{j-2i}(Y)(-i).\]

In fact, we have a canonical splitting
\[ \pi_* A^{\cM}_{\P(\cE)} = \bigoplus_{i=0}^{r-1} A^{\cM}_Y[-2i](-i)\]
in $D^b(\cM(Y))$. Indeed, we have, for any $0\leq i\leq r-1$ a morphism
\[ A_Y^{\cM}[-2i](-i) \to \pi_* A^{\cM}_{\P(\cE)}[-2i](-i) \xrightarrow[]{\xi^i} \pi_* A^{\cM}_{\P(\cE)},\]
where the first arrow is defined by adjunction for $\pi^*,\pi_*$. Taking the direct sum, we get a morphism
\[ \bigoplus_{i=0}^{r-1} A_Y^{\cM}[-2i](-i) \to \pi_* A^{\cM}_{\P(\cE)},\]
and so it suffices to prove that the cone is zero. This can be done at the level of underlying $A$-complexes. Moreover, it is equivalent to proving that $i_y^*$ applied to the cone is $0$ for all $y\in Y$, but this is clear from the standard formula for the $A$-cohomology of projective space. So we have an isomorphism of pure objects of weight $j$ in $\cM_{\rm cons}({\rm Spec}(k))$:
\[ H^j_{\cM}(\P(\cE)) = \bigoplus_{i=0}^{r-1} H^{j-2i}_{\cM}(Y)(-i).\]

\subsubsection{Cohomology of blow-ups of smooth varieties} Let $Z \subseteq Y$ be a smooth subvariety of codimension $c$ in a smooth projective variety $Y$ and let $\pi \colon {\rm Bl}_Z(Y) \to Y$ denote the blow-up of $Y$ along $Z$. We have
\[ H^j({\rm Bl}_Z(Y)) = H^j(Y) \oplus \left(\bigoplus_{i=1}^{c-1} H^{j-2i}(Z) e^i\right),\]
where $e$ is the class of the exceptional divisor $E_Z$. This gives an isomorphism of pure Hodge structures
\[ H^j({\rm Bl}_Z(Y)) = H^j(Y) \oplus \left(\bigoplus_{i=1}^{c-1} H^{j-2i}(Z)(-i)\right).\]

Using a similar argument to the above, we can write
\[ \pi_* A^{\cM}_{{\rm Bl}_Z(Y)} = A^{\cM}_{Y} \oplus \bigoplus_{i=1}^{c-1} A^{\cM}_Z[-2i](-i),\]
and hence we have the decomposition of pure objects in $\cM_{\rm cons}({\rm Spec}(k))$:
\[ H^j_{\cM}({\rm Bl}_Z(Y)) = H^j_{\cM}(Y) \oplus \left(\bigoplus_{i=1}^{c-1} H^{j-2i}_{\cM}(Z)(-i)\right).\]

 We need to analyze the primitive cohomology of the blow-up when we take $Z$ to be a point:

\begin{lem} \label{lem-BlowupPrim} The cohomology ring of ${\rm Bl}_y(Y)$ satisfies the following properties:
\begin{enumerate} \item For all $k$, we have
\[ H^k ({\rm Bl}_y(Y)) = b_y^*(H^k(Y)) \oplus \begin{cases} \Q e^{\frac{k}{2}} & k \in 2\Z \cap [2,2\dim(Y)-2] \\ 0 & \text{ otherwise}\end{cases},\]
where $e = c_1(E_y)$. 
\item The cup product is determined by the fact that $b_y^*(-)$ commutes with cup-product: $e\cup b_y^*(-) = 0$, and $e^i \cup e^j = \begin{cases} e^{i+j} & i+j < \dim Y \\ (-1)^{\dim Y} & i+j=\dim Y \end{cases}$.

\item Let $L$ be a very ample line bundle on $Y$. We have the equality for $k\leq \dim Y$: \[H^{k}_{\rm prim}({\rm Bl}_y(Y)) = \begin{cases} b_y^*(H^{k}_{\rm prim}(Y)) & k \neq 2\\ b_y^*(H^2_{\rm prim}(Y)) \oplus (\Q \cdot c_1(L) + 2 (\frac{1}{2} L)^{\dim Y} e) & k = 2
\end{cases},\]
where on ${\rm Bl}_y(Y)$ the primitive classes are determined by cupping with $b_y^*(c_1(L)) - 2e$ and on $Y$ they are determined by cupping with $c_1(L)$.
\end{enumerate}\end{lem}
\begin{proof} The first two results are standard (see \cite{GH}*{Ch. 4}.) For the third point, note that the claim is obvious if $k$ is odd.

For simplicity of notation, we replace $k$ with $\dim Y - k$ for the remainder of the argument, and assume $\dim Y - k$ is even.

Write $\dim Y -k = 2m$, and consider a primitive element $b_y^*(\eta) + \alpha e^m$ in $H^{\dim Y - k}({\rm Bl}_y(Y))$ with $\alpha \in \Q$. Then
\begin{equation*}\begin{split} 0 =& (b_y^*(c_1(L)) - 2e)^{k+1} \cup (b_y^*(\eta) + \alpha e^m) \\ =& \begin{cases} b_y^*(c_1(L)^{k+1} \cup \eta) + (-2)^{k+1} \alpha e^{m+k+1} & m+k+1 < \dim Y \\ b_y^*(c_1(L)^{k+1} \cup \eta) + (-2)^{k+1} \alpha e^{m+k+1} & m+k+1 = \dim Y\end{cases}. \end{split}\end{equation*}

Note that $m+k+1 = d$ implies $2m+2k+2 = 2d$. By definition $2m = \dim Y - k$, so this gives $\dim Y + k +2 = 2\dim Y$, so that $k = \dim Y - 2$ in that case.

In the first case, due to the fact that the description of cohomology is a direct sum, we get that $b_y^*(c_1(L)^{k+1} \cup \eta) = 0$ and that $(-2)^{k+1} \alpha = 0$. Thus, $\eta \in H^{\dim Y -k}_{\rm prim}(Y)$ and $\alpha = 0$.

In the second case, we get
\[ b_y^*(c_1(L)^{\dim Y -1} \cup \eta) + (-2)^{\dim Y -1} \alpha (-1)^{\dim Y} = 0,\]
where $\eta \in H^2(Y) = \Q \cdot c_1(L) \oplus H^2_{\rm prim}(Y)$. We can ignore the part lying in the primitive cohomology, as by definition it is annihilated by cupping with $c_1(L)^{\dim Y -1}$. So we can assume $\eta = \beta c_1(L)$ for some $\beta \in \Q$. We get the equality
\[ \beta L^{\dim Y} + (-2)^{\dim Y -1} \alpha (-1)^{\dim Y} = 0,\]
and so $\alpha$ is determined by $\beta$ by the formula $\alpha = 2 (\frac{1}{2} L)^{\dim Y}\beta$. 
\end{proof}

\subsubsection{Cohomology of Hilbert scheme of two points} Let $Y$ be a smooth complex projective variety and consider ${\rm Hilb}^2(Y)$. There are two ways to approach the cohomology computations of this space. As mentioned above, ${\rm Hilb}^2(Y)$ is the quotient of $\Phi = {\rm Bl}_{\Delta}(Y\times Y)$ by the natural $\mathfrak S_2$-action. So we have
\[ H^j({\rm Hilb}^2(Y)) = H^j(\Phi)^{\mathfrak S_2}.\]

 As the diagonal $\Delta$ is $\mathfrak S_2$-invariant, we conclude using the description of the cohomology of ${\rm Bl}_\Delta(Y\times Y)$ above that
\[ H^j({\rm Hilb}^2(Y)) = H^j(\Phi)^{\mathfrak S_2} = H^j(Y\times Y)^{\mathfrak S_2} \oplus \left(\bigoplus_{i=1}^{\dim Y -1} H^{j-2i}(Y) e^i\right)\]
\[ \cong H^j({\rm Sym}^2(Y)) \oplus \left(\bigoplus_{i=1}^{\dim Y -1} H^{j-2i}(Y)(-i)\right)\]
where $e=c_1(E_{\Delta})$ is the class of the exceptional divisor on the blow-up.

We can see this another way (following the computation in \cite{Totaro}): we have the resolution diagram
\[ \begin{tikzcd} \P(T^* Y) \ar[r] \ar[d] & {\rm Hilb}^2(Y) \ar[d] \\ Y \ar[r] & {\rm Sym}^2 (Y)\end{tikzcd}, \]
which leads to the long exact sequence
\[ \dots \to H^j({\rm Sym}^2(Y)) \to H^j(Y) \oplus H^j({\rm Hilb}^2(Y)) \to H^j(\P(T^* Y)) \to \dots.\]

As all four spaces have pure cohomology (using that ${\rm Sym}^2(Y)$ is a projective rational homology manifold), this long exact sequence breaks into short exact sequences. So we get (split, by polarizability) short exact sequences
\[ 0 \to H^j({\rm Sym}^2(Y)) \to H^j(Y) \oplus H^j({\rm Hilb}^2(Y)) \to H^j(\P(T^* Y)) \to 0.\]

By the splitting of this short exact sequence and the computation of the cohomology of $\P(T^*Y)$ above, we conclude
\[ H^j({\rm Hilb}^2(Y)) \cong H^j({\rm Sym}^2(Y)) \oplus \bigoplus_{i = 1}^{\dim Y -1} H^{j-2i}(Y)(-i).\]

Clearly this decomposition also holds in any theory of mixed sheaves.

\subsubsection{Cohomology of some  symmetric products} Let $C$ be a smooth projective curve over $\C$. Recall that the constructions ${\rm Sym}^k C = C^{(k)}$ is key to the resolution of the higher secant varieties of $C$. The cohomology of these spaces was classically understood by MacDonald \cite{MacDonald}. As $C^{(k)}$ is the $\mathfrak S_k$-quotient of $C^{\times k}$, it follows that
\[ H^j(C^{(k)}) = H^j(C^{\times k})^{\mathfrak S_k}\]
where the action of $\mathfrak S_k$ on cohomology is signed. The simplest way to express this invariant subspace is via the language of \emph{super vector spaces}, where we write $$H^\bullet(C) = H^{\bullet}_{\rm even}(C) \oplus H^{\bullet}_{\rm odd}(C).$$ Then we have
\[ H^j(C^{(k)}) = {\rm Gr}^W_j {\rm Sym}^k H^\bullet(C),\]
where the symmetric power is taken in the category of super vector spaces. Specifically, it acts as actual symmetric power on the even part and as the exterior power on the odd part. We conclude that for all $0 \leq j \leq 2k$, we get
\begin{equation} \label{eq-cohSymmPower} H^j(C^{(k)}) \cong \bigoplus_{\max\{0,j-k\} \leq i \leq \lfloor\frac{j}{2}\rfloor} \bigwedge^{j-2i} H^1(C)(-i).\end{equation}

Here we use the fact that $H^\bullet_{\rm even}(C) = \Q^H \oplus \Q^H(-1)$, so its symmetric power agrees with a Tate twist. 

This discussion holds for an arbitrary smooth projective variety $Y$, except the even and odd parts are more complicated in the higher dimensional setting. We explain here the result for ${\rm Sym}^2(Y)$, when $\dim Y \geq 2$. Then $H^\bullet(Y) = H^\bullet_{\rm even}(Y) \oplus H^\bullet_{\rm odd}(Y)$ and
\[ H^j({\rm Sym}^2(Y)) = {\rm Gr}^W_j {\rm Sym}^2 H^\bullet(C),\]
where
\begin{equation} \label{eq-cohSym2} {\rm Sym}^2(H^\bullet(C)) \cong \bigoplus_{i < j} H^i(Y)\otimes H^j(Y) \oplus \bigoplus_{i} {\rm Sym}^2(H^{2i}(Y)) \oplus \bigoplus_i \bigwedge^2 H^{2i+1}(Y).\end{equation}

These decompositions also hold in any theory of mixed sheaves. To make sense of symmetric powers, we have the following:

\begin{rmk} \label{rmk-SchurFunctors} One can in fact perform arbitrary tensor operations (for example, Schur functors) in a theory of mixed sheaves which agrees with the tensor operation applied to the underlying $A$-vector space, following the construction \cite{SaitoMixedSheaves}*{Pf. of Lem. 5.1, (1.5.5)}. Related discussions can be found in \cite{SymmPowerHodge}, and the result is not quite trivial. \end{rmk}

\section{General Results} \label{sec-GeneralResult}
We recall here the set-up and main result of \cite{CDOIsolated}. Again, we work with a theory of $A$-mixed sheaves and consider morphisms of $k$-varieties, where $A$ is a subfield of $\R$ and $k$ is a subfield of $\C$.

Throughout this section, we work with the Cartesian square
\[\begin{tikzcd} \widetilde{Z} \ar[r] \ar[d,"p"] & \widetilde{X}\ar[d,"f"] \\ Z \ar[r,"i"] & X\end{tikzcd}\]
with the following properties:
\begin{enumerate} \item $f$ is projective and $f_{\C}$ isomorphism over $X_\C \setminus Z_\C$,
\item $\widetilde{X}_\C,\widetilde{Z}_\C$ are connected rational homology manifolds,
\item $\ell$ is an $f$-relatively ample line bundle on $\widetilde{X}_\C$.
\end{enumerate}

We fix the following notation for the rest of this section:
\begin{center}
    $n = \dim X = \dim \widetilde{X}, d_Z = \dim Z, d_{\widetilde{Z}} = \dim \widetilde{Z}, c_Z = n -d_Z, c_{\widetilde{Z}} = n-d_{\widetilde{Z}}$.
\end{center} 
The map $p$ has relative dimension $d=c_Z - c_{\widetilde{Z}}$, which we assume is strictly positive, or equivalently, that $c_Z \geq 2$.

As $\widetilde{Z}$ is a rational homology manifold, the direct image $p_* A_{\widetilde{Z}}^{\cM}[d_{\widetilde{Z}}]$ is a pure complex of weight $d_{\widetilde{Z}}$ and satisfies the Lefschetz isomorphisms and Lefschetz decompositions by \theoremref{thm-HardLefschetz} above.

For ease of notation, we define $\delta = c_{\widetilde{Z}}-1$ (the most clear case is for $\delta =0$, when $\widetilde{Z}$ is a divisor in $\widetilde{X}$). The main result of \cite{CDOIsolated} is the following:

\begin{thm}[\cite{CDOIsolated}*{Thm. 3.4}] \label{thm-generalSmooth} In the setting above, we have for all $0 < j < c_Z-1$ isomorphisms of pure objects of weight $n-j-1$:
\[ \cH^{-j}A_X^{\cM}[n] \cong \begin{cases} \bigoplus\limits_{r=0}^{\delta} i_*(\cH^{\delta-j-2r}p_* A_{\widetilde{Z}}^{\cM}[d_{\widetilde{Z}}])_{\rm prim}(-r) & \delta \leq j \\ \bigoplus\limits_{r = 0}^{j} i_*(\cH^{j-\delta-2r}p_* A_{\widetilde{Z}}^{\cM}[d_{\widetilde{Z}}])_{\rm prim}(j-\delta-r) & \delta > j\end{cases}\]
and an isomorphism
\[ W_{n-1} \cH^0 A_X^{\cM}[n] = {\rm Gr}^W_{n-1} \cH^0A_X^{\cM}[n] \cong i_*(\cH^{-\delta}p_* A_{\widetilde{Z}}^{\cM}[d_{\widetilde{Z}}])_{\rm prim}(-\delta).\]  

Dually, we have for all $0 < j < c_Z -1$ isomorphisms of pure objects of weight $n+j+1$:
\[ \cH^{j}\mathbf D_X^{\cM} \cong \begin{cases} \bigoplus\limits_{r=0}^{\delta} i_*(\cH^{\delta-j-2r}p_* A_{\widetilde{Z}}^{\cM}[d_{\widetilde{Z}}])_{\rm prim}(-j-r-1) & \delta \leq j \\ \bigoplus\limits_{r = 0}^{j} i_*(\cH^{j-\delta-2r}p_* A_{\widetilde{Z}}^{\cM}[d_{\widetilde{Z}}])_{\rm prim}(-\delta-r-1) & \delta > j\end{cases}\]
and an isomorphism
\[ \cH^0\mathbf D_X^{\cM}/{\rm IC}_X^{\cM} = {\rm Gr}^W_{n+1} \cH^0\mathbf D_X^{\cM} \cong i_*(\cH^{-\delta}p_* A_{\widetilde{Z}}^{\cM}[d_{\widetilde{Z}}])_{\rm prim}(-\delta-1).\]  
\end{thm}

\subsection{Local cohomological defect and related invariants} For the remainder of this section, we specialize our discussion to the following situation: the map $p \colon \widetilde{Z} \to Z$ is a smooth morphism between smooth varieties of positive dimension. We let $p^{-1}(x)= \cF_x$ denote the fiber over $x$, which is a smooth variety of dimension $d = c_Z - c_{\widetilde{Z}}$.

Under this assumption, $\cH^j p_* A_{\widetilde{Z}}^{\cM}[d_{\widetilde{Z}}]$ is a smooth, pure object in $\cM(Z)$ of weight $d_{\widetilde{Z}} +j$ for all $j\in \Z$, which we denote by $\mathbf V_{\cM}^{d+j}$. Its underlying $A$-perverse sheaf is the (shifted) local system $\mathbf V^{d+j}$ whose fiber over $x$ is the singular cohomology $H^{d+j}(\cF_x,A)$. Below, we drop this $A$ from the notation for (primitive) singular cohomology.

With this notation, for $j\geq 0$, we have that $(\cH^{-j} p_* A_{\widetilde{Z}}^{\cM}[d_{\widetilde{Z}}])_{\rm prim}$ is the subobject of primitive classes $\mathbf V^{d-j}_{\cM,\rm prim}$ whose fiber over $x\in Z$ is
\[ H^{d-j}_{\cM,\rm prim}(\cF_x) =\ker(c_1(\ell\vert_{\cF_x})^{j+1}\colon H^{d-j}_{\cM}(\cF_x) \to H^{d+j+2}_{\cM}(\cF_x)).\]

\begin{cor} \label{cor-generalSit1} Assume as above $p\colon \widetilde{Z} \to Z$ is a smooth (with geometrically connected fibers) morphism between smooth varieties of positive dimension and $\widetilde{X}_{\C}$ is a connected rational homology manifold. Then
\begin{enumerate}
    \item ${\rm lcdef}(X) = {\rm lcdef}_{\rm gen}(X) = {\rm lcdef}_{\rm gen}^{>0}(X) \leq c_Z-2$.
    \item For all $0 < j \leq c_Z -2$, we have an isomorphism of pure objects of weight $n +j+1$ in $\cM(X)$:
    \[ \cH^{-j} \cK_{\cM,X}^\bullet(-j-1) \cong \cH^j \mathbf D_X^{\cM} \cong \begin{cases} \bigoplus\limits_{a=0}^{\delta} i_* \mathbf V_{\cM,\rm prim}^{d -(j-\delta+2a)}(-a-j-1) & j \geq \delta \\ \bigoplus\limits_{a=0}^{j} i_* \mathbf V^{d-(\delta-j+2a)}_{\cM,\rm prim} (-a-\delta-1) & j < \delta\end{cases}.\]
    \item We have ${\rm Gr}^W_i \cH^0 \mathbf D_X^{\cM} \neq 0 \implies i \in [n, n+1]$, an isomorphism
    \[ \cH^0 \cK_{\cM,X}^\bullet(-1) \cong {\rm Gr}^W_{n +1} \cH^0 \mathbf D_X^{\cM} \cong i_* \mathbf V_{\cM,\rm prim}^{d-\delta}(-\delta-1)\] and an isomorphism in $D^b(\cM(X))$ 
\[ f_* A_{\widetilde{X}}^{\cM}[n] \cong {\rm IC}_X^{\cM} \oplus  i_* \mathbf V_{\cM}^{c_Z}\oplus \left(\bigoplus_{\ell=1}^{d-c_{\widetilde{Z}}} \left(i_* \mathbf V_{\cM}^{c_Z+\ell}[-\ell] \oplus i_* \mathbf V_{\cM}^{c_Z+\ell}(\ell)[\ell]\right)\right).\]
    \item \textcolor{black}{If $X$ is not a rational homology manifold, then} the RHM defect object $\cK_{\cM,X}^\bullet$ is a pure complex of weight $n -1$.
\end{enumerate}
\end{cor}
\begin{proof} We prove the claim about the local cohomological defect. It suffices to prove the claim for the underlying $A$-perverse sheaves, so we omit $\cM$ from the notation. Consider the triangle
\[ A_X[n] \to f_* A_{\widetilde{X}}[n] \to i_* S' \xrightarrow[]{+1},\]
where we know the cone is supported on $Z \subseteq X$ by assumption on $f$. If we apply $i^*$ to this, we get
\[ A_Z[n] \to p_* A_{\widetilde{Z}}[d_{\widetilde{Z}}][c_{\widetilde{Z}}] \to S' \xrightarrow[]{+1}.\]

By looking at the long exact sequence in cohomology, we have surjections
\[ \cH^{-j-1} i_* S' \to \cH^{-j}A_X[n]\]
and so it suffices to prove vanishing $\cH^j i_* S'$ for all $j\leq - c_Z$.

In the notation above, we have the identification
\[\cH^j (p_* A_{\widetilde{Z}}[d_{\widetilde{Z}}][c_{\widetilde{Z}}]) = \mathbf V^{c_Z+j},\]
which vanishes for $j < -c_Z$. By the long exact sequence in cohomology, we have
\[ 0 \to \cH^{-c_Z-1} S' \to A_Z[d_Z] \to \mathbf V^{0} \to \cH^{-c_Z} S' \to 0,\]
and the middle map is an isomorphism because we assumed $\cF_{x,\C}$ is connected. Thus, $\cH^{j} S' = 0$ for all $j \leq -c_Z$, as desired.

The identifications of $\cH^j \mathbf D_X^{\cM}$ and ${\rm Gr}^W_{n+1} \cH^0 \mathbf D_X^{\cM}$ follow immediately from \theoremref{thm-generalSmooth}.

We now prove the description of $f_* A_{\widetilde{X}}^{\cM}[n]$, which is a consequence of the decomposition theorem. Indeed, we have
\[ f_* A_{\widetilde{X}}^{\cM}[n] \cong \bigoplus_{j \in \Z} \cH^j f_* A_{\widetilde{X}}^{\cM}[n] [-j].\]

In the above notation, using the fact that $A^{\cM}_X[n] \in D^{\leq 0}(\cM(X))$ we see for $j>0$ that the natural map
\[ \cH^j f_* A_{\widetilde{X}}^\cM[n] \to i_* \cH^j S'\]
is an isomorphism. Also, by the discussion above, the right hand side can be identified with $i_* \cH^{c_{\widetilde{Z}}+j}p_* A_{\widetilde{Z}}^{\cM}[d_{\widetilde{Z}}]$. Thus, we have isomorphisms for $j>0$:
\[ \cH^j f_* A_{\widetilde{X}}^{\cM}[n] \cong i_* \mathbf V_{\cM}^{c_Z+j},\]
and for $j=0$, we have
\[ \cH^0 f_* A_{\widetilde{X}}^{\cM}[n] \cong {\rm IC}_X^{\cM} \oplus i_* \mathbf V_{\cM}^{c_Z}.\]

For the negative cohomology modules, we can use the Lefschetz isomorphisms: for $j>0$, we have isomorphisms
\[ \cH^{-j} f_* A_{\widetilde{X}}^{\cM}[n] \cong i_*\mathbf V_{\cM}^{c_Z+j}(j).\]

It is a simple computation to check that the semi-smallness defect of the map $f$ is equal to $\max\{0,d - c_{\widetilde{Z}}\}$, hence the bound on the direct sum.

The very last claim follows by definition of pure complexes.
\end{proof}

\subsection{Application to singularity invariants, $c(X)$ and ${\rm HRH}(X)$}
We keep the notation as in the previous subsection, in particular, the set-up remains the same and $p \colon \widetilde{Z} \to Z$ is assumed to be a smooth morphism between smooth varieties of positive dimension.

We begin with a simple application to intersection cohomology. 

\begin{thm} \label{thm-IHGeneral1} We have the following isomorphisms for all $j\in \Z$:
\[ H_{\cM}^j(\widetilde{X}) \cong {\rm IH}_{\cM}^j(X) \oplus \left(\bigoplus_{\ell=0}^{d -c_{\widetilde{Z}}} H^{j-n -\ell}_{\cM}(\mathbf V_{\cM}^{c_Z+\ell})\right) \oplus \left(\bigoplus_{\ell=0}^{d -c_{\widetilde{Z}}-1} H^{j-c_Z-c_{\widetilde{Z}}-\ell}_{\cM}(\mathbf V_{\cM}^\ell)(-c_{\widetilde{Z}})\right).\]

\[ H^j_{\cM,c}(\widetilde{X}) \cong {\rm IH}^j_{\cM,c}(X) \oplus \left(\bigoplus_{\ell=0}^{d -c_{\widetilde{Z}}} H^{j-n -\ell}_{\cM,c}(\mathbf V_{\cM}^{c_Z+\ell})\right) \oplus \left(\bigoplus_{\ell=0}^{d -c_{\widetilde{Z}}-1} H^{j-c_Z-c_{\widetilde{Z}}-\ell}_{\cM,c}(\mathbf V_{\cM}^\ell)(-c_{\widetilde{Z}})\right).\]
\end{thm}
\begin{proof} Recall that \corollaryref{cor-generalSit1} gives a decomposition
\[ f_* A_{\widetilde{X}}^{\cM}[\dim \widetilde{X}] \cong {\rm IC}_X^{\cM} \oplus i_* \mathbf V_{\cM}^{c_Z} \oplus \left(\bigoplus_{\ell=1}^{d - c_{\widetilde{Z}}}\left( i_* \mathbf V_{\cM}^{c_Z+\ell}[-\ell] \oplus i_* \mathbf V_{\cM}^{c_Z+\ell}(\ell)[\ell] \right)\right).\]

We would like to rewrite the Tate twisted terms as follows: first, we replace $\ell$ by $d_Z - c_{\widetilde{Z}} - \ell$, which gives
\[ \bigoplus_{\ell=1}^{d - c_{\widetilde{Z}}} i_* \mathbf V_{\cM}^{c_Z+\ell}(\ell)[\ell] = \bigoplus_{\ell=0}^{d - c_{\widetilde{Z}}-1} i_* \mathbf V_{\cM}^{d + (d_Z - \ell)}(d_Z-c_{\widetilde{Z}} - \ell)[d_Z-c_{\widetilde{Z}} -\ell].\] 

We can then use Hard Lefschetz for the fibers, which says
\[ \mathbf V_{\cM}^{d + (d_Z-\ell)} \cong \mathbf V_{\cM}^{d - (d_Z-\ell)} (\ell - d_Z),\]
and so we get
\[ \bigoplus_{\ell=0}^{d - c_{\widetilde{Z}}-1} i_* \mathbf V_{\cM}^{d + (d_Z - \ell)}(d_Z-c_{\widetilde{Z}} - \ell)[d_Z-c_{\widetilde{Z}} -\ell] \cong \bigoplus_{\ell=1}^{d -c_{\widetilde{Z}}-1} i_* \mathbf V_{\cM}^{d -(d_Z-\ell)} (-c_{\widetilde{Z}})[d_Z-c_{\widetilde{Z}} -\ell].\]

Thus, if $a\colon X \to {\rm Spec}(k)$ is the structure map, the conclusion follows by applying $\cH^{j-n}a_*(-)$ (resp. $\cH^{j-n}a_!(-)$, using that $f$ is projective) to both sides of this isomorphism.
\end{proof}

For the remaining singularity invariants, we take $k = \C, A = \Q$ and $\cM(X) = {\rm MHM}(X,\Q)$. Many of the arguments are completely identical to those found in \cite{CDOIsolated}. The proofs we write below will show how to reduce the statement to one which is proven in exactly the same way as in \emph{loc. cit}.

\begin{cor} \label{cor-kCCI} The following hold:
\begin{enumerate}
    \item If $\delta >0$ and $d >1$ or $\delta >1$ and $d =1$, then $c(X) \leq 0$, with equality if and only if for some (hence any) $x\in Z$, we have
\[ F^{d-b}H^{d-b}(\cF_x) = 0 \text{ for } 0 \leq b\leq d-1.\]
\item If $\delta = d =1$, then $c(X) \geq 0$ if and only if $X$ is CCI if and only if for some (hence any) $x\in Z$, we have $$H^1(\cF_x) = 0.$$
\item If $\delta = 0$, then we have $c(X) \geq k$ if and only if for some (hence any) $x\in Z$, we have 
\[ F^{d-b-k} H^{d-b}_{\rm prim}(\cF_x) = 0 \text{ for all } 0 < b < d.\]
\end{enumerate}
\end{cor}
\begin{proof} By definition, $c(X)\geq k$ if and only if
\[ {\rm Gr}^F_{p-n}{\rm DR}_X(\cH^j \mathbf D_X^H) = 0\]
 for all $p\leq k$ and $j > 0$.

We can rewrite these conditions using the result of \theoremref{thm-generalSmooth}: we have $c(X) \geq k$ if and only if for all $p \leq k$, we have vanishing
\[ {\rm Gr}^F_{p+a+j+1-n} {\rm DR}\left(\bigoplus_{a=0}^{\delta} i_* \mathbf V_{\rm prim}^{d-(j-\delta+2a)}\right) = 0 \text{ for all } \delta \leq j \leq c_Z-2,\]
\[ {\rm Gr}^F_{p+a+\delta+1-n} {\rm DR}\left(\bigoplus_{a=0}^j i_* \mathbf V_{\rm prim}^{d-(\delta-j+2a)}\right) = 0 \text{ for all } 0 < j < \delta.\]

Using the fact that $i_*$ commutes with ${\rm DR}$, we can rewrite this as requiring that, for all $p\leq k$, we have the vanishing
\[ \bigoplus_{a=0}^{\delta}{\rm Gr}^F_{p+a+j+1-n} {\rm DR}_Z(\mathbf V_{\rm prim}^{d-(j-\delta+2a)}) = 0 \text{ for all } \delta \leq j \leq c_Z-2,\]
\[ \bigoplus_{a=0}^j{\rm Gr}^F_{p+a+\delta+1-n} {\rm DR}_Z(  \mathbf V_{\rm prim}^{d-(\delta-j+2a)}) = 0 \text{ for all } 0 < j < \delta.\]

Now, as each direct summand is a Hodge module on a smooth variety, these conditions are further equivalent (using \lemmaref{lem-LowestHodge} and the fact that $d_Z -n = -c_Z$) to
\[ F_{k+a+j+1-c_Z} \mathbf V_{\rm prim}^{d-(j-\delta+2a)} = 0 \text{ for all } 0 \leq a\leq \delta \leq j \leq c_Z-2,\]
\[ F_{k+a+\delta+1-c_Z} \mathbf V_{\rm prim}^{d-(\delta-j+2a)} = 0 \text{ for all } 0 \leq a\leq j <\delta,\]
where we only consider $0 < j \leq c_Z -2$.

Now, as this is a variation of Hodge structures, the vanishing is equivalent to the same vanishing of any given fiber. By rewriting to follow the convention of using decreasing Hodge filtrations for Hodge structures, we see that $c(X) \geq k$ is equivalent to the vanishing for some (hence any) $x\in Z$:
\[ F^{c_Z -k-a-j-1} H_{\rm prim}^{d-(j-\delta+2a)}(\cF_x) = 0 \text{ for all } 0 \leq a\leq \delta \leq j \leq c_Z-2,\]
\[ F^{c_Z-k-a-\delta-1} H^{d-(\delta-j+2a)}_{\rm prim}(\cF_x) = 0 \text{ for all } 0 \leq a\leq j <\delta,\]
where again we only consider $0 < j \leq c_Z -2$. At this point, the proof is exactly identical to that of \cite{CDOIsolated}*{Cor. 3.7}, the only difference being that $c_Z$ replaces $n$.
\end{proof}

\begin{cor} \label{cor-HRH} In the above setting, the following hold:
\begin{enumerate}
    \item If $\delta > 0$, then $c(X) \geq 0$ is equivalent to ${\rm HRH}(X) = 0$.
    \item If $\delta = 0$, then we have ${\rm HRH}(X) \geq k$ if and only if $c(X) \geq k$ and, moreover, for some (hence any) $x\in Z$, we have
\[ F^{d-k} H^{d}_{\rm prim}(\cF_x) = 0.\]
\end{enumerate}
\end{cor}
\begin{proof} The variety $X$ satisfies ${\rm HRH}(X) \geq k$ if $c(X) \geq k$ and if, moreover,
\[ {\rm Gr}^F_{p-n}{\rm DR}_X(\cH^0 \mathbf D_X^H/{\rm IC}_X^H) = 0\]
for all $p\leq k$. By \theoremref{thm-generalSmooth}, this is equivalent to
\[ {\rm Gr}^F_{p-n}{\rm DR}(i_* \mathbf V^{d-\delta}_{\rm prim}(-\delta-1)) = 0 \text{ for all } p \leq k,\]
or using similar reasoning to the above,
\[ F_{k+\delta+1-c_Z} \mathbf V^{d-\delta}_{\rm prim} = 0,\]
which is equivalent to 
\[ F^{c_Z -k-\delta-1} H^{d-\delta}_{\rm prim}(\cF_x) = F^{d-k} H^{d-\delta}_{\rm prim}(\cF_x) = 0\]
for some (hence any) $x\in Z$. At this point the proof is analogous to \cite{CDOIsolated}*{Cor. 3.8}.
\end{proof}

\subsection{Application to computing Hodge-Lyubeznik numbers} We keep the notation as in the previous two subsections. Our next goal is to compute $\lambda_{r,s}^{p,q}(\cO_{X,x})$ for a point $x\in X$. Note that the (intersection) Hodge-Lyubeznik numbers are only interesting for points $x\in X_{\rm nRS}$. So assume $X_{\rm nRS} \neq \emptyset$, in which case $X_{\rm nRS} = Z$.

\begin{thm} \label{thm-HLSit1} Let $x\in Z$. Then the following statements hold:
\begin{enumerate}
    \item If $s < n$, then we have that $\lambda_{r,s}^{p,q}(\cO_{X,x})$ is non-zero only for $$r = d_Z\textrm{ and }p+q=-s+1+d_Z,$$ and in that case, we have equality
\[\lambda_{d_Z,s}^{p,q}(\cO_{X,x}) = \begin{cases} \sum\limits_{a=0}^{\delta} h_{\rm prim}^{-q-a,-p-a}(\cF_x) & n -s \geq \delta \\ \sum\limits_{a=0}^{n -s} h_{\rm prim}^{n-s-\delta-q-a,n-s-\delta-p-a} (\cF_x) & n -s < \delta \end{cases}.\]
\item For $s = n$, we have $\lambda_{r,n}^{p,q}(\cO_{X,x}) = 0$ for $r\leq d_Z+1$. For $r\geq d_Z+2$, we have
\[ \lambda_{r,n}^{p,q}(\cO_{X,x}) = {\rm I}\lambda_r^{p,q}(\cO_{X,x}).\]
\item The intersection Hodge-Lyubeznik numbers are non-zero only for $$r\geq d_Z+1\textrm{ and }p+q = r-n,$$ and are given by the following:
\[ {\rm I}\lambda_r^{p,q}(\cO_{X,x}) = \begin{cases} \sum\limits_{a=0}^{\delta} h^{-q-a,-p-a}_{\rm prim}(\cF_x) & r -d_Z > \delta \\ \sum\limits_{a=0}^{r-d_Z-1} h^{d-a+p,d-a+q}_{\rm prim}(\cF_x) & r-d_Z \leq \delta \end{cases}.\]
\end{enumerate}
\end{thm}

\begin{proof} Let $x\in Z$ \textcolor{black}{and set $i_x:\left\{x\right\} \to X$ and $\iota_x:\left\{x\right\} \to Z$ to be the embeddings}. First of all, for $s < n$, we can use \corollaryref{cor-generalSit1}. We have
\[\begin{split}
    i_x^! \cH^{n -s} \mathbf D_X^H(n) \cong & \begin{cases} \bigoplus\limits_{a=0}^\delta i_x^!i_* \mathbf V_{\rm prim}^{d-(n -s-\delta+2a)}(s-a-1) & n -s \geq \delta \\ \bigoplus\limits_{a=0}^{n-s} i_x^!i_* \mathbf V_{\rm prim}^{d-(\delta-n +s+2a)}(n -\delta-a-1) & n -s < \delta \end{cases}\\
    = & \begin{cases} \bigoplus\limits_{a=0}^{\delta} \iota_x^!\mathbf V_{\rm prim}^{d-(n -s-\delta+2a)}(s-a-1)& n -s\geq \delta \\  \bigoplus\limits_{a=0}^{n-s} \iota_x^!\mathbf V_{\rm prim}^{d-(\delta-n +s+2a)}(n-\delta-a-1) & n -s < \delta \end{cases}
\end{split} \]
which is only non-zero for $0 < n -s \leq c_Z -2$. Thus, we only need to consider $d_Z +2 \leq s < n$.

By \exampleref{eg-restrictVHS}, this is isomorphic to
\[ \begin{cases} \bigoplus\limits_{a=0}^{\delta} H_{\rm prim}^{d-(n -s-\delta+2a)}(\cF_x)(s-a-1-d_Z)[-d_Z]& n -s\geq \delta \\  \bigoplus\limits_{a=0}^{n-s} H_{\rm prim}^{d-(\delta-n +s+2a)}(\cF_x)(c_Z-\delta-a-1)[-d_Z] & n -s < \delta \end{cases}\]
and so we see that for $s < n$ we get vanishing $\lambda_{r,s}^{p,q}(\cO_{X,x}) = 0$ unless $r= d_Z$. If $r=d_Z$, we get\small
\[\begin{split} \lambda_{d_Z,s}^{p,q}(\cO_{X,x}) = & \begin{cases} \sum\limits_{a=0}^{\delta} \dim_{\C} {\rm Gr}^F_{-p} {\rm Gr}^W_{p+q} H_{\rm prim}^{d-(n -s-\delta+2a)}(\cF_x)(s-a-1-d_Z) & n -s \geq \delta \\ \sum\limits_{a=0}^{n-s} \dim_{\C} {\rm Gr}^F_{-p} {\rm Gr}^W_{p+q}H_{\rm prim}^{d-(\delta-n +s+2a)}(\cF_x)(c_Z-\delta-a-1) & n -s < \delta\end{cases}\\
= & \begin{cases} \sum\limits_{a=0}^{\delta} \dim_{\C} {\rm Gr}^F_{d_Z+a+1-s-p} {\rm Gr}^W_{p+q+2(s-a-1-d_Z)} H_{\rm prim}^{d-(n -s-\delta+2a)}(\cF_x) & n -s \geq \delta \\ \sum\limits_{a=0}^{n-s} \dim_{\C} {\rm Gr}^F_{a+\delta+1-c_Z-p} {\rm Gr}^W_{p+q+2(c_Z-\delta-a-1)}H_{\rm prim}^{d-(\delta-n +s+2a)}(\cF_x) & n -s < \delta\end{cases}\end{split}\]
\normalsize
which, by purity, is only non-zero if 
\begin{equation} \label{eq-Weights} \begin{cases} p+q +2(s-a-1-d_Z) = d-(n -s -\delta+2a) & n -s \geq \delta \\ p+q +2(c_Z-\delta-a-1) = d-(\delta-n +s+2a) & n -s < \delta \end{cases},\end{equation}
which is easily simplified in both cases to $p+q = -s+1+d_Z$.

In the first case, we can write the index of the Hodge filtration as $d_Z+a+1-s-p = q+a$. In the second case, we can write the index as $a+\delta+1-c_Z-p = a+q+s+\delta-n$.

To rewrite the expression in terms of the Hodge bi-grading, we first have to negate these indices, and then compute the other index (using equality \eqref{eq-Weights}) as: in the first case,
\[ d -(n-s-\delta+2a) + (q+a) = -p-a \]
and in the second case,
\[ d - (\delta -n +s+2a) + a+q+s+\delta-n = n -s-\delta-a-p.\]

Summarizing, we have that for $s<n$, the non-vanishing of $\lambda_{d_Z,s}^{p,q}(\cO_{X,x})$ is only possible if $p+q=-s+1+d_Z$, and in that case, it is equal to
\[\begin{cases} \sum\limits_{a=0}^{\delta} h_{\rm prim}^{-q-a,-p-a}(\cF_x) & n -s \geq \delta \\ \sum\limits_{a=0}^{n -s} h_{\rm prim}^{n-s-\delta-q-a,n-s-\delta-p-a} (\cF_x) & n -s < \delta\end{cases}.\]

For $s= n$, we want to compute
\[ \lambda_{r,n}^{p,q}(\cO_{X,x}) = \dim_{\C} {\rm Gr}^F_{- p} {\rm Gr}^W_{p+q} \cH^r_x(\cH^{0} \mathbf D_X^H(n)), \]
which we can do in two steps. Indeed, we have the short exact sequence
\[ 0 \to {\rm IC}_X^H \to \cH^0 \mathbf D_X^H \to i_* \mathbf V_{\rm prim}^{d-\delta}(-\delta-1) \to 0,\]
to which we apply $i_x^!$. As $\mathbf V^{d-\delta}_{\rm prim}(-\delta-1)$ is a variation of Hodge structure, we know by \exampleref{eg-restrictVHS} that
\[ i_x^! i_* \mathbf V^{d-\delta}_{\rm prim}(-\delta-1) = \iota_x^! \mathbf V^{d-\delta}_{\rm prim}(-\delta-1) \cong H^{d-\delta}_{\rm prim}(\cF_x)(-\delta-d_Z-1)[-d_Z].\]

So we have the exact sequence
\small
\begin{equation} \label{eq-ESH0D} 0 \to \cH^{d_Z}_x{\rm IC}_X^H \to \cH^{d_Z}_x \cH^0 \mathbf D_X^H \to H^{d-\delta}_{\rm prim}(\cF_x)(-\delta-d_Z-1) \to \cH^{d_Z+1}_x{\rm IC}_X^H \to \cH^{d_Z+1}_x \cH^0\mathbf D_X^H\to 0,\end{equation}
\normalsize
which we will Tate twist by $n$, and we have isomorphisms
\begin{equation} \label{eq-isoHL} \cH^j_x {\rm IC}_X^H(n) \cong \cH^j_x \cH^0\mathbf D_X^H(n) \text{ for all } j \notin \{d_Z,d_Z+1\}.\end{equation}

By \theoremref{thm-generalSmooth}, we have 
\[
f_*\Q_{\widetilde{X}}^H[n] \simeq {\rm IC}_X^H\oplus i_* \mathbf V^{c_Z} \oplus \bigoplus_{\ell>0} \left( i_* \mathbf V^{c_Z + \ell}[-\ell] \oplus i_* \mathbf V^{c_Z+\ell}(\ell)[\ell] \right).\]
Combined with base change we have by applying $i_x^*$ the decomposition:
\[a_*\Q_{\cF_x}[n] \simeq i_x^*{\rm IC}_X^H\oplus \iota_x^* \mathbf V^{c_Z} \oplus \bigoplus_{\ell>0} \left( \iota_x^*\mathbf V^{c_Z + \ell}[-\ell] \oplus \iota_x^*\mathbf V^{c_Z+\ell}(\ell)[\ell] \right).\]

Taking cohomology we have 
\[ H^{n+j}(\cF_x) \simeq \cH^j i_x^*{\rm IC}_X^H\oplus \cH^j \iota_x^* \mathbf V^{c_Z} \oplus \bigoplus_{\ell>0} \left( \cH^{j-\ell}\iota_x^*\mathbf V^{c_Z + \ell}\oplus \cH^{j+\ell}\iota_x^*\mathbf V^{c_Z+\ell}(\ell) \right).\]

For fixed $j$, we get, by \exampleref{eg-restrictVHS} the isomorphisms
\[ H^{n+j}(\cF_x) \simeq \begin{cases} \cH^j i_x^*{\rm IC}_X^H\oplus H^{n+j}(\cF_x) & j> -d_Z\\ \cH^j i_x^*{\rm IC}_X^H \oplus H^{c_Z-d_Z-j}(\cF_x)(-d_Z-j) & j < -d_Z \\ \cH^{-d_Z} i_x^*{\rm IC}_X^H\oplus H^{c_Z}(\cF_x) & j = -d_Z\end{cases}\]

We conclude that there is vanishing 
\[ \cH^j i_x^*{\rm IC}_X^H = 0 \text{ for all } j \geq -d_Z.\]

We are interested in $i_x^! {\rm IC}_X^H$, and so we must dualize. By pure polarizability, this just amounts to tracking some Tate twists.

We conclude $\cH^j \iota_x^!{\rm IC}_X^H = 0$ for all $j \leq d_Z$ and that we have an isomorphism for all $j < -d_Z$
\[ H^{n+j}(\cF_x)(n+j) \cong \cH^{-j} i_x^! {\rm IC}_X^H(n) \oplus H^{c_Z-d_Z-j}(\cF_x)(c_Z),\]
or, by negating $j$, we have isomorphism for $j > d_Z$:
\[ H^{n-j}(\cF_x)(n-j) \cong \cH^{j} i_x^! {\rm IC}_X^H(n) \oplus H^{c_Z-d_Z+j}(\cF_x)(c_Z).\]

If we write $j = d_Z +k$ for some $k>0$, then we get
\[ H^{c_Z-k}(\cF_x)(c_Z-k) \cong \cH^{d_Z + k} i_x^! {\rm IC}_X^H(n) \oplus H^{c_Z+k}(\cF_x)(c_Z),\]
and we have the surjection by Hard Lefschetz:
\[ c_1(\ell_x)^k \colon H^{c_Z-k}(\cF_x)(c_Z-k) \to H^{c_Z+k}(\cF_x)(c_Z),\]
so that we have an identification
\begin{equation} \label{eq-IsoBeforeLefschetz} \cH^{d_Z+k} i_x^! {\rm IC}_X^H(n) \cong \ker\left(c_1(\ell_x)^k \colon H^{c_Z-k}(\cF_x) \to H^{c_Z+k}(\cF_x)(k)\right)(c_Z-k).\end{equation}

By writing $c_Z-k = d + \delta -(k-1)$, we have the Lefschetz decomposition
\[H^{c_Z-k}(\cF_x) = \begin{cases} \bigoplus\limits_{a\geq 0} H^{d-(k-1-\delta+2a)}_{\rm prim}(\cF_x)(-a) & k-1 \geq \delta \\ \bigoplus\limits_{a\geq 0} H^{d-(\delta-k+2a+1)}_{\rm prim}(\cF_x)(k-a-\delta-1) & k-1 < \delta\end{cases}\]
which allows us rewrite the equality \eqref{eq-IsoBeforeLefschetz} as
\begin{equation}\label{fact}
\cH^{d_Z+k} i_x^! {\rm IC}_X^H(n) \cong \begin{cases} \bigoplus\limits_{a=0}^\delta H^{d-((k-1)-\delta+2a)}_{\rm prim}(\cF_x)(c_Z-k-a) & k > \delta \\ \bigoplus\limits_{a=0}^{k-1} H^{d-(\delta-k+2a+1)}_{\rm prim}(\cF_x)(c_Z-a-\delta-1) & k \leq \delta \end{cases}.
\end{equation} 

In summary, we have
\[ \cH^r i_x^! {\rm IC}_X^H (n) = \begin{cases} 0 & r \leq d_Z \\ \bigoplus\limits_{a=0}^\delta H^{d-(r-d_Z-1-\delta+2a)}_{\rm prim}(\cF_x)(n-r-a) & r-d_Z > \delta \\ \bigoplus\limits_{a=0}^{r-d_Z-1} H^{d-(\delta-r+d_Z+2a+1)}_{\rm prim}(\cF_x)(c_Z-a-\delta-1) & 0 < r-d_Z \leq \delta \end{cases}.\]

Using this computation for $r = d_Z, d_Z+1$, we have the exact sequence (which is simply a restatement of the exact sequence \eqref{eq-ESH0D})\footnotesize
\[0 \to \cH^{d_Z}_x \cH^0 \mathbf D_X^H \to H^{d-\delta}_{\rm prim}(\cF_x)(-\delta-d_Z-1) \to  H^{d-\delta}_{\rm prim}(\cF_x)(-\delta-d_Z-1) \to \cH^{d_Z+1}_x \cH^0\mathbf D_X^H\to 0.\]
\normalsize

\begin{claim}
    We have $\cH_x^{d_Z} \cH^0 \mathbf D^H_X = \cH_x^{d_Z+1} \cH^0 \mathbf D_X^H = 0$.
\end{claim} 
\begin{proof} To see this, note that it suffices by comparing dimensions in the exact sequence above to prove $\cH_x^{d_Z+1}\cH^0 \mathbf D_X = 0$. To do this, consider the spectral sequence
\[ E_2^{p,q} = \cH^p_x \cH^q \mathbf D_X^H \implies \cH^{p+q}_x \mathbf D_X^H.\]

By the isomorphism \eqref{eq-DXRule}, we see that $\cH^{p+q}_x \mathbf D_X^H = 0$ unless $p+q = n$. Hence $E_\infty^{p,q} = 0$ for all $p+q \neq n$.

On the other hand, by what we have shown above, we have $E_2^{p,q} =0 $ unless $q >0$ and $p= d_Z$ or $q = 0$ and $p\geq d_Z$. It is easy to see then that
\[ E_2^{d_Z+1,0} = E_\infty^{d_Z+1,0}.\]

As we are assuming $c_Z \geq 2$, we have that $d_Z+1 \neq n$, and so we get the vanishing $\cH^{d_Z+1}_x \cH^0\mathbf D_X^H = E_2^{d_Z+1,0} = E_\infty^{d_Z+1,0} = 0$.
\end{proof}

For $r \geq d_Z +2$, we can perform the same analysis as we did in the $s < n$ case to see that $\lambda_{r,n}^{p,q}(\cO_{X,x}) \neq 0$ only if $p+q = r- n$. We can also rewrite in terms of the Hodge bigrading, similarly to the above, though we omit that computation as it is the same as above. This completes the proof.
\end{proof}

\section{Cohomology and Singularities of Secant Varieties} \label{sec-2Sec}
Let $Y$ be a smooth projective variety with $L$ a $3$-very ample line bundle on $Y$ defining an embedding $Y \subseteq \P^N$. Recall that, in this case, we have the diagram
\begin{equation}\label{ds}
    \begin{tikzcd} \Phi \ar[d,"p"] \ar[r] & \mathbf P \ar[d,"t"] \\ Y \ar[r,"i"] & \Sigma \end{tikzcd},
\end{equation}
where $\mathbf P \to \Sigma$ can be identified with the blow-up of $\Sigma$ along $Y$, and it is a log-resolution. We will assume that $\Sigma \neq \mathbf P^N$, and in this situation, we have $\Sigma_{\rm sing} = Y$. Also recall that $\dim \Sigma=2\dim Y+1$. We resume the notation from Subsection \ref{prelimsecant}. 
%However, we abuse the notation to denote both the classes of the exceptional divisors $c_1(E_{\Delta})$ and $c_1(E_y)$ on ${\rm Bl}_{\Delta}(Y\times Y)$ and ${\rm Bl}_y(Y)$ respectively by $e$ (and irrespective of $y\in Y$).

\subsection{Singularity invariants and Hodge-Lyubeznik numbers} The blow-up diagram shows that we completely understand the higher cohomology modules $\cH^j\mathbf D_{\Sigma}^{\cM}$ and ${\rm Gr}^W_{n+1} \cH^0 \mathbf D_{\Sigma}^{\cM}$. The following proves \theoremref{thm-Secants} from the introduction via \lemmaref{lem-DXLocCoh} (upon taking $k= \C$ and $\cM(-) = {\rm MHM}(-,\Q)$).

\begin{thm}\label{thmA} In the setting described above, we have the following:
\begin{itemize}
    \item Assume $\dim Y \geq 2$, then
\begin{enumerate}
    \item ${\rm lcdef}(\Sigma) = {\rm lcdef}_{\rm gen}(\Sigma) = {\rm lcdef}_{\rm gen}^{>0}(\Sigma) \in \{\dim Y-2, \dim Y -1\}$.  We have ${\rm lcdef}(\Sigma) = \dim Y -1$ if and only if $H^1(Y,\cO_Y) \neq 0$.
    \item $\Sigma_{\rm nRS} = Y$, which is equal to $\Sigma_{\rm nCCI}$ if $\Sigma$ is not CCI.
    \item For all $j>0$, we have an isomorphism of pure objects of $\cM(\Sigma)$ of weight $\dim \Sigma +j+1$:
    \[ \cH^j \mathbf D_\Sigma^{\cM} \cong i_* \mathbf V_{\cM,\rm prim}^{\dim Y -j}(-j-1).\]
    \item We have ${\rm Gr}^W_i \cH^0 \mathbf D_\Sigma^{\cM} \neq 0 \implies i \in [\dim \Sigma, \dim \Sigma+1]$, an isomorphism
    \[ {\rm Gr}^W_{\dim \Sigma +1} \cH^0 \mathbf D_\Sigma^{\cM} \cong i_* \mathbf V_{\cM,\rm prim}^{\dim Y}(-1),\]
    and a short exact sequence
    \[ 0 \to {\rm IC}_\Sigma^{\cM} \to \cH^0 f_* A_{\P}^{\cM}[\dim \Sigma] \to i_* \mathbf V_{\cM}^{\dim Y+1} \to 0.\]
\end{enumerate}
\item If $Y$ is a curve, then
\begin{enumerate}
    \item ${\rm lcdef}(\Sigma) = {\rm lcdef}_{\rm gen}(\Sigma) = {\rm lcdef}_{\rm gen}^{>0}(\Sigma)= 0$.
    \item We have ${\rm Gr}^W_i \cH^0 \mathbf D_\Sigma^{\cM} \neq 0 \implies i \in [3, 4]$, an isomorphism
    \[ {\rm Gr}^W_{4} \cH^0 \mathbf D_\Sigma^{\cM} \cong i_*(H^1_{\cM}(Y)\boxtimes A_Y^{\cM}[1] )(-1) ,\]
    and a short exact sequence
    \[ 0 \to {\rm IC}_\Sigma^{\cM} \to \cH^0 f_* A_{\P}^{\cM}[3] \to i_* A_Y^{\cM}[1](-1) \to 0.\]
    \item $\Sigma$ is a rational homology manifold if and only if $Y\cong \P^1$.
\end{enumerate}
\end{itemize}
\end{thm}
\begin{proof} All statements follow from the identification of $\cH^j \mathbf D_{\Sigma}^{\cM}$ with the primitive cohomology of the family $\{{\rm Bl}_y(Y)\}_{y\in Y}$, and then from \lemmaref{lem-BlowupPrim}.

For the identification, we use \corollaryref{cor-generalSit1}.

For the computation of ${\rm lcdef}(\Sigma)$, we use the isomorphism $\cH^j \mathbf D_\Sigma^{\cM} \cong i_* \mathbf V_{\cM,\rm prim}^{\dim Y -j}(-j-1)$, which is non-zero only for $j \leq \dim Y -1$. 

The value for $j = \dim Y -1$ corresponds to $\mathbf V_{\cM,\rm prim}^{1}(-\dim Y)$. As $H^1$ is always primitive, this has underlying $\Q$-local system given by $\{ H^1({\rm Bl}_y(Y),\Q)\}_{y\in Y}$. By \lemmaref{lem-BlowupPrim}, we have identification $H^1({\rm Bl}_y(Y),\Q) \cong b_y^*(H^1(Y,\Q))$. So we get the claimed equivalence.

If $H^1(Y) = 0$ and $\dim Y \geq 2$, then ${\rm lcdef}(\Sigma) \leq \dim Y -2$. However, we always have that $\cH^{\dim Y -2} \mathbf D_\Sigma^{\cM} \neq 0$. Indeed, its underlying $\Q$-local system is $\{ H^2_{\rm prim}({\rm Bl}_y(Y),\Q)\}_{y\in Y}$. This is always non-zero as shown in \lemmaref{lem-BlowupPrim}.

Now, for $Y$ a curve, note that $\mathbf V_{\cM}^2$ has underlying local system with constant value $H^2(Y,\Q)$. This is simply $\Q_Y[1]$. Similarly, $\mathbf V^1_{\cM}$ has underlying $\Q$-local system with constant value $H^1(Y)$, given by $H^1(Y) \boxtimes\Q_Y[1] $ As ${\rm lcdef}(\Sigma) = 0$, we see that $\Sigma$ is a rational homology manifold if and only if ${\rm Gr}^W_4 \cH^0 \mathbf D_{\Sigma}^{\cM} =0$ if and only if $H^1_{\cM}(Y) = 0$, which by definition is equivalent to $Y$ having genus $0$.
\end{proof}

From now on, assume  $k= \C$ and $\cM(-) = {\rm MHM}(-,\Q)$. We make one more observation concerning the family $p \colon \Phi \to Y$. 

\begin{lem} \label{lem-trivialMonodromy} For any $k \in \Z$, the local systems
\[ \mathbf V^k \text{ and } \mathbf V^k_{\rm prim}\]
have trivial monodromy.
\end{lem}
\begin{proof} It suffices to prove the claim for $\mathbf V^k$, as $\mathbf V^k_{\rm prim}$ is a sub-local system. To do this, we use the \emph{Global invariant cycles theorem} \cite{DeligneHodgeII}*{4.1.1} (see also the exposition surrounding \cite{DCM}*{Thm. 1.2.4}) which says that the image of restriction map
\[ H^k(\Phi) \to H^k(p^{-1}(y))\]
is the monodromy invariant part of the fiber. As usual, we identify $\Phi \cong {\rm Bl}_\Delta(Y \times Y)$ and $p^{-1}(y) \cong {\rm Bl}_y(Y)$. We need to show that the restriction map is surjective.

By Subsection \ref{subsec-ClassicalCohomology}, we have
\[ H^k(\Phi) \cong b^* H^k(Y \times Y) \oplus \left(\bigoplus_{i=1}^{\dim Y -1} H^{k-2i}(Y)E^i\right) ,\]
\[ H^k({\rm Bl}_y(Y)) \cong b_y^* H^k(Y) \oplus \left(\bigoplus_{i=1}^{\dim Y -1} H^{k-2i}(\{y\})e^i\right),\]
where $E = c_1(\cO(E_\Delta))$ is the first Chern class of the exceptional divisor for the blow-up $\Phi \to Y \times Y$, and similarly $e = c_1(\cO(E_y))$ is the first Chern class of the exceptional divisor for the blow-up ${\rm Bl}_y(Y)$.

We have the K\"{u}nneth formula $$H^k(Y\times Y) \cong \bigoplus_{i =0}^k \pi_1^* H^i(Y) \otimes \pi_2^* H^{k-i}(Y),$$
and in particular, we have $b^* \pi_1^* H^k(Y) \subseteq H^k(\Phi)$. If we apply $i_y^*$ to this, then we can use that $\pi_1 \circ b \circ i_y \colon {\rm Bl}_y(Y) \times \{y\} \to Y$ can be identified with $b_y \colon {\rm Bl}_y(Y) \to Y$, and so we see that $b_y^* H^k(Y)$ is contained in the image of the restriction map $i_y^*$. This proves the desired surjectivity if $k\notin 2\Z \cap [2,2\dim Y-2]$.

Assume $k = 2i$ with $1\leq i\leq \dim Y -1$. 
Thus, given $e^i \alpha \in e^i \cdot H^{k-2i}(\{y\}) = e^i \cdot \Q$, we see 
\[ E^i \cdot \alpha \in E^i \cdot H^{k-2i}(Y) = E^i \cdot H^0(Y) = E^i \cdot \Q\]
restricts to $e^i \alpha$, by naturality of the Chern class under $i_y^*$ and the fact that intersecting the fiber ${\rm Bl}_y(Y)$ with $E_\Delta$ gives $E_y$ (see a similar argument \cite{OR}*{Pf. of Claim 6.4}).
\end{proof}

\begin{thm}[$=$ \theoremref{cpc}]\label{cpcn} Secant varieties satisfy the Cappell-Shaneson Conjecture if $L$ is 3-very ample.
\end{thm}

\begin{proof}
Recall from \corollaryref{cor-generalSit1}(3) that \[ t_* \Q^H_{\P}[2\dim Y+1] \cong {\rm IC}_{\Sigma}^{H} \oplus  i_* \mathbf V^{\dim Y+1}\oplus \left(\bigoplus_{\ell=1}^{\dim Y-1} \left(i_* \mathbf V^{\dim Y+1+\ell}[-\ell] \oplus i_* \mathbf V^{\dim Y+1+\ell}(\ell)[\ell]\right)\right).\]
Now, \cite{AMS}*{Cor. 4.11} shows that if ${\rm IC}_{\Sigma}^H[-\dim\Sigma]$ belongs to the $ K_0({\rm MHM}({\rm pt}))$-submodule generated by ${\rm Image}(\chi_{\rm Hdg})$, then the conjecture holds for $\Sigma$. Consequently the assertion follows by the above decomposition and \lemmaref{lem-trivialMonodromy}.
\end{proof}

\begin{cor}[$=$ \corollaryref{cor-chrh}]\label{cor-chrhn} Assume $L$ is 3-very ample and $\Sigma\neq\P^N$. Then we have:
\begin{enumerate}
    \item $c(\Sigma)=\begin{cases}
        \infty & \dim Y=1; \textrm{ or $\dim Y=2$ and $H^1(Y,\cO_Y) = 0$,}\\
        0 & \dim Y\geq 3\textrm{ and }H^{i}(Y,\cO_Y) = 0 \text{ for all } 0 < i < \dim Y,\\
        -1 & \textrm{otherwise.}
    \end{cases}$

    \smallskip
    
    \noindent In particular, $c(\Sigma)\geq 0$ if and only if $H^{i}(Y,\cO_Y) = 0 \text{ for all } 0 < i < \dim Y$.

    \smallskip
    
    \item ${\rm HRH}(\Sigma)=\begin{cases}
        \infty & Y\cong\P^1,\\
        0 & \dim Y\geq 2\textrm{ and } H^{i}(Y,\cO_Y) = 0 \text{ for all } i>0,\\
        -1 & \textrm{otherwise}.
    \end{cases}$

    \smallskip
    
    \noindent In particular, ${\rm HRH}(\Sigma)\geq 0$ if and only if $H^{i}(Y,\cO_Y) = 0 \text{ for all } i>0$.
\end{enumerate}
\end{cor}
\begin{proof} We use the fact that the fiber $p^{-1}(y)$ can be identified with ${\rm Bl}_y(Y)$.

As above, the condition $c(\Sigma)\geq k$ is equivalent to $F^{\dim Y -k-j} H^{n-j}_{\rm prim}({\rm Bl}_y(Y)) =0 \text{ for all } j >0$. For $j \neq \dim Y - 2$, by \lemmaref{lem-BlowupPrim} above, this is equivalent to $F^{ \dim Y -k - j} H^{\dim Y -j}_{\rm prim}(Y) = 0$. 

For $j = \dim Y -2$, the non-zero element $c_1(L) +2 (\frac{1}{2} L)^{\dim Y} e$ lies in $H^{1,1}_{\rm prim}$, hence in $F^1 H^2_{\rm prim}$. So $c(\Sigma) < 1$.

Note that the outermost terms of each row in the Hodge diamond are automatically primitive, so the condition that $c(\Sigma) = 0$ is equivalent to $H^{\dim Y -j,0}(Y) = 0$ for all $0 < j < \dim Y$.

Assuming $c(\Sigma)= 0$, the condition for ${\rm HRH}(Y) \geq 0$ is equivalent, by the same logic above, to the vanishing $H^{\dim Y,0}(Y) = 0$.
\end{proof}

\begin{rmk}\label{wSigma}
    Using the computations of \corollaryref{cor-kCCI}, one can also show that when $L$ is 3-very ample and $\Sigma\neq\P^N$, we have $$w(\Sigma)=\begin{cases}
        \dim Y-1-\mu^{\dim Y}_{\rm prim}(Y) &  H^{\dim Y}_{\rm prim}(Y)\neq 0 \\
        \infty & \dim{Y}\neq 2 \text{ and } H^{\dim Y}_{\rm prim}(Y)= 0 \\
        0 &  \dim{Y}=2 \text{ and } H^{2}_{\rm prim}(Y)= 0
    \end{cases}.$$ In particular, $w(\Sigma)\geq 0$ if and only if $H^{\dim Y}(\cO_Y)=0$. As mentioned in the introduction, this refines \cite{ChouSong}*{Thm. 1.4}.
\end{rmk}

 Finally, we have the computation of the Hodge-Lyubeznik numbers:

 \begin{cor}[$=$ \corollaryref{hln}]\label{hlnn} \textcolor{black}{Assume $L$ is 3-very ample and $\Sigma\neq\P^N$. Further assume $\Sigma_{\rm nRS}\neq\emptyset$ (equivalently $Y\ncong\P^1$, whence $\Sigma_{\rm nRS}=Y$ by \theoremref{thm-Secants}).} Let $y\in Y \subseteq \Sigma$. Then:
\begin{enumerate}
    \item For $\dim Y + 2 \leq s < \dim \Sigma$, we have the equality
\[ \lambda_{\dim \Sigma,s}^{u,v}(\cO_{\Sigma,y}) = \begin{cases} h^{-v,-u}_{\rm prim}(Y) & u+v = \dim Y +1-s \text{ and } (u,v) \neq (-1,-1) \\ 1 + h^{1,1}_{\rm prim}(Y) & \, s = \dim Y +3 \text{ and } (u,v) = (-1,-1)\end{cases}\]
and all other numbers vanish for such $s$.
\item For $s= \dim \Sigma$, \textcolor{black}{the Hodge-Lyubeznik numbers $\lambda_{r,\dim \Sigma}^{u,v}(\cO_{\Sigma,y})$ are potentially non-zero only for $r\geq \dim Y +2$, in which case} we have the equality
\[ \lambda_{r,\dim \Sigma}^{u,v}(\cO_{\Sigma,y}) = \begin{cases} h^{-v,-u}_{\rm prim}(Y) & u+v = r-\dim \Sigma, (u,v) \neq (-1,-1), \textcolor{black}{\dim Y\geq 2} \\ 1+h^{1,1}_{\rm prim}(Y) & r = \dim \Sigma -2, (u,v) = (-1,-1), \textcolor{black}{\dim Y\geq 2}\\
\textcolor{black}{1} & \textcolor{black}{r=3, (u,v)=(0,0),\dim Y=1}\\
\textcolor{black}{0}& \textcolor{black}{\textrm{otherwise}}\end{cases}.\]
%and all other numbers vanish.
\item The intersection Hodge-Lyubeznik numbers are \textcolor{black}{potentially} non-zero only for $r\geq \dim Y +1$, in which case they are given by the following:
\[ {\rm I}\lambda_r^{u,v}(\cO_{\Sigma,y}) = \begin{cases} h^{-v,-u}_{\rm prim}(Y) & u+v = r-\dim \Sigma, (u,v) \neq (-1,-1), \textcolor{black}{\dim Y\geq 2} \\ 1+h^{1,1}_{\rm prim}(Y) & r = \dim \Sigma -2,(u,v) = (-1,-1), \textcolor{black}{\dim Y\geq 2}\\
\textcolor{black}{h^1(Y,\cO_Y)} & \textcolor{black}{(u,v)\in\left\{(-1,0),(0,-1)\right\}, r=2,\dim Y=1}\\
1 & r = 3, (u,v) = (0,0), \dim Y = 1 \\ 
\textcolor{black}{0} & \textcolor{black}{\textrm{otherwise}}\end{cases}.\]
\end{enumerate}
\end{cor}

\begin{proof}
    The statements immediately follow from \theoremref{thm-HLSit1} above. The relation to the primitive cohomology of $Y$ follows from \lemmaref{lem-BlowupPrim}.
\end{proof}

\subsection{Generation level of Hodge filtrations on local cohomology} We give here a computation of ${\rm gl}(\cH^j_{\Sigma}(\cO_{\P^N}),F)$, where $Y \hookrightarrow \P^N$ is the embedding determined by $L$. At this point, we continue to assume $k= \C$ and $\cM(-) = {\rm MHM}(-,\Q)$. We need the following

\begin{claim}\label{acc}
    $R^ip_*\Omega_{\Phi}^{2\dim Y+1-i}=\bigoplus\limits_{l=0}^{\dim Y-1}{\rm Gr}^{\dim Y+1+l-i}_{F}H^{\dim Y+l+1}({\rm Bl}_yY)\otimes\Omega_Y^{\dim Y-l}$.
\end{claim}
\begin{proof}
    Apply the decomposition theorem for the projective morphism  $p:\Phi\to Y$ to see that $$p_*\Q_{\Phi}^H[2\dim Y]=\bigoplus_{l=-\dim Y}^{\dim Y}{\bf V}^{\dim Y+l}[-l].$$
    Applying $\cH^1{\rm Gr}_{-(2\dim Y+1-i)}^F{\rm DR}$ both sides and using \lemmaref{lem-trivialMonodromy}, we obtain 
    \[
    \begin{array}{ll}
       R^ip_*\Omega_{\Phi}^{2\dim Y+1-i}  &= \bigoplus\limits_{l=-\dim Y}^{\dim Y}\cH^{-(l-1)}{\rm Gr}_{-(2\dim Y+1-i)}^F{\rm DR}({\bf V}^{\dim Y+l}) \\
         &= \bigoplus\limits_{l=1}^{\dim Y}\cH^{-(l-1)}{\rm Gr}_{-(2\dim Y+1-i)}^F{\rm DR}({\bf V}^{\dim Y+l})\\
         &= \bigoplus\limits_{l=1}^{\dim Y}{\rm Gr}_{-(\dim Y-i+l)}^FH^{\dim Y+l}({\rm Bl}_yY)\otimes\Omega_Y^{\dim Y-l+1}\\
         & = \bigoplus\limits_{l=0}^{\dim Y-1}{\rm Gr}^{\dim Y+1+l-i}_{F}H^{\dim Y+l+1}({\rm Bl}_yY)\otimes\Omega_Y^{\dim Y-l}.
    \end{array}
    \]
    The above completes the proof of the claim.
\end{proof}

\begin{cor}[$=$ \corollaryref{cor-GenLevel2Secants}] \label{cor-GenLevel2Secantsn} 
Assume $L$ is 3-very ample and $\Sigma\neq\P^N$. Then:
\begin{enumerate}
    \item If %$\dim Y > 2$ and $j>0$
    $\dim Y\geq 2$ and $0<j\leq {\rm lcdef}(\Sigma)$, there are equalities
\[ {\rm gl}(\cH^{q+j}_{\Sigma}(\cO_{\P^N}),F) = \begin{cases} \dim Y - j - \mu^{\dim Y-j}_{\rm prim}(Y) & j \neq \dim Y -2 \\ 2 & j =\dim Y -2, H^2(Y,\cO_Y)\neq 0 \\ 1 & j = \dim Y -2, H^2(Y,\cO_Y) = 0 \end{cases}.\]
\item ${\rm gl}({\rm Gr}^W_{N+q+1}\cH^{q}_{\Sigma}(\cO_{\P^N}),F) = \begin{cases}
        \dim Y- \mu^{\dim Y}_{\rm prim}(Y) & \dim Y\geq 3\\
        2 & \dim Y=2, H^2(Y,\cO_Y)\neq 0\\
        1 & \dim Y=2, H^2(Y,\cO_Y)=0\\
        1 & \dim Y=1, H^1(Y,\cO_Y)\neq 0\\
        -\infty & \textrm{otherwise $($i.e., $Y\cong\P^1)$}
        \end{cases}$. 
\item For $Y$ of any dimension, we have
$ {\rm gl}( \cH^q_{\Sigma}(\cO_{\P^N}),F) \leq \dim Y$.
\item Assume $L$ satisfies $(Q'_p)$-property (\definitionref{def-pos}) for some $0\leq p\leq \dim Y-1$. Then
\begin{enumerate}
    \item ${\rm gl}({\rm IC}_{\Sigma}^H(-q),F)\leq \dim Y-p-1$.
     \item Assume $L$ satisfies $(Q'_{\dim Y-1})$-property. Then ${\rm gl}({\rm IC}_{\Sigma}^H(-q),F)=0$ and $${\rm gl}(\cH^q_{\Sigma}(\cO_{\P^N}),F)=\begin{cases}
        \max\left\{0,\dim Y- \mu^{\dim Y}_{\rm prim}(Y)\right\} & \dim Y\geq 3\\
        2 & \dim Y=2, H^2(Y,\cO_Y)\neq 0\\
        1 & \dim Y=2, H^2(Y,\cO_Y)=0\\
        1 & \dim Y=1, H^1(Y,\cO_Y)\neq 0\\
        0 & \textrm{otherwise $($i.e., $Y\cong\P^1)$} 
    \end{cases}.$$
\end{enumerate}
\end{enumerate}
\end{cor}

\begin{proof}
As in \lemmaref{lem-preciseGL}, it suffices to compute $\cH^0 {\rm Gr}^F_{i-N} {\rm DR}(-)$. Recall that $q = N - (2\dim Y +1)$ denotes the codimension of $\Sigma$ in $\P^N$. We have the isomorphism of \lemmaref{lem-DXLocCoh} given by
\[ (\cH^{q+j}_{\Sigma}(\cO_{\P^N}),F,W) \cong {i_{\Sigma}}_* \cH^j \mathbf D_{\Sigma}^H(-q),\]
where $i_{\Sigma}\colon \Sigma \to \P^N$ is the closed embedding.

For $j > 0$, by \theoremref{thm-Secants} we see that
\[ (\cH^{q+j}_{\Sigma}(\cO_{\P^N}),F) \cong \iota_*\mathbf V_{\rm prim}^{\dim Y -j}(-q-j-1),\]
and
\[ ({\rm Gr}^W_{N+q+1} \cH^{q}_{\Sigma}(\cO_{\P^N}),F) \cong \iota_* \mathbf V_{\rm prim}^{\dim Y}(-q-1),\]
where $\iota \colon Y \to \P^N$ is the closed embedding.

By \lemmaref{lem-trivialMonodromy}, the local systems on the right hand side have trivial monodromy, hence we have isomorphisms
\[ (\cH^{q+j}_{\Sigma}(\cO_{\P^N}),F) \cong \iota_*(H^{\dim Y -j}_{\rm prim}({\rm Bl}_y(Y)) \boxtimes \Q_Y^H[\dim Y]) (-q-j-1) \textrm{ for $j>0$},\]
and
\[ ({\rm Gr}^W_{N+q+1} \cH^{q}_{\Sigma}(\cO_{\P^N}),F) \cong \iota_*(H^{\dim Y}_{\rm prim}({\rm Bl}_y(Y)) \boxtimes \Q_Y^H[\dim Y]) (-q-1).\]

By taking $\ell =0$ in isomorphism \eqref{eq-grDRTrivialMonodromy} and switching to decreasing Hodge filtrations on Hodge structures, and accounting for the Tate twist, we have isomorphisms
\[ \cH^0 {\rm Gr}^F_{i-N} {\rm DR}( \cH^{q+j}_{\Sigma}(\cO_{\P^N}),F) \cong {\rm Gr}_F^{N- \dim Y -i-q-j-1}H^{\dim Y-j}_{\rm prim}({\rm Bl}_y(Y)) \otimes \Omega_Y^{\dim Y}\textrm{ for $j>0$},\]
\[ \cH^0 {\rm Gr}^F_{i-N} {\rm DR}( {\rm Gr}^W_{N+q+1}\cH^{q}_{\Sigma}(\cO_{\P^N}),F) \cong {\rm Gr}_F^{N- \dim Y -i-q-1}H^{\dim Y}_{\rm prim}({\rm Bl}_y(Y)) \otimes \Omega_Y^{\dim Y}.\]

Now, $N - \dim Y -q -1 = \dim Y$, so this can be rewritten
\[ \cH^0 {\rm Gr}^F_{i-N} {\rm DR}( \cH^{q+j}_{\Sigma}(\cO_{\P^N}),F) \cong {\rm Gr}_F^{\dim Y -i-j}H^{\dim Y-j}_{\rm prim}({\rm Bl}_y(Y)) \otimes \Omega_Y^{\dim Y}\textrm{ for $j>0$},\]
\[ \cH^0 {\rm Gr}^F_{i-N} {\rm DR}( {\rm Gr}^W_{N+q+1}\cH^{q}_{\Sigma}(\cO_{\P^N}),F) \cong {\rm Gr}_F^{\dim Y -i}H^{\dim Y}_{\rm prim}({\rm Bl}_y(Y)) \otimes \Omega_Y^{\dim Y}.\]

%If we set $\mu^\ell_{\rm prim}({\rm Bl}_y(Y)) = \min\{ p \mid {\rm Gr}_F^p H^\ell_{\rm prim}({\rm Bl}_y(Y)) \neq 0\}$, then 
Now, \lemmaref{lem-preciseGL} gives
\[ {\rm gl}(\cH^{q+j}_{\Sigma}(\cO_{\P^N}),F) = \dim Y - j-\mu^{\dim Y -j}_{\rm prim}({\rm Bl}_y(Y)) \text{ for } j > 0,\]\[ {\rm gl}({\rm Gr}^W_{N+q+1}\cH^{q}_{\Sigma}(\cO_{\P^N}),F) = \dim Y-\mu^{\dim Y }_{\rm prim}({\rm Bl}_y(Y)).\]

By \lemmaref{lem-BlowupPrim}, we see that when $\dim Y\geq 2$ (which holds if ${\rm lcdef}(\Sigma)\geq 1$)
\[ \mu^{\ell}_{\rm prim}({\rm Bl}_y(Y)) = \begin{cases} \mu^{\ell}_{\rm prim}(Y) & \ell \neq 2 \\ 0 & \ell =2, \quad H^2(Y,\cO_Y) \neq 0 \\ 1 & \ell =2,  \quad H^2(Y,\cO_Y) = 0\end{cases},\]
which gives the formula (1) and (2) in \corollaryref{cor-GenLevel2Secants}.

To give the bound on ${\rm gl}(\cH^q_\Sigma(\cO_{\P^N}),F)$, we use the bound on the generating level of $i_*{\rm IC}_{\Sigma}^H(-q) = W_{N+q} \cH^q_{\Sigma}(\cO_{\P^N})$ from \lemmaref{lem-genLevelIC}. This says that $${\rm gl}(W_{N+q} \cH^q_{\Sigma}(\cO_{\P^N}),F) \leq \dim Y,$$ since the maximal dimension of a fiber of $t\colon \P \to \Sigma$ is $\dim Y$.

Finally, by applying ${\rm Gr}^F_{i-N} {\rm DR}(-)$ to the short exact sequence
\[ 0 \to i_*{\rm  IC}_\Sigma(-q) \to \cH^q_{\Sigma}(\cO_{\P^N}) \to {\rm Gr}^W_{N+q+1} \cH^q_{\Sigma}(\cO_{\P^N}) \to 0,\]
(where we used the result of \theoremref{thm-Secants} which says that the weight filtration on $\cH^q_{\Sigma}(\cO_{\P^N})$ only jumps at $N+q$ and $N+q+1$), we get the exact sequence
\small
\begin{equation}\label{genlevelexact}
    \cH^0 {\rm Gr}^F_{i-N} {\rm DR}(i_* {\rm IC}_{\Sigma}^H(-q)) \to \cH^0 {\rm Gr}^F_{i-N} {\rm DR}(\cH^q_{\Sigma}(\cO_{\P^N})) \to {\rm Gr}^F_{i-N}{\rm DR}({\rm Gr}^W_{N+q+1} \cH^q_{\Sigma}(\cO_{\P^N})) \to 0.
\end{equation} 

\normalsize
By what was just discussed, for $i> \dim Y$, the leftmost term vanishes, and by non-negativity of $\mu^{\dim Y}_{\rm prim}({\rm Bl}_y(Y))$, we see that for $i> \dim Y$ the rightmost term vanishes, too. This shows that ${\rm gl}(\cH^q_{\Sigma}(\cO_{\P^N}),F) \leq \dim Y$, as claimed. This proves (3).

Now assume $L$ satisfies the \textcolor{black}{$(Q'_{p})$-property for some $0\leq p\leq \dim Y-1$.} By \propositionref{glv}, we have $$R^it_*\Omega_{\P}^{2\dim Y+1-i}(\log \Phi)(-\Phi)=0\textrm{ for all $i\geq \dim Y-p$}.$$ Consequently the exact sequence $$0\to \Omega_{\P}^{2\dim Y+1-i}(\log \Phi)(-\Phi)\to \Omega_{\P}^{2\dim Y+1-i}\to \Omega_{\Phi}^{2\dim Y+1-i} \to 0 $$
shows 
\begin{equation}\label{positivity}
    R^it_*\Omega_{\P}^{2\dim Y+1-i}\cong R^ip_*\Omega_{\Phi}^{2\dim Y+1-i}\textrm{ for $i\geq \dim Y-p$}.
\end{equation}
Once again, recall from \corollaryref{cor-generalSit1}(3) that
\[ t_* \Q^H_{\P}[2\dim Y+1] \cong {\rm IC}_{\Sigma}^{H} \oplus  i_* \mathbf V^{\dim Y+1}\oplus \left(\bigoplus_{\ell=1}^{\dim Y-1} \left(i_* \mathbf V^{\dim Y+1+\ell}[-\ell] \oplus i_* \mathbf V^{\dim Y+1+\ell}(\ell)[\ell]\right)\right).\]
Taking $\cH^0{\rm Gr}_{i-N}^F{\rm DR}$ both sides (and taking account of the Tate twist and reindexing), we obtain
$$R^it_*\Omega_{\P}^{2\dim Y+1-i}\cong \cH^0{\rm Gr}_{i-N}^F{\rm DR}({\rm IC}_{\Sigma}^H(-q))\oplus\left(\bigoplus_{l=0}^{\dim Y-1}{\rm Gr}^{\dim Y+1+l-i}_{F}H^{\dim Y+l+1}({\rm Bl}_yY)\otimes\Omega_Y^{\dim Y-l}\right).$$
Applying \eqref{positivity} and Claim \ref{acc}, we conclude (since the right-most summand as well as the left hand side are both locally free) $$\cH^0{\rm Gr}_{i-N}^F{\rm DR}({\rm IC}_{\Sigma}^H(-q))=0\textrm{ for }i>\dim Y-p-1$$ and the assertion (4)(a) follows. Consequently (4)(b) follows immediately from the exact sequence \eqref{genlevelexact} and the previous parts.
\end{proof}

\subsection{Intersection cohomology} We proceed to computing the intersection cohomology for $\Sigma$. Note that our computation holds in any category of mixed sheaves, although \theoremref{thm-IHSecant} is stated for Hodge structures. Thus, in this section we consider a theory of mixed sheaves $\cM(-)$ over $k$ and let $A$ be a subfield of $\R$.

Using that $\P$ is a $\P^1$-bundle over ${\rm Hilb}^2(Y)$, we have by Subsection \ref{subsec-ClassicalCohomology}
\[ H^i_{\cM}(\P) \cong H^i_{\cM}({\rm Hilb}^2(Y)) \oplus H^{i-2}_{\cM}({\rm Hilb}^2(Y))(-1)\]
which further decomposes as 
\begin{equation} \label{eq-cohP}
\begin{split}
    H^i_{\cM}(\P) \cong &  \, H^i_{\cM}({\rm Sym}^2(Y)) \oplus H^{i-2}_{\cM}({\rm Sym}^2(Y))(-1)\\
    & \oplus \left(\bigoplus_{0 <\ell \leq \lfloor \frac{i}{2}\rfloor} H^{i-2\ell}_{\cM}(Y)(-\ell)\right) \oplus \left(\bigoplus_{0 < \ell <\lfloor \frac{i}{2}\rfloor} H^{i-2(\ell+1)}_{\cM}(Y)(-\ell-1)\right).
\end{split}
\end{equation}

We have by the isomorphism \eqref{eq-cohSym2}
\[ H^{2k+1}_{\cM}({\rm Sym}^2(Y)) \cong \bigoplus_{i \leq k} H^i_{\cM}(Y) \otimes H^{2k+1-i}_{\cM}(Y)\]
\[ H^{4k}_{\cM}({\rm Sym}^2(Y))  \cong \left(\bigoplus_{i < 2k} H^i_{\cM}(Y) \otimes H^{4k-i}_{\cM}(Y)\right) \oplus {\rm Sym}^2(H^{2k}_{\cM}(Y))\]
and for $k$ odd, we have
\[ H^{2k}_{\cM}({\rm Sym}^2(Y)) \cong \left(\bigoplus_{i < k} H^i_{\cM}(Y) \otimes H^{2k-i}_{\cM}(Y)\right) \oplus \bigwedge^2 H^k_{\cM}(Y).\]

Using \lemmaref{lem-trivialMonodromy} and \theoremref{thm-IHGeneral1}, we can prove \theoremref{thm-IHSecant} from the introduction. Before providing the proof, let us first note the following:

\begin{comment}
\begin{thm} Let $L$ be a $3$-very ample line bundle on a smooth projective variety $Y$. Let $\Sigma$ be the secant variety for the embedding determined by $L$. Then for all $j \leq 2\dim Y +1 = \dim \Sigma$, we have
\[ {\rm IH}^j(\Sigma) = \bigoplus_{\ell=1}^{\dim Y -1} H^{j-2\ell}(Y)(-\ell) \oplus \bigoplus_{\ell = \lceil \frac{j}{2} \rceil +1}^{\min\{j,\dim Y\}} \left(H^{\ell}_{\rm prim}(Y) \otimes H^{j-\ell}(Y)\right)\oplus \mathcal Q_j,\]
where $\mathcal Q_j$ depends on the class of $j$ mod $4$.

If $j=2k+1$ is odd, then
\[ \mathcal Q_j = H^{k}(Y) \otimes H^{k+1}(Y).\]

If $j= 4k$, then
\[ \mathcal Q_j = \bigwedge^2 H^{2k-1}(Y)(-1) \oplus {\rm Sym}^2 H^{2k}(Y),\]

Finally, if $j = 2k$ with $k$ odd, then
\[  =  {\rm Sym}^2 H^{k-1}(Y)(-1) \oplus \bigwedge^2 H^{k}(Y).\]

For $j >  \dim \Sigma$, the space ${\rm IH}^j(\Sigma)$ is determined by Poincar\'{e} duality applied to the above computation.
\end{thm}
\end{comment}

\begin{rmk}\label{brogrmk} \theoremref{thm-IHSecant} recovers the computation of \cite{Brogan}*{} for the $2$-secant varieties of curves. Indeed, if $\dim Y =1$, then for any $j\leq 3$, the formula simplifies to
\[ {\rm IH}^j(\Sigma) = \left(\bigoplus_{\ell = \lceil \frac{j}{2} \rceil +1}^{\min\{j,\dim Y\}} H^{\ell}_{\rm prim}(Y) \otimes H^{j-\ell}(Y)\right)\oplus \mathcal Q_j,\]
where $\mathcal Q_j$ depends on the class of $j$ mod $4$. However, the first direct sum is always 0: this is clear for $j=0$ and for $j \geq 1$, we have $\lceil \frac{j}{2}\rceil \geq 1$, but then the interval $[\lceil \frac{j}{2}\rceil +1,1]$ is empty.

If $j=2k+1$ is odd with $k =0,1$, then $\mathcal Q_j = H^{k}(Y) \otimes H^{k+1}(Y)$.

If $j= 4k$ for $k=0$ or $k=1$, then
$\mathcal Q_j = {\rm Sym}^2 H^{2k}(Y) = \Q^H(-2k)$.

Finally, if $j = 2$, then
$\mathcal Q_j =  {\rm Sym}^2 H^0(Y)(-1) \oplus \bigwedge^2 H^{1}(Y)$.
\end{rmk}

Recall the statement of \theoremref{thm-IHSecant} from the Introduction:

\begin{thm}[$=$ \theoremref{thm-IHSecant}] \label{thm-IHSecantn} Assume $L$ is 3-very ample and $\Sigma\neq\P^N$. Then for all $j \leq 2\dim Y +1 = \dim \Sigma$, we have the following (non-canonical) isomorphisms of pure weight $j$ Hodge structures:
\[ {\rm IH}^j(\Sigma) \cong \bigoplus_{\ell=1}^{\dim Y -1} H^{j-2\ell}(Y)(-\ell) \oplus \bigoplus_{\ell = \lceil \frac{j}{2} \rceil +1}^{\min\{j,\dim Y\}}\left( H^{\ell}_{\rm prim}(Y) \otimes H^{j-\ell}(Y)\right)\oplus \mathcal Q_j,\]
where $\mathcal Q_j$ depends on the class of $j$ mod $4$ as follows:
\[\mathcal Q_j=
\begin{cases}
    H^{k}(Y) \otimes H^{k+1}(Y) & \textrm{if $j=2k+1$ is odd};\\
    \bigwedge^2 H^{2k-1}(Y)(-1) \oplus {\rm Sym}^2 H^{2k}(Y) & \textrm{if $j=4k$};\\
    {\rm Sym}^2 H^{k-1}(Y)(-1) \oplus \bigwedge^2 H^{k}(Y) & \textrm{if $j = 2k$ with $k$ odd}.
\end{cases}
\]
\end{thm}

\begin{proof} By \eqref{eq-cohP}, we have one description of $H^j_{\cM}(\P)$ in terms of the cohomology of ${\rm Sym}^2(Y)$ and the cohomology of $Y$.

We get another description involving ${\rm IH}^j_{\cM}(\Sigma)$ by the Decomposition theorem. Indeed, by \theoremref{thm-IHGeneral1}, we have
\[ H^j_{\cM}(\P) \cong {\rm IH}^j_{\cM}(\Sigma) \oplus  \left(\bigoplus_{\ell = 0}^{\dim Y -1} H^{j-\dim \Sigma -\ell}_{\cM}(\mathbf V^{\dim Y + 1 +\ell})\right) \oplus \left(\bigoplus_{\ell = 0}^{\dim Y-2}H^{j-\dim Y-2-\ell}_{\cM}(\mathbf V^\ell)(-1)\right).\]

By semi-simplicity of the category of pure Hodge structures of weight $j$, it suffices to compare
\[\left(\bigoplus_{\ell = 0}^{\dim Y-1} H^{j-\dim \Sigma -\ell}_{\cM}(\mathbf V^{\dim Y + 1 +\ell})\right) \oplus \left(\bigoplus_{\ell=0}^{\dim Y -2} H^{j-\dim Y-2-\ell}_{\cM}(\mathbf V^\ell)(-1)\right)\]
with the description of $H^j_{\cM}(\P)$ from \eqref{eq-cohP}. To do this, we must compute the cohomology of the local systems $\mathbf V^\ell$.

By \lemmaref{lem-trivialMonodromy}, we have isomorphisms
\[ \mathbf V^k \cong  \begin{cases} \left(H^k_{\cM}(Y) \otimes A_Y^{\cM}\right) \oplus A_Y^{\cM}(-\frac{k}{2}) & k \in 2\Z \cap [2,2\dim(Y)-2] \\ H^k_{\cM}(Y) \otimes A_Y^{\cM} & \text{otherwise}\end{cases},\]
and thus we have isomorphisms for $\ell \geq 0$ (where ``Case 1a'' below means $\dim Y +\ell +1 \in 2\Z \cap [\dim Y +1,2\dim Y -2]$)
\footnotesize
\[ H^{j-\dim \Sigma -\ell}_{\cM}(\mathbf V^{\dim Y +1+\ell}) = \begin{cases} \left( H^{\dim Y +1+\ell}_{\cM}(Y) \otimes H^{j-\ell-\dim Y-1}_{\cM}(Y) \right) \oplus H^{j-\ell-\dim Y -1}_{\cM}(Y)(\frac{-\ell-\dim Y-1}{2})& \text{Case 1a} \\ H^{\dim Y +1+\ell}_{\cM}(Y) \otimes H^{j-\ell-\dim Y-1}_{\cM}(Y)  & \text{otherwise}\end{cases}\]
\normalsize
and (where ``Case 1b'' means $\ell \in 2\Z \cap [2,\dim Y -1)$)
\[ H^{j-\dim Y-2-\ell}_{\cM}(\mathbf V^\ell) =  \begin{cases} \left( H^{\ell}_{\cM}(Y) \otimes H^{j-2-\ell}_{\cM}(Y) \right) \oplus H^{j-2-\ell}_{\cM}(Y)(-\frac{\ell}{2})& \text{Case 1b} \\ H^{\ell}_{\cM}(Y) \otimes H^{j-2-\ell}_{\cM}(Y)  & \text{otherwise}\end{cases}.\]

We first handle the special cases: for the first direct sum, we have
\[ \bigoplus_{\dim Y + \ell \in (2\Z +1)\cap [\dim Y,2\dim Y -3]} H^{j-\ell-\dim Y -1}_{\cM}(Y)\left(\frac{-\ell-\dim Y-1}{2}\right),\]
and so if we write $\dim Y +\ell = 2k+1$, this becomes
\[ \bigoplus_{k = \lceil \frac{\dim Y -1}{2}\rceil}^{\dim Y -2} H^{j-(2k+1) -1}_{\cM}(Y)\left(\frac{-(2k+1)-1}{2}\right) = \bigoplus_{k = \lceil \frac{\dim Y -1}{2}\rceil}^{\dim Y -2} H^{j-2(k+1)}_{\cM}(Y)\left(-k-1\right).\]

In the second direct sum, we have (where the equality follows by letting $\ell = 2k$)
\[ \bigoplus_{\ell \in 2\Z \cap [2,\dim Y -1)} H^{j-2-\ell}_{\cM}(Y)\left(\frac{-\ell}{2}\right)=\bigoplus_{k \in [1,\lfloor \frac{\dim Y -1}{2}\rfloor)} H^{j-2(k+1)}_{\cM}(Y)\left(-k\right).\]

So when we take the direct sum of both terms (with the Tate twist by $(-1)$ on the second direct sum), we get
\[ \bigoplus_{k = 1}^{\dim Y -2} H^{j-2(k+1)}_{\cM}(Y)(-k-1).\]

This direct sum is exactly one of the direct sums appearing in the description of $H^j_{\cM}(\P)$ using  \eqref{eq-cohP}, so we can remove it from both descriptions.

We are left with the goal of comparing
\[\left(\bigoplus_{\ell=0}^{\dim Y -1} H^{\dim Y +1+\ell}_{\cM}(Y) \otimes H^{j-\dim Y - 1-\ell}_{\cM}(Y)\right) \oplus \left(\bigoplus_{\ell=0}^{ \dim Y -2} H^{\ell}_{\cM}(Y) \otimes H^{j-2-\ell}_{\cM}(Y)(-1)\right)\]
with
\[ H^j_{\cM}({\rm Sym}^2(Y)) \oplus H^{j-2}_{\cM}({\rm Sym}^2(Y))(-1) \oplus \left(\bigoplus_{0 <\ell \leq \lfloor \frac{i}{2}\rfloor} H^{i-2\ell}_{\cM}(Y)(-\ell)\right),\]
though we will really just compare with $H^j_{\cM}({\rm Sym}^2(Y)) \oplus H^{j-2}_{\cM}({\rm Sym}^2(Y))(-1)$.

At this point, we will break into the three cases in the theorem statement. 

\noindent\textbf{Case 1: $j = 2k+1$.} Then the direct sum
\[\bigoplus_{\ell=0}^{\dim Y -1} H^{\dim Y +1+\ell}_{\cM}(Y) \otimes H^{j-\dim Y - 1-\ell}_{\cM}(Y) = \bigoplus_{\ell=\dim Y +1}^{2\dim Y} H^{\ell}_{\cM}(Y) \otimes H^{j-\ell}_{\cM}(Y)\]
is the easiest. Indeed, we have
\[ H^j_{\cM}({\rm Sym}^2(Y)) = \bigoplus_{i = k+1}^{2k+1} H^i_{\cM}(Y) \otimes H^{2k+1-i}_{\cM}(Y),\]
and so after removing the direct sum of interest, we are left with
\[ \bigoplus_{i=k+1}^{\min\{j,\dim Y\}} H^i_{\cM}(Y) \otimes H^{2k+1-i}_{\cM}(Y).\]

Similarly, to handle the direct sum $$\bigoplus_{\ell=0}^{ \dim Y -2} H^{\ell}_{\cM}(Y) \otimes H^{j-2-\ell}_{\cM}(Y)(-1),$$ we can consider $H^{j-2}_{\cM}({\rm Sym}^2(Y))(-1) = H^{2(k-1)+1}_{\cM}({\rm Sym}^2(Y))(-1)$. We have seen above that this is equal to 
\[\bigoplus_{i \leq k-1} H^i_{\cM}(Y) \otimes H^{2k-1-i}_{\cM}(Y) (-1),\]
and thus we need to compare
\[ \bigoplus_{\ell = k}^{\dim Y -2} H^{\ell}_{\cM}(Y) \otimes H^{2k-1-\ell}_{\cM}(Y)(-1)\textrm{ with }\bigoplus_{i=k+1}^{\min\{j,\dim Y\}} H^i_{\cM}(Y) \otimes H^{2k+1-i}_{\cM}(Y).\] As $\ell \leq \dim Y -2$, we have the Lefschetz decomposition
\[ H^{\ell+2}_{\cM}(Y) = H^{\ell+2}_{\cM,\rm prim}(Y) \oplus H^\ell_{\cM}(Y)(-1),\]
and so we see that the direct sum we are interested in is a direct summand of
\[ \bigoplus_{\ell = k}^{\dim Y -2} H^{\ell+2}_{\cM}(Y) \otimes H^{2k-1-\ell}_{\cM}(Y)=\bigoplus_{\ell = k+2}^{\dim Y} H^{\ell}_{\cM}(Y) \otimes H^{2k+1-\ell}_{\cM}(Y),\]
and finally we can cancel this out yielding the description in the theorem statement.

\noindent\textbf{Case 2: $j = 4k$.} Then the direct sum
\[\bigoplus_{\ell=0}^{\dim Y -1} H^{\dim Y +1+\ell}_{\cM}(Y) \otimes H^{j-\dim Y - 1-\ell}_{\cM}(Y) = \bigoplus_{\ell=\dim Y +1}^{2\dim Y} H^{\ell}_{\cM}(Y) \otimes H^{j-\ell}_{\cM}(Y)\]
is again a direct summand of $H^j_{\cM}({\rm Sym}^2(Y))$. Indeed, we have
\[ H^j_{\cM}({\rm Sym}^2(Y)) = \left(\bigoplus_{i = 2k+1}^{4k} H^i_{\cM}(Y) \otimes H^{4k-i}_{\cM}(Y)\right) \oplus {\rm Sym}^2 H^{2k}_{\cM}(Y),\]
and so after removing that direct sum and ignoring ${\rm Sym}^2 H^{2k}_{\cM}(Y)$, we are left with
\[ \bigoplus_{i=2k+1}^{\min\{j,\dim Y\}} H^i_{\cM}(Y) \otimes H^{4k-i}_{\cM}(Y).\]

Similarly, to handle the direct sum $$\bigoplus_{\ell=0}^{ \dim Y -2} H^{\ell}_{\cM}(Y) \otimes H^{j-2-\ell}_{\cM}(Y)(-1),$$ we can consider $H^{j-2}_{\cM}({\rm Sym}^2(Y))(-1) = H^{2(2k-1)}_{\cM}({\rm Sym}^2(Y))(-1)$. We have seen above that this is equal to 
\[\left(\bigoplus_{i < 2k-1} H^i_{\cM}(Y) \otimes H^{4k-2-i}_{\cM}(Y) (-1)\right) \oplus \bigwedge^2 H^{2k-1}_{\cM}(Y).\]

As suggested by the theorem statement, we will ignore $\bigwedge^2 H^{2k-1}_{\cM}(Y)$ and cancel the first direct sum in the one we are interested in, so we need to compare
\[ \bigoplus_{\ell = 2k-1}^{\dim Y -2} H^{\ell}_{\cM}(Y) \otimes H^{4k-2-\ell}_{\cM}(Y)(-1)
\textrm{ with }\bigoplus_{i=2k+1}^{\min\{j,\dim Y\}} H^i_{\cM}(Y) \otimes H^{4k-i}_{\cM}(Y).\] As $\ell \leq \dim Y -2$, we have the Lefschetz decomposition
\[ H^{\ell+2}_{\cM}(Y) = H^{\ell+2}_{\cM,\rm prim}(Y) \oplus H^\ell_{\cM}(Y)(-1),\]
and so we see that the direct sum we are interested in is a direct summand of
\[ \bigoplus_{\ell = 2k-1}^{\dim Y -2} H^{\ell+2}_{\cM}(Y) \otimes H^{4k-2-\ell}_{\cM}(Y)=\bigoplus_{\ell = 2k+1}^{\dim Y} H^{\ell}_{\cM}(Y) \otimes H^{4k-\ell}_{\cM}(Y),\]
and finally we can cancel this out yielding the description in the theorem statement.

\noindent\textbf{Case 3: $j= 2k$ with $k$ odd.} Then, as above, we first remove the direct sum
\[\bigoplus_{\ell=0}^{\dim Y -1} H^{\dim Y +1+\ell}_{\cM}(Y) \otimes H^{j-\dim Y - 1-\ell}_{\cM}(Y)\] from $H^j_{\cM}({\rm Sym}^2(Y))$. Indeed, we have
\[ H^j_{\cM}({\rm Sym}^2(Y)) = \bigoplus_{i = k+1}^{2k} H^i_{\cM}(Y) \otimes H^{2k-i}_{\cM}(Y) \otimes \bigwedge^2 H^k_{\cM}(Y),\]
and so after removing the direct sum we have mentioned and ignoring $\bigwedge^2 H^k_{\cM}(Y)$, we are left with
\[ \bigoplus_{i=2k+1}^{\min\{j,\dim Y\}} H^i_{\cM}(Y) \otimes H^{2k-i}_{\cM}(Y).\]

Similarly, to handle the direct sum $$\bigoplus_{\ell=0}^{ \dim Y -2} H^{\ell}_{\cM}(Y) \otimes H^{j-2-\ell}_{\cM}(Y)(-1),$$ we can consider $H^{j-2}_{\cM}({\rm Sym}^2(Y))(-1) = H^{2(k-1)}_{\cM}({\rm Sym}^2(Y))(-1)$. We have seen above that this is equal to 
\[\bigoplus_{i < k-1} H^i_{\cM}(Y) \otimes H^{2(k-1)-i}_{\cM}(Y) (-1) \oplus {\rm Sym}^2 H^{k-1}_{\cM}(Y) (-1),\]
and thus we need to compare
\[ \bigoplus_{\ell = k-1}^{\dim Y -2} H^{\ell}_{\cM}(Y) \otimes H^{2k-\ell}_{\cM}(Y)(-1)
\textrm{ with }\bigoplus_{i=2k+1}^{\min\{j,\dim Y\}} H^i_{\cM}(Y) \otimes H^{2k-i}_{\cM}(Y).\] As $\ell \leq \dim Y -2$, we have the Lefschetz decomposition
\[ H^{\ell+2}_{\cM}(Y) = H^{\ell+2}_{\cM,\rm prim}(Y) \oplus H^\ell_{\cM}(Y)(-1),\]
and so we see that the direct sum we are interested in is a direct summand of
\[ \bigoplus_{\ell = k-1}^{\dim Y -2} H^{\ell+2}_{\cM}(Y) \otimes H^{2(k-1)-\ell}_{\cM}(Y)=\bigoplus_{\ell = k+1}^{\dim Y} H^{\ell}_{\cM}(Y) \otimes H^{2k-\ell}_{\cM}(Y),\]
and finally we can cancel this out yielding the description in the theorem statement.
\end{proof}

\begin{rmk} \label{rmk-2SecGaloisIH} Working in $\cM(-)$ as in Example \ref{eg-SysReal}, we see that the isomorphism in the statement of \theoremref{thm-IHSecant} is in fact an isomorphism of $G = {\rm Gal}(\overline{k}/k)$-representations.
\end{rmk}

\subsection{Singular cohomology} We now compute the singular cohomology for $\Sigma$ the secant variety of lines associated to a smooth projective variety $Y$ with $3$-very ample line bundle $L$. We continue with the notation of the previous subsections of Section \ref{sec-2Sec} and recall the diagram:
\begin{equation*}
    \begin{tikzcd} F_y:={\rm Bl}_yY\ar[r]\ar[d] &\Phi \ar[r] \ar[d,"p"] & \P \ar[d,"t"]\ar[dr, "\beta"] & \\ \left\{y\right\}\ar[r] & Y \ar[r, "i"]\arrow[rr, bend right=20, swap, "\iota"] & \Sigma \ar[r,"i_{\Sigma}"] & \P^N\end{tikzcd}
\end{equation*}

\begin{thm}[$=$ \theoremref{sing-coh2}]\label{sing-coh2n} Assume $L$ is 3-very ample and $\Sigma\neq\P^N$. Then, for all $j > 0$, we have ${\rm Gr}^W_w H^j(\Sigma) = 0$ for all $w \notin \{j-1,j\}$, and we have (non-canonical) isomorphisms of mixed Hodge structures
\[ \bigoplus_{j\in \Z} {\rm Gr}^W_{j-1} H^j(\Sigma) \cong \bigoplus_{\substack{\ell \text{ odd} \\ a \leq \dim Y - \ell}} {\rm Sym}^2 P_\ell(-a) \oplus \bigoplus_{\substack{0 < \ell \text{ even} \\ a \leq \dim Y - \ell}} \bigwedge^2 P_\ell(-a)\oplus\bigoplus_{\substack{0 < \ell_1 < \ell_2 \\ a \leq \dim Y - \ell_2}} P_{\ell_1} \otimes P_{\ell_2}(-a),\]
and
\begin{equation*}
    \begin{split}
        \bigoplus_{j\in \Z} {\rm Gr}^W_{j} H^j(\Sigma) \cong & \left(\bigoplus_{j}\bigoplus_{k=1}^{\dim Y -1} H^{j-2-2k}(Y)(-k-1) \right) \oplus \bigoplus_{  0 \leq  a\leq 2\dim Y+1} \Q^H(-a)\\
& \oplus \bigoplus_{\substack{\ell_1 < \ell_2 \\ a \leq \dim Y - \ell_2}} P_{\ell_1} \otimes P_{\ell_2}(\ell_1 - \dim Y-a-1)\\
& \oplus \bigoplus_{\substack{0 < \ell \text{ even} \\  \dim Y - \ell < a\leq 2(\dim Y - \ell)+1}} {\rm Sym}^2 P_\ell (-a) \oplus \bigoplus_{\substack{\ell \text{ odd} \\  \dim Y -\ell < a\leq 2(\dim Y - \ell)+1}} \bigwedge^2 P_\ell (-a).
    \end{split}
\end{equation*}
\end{thm}

\begin{proof} The Cartesian square in the middle of the diagram gives an exact triangle (\cite{PetersSteenbrink}*{Cor. A.14}):
\[ H^\bullet_{\cM}(\Sigma) \to H^\bullet_{\cM}(Y) \oplus H^\bullet_{\cM}(\P) \to H^\bullet_{\cM}(\Phi) \xrightarrow[]{+1}.\]

Note that the inclusion $\Phi \to \P$ commutes with the maps down to ${\rm Hilb}^2(Y)$ \cite{Ullery}*{(1.4)}. Thus, using the decomposition
\[ H^\bullet_{\cM}(\P) = H^{\bullet}_{\cM}({\rm Hilb}^2(Y)) \oplus H^{\bullet-2}_{\cM}({\rm Hilb}^2(Y))\xi,\]
to understand the restriction $H^\bullet_{\cM}(\P) \to H^\bullet_{\cM}(\Phi)$, it suffices to understand the map $$H^\bullet_{\cM}({\rm Hilb}^2(Y)) \to H^\bullet_{\cM}(\Phi)$$ and to where the hyperplane class $\xi$ maps. The first map is easy: as mentioned in Subsection \ref{subsec-ClassicalCohomology}, it is an $\mathfrak S_2$-quotient map, hence, the map on cohomology corresponds to the inclusion of the $\mathfrak S_2$-invariants.

For the tautological class $\xi$, note that the map $\beta \colon \P \to \P^N$ satisfies $\beta^* \cO_{\P^N}(1) = \cO_{\P}(1)$ \cite{Ullery}*{Pg. 7}. Hence, the restriction to $\Phi$ is the same as the pull-back from $\P^N$. But we could first pullback to $Y \subseteq \P^N$, and then pullback along $\Phi \to Y$. By definition, $\cO_{\P^N}(1)\vert_Y = L$. Hence, if we write
\[ H^\bullet_{\cM}(\Phi) = b^* H^{\bullet}_{\cM}(Y\times Y) \oplus \left(\bigoplus_{k = 1}^{\dim Y -1} H^{\bullet-2k}_{\cM}(\Delta)  E^k\right),\]
where $\Delta \subseteq Y\times Y$ is the diagonal (the center of the blow up) and $E = [E_{\Delta}]$ is the class of the exceptional divisor of the blow-up morphism $b\colon \Phi \to Y \times Y$, then
\[ \xi \mapsto b^* \pi_2^*c_1(L),\]
with no exceptional part. Moreover, from the description of $H^\bullet_{\cM}(\Phi)$ above, we see that the $\mathfrak S_2$-invariant part is
\[ H^\bullet_{\cM}(Y\times Y)^{\mathfrak S_2} \oplus \left(\bigoplus_{k=1}^{\dim Y -1} H^{\bullet-2k}_{\cM}(\Delta)E^k\right),\]
as $\Delta$ is $\mathfrak S_2$-invariant (each point is a fixed point), which gives another computation of $H^\bullet_{\cM}({\rm Hilb}^2(Y))$.

We see that the map $H^\bullet_{\cM}({\rm Hilb}^2(Y)) \to H^\bullet(\P) \to H^{\bullet}_{\cM}(\Phi)$ hits the entire direct sum $$\bigoplus_{k=1}^{\dim Y -1} H^{\bullet-2k}_{\cM}(\Delta)E^k,$$ and so the corresponding direct sum in $H^{\bullet-2}_{\cM}({\rm Hilb}^2(Y))\xi$ must contribute to the kernel.

Focusing now on the remaining terms, we have
\[ \left(H^0_{\cM}(Y)\otimes H^\bullet_{\cM}(Y)\right) \oplus H^\bullet_{\cM}(Y\times Y)^{\mathfrak S_2} \oplus H^{\bullet-2}_{\cM}(Y\times Y)^{\mathfrak S_2}\pi_2^*(c_1(L)) \to H^\bullet_{\cM}(Y\times Y),\]
where in the first term we included the $H^0_{\cM}(Y)$ part to make clear how $H^\bullet_{\cM}(Y)$ maps under the K\"{u}nneth decomposition, and we rewrote $\xi$ with the class which it maps to under the pull-back map.

We first take the quotient by the middle object, which we can think of as the (signed) symmetric tensors. Recall that the $\mathfrak S_2$-action is defined so that the non-trivial permutation acts on $v_1\otimes v_2$ (where $v_i \in H^{d_i}(Y,A)$) by $(-1)^{d_1d_2} v_2 \otimes v_1$. In this way, a symmetric tensor is a sum of elements of the form $v_1\otimes v_2 + (-1)^{d_1d_2} v_2 \otimes v_1$.

Hence, we can decompose $H^\bullet(Y\times Y,A)$ into the direct sum of symmetric and anti-symmetric tensors. Given any simple tensor $v_1\otimes v_2$, this corresponds to writing
\[  v_1 \otimes v_2 = \left(\frac{v_1\otimes v_2 + (-1)^{d_1d_2+1} v_2 \otimes v_1}{2}\right) +\left(\frac{v_1\otimes v_2 + (-1)^{d_1d_2}v_2 \otimes v_1}{2}\right),\]
so that the first summand on the right is an anti-symmetric tensor and the second one is a symmetric tensor. When we take the quotient of $H^\bullet(Y\times Y)$ by $H^{\bullet}(Y \times Y)^{\mathfrak S_2}$, it is equivalent to project the simple tensors $v\otimes w$ to the first summand in this expression.

We study the induced morphism $H^0(Y) \otimes H^\bullet(Y)$ to this quotient. It sends
\[ 1 \otimes v \to \frac{1\otimes v - v\otimes 1}{2},\]
because $1$ always has even degree. This is injective except for $\bullet =0$.

Finally, we study the morphism $H^{\bullet-2}(Y\times Y,A)^{\mathfrak S_2} \pi_2^*(c_1(L))$ to the quotient obtained by modding out by both the $\mathfrak S_2$-invariants and by the image of $H^0(Y,A) \otimes H^\bullet(Y,A)$. We consider a symmetric tensor $v_1\otimes v_2 + (-1)^{d_1d_2}v_2 \otimes v_1$, which we then cup with $\pi_2^*(c_1(L))$, yielding the element
\[ v_1 \otimes (v_2 c_1(L)) + (-1)^{d_1d_2} v_2 \otimes (v_1 c_1(L)).\]

By taking the projection to the quotient, we see that the symmetric tensor maps to
\small 
 \begin{equation}\label{eq-computeCup}
     \begin{split}
         \frac{v_1 \otimes v_2 c_1(L) + (-1)^{d_1d_2+1} (v_2 c_1(L)) \otimes v_1}{2} +(-1)^{d_1d_2}\frac{v_2 \otimes v_1c_1(L) + (-1)^{d_1d_2+1} (v_1 c_1(L)) \otimes v_2}{2}\\
         =\frac{v_1 \otimes v_2 c_1(L) + (-1)^{d_1d_2+1} (v_2 c_1(L)) \otimes v_1}{2}- \frac{(v_1 c_1(L)) \otimes v_2+ (-1)^{d_1d_2+1}v_2 \otimes v_1c_1(L)  }{2},
     \end{split}
 \end{equation}
\normalsize
which is not obvious a priori.

We introduce the following notation: the anti-symmetric tensor $v_1 \otimes v_2 + (-1)^{d_1d_2+1} v_2 \otimes v_1$ will be denoted $v_1 \twedge v_2$, and the symmetric tensor will be denoted $v_1\totimes v_2$.

Then the morphism can be rewritten as
\[ v_1 \totimes v_2 \mapsto v_1\twedge v_2 c_1(L) - v_1 c_1(L) \twedge v_2,\]
and we keep in mind that the anti-symmetric tensors moreover satisfy $1\twedge v = v \twedge 1 = 0$. For the time being, we take the direct sum over all $\bullet$, and at the end of the discussion, we will focus again on weighted pieces.

By the Lefschetz decomposition, we know an arbitrary anti-symmetric tensor is a sum of terms of the form $v_1 c_1(L)^{a_1} \twedge v_2 c_1(L)^{a_2}$, where $v_i \in H^{\ell_i}_{\rm prim}(Y,A)$ and $a_i \leq \dim Y - \ell_i$. In our quotient, we can further reduce these to
\[ v_1 \twedge v_2 c_1(L)^{a_1+a_2} = v_1 c_1(L)^{a_1+a_2} \twedge v_2,\]
where now this class is 0 if either $\ell_1$ or $\ell_2$ is $0$ or  if $a_1+a_2 > \dim Y - \max\{\ell_1,\ell_2\}$. 

Hence, we see that the quotient is spanned by anti-symmetric tensors $v_1\twedge v_2 c_1(L)^a$ with $v_i \in H^{\ell_i}_{\rm prim}(Y,A)$, $\ell_2\geq \ell_1 > 0$ and $a \leq \dim Y -\ell_2$. We see some interesting behavior for diagonal terms: let $v\in H^\ell_{\rm prim}(Y,A)$ with $\ell$ even and $a\in \Z$, and consider the relation
\[ v \twedge v c_1(L)^a = - (v c_1(L)^a \twedge v) = - v \twedge v c_1(L)^a\]
which shows that even terms behave like wedge powers, as expected.

Now that we have identified how the morphism behaves on underlying $A$-vector spaces, we can return to mixed sheaf coefficients.

Thus, if $\ell_1 = \ell_2 = \ell$, we get a term of the form
\[ \begin{cases} {\rm Sym}^2 H^{\ell}_{\cM,\rm prim}(Y) (-a) & \ell \text{ odd} \\ \bigwedge^2 H^{\ell}_{\cM,\rm prim}(Y) (-a) & \ell \text{ even} \end{cases},\]
for all $a \leq \dim Y - \ell$, and for $\ell_2 > \ell_1$, we get at term of the form $H^{\ell_1}_{\cM,\rm prim}(Y) \otimes H^{\ell_2}_{\cM,\rm prim}(Y) (-a)$ for all $a\leq \dim Y - \ell_2$.

A similar argument is more difficult for the kernel. However, with the cokernel determined, we can understand the kernel by semi-simplicity. Let $P_\ell = H^{\ell}_{\cM,\rm prim}(Y)$ for convenience.

Indeed, we can rewrite the target space (the codomain) of anti-symmetric tensors modulo $H^0_{\cM}(Y) \otimes H^\bullet_{\cM}(Y)$ as
\[ \left(\bigoplus_{0 < i < j} H^i_{\cM}(Y) \otimes H^j_{\cM}(Y)\right) \oplus \left(\bigoplus_{m\geq 0} \bigwedge^2 H^{2m}_{\cM}(Y) \oplus {\rm Sym}^2 H^{2m+1}_{\cM}(Y)\right),\]
and we can rewrite this using primitive pieces as the direct sum of the following:
\begin{equation} \label{eq-PrimCoDomain1} \bigoplus_{\substack{0 < \ell_1 + 2a_1 < \ell_2 +2a_2 \\ a_i \leq \dim Y - \ell_i}} P_{\ell_1} \otimes P_{\ell_2}(-a_1-a_2)\end{equation} 
\begin{equation} \label{eq-PrimCoDomain2} \oplus \bigoplus_{\substack{\ell_1 < \ell_2 \\ \ell_1 + 2a_1 = \ell_2 + 2a_2 \\ a_i \leq \dim Y - \ell_i}} P_{\ell_1} \otimes P_{\ell_2} (-a_1-a_2)\end{equation}
\begin{equation} \label{eq-PrimCoDomain3} \bigoplus_{\substack{\ell \text{ even} \\  a\leq \dim Y - \ell}} \bigwedge^2 P_\ell (-2a) \oplus \bigoplus_{\substack{\ell \text{ odd} \\  a\leq \dim Y - \ell}} {\rm Sym}^2 P_\ell (-2a).\end{equation}

We will use the symmetry of tensor products to rewrite this so that $\ell_1 \leq \ell_2$ in each summand, which just has the effect of also allowing for the inequality $\ell_1 +2a_1 > \ell_2 + 2a_2$ for $\ell_1 < \ell_2$. As the direct sum \eqref{eq-PrimCoDomain2} handles the case with equality, this means we can rewrite without any condition on $\ell_1+2a_1$ compared to $\ell_2 +2a_2$. The only conditions are $\ell_1,a_1$ not both zero and $a_i \leq \dim Y - \ell_i$. The result is
\begin{equation} \label{eq-SymPrimCoDomain1}\bigoplus_{\substack{\ell_1 < \ell_2 \\ \ell_1 + a_1 \neq 0 \\ a_i \leq \dim Y - \ell_i}} P_{\ell_1} \otimes P_{\ell_2} (-a_1-a_2)=\bigoplus_{\substack{0 < \ell_1 < \ell_2 \\ a \leq \dim Y - \ell_2}} P_{\ell_1} \otimes P_{\ell_2}(-a)\oplus \bigoplus_{\substack{\ell_1 < \ell_2 \\ a_1 > 0 \\ a_i \leq \dim Y - \ell_i}} P_{\ell_1} \otimes P_{\ell_2} (-a_1-a_2)\end{equation}
\begin{align} \label{eq-SymPrimCoDomain2} 
\begin{split}
     \bigoplus\limits_{\substack{a_1 < a_2 \\ \ell + a_1 \neq 0 \\ a_2 \leq \dim Y -\ell}} P_\ell \otimes P_\ell (-a_1-a_2) =&
     \bigoplus\limits_{\substack{\ell \text{ odd} \\ 0<a \leq \dim Y - \ell}} {\rm Sym}^2 P_\ell(-a)\oplus\bigoplus\limits_{\substack{\ell \text{ odd} \\  0< a\leq \dim Y - \ell}} \bigwedge\limits^2 P_\ell (-a)\\
     & \oplus \bigoplus\limits_{\substack{0 < \ell \text{ even} \\  0 < a\leq \dim Y - \ell}} {\rm Sym}^2 P_\ell (-a)\\
     & \oplus \bigoplus\limits_{\substack{0 < \ell \text{ even} \\ 0<a \leq \dim Y - \ell}} \bigwedge\limits^2 P_\ell(-a)\\
     & \oplus \bigoplus\limits_{\substack{0 < a_1 < a_2 \\ a_2 \leq \dim Y -\ell}} P_\ell \otimes P_\ell (-a_1-a_2)
\end{split}
\end{align}
\begin{equation} \label{eq-SymPrimCoDomain3} 
\begin{split}
    \bigoplus\limits_{\substack{\ell \text{ even} \\  a\leq \dim Y - \ell}} \bigwedge\limits^2 P_\ell (-2a) \oplus \bigoplus\limits_{\substack{\ell \text{ odd} \\  a\leq \dim Y - \ell}} {\rm Sym}^2 P_\ell (-2a) 
     =& \cancelto{0}{ \bigwedge\limits^2 P_0}\oplus \bigoplus\limits_{\substack{0<\ell \text{ even} \\  0\leq \dim Y - \ell}} \bigwedge\limits^2 P_\ell\\
      & \oplus  \bigoplus\limits_{\substack{\ell \text{ even} \\  0<a\leq \dim Y - \ell}} \bigwedge\limits^2 P_\ell (-2a)\\
      & \oplus\bigoplus\limits_{\substack{\ell \text{ odd} \\  0\leq \dim Y - \ell}} {\rm Sym}^2 P_\ell\\
      & \oplus  \bigoplus\limits_{\substack{\ell \text{ odd} \\  0<a\leq \dim Y - \ell}} {\rm Sym}^2 P_\ell (-2a)
\end{split}
\end{equation}

The cokernel is already written in our desired way as the direct sum of the following:
\begin{equation}\label{eq-PrimCokernel1} \bigoplus_{\substack{\ell \text{ odd} \\ a \leq \dim Y - \ell}} {\rm Sym}^2 P_\ell(-a)=\bigoplus_{\substack{\ell \text{ odd} \\ 0 \leq \dim Y - \ell}} {\rm Sym}^2 P_\ell\oplus \bigoplus_{\substack{\ell \text{ odd} \\ 0<a \leq \dim Y - \ell}} {\rm Sym}^2 P_\ell(-a)\end{equation}
\begin{equation} \label{eq-PrimCokernel2}\bigoplus_{\substack{0 < \ell \text{ even} \\ a \leq \dim Y - \ell}} \bigwedge^2 P_\ell(-a)=\bigoplus_{\substack{0 < \ell \text{ even} \\ 0 \leq \dim Y - \ell}} \bigwedge^2 P_\ell\oplus \bigoplus_{\substack{0 < \ell \text{ even} \\ 0<a \leq \dim Y - \ell}} \bigwedge^2 P_\ell(-a)\end{equation}
\begin{equation} \label{eq-PrimCokernel3} \bigoplus_{\substack{0 < \ell_1 < \ell_2 \\ a \leq \dim Y - \ell_2}} P_{\ell_1} \otimes P_{\ell_2}(-a).\end{equation}

For the domain, the space of symmetric tensors (twisted by $(-1)$ because of the class $\xi$) is 
\[ \left(\bigoplus_{i < j} H^i_{\cM}(Y) \otimes H^j_{\cM}(Y) \oplus \bigoplus_{m\geq 0} {\rm Sym}^2 H^{2m}_{\cM}(Y) \oplus \bigwedge^{2} H^{2m+1}_{\cM}(Y)\right) (-1)\]
we similarly write using primitive pieces (and using the tensor symmetry)
\begin{equation} \label{eq-SymPrimDomain1} \bigoplus_{\substack{\ell_1 < \ell_2 \\ a_i \leq \dim Y - \ell_i}} P_{\ell_1} \otimes P_{\ell_2}(-a_1-a_2-1)\end{equation}
\begin{equation} \label{eq-SymPrimDomain2} \bigoplus_{\substack{a_1 < a_2 \\ a_2 \leq \dim Y - \ell}} P_{\ell} \otimes P_{\ell}(-a_1-a_2-1)\end{equation}
\begin{equation} \label{eq-SymPrimDomain3} \bigoplus_{\substack{\ell \text{ even} \\  a\leq \dim Y - \ell}} {\rm Sym}^2 P_\ell (-2a-1) \oplus \bigoplus_{\substack{\ell \text{ odd} \\  a\leq \dim Y - \ell}} \bigwedge^2 P_\ell (-2a-1).\end{equation}

We now proceed with the cancellation: canceling the cokernel in the codomain we get

\begin{equation} \label{eq-SymPrimImage1}\bigoplus_{\substack{\ell_1 < \ell_2 \\ a_1 > 0 \\ a_i \leq \dim Y - \ell_i}} P_{\ell_1} \otimes P_{\ell_2} (-a_1-a_2)\end{equation}
\begin{equation} \label{eq-SymPrimImage2} \bigoplus_{\substack{0 < a_1 < a_2 \\ a_2 \leq \dim Y -\ell}} P_\ell \otimes P_\ell (-a_1-a_2)\end{equation}
\begin{equation} \label{eq-SymPrimImage3} \bigoplus_{\substack{\ell \text{ even} \\ 0<   a\leq \dim Y - \ell}} \bigwedge^2 P_\ell (-2a) \oplus \bigoplus_{\substack{\ell \text{ odd} \\  0< a\leq \dim Y - \ell}} {\rm Sym}^2 P_\ell (-2a).\end{equation}
\begin{equation} \label{eq-SymPrimImage4} \bigoplus_{\substack{\ell \text{ odd} \\  0< a\leq \dim Y - \ell}} \bigwedge^2 P_\ell (-a) \oplus \bigoplus_{\substack{0 < \ell \text{ even} \\  0 < a\leq \dim Y - \ell}} {\rm Sym}^2 P_\ell (-a).\end{equation}

To cancel this, we first focus on \eqref{eq-SymPrimImage3}. These appear in \eqref{eq-SymPrimDomain2} by taking $a_1' = a-1, a_2' = a$. When we mod out by this, we get
\begin{equation} \label{eq-SymPrimDomain1'} \bigoplus_{\substack{\ell_1 < \ell_2 \\ a_i \leq \dim Y - \ell_i}} P_{\ell_1} \otimes P_{\ell_2}(-a_1-a_2-1)\end{equation}
\begin{equation} \label{eq-SymPrimDomain2'} \bigoplus_{\substack{a_1 < a_2-1 \\ a_2 \leq \dim Y - \ell}} P_{\ell} \otimes P_{\ell}(-a_1-a_2-1)\end{equation}
\begin{equation} \label{eq-SymPrimDomain3'} \bigoplus_{\substack{\ell \text{ even} \\  a\leq \dim Y - \ell}} {\rm Sym}^2 P_\ell (-2a-1) \oplus \bigoplus_{\substack{\ell \text{ odd} \\  a\leq \dim Y - \ell}} \bigwedge^2 P_\ell (-2a-1).\end{equation}
\begin{equation} \label{eq-SymPrimDomain4'} \bigoplus_{\substack{\ell \text{ even} \\  0 < a\leq \dim Y - \ell}} {\rm Sym}^2 P_\ell (-2a) \oplus \bigoplus_{\substack{\ell \text{ odd} \\ 0<  a\leq \dim Y - \ell}} \bigwedge^2 P_\ell (-2a).\end{equation}
and we can collect the last two direct sums into the following:
\begin{equation} \label{eq-SymPrimDomain5'} \bigoplus_{\substack{\ell \text{ even} \\  0< a\leq 2(\dim Y - \ell)+1}} {\rm Sym}^2 P_\ell (-a) \oplus \bigoplus_{\substack{\ell \text{ odd} \\  0 < a\leq 2(\dim Y - \ell)+1}} \bigwedge^2 P_\ell (-a).\end{equation}

The space which we want to cancel is now the direct sum of:
\begin{equation} \label{eq-SymPrimImage1'}\bigoplus_{\substack{\ell_1 < \ell_2 \\ a_1 > 0 \\ a_i \leq \dim Y - \ell_i}} P_{\ell_1} \otimes P_{\ell_2} (-a_1-a_2)\end{equation}
\begin{equation} \label{eq-SymPrimImage2'} \bigoplus_{\substack{0 < a_1 < a_2 \\ a_2 \leq \dim Y -\ell}} P_\ell \otimes P_\ell (-a_1-a_2)\end{equation}
\begin{equation} \label{eq-SymPrimImage3'} \bigoplus_{\substack{\ell \text{ odd} \\  0< a\leq \dim Y - \ell}} \bigwedge^2 P_\ell (-a) \oplus \bigoplus_{\substack{0 < \ell \text{ even} \\  0 < a\leq \dim Y - \ell}} {\rm Sym}^2 P_\ell (-a).\end{equation}

Easily the sum \eqref{eq-SymPrimImage3'} cancels from \eqref{eq-SymPrimDomain5'}, and so the spaces are now
\begin{equation} \label{eq-SymPrimDomain1''} \bigoplus_{\substack{\ell_1 < \ell_2 \\ a_i \leq \dim Y - \ell_i}} P_{\ell_1} \otimes P_{\ell_2}(-a_1-a_2-1)\end{equation}
\begin{equation} \label{eq-SymPrimDomain2''} \bigoplus_{\substack{a_1 < a_2-1 \\ a_2 \leq \dim Y - \ell}} P_{\ell} \otimes P_{\ell}(-a_1-a_2-1)\end{equation}
\begin{equation} \label{eq-SymPrimDomain3''} \bigoplus_{\substack{0 < \ell \text{ even} \\  \dim Y - \ell < a\leq 2(\dim Y - \ell)+1}} {\rm Sym}^2 P_\ell (-a) \oplus \bigoplus_{\substack{\ell \text{ odd} \\  \dim Y -\ell < a\leq 2(\dim Y - \ell)+1}} \bigwedge^2 P_\ell (-a).\end{equation}
\begin{equation} \label{eq-SymPrimDomain4''} \bigoplus_{  0 < a\leq 2\dim Y+1} {\rm Sym}^2 P_0 (-a),\end{equation}
which we will rewrite below as $\bigoplus_{0 < a \leq 2\dim Y +1} A^{\cM}(-a)$.

We want to cancel from this the space
\begin{equation} \label{eq-SymPrimImage1''}\bigoplus_{\substack{\ell_1 < \ell_2 \\ a_1 > 0 \\ a_i \leq \dim Y - \ell_i}} P_{\ell_1} \otimes P_{\ell_2} (-a_1-a_2)\end{equation}
\begin{equation} \label{eq-SymPrimImage2''} \bigoplus_{\substack{0 < a_1 < a_2 \\ a_2 \leq \dim Y -\ell}} P_\ell \otimes P_\ell (-a_1-a_2)\end{equation}

These easily cancel by sending $(\ell_1,\ell_2,a_1,a_2)$ to $(\ell_1,\ell_2,a_1-1,a_2)$ and similarly $(\ell,a_1,a_2) \mapsto (\ell,a_1-1,a_2)$. The only terms which do not get cancel by this are those with $a_1 = \dim Y - \ell_1$. Note that in \eqref{eq-SymPrimDomain2''} it is impossible to have $a_1 = \dim Y - \ell$ because $a_1 < a_2-1$ would then imply $\dim Y - \ell < a_2$, contradicting choice of $a_2$. 

Thus, the end result is
\[ \bigoplus_{\substack{\ell_1 < \ell_2 \\ a \leq \dim Y - \ell_2}} P_{\ell_1} \otimes P_{\ell_2}(\ell_1 - \dim Y-a-1)\]
\[  \bigoplus_{\substack{0 < \ell \text{ even} \\  \dim Y - \ell < a\leq 2(\dim Y - \ell)+1}} {\rm Sym}^2 P_\ell (-a) \oplus \bigoplus_{\substack{\ell \text{ odd} \\  \dim Y -\ell < a\leq 2(\dim Y - \ell)+1}} \bigwedge^2 P_\ell (-a)\]
\[ \bigoplus_{0 < a \leq 2\dim Y +1} A^{\cM}(-a).\]

Recalling that we argued that the direct sum $\bigoplus_{k=1}^{\dim Y -1} H^{\bullet-2-2k}_{\cM}(Y)(-1-k)$ contributes to the kernel for each $\bullet$ and similarly that $H^0_{\cM}(Y) \otimes H^0_{\cM}(Y)$ lies in the kernel, we obtain the complete description of the singular cohomology of $\Sigma$ asserted in the theorem statement.
\end{proof}

\begin{comment}
In summary, we have the following description of the singular cohomology of $\Sigma$:
\begin{thm}\label{sing-coh2} For all $w > 0$, we have
\[ \bigoplus_{w\in \Z} {\rm Gr}^W_{w-1} H^w(\Sigma) = \]
\[\bigoplus_{\substack{\ell \text{ odd} \\ a \leq \dim Y - \ell}} {\rm Sym}^2 P_\ell(-a) \oplus \bigoplus_{\substack{0 < \ell \text{ even} \\ a \leq \dim Y - \ell}} \bigwedge^2 P_\ell(-a)\oplus\bigoplus_{\substack{0 < \ell_1 < \ell_2 \\ a \leq \dim Y - \ell_2}} P_{\ell_1} \otimes P_{\ell_2}(-a).\]

Similarly, we have
\[ \bigoplus_{w\in \Z} {\rm Gr}^W_{w} H^w(\Sigma) = \]
\[ \left(\bigoplus_{w}\bigoplus_{k=1}^{\dim Y -1} H^{w-2-2k}(Y)(-k-1) \right) \oplus \bigoplus_{  0 \leq  a\leq 2\dim Y+1} \Q^H(-a)\]
\[\oplus \bigoplus_{\substack{\ell_1 < \ell_2 \\ a \leq \dim Y - \ell_2}} P_{\ell_1} \otimes P_{\ell_2}(\ell_1 - \dim Y-a-1)\]
\[\oplus \bigoplus_{\substack{0 < \ell \text{ even} \\  \dim Y - \ell < a\leq 2(\dim Y - \ell)+1}} {\rm Sym}^2 P_\ell (-a) \oplus \bigoplus_{\substack{\ell \text{ odd} \\  \dim Y -\ell < a\leq 2(\dim Y - \ell)+1}} \bigwedge^2 P_\ell (-a).\]
\end{thm}
\end{comment}

\begin{rmk} \label{rmk-2SecGaloisSing} Working in $\cM(-)$ as in \exampleref{eg-SysReal}, we see that the isomorphism in the statement of \theoremref{sing-coh2} is in fact an isomorphism of $G = {\rm Gal}(\overline{k}/k)$-representations. So the Galois representation on the (${\rm Gr}^W_\bullet$ of) singular cohomology of $\Sigma$ is completely determined by that of $Y$.
\end{rmk}

We make explicit this formula in the low dimensional cases below. 

\begin{eg} When $\dim Y = 1$, we see that all cohomologies are pure, though $H^3(\Sigma)$ is pure of weight $2$. We have
\[H^0(\Sigma) = \Q^H\]
\[H^1(\Sigma) = 0\]
\[H^2(\Sigma) = \Q^H(-1)\]
\[H^3(\Sigma) = {\rm Gr}^W_2 H^3(\Sigma) = {\rm Sym}^2 H^1(Y)\]
\[H^4(\Sigma) = \Q^H(-2) \oplus \bigwedge^2 H^1(Y)(-1)\]
\[H^5(\Sigma) = H^1(Y)(-2)\]
\[H^6(\Sigma) = \Q^H(-3).\]
This agrees with the formula \cite{Brogan}*{}.
\end{eg}

\begin{eg}\label{explicit2coh} Let $\dim Y = 2$. By the formula, we see that $H^\bullet(\Sigma)$ is pure except for $\bullet \in \{3,4,5\}$. We have
\[ {\rm Gr}^W_2 H^3(\Sigma) = H^3(\Sigma) = {\rm Sym}^2 H^1(Y),\]
\[ {\rm Gr}^W_3 H^4(\Sigma) = H^1(Y) \otimes H^2_{\rm prim}(Y),\]
\[ {\rm Gr}^W_4 H^4(\Sigma) = \Q^H(-2) \oplus \Q^H(-2)\]
\[ {\rm Gr}^W_4 H^5(\Sigma) = {\rm Sym}^2 H^1(Y) (-1) \oplus \bigwedge^2 H^2_{\rm prim}(Y).\]
\[ {\rm Gr}^W_5 H^5(\Sigma) = H^1(Y)(-2),\]
and the rest are given by:
\[ H^0(\Sigma) = \Q^H,\]
\[ H^1(\Sigma) = 0,\]
\[ H^2(\Sigma) = \Q^H(-1),\]
\[ H^6(\Sigma) = \Q^H(-3) \oplus H^2(Y)(-2) \oplus {\rm Sym}^2 H^2_{\rm prim}(Y)(-1) \oplus \bigwedge^2 H^1(Y)(-2), \]
\[ H^7(\Sigma) = H^3(Y)(-2) \oplus H^2_{\rm prim}(Y) \otimes H^1(Y)(-2) \oplus H^1(Y)(-3)\]
\[ H^8(\Sigma) = \Q^H(-4)^{\oplus 2} \oplus H^2_{\rm prim}(Y)(-3) \oplus \bigwedge^2 H^1(Y)(-3),\]
\[ H^9(\Sigma) = H^1(Y)(-4)\]
\[ H^{10}(\Sigma) = \Q^H(-5).\]
\end{eg}

\subsection{Defect in $\Q$-factoriality}

We can apply our computations of intersection and singular cohomology to the study of the $\Q$-factoriality defect $\sigma(\Sigma)$ using \cite{PPFactorial}. For this subsection, assume $k = \C$ and $A=\Q$.

First we will need the following easy lemma:
\begin{lem} \label{lem-rationalwedge2} Let $V$ be a non-zero polarizable pure $\Q$-Hodge structure of weight $1$. Then the intersection
\[ \bigwedge^2 V_\Q \cap \left(\bigwedge^2 V\right)^{(1,1)}\]
contains any polarization of $V$, hence, is non-zero.
\end{lem}
\begin{proof} By definition, a polarization on a weight $1$ $\Q$-Hodge structure $V$ is a skew-symmetric bilinear form $S \colon V_{\Q} \otimes_\Q V_{\Q} \to \Q$ whose extension to $\C$ satisfies
\begin{equation} \label{eq-vanishing} S(V^{p,q},V^{p',q'}) =0  \text{ unless } p=p', \, q=q',\end{equation}
as well as some positivity condition.

In other words, by skew symmetry, $S$ determines an element of $\bigwedge^2 V_{\Q}$ such that, when viewed in $\bigwedge^2 V$, by the vanishing \eqref{eq-vanishing} it has no $(2,0)$ or $(0,2)$ part, as desired.
\end{proof}

We are now ready to prove

\begin{cor}[$=$ \corollaryref{corFactorial}] \label{corFactorialn} Assume $L$ is 3-very ample and $\Sigma\neq\P^N$. Further assume $\Sigma$ is normal. Then:
\begin{enumerate}
    \item $\sigma(\Sigma)<\infty$ if and only if $\mathrm{Alg}_{\Q}^1(\Sigma)\subseteq \mathrm{CDiv}_{\Q}(\Sigma)$ if and only if $H^1(Y,\mathcal{O}_Y)=0$.
     \item The value of $\sigma(\Sigma)$ is given by \[ \sigma(\Sigma) = \begin{cases} \infty & H^1(Y,\mathcal{O}_Y)\neq 0\\ h^{1,1}_{\Q}(Y) & \dim Y \geq 2 \, \textrm{ and }\, H^1(Y,\mathcal{O}_Y)=0\\ 0 & \dim Y=1 \, \textrm{ and }\, H^1(Y,\mathcal{O}_Y)=0\, (\textrm{i.e. $Y \cong \P^1$})  \end{cases}.\]
     \item The following are equivalent:
     \begin{itemize}
     \item[(a)] $\Sigma$ is locally analytically $\Q$-factorial,
         \item[(b)] $\Sigma$ is ${\Q}$-factorial,
         \item[(c)] the restriction of the cycle class map to Cartier classes $${\rm cl}_{\Q}^{\rm res}:{\rm CDiv}_{\Q}(\Sigma)\to {\rm IH}^2(\Sigma,\Q)\cap {\rm IH}^{1,1}(\Sigma)$$ is surjective,
         \item[(d)] $(Y,L)\cong({\P}^1,\mathcal{O}_{\P^1}(d))$ with $d\geq 4$.
     \end{itemize}
     Moreover, $\Sigma$ is never factorial.
     \item In addition, assume $\dim Y\geq 2$ and  $L$ satisfies $(U_2)$-property (\definitionref{def-pos}). Then $\sigma^{\rm an}(\Sigma;y)<\infty$ for any (equivalently for all) $y\in Y$ if and only if $H^1(Y,\cO_Y)=0$, in which case, for $y\in Y$ we have
    \[\sigma^{\rm an}(\Sigma;y)\leq h^2(Y)=2h^2(Y,\cO_Y)+h^{1,1}(Y)
    \]
     with equality if $H^2(Y,\cO_Y)=0$.
\end{enumerate}
\end{cor}    
\begin{proof} Throughout we apply \theoremref{sing-coh2} without further reference. 

(1) By \cite{PPFactorial}*{Thm. A}, the $\Q$-factoriality defect is finite if and only if $\mathrm{Alg}_{\Q}^1(\Sigma)\subseteq \mathrm{CDiv}_{\Q}(\Sigma)$ if and only if $h^1(\Sigma) = h^{2\dim \Sigma-1}(\Sigma)$. By our computation of singular cohomology, it is easy to see that $h^1(\Sigma) = 0$.

Now we must compute $H^{2\dim \Sigma -1}(\Sigma) = H^{4 \dim Y +1}(\Sigma)$. 

We first calculate ${\rm Gr}^W_{4\dim Y} H^{4\dim Y+1}(\Sigma)$. If the following direct sum
\[\left(\bigoplus_{\substack{\ell \text{ odd} \\ a \leq \dim Y - \ell}} {\rm Sym}^2 P_\ell(-a)\right) \oplus \left(\bigoplus_{\substack{0 < \ell \text{ even} \\ a \leq \dim Y - \ell}} \bigwedge^2 P_\ell(-a)\right)\] contributes to it, then we must have $$2\ell+2a=4\dim Y\implies \ell+a=2\dim Y.$$ But $a\leq \dim Y-\ell$ by assumption, whence $$2\dim Y-\ell\leq \dim Y-\ell\implies \dim Y\leq 0$$ which is a contradiction.
This shows that the only possible contributions coming from ${\rm Gr}^W_{4\dim Y} H^{4\dim Y+1}(\Sigma)$ is from the following 
\[\bigoplus_{\substack{0 < \ell_1 < \ell_2 \\ a \leq \dim Y - \ell_2}} P_{\ell_1} \otimes P_{\ell_2}(-a),\,\textrm{ with }\, \ell_1+\ell_2+2a=4\dim Y.
\]
Notice that $$\ell_1+\ell_2=4\dim Y-2a<2\ell_2\implies 2\dim Y-a<\ell_2.$$ But $\ell_2\leq  \dim Y-a$ which gives the contradiction $\dim Y<0$. We conclude that $${\rm Gr}^W_{4\dim Y} H^{4\dim Y+1}(\Sigma)=0.$$

Now we calculate ${\rm Gr}^W_{4\dim Y+1} H^{4\dim Y+1}(\Sigma)$. Notice that $4\dim Y+1$ is odd, so the only possible contributions are coming from the following direct sums (taking $w = 4\dim Y +1$)
\[
\left(\bigoplus_{k=1}^{\dim Y -1} H^{4\dim Y +1-2-2k}(Y)(-k-1)\right)\oplus \left(\bigoplus_{\substack{\ell_1 < \ell_2 \\ a \leq \dim Y - \ell_2}} P_{\ell_1} \otimes P_{\ell_2}(\ell_1 - \dim Y-a-1)\right).
\]

The first term is zero unless $4\dim Y +1-2-2k\leq 2\dim Y$, or in other words, unless $\dim Y - \frac{1}{2} \leq k$, but this is impossible as $k\leq \dim Y -1$ by assumption.

In the second direct sum, we want $\ell_1 < \ell_2$ and $a\leq \dim Y -\ell_2$ such that
\[ 4\dim Y +1 = \ell_1 + \ell_2 + 2(\dim Y +a+1-\ell_1) = \ell_2 + 2\dim Y +2a+2 - \ell_1,\]
or by rearranging,
\[ 2(\dim Y - a) -1 = \ell_2 - \ell_1,\]
but by assumption $2(\dim Y -a) \geq 2\ell_2$, so this would give
\[2 \ell_2 -1 \leq \ell_2 - \ell_1 \implies \ell_2 + \ell_1 \leq 1.\]

As $\ell_1 < \ell_2$, the only possibility is for equality to hold, so that $\ell_1 = 0, \ell_2 = 1$ and $\dim Y - a = \ell_2$, giving $a = \dim Y -1$. Thus, we get a contribution
\[ P_0 \otimes P_1 (-\dim Y - (\dim Y - 1)-1) = H^1(Y)(-2\dim Y).\]

Hence, we conclude
\[ H^{4\dim Y+1}(\Sigma) = H^1(Y)(-2\dim Y),\]
and the assertion follows. 

%Assume below that $H^1(Y) = 0$.

(2) By \cite{PPFactorial}*{Thm. A}, the ${\Q}$-factoriality defect $\sigma(\Sigma)$ coincides with $$\sigma'(\Sigma)=\dim_{\Q} \left({\rm IH}^2(\Sigma,\Q)\cap {\rm IH}^{1,1}(\Sigma)\right)-\dim_{\Q}\left({\rm ker}(H^2(\Sigma,\Q)\to H^2(\Sigma,\mathcal{O}_{\Sigma}))\right)$$ when $\sigma(\Sigma)<\infty$. 

We have the formula for ${\rm IH}^2(\Sigma)$: it is
\[ \begin{cases} {\rm Sym}^2 H^0(Y)(-1) \oplus \bigwedge^2 H^1(Y) & \dim Y =1 \\  H^0(Y)(-1) \oplus H^2_{\rm prim}(Y) \otimes H^0(Y) \oplus {\rm Sym}^2 H^0(Y)(-1) \oplus \bigwedge^2 H^1(Y) & \dim Y \geq 2 \end{cases}.\]
%and so under our assumption, we have
%\[ {\rm IH}^2(\Sigma) = \begin{cases} {\rm Sym}^2 H^0(Y)(-1) & \dim Y =1 \\  H^0(Y)(-1) \oplus H^2_{\rm prim}(Y) \otimes H^0(Y) \oplus {\rm Sym}^2 H^0(Y)(-1)& \dim Y \geq 2 \end{cases}.\]

In terms of the $(1,1)$ parts, note that $H^{1,1}_{\rm prim}(Y) = H^{1,1}(Y)/(\C \cdot c_1(L))$, and as $c_1(L)$ is always defined over $\Q$, we also have $H^{1,1}_{\rm prim}(Y)\cap H^2(Y,\Q) = (H^{1,1}(Y)\cap H^2(Y,\Q))/\Q \cdot c_1(L)$ hence has dimension $h_{\Q}^{1,1}(Y) -1$. 

For $\wedge^2 H^1(Y)$, we see that $(\wedge^2 H^1(Y))_\Q = \wedge^2 H^1(Y,\Q)$. However, the intersection with $H^{1,1}(Y)$ is not easy to compute, so we leave that written as \[\left(\bigwedge^2 h^{1,1}(Y)\right)_{\Q} = \dim_{\Q} \left(\left(\bigwedge^2 H^1(Y)\right)^{(1,1)} \bigcap \bigwedge^2 H^1(Y,\Q)\right).\] As $H^1(Y)$ is always a polarizable $\Q$-Hodge structure of weight $1$, this intersection vanishes if and only if $H^1(Y) = 0$ by \lemmaref{lem-rationalwedge2}.

We conclude
\begin{equation}\label{dimc}
    \dim {\rm IH}^{1,1}_{\Q}(\Sigma) = \begin{cases} 1+(\bigwedge^2 h^{1,1}(Y))_{\Q} & \dim Y =1 \\   h^{1,1}_{\Q}(Y) + (\bigwedge^2 h^{1,1}(Y))_{\Q}+1 & \dim Y \geq 2\end{cases}.
\end{equation}  

Using our cohomology computation, it is easy to see that $H^2(\Sigma,\Q) = \Q^H(-1)$ is one dimensional (for example, spanned by $c_1(L')$ for the very ample $L'$ corresponding to the embedding $\Sigma \hookrightarrow \P^N$), and we know that $c_1(L') \mapsto 0$ under the map to $H^2(\Sigma,\cO_{\Sigma})$ (by the long exact sequence induced by the exponential short exact sequence). Thus,
\[ \ker(H^2(\Sigma,\Q)  \to H^2(\Sigma,\cO_{\Sigma})) = H^2(\Sigma,\Q)\]
is $1$-dimensional. By subtracting $1$ from the above dimension \eqref{dimc}, we conclude that 
\begin{equation}\label{sigma'}
    \sigma'(\Sigma)=\begin{cases} (\bigwedge^2 h^{1,1}(Y))_{\Q} & \dim Y=1  \\ h^{1,1}_{\Q}(Y)+(\bigwedge^2 h^{1,1}(Y))_{\Q} & \dim Y \geq 2\end{cases}.
\end{equation}
The values of ${\Q}$-factoriality defect follows from the above by assuming $H^1(Y) =0$, which forces $(\bigwedge^2 h^{1,1}(Y))_\Q = 0$, since this is the only case in which $\sigma(\Sigma)$ is finite by (1).

(3) It is well-known that (a)$\implies$(b). The equivalence of (b) and (d) follows from (2). To see (d)$\implies$(c), recall from \cite{PPFactorial}*{Thm. 2.1} the cycle class map $${\rm cl}_{\Q}:{\rm Div}_{\Q}(\Sigma)\to {\rm IH}^2(\Sigma,\Q)\cap {\rm IH}^{1,1}(\Sigma)$$ is surjective, whence the conclusion follows as in this case ${\rm CDiv}_{\Q}(\Sigma)={\rm Div}_{\Q}(\Sigma)$. To see (c)$\implies$(d), we use the commutative diagram 
\[
\begin{tikzcd}
    {\rm CDiv}_{\Q}(\Sigma)\arrow[d, two heads, swap]\arrow[r, hook] & {\rm Div}_{\Q}(\Sigma)\arrow[d, two heads, "{\rm cl}_{\Q}"]\\
    {\rm ker}(H^2(\Sigma,\Q)\to H^2(\Sigma,\mathcal{O}_{\Sigma}))\arrow[r, hook] & {\rm IH}^2(\Sigma,\Q)\cap {\rm IH}^{1,1}(\Sigma)
\end{tikzcd}
\]
with the indicated maps surjective/injective by \cite{PPFactorial}*{p. 14}. Thus the composed map ${\rm cl}_{\Q}^{\rm res}:{\rm CDiv}_{\Q}(\Sigma)\to {\rm IH}^2(\Sigma,\Q)\cap {\rm IH}^{1,1}(\Sigma)$ is surjective implies that the bottom horizontal map is an isomorphism which by the equation \eqref{sigma'} implies (d). Now, (d) implies that $\Sigma$ is a rational homology manifold of dimension 3 by \theoremref{thm-Secants} (or \corollaryref{cor-chrh}), hence (a) follows by \cite{PPFactorial}*{Cor. F} (since we know that in this case $\Sigma$ has rational singularities by \cite{ChouSong}*{Cor. 1.5}).

Lastly, if $\Sigma$ is factorial, then in particular it is ${\Q}$-factorial and Gorenstein, whence by \cite{OR}*{Thm. F} we have $(Y,L)\cong (\P^1,\mathcal{O}_{\P^1}(4))$. In this case $\Sigma$ is defined by $\{f=0\}$ where $f$ is the determinant of the Hankel matrix
\[\begin{bmatrix}
Z_0 & Z_1 & Z_2\\
Z_1 & Z_2 & Z_3\\
Z_2 & Z_3 & Z_4
\end{bmatrix}\] in ${\P^4}$. However, in this case the homogeneous coordinate ring of $\Sigma$ has a non-trivial divisor class group by \cite{Singh}*{Thm. 3.1}.

(4) \textcolor{black}{Assume $L$ satisfies $(U_2)$-property. By \cite{ORS}*{Prop. 4.3}, we have $$R^i\cO_{\P}(-\Phi)=0\textrm{ for all $1\leq i\leq 2$}.$$ Consequently the exact sequence $$0\to \cO_{\P}(-\Phi)\to \cO_{\P}\to \cO_{\Phi} \to 0 $$
shows $R^1t_*\cO_{P}=0$ if and only if $R^1p_*\cO_{\Phi}=0$, which by \cite{ChouSong}*{Lem. 2.2} is equivalent to the vanishing $H^1(Y,\cO_Y)=0$. In this case \textcolor{black}{$R^2p_*\cO_{\Phi}\cong H^2(Y,\cO_Y)\otimes \cO_Y$,} whence the conclusion follows by dualizing \eqref{fact} and using \cite{PPFactorial}*{Thm. B}.}
\end{proof}

\begin{rmk}
We note that $\Sigma:=\Sigma(\P^1,\cO_{\P^1}(d))$ for $d\geq 4$ is a Mori Dream Space in the sense of \cite{HK}. Indeed, in this case $\Sigma$ has ${\Q}$-factorial singularities by above, and moreover it is Fano with log-terminal singularities by \cite{ENP}*{Thm. 1.1}. Such varieties are proven to be Mori Dream Spaces in \cite{BCHM}.
\end{rmk}

\color{black}
We record various equivalent characterizations of secant varieties of rational normal curves proven above (for more such characterizations, see \cite{OR}*{Cor. E}):

\begin{cor}\label{p1}
Assume $L$ is 3-very ample and $\Sigma\neq\mathbb{P}^N$. Then the following are equivalent:
\begin{enumerate}
    \item[(a)] $\Sigma$ is a rational homology manifold.
    \item[(b)] $(Y,L)\cong({\P}^1,\mathcal{O}_{\P^1}(d))$ with $d\geq 4$.
\end{enumerate}
If $\Sigma$ is normal, then any of the above is equivalent to any of the following:
\begin{itemize}
     \item[(c)] $\Sigma$ is locally analytically $\Q$-factorial.
         \item[(d)] $\Sigma$ is ${\Q}$-factorial.
         \item[(e)] the restriction of the cycle class map to Cartier classes $${\rm cl}_{\Q}^{\rm res}:{\rm CDiv}_{\Q}(\Sigma)\to {\rm IH}^2(\Sigma,\Q)\cap {\rm IH}^{1,1}(\Sigma)$$ is surjective.
     \end{itemize}
     If $L$ satisfies $(Q'_{\dim Y-1})$-property, then (a)-(e) is equivalent to
     \begin{enumerate}
         \item[(f)] ${\rm gl}(\cH^j_{\Sigma}(\cO_{\P^N}),F)=0$ for some $j$.
     \end{enumerate}
\end{cor}

We end this section by applying our results to secant varieties of $\P^2$:

\begin{eg}\label{secp2}
    Consider $\Sigma:=\Sigma(\P^2,\cO_{\P^2}(d))$ with $d\geq 3$. We apply our results to obtain the singularity invariants of $\Sigma$.
    
    \noindent {\it Hodge-theoretic singularity invariants.} Using \theoremref{thm-Secants}, \corollaryref{cor-chrh}, \corollaryref{hln} and \corollaryref{cor-GenLevel2Secants}, we find:
    \begin{itemize}
        \item ${\rm lcdef}(\Sigma) = {\rm lcdef}_{\rm gen}(\Sigma) = {\rm lcdef}_{\rm gen}^{>0}(\Sigma)=0$.
        \item $c(\Sigma)=\infty$, ${\rm HRH}(\Sigma)=0$, and $\Sigma_{{\rm nRS}}=\P^2$.
        \item For any $y\in\P^2$, the Hodge-Lyubeznik numbers are given by $$\lambda^{p,q}_{r,s}(\cO_{\Sigma,y})=0\textrm{ for all }p,q,r,s;$$ and the intersection Hodge-Lyubeznik numbers are given by 
        \[
        {\rm I}\lambda_r^{p,q}(\cO_{\Sigma,y})=\begin{cases}
            1 & (r,p,q)=(3,-1,-1)\\
            0 & \textrm{otherwise}
        \end{cases}.
        \]
        \item Set $N=\binom{d+2}{2}-1$ and $q=N-5$. Since $\Sigma$ is CCI (second item above), we see that $\cH_{\Sigma}^j(\cO_{\P^N})$ is non-vanishing only for $j=q$. We have the generation levels of $({\rm IC}_{\Sigma}^H(-q),F)$ and $(\cH_{\Sigma}^q(\cO_{\P^N}),F)$ given by 
        \[{\rm gl}({\rm IC}_{\Sigma}^H(-q),F)=0\textrm{ and }{\rm gl}(\cH_{\Sigma}^q(\cO_{\P^N}),F)=1.\]
    \end{itemize}
    {\it Intersection cohomology and intersection Hodge numbers.} By Poincar\'e duality, we only need to compute ${\rm IH}^j(\Sigma)$ for $0\leq j\leq 5$, which by \theoremref{thm-IHSecant} are given by: 
    \begin{center}
        \begin{tabular}{c|c|c|c|c|c|c}
        \hline
           $j$ & 0 & 1 & 2 & 3 & 4 & 5  \\
           \hline
           ${\rm IH}^j(\Sigma)$ & $\Q^H$ & $0$ & $\Q^H(-1)^{\oplus 2}$ & $0$ & $\Q^H(-2)^{\oplus 2}$  & $0$\\
           \hline
        \end{tabular}
    \end{center}
    Consequently, the intersection Hodge numbers ${\rm I\underline{h}}^{p,q}(\Sigma):=\dim {\rm Gr}^p_F{\rm IH}^{p+q}(\Sigma)$ are given by 
    \[{\rm I\underline{h}}^{p,q}(\Sigma)=\begin{cases}
        1 & (p,q)\in\left\{(0,0),(5,5)\right\}\\
        2 & (p,q)\in\left\{(1,1), (2,2),(3,3),(4,4)\right\}\\
        0 & \textrm{otherwise}
    \end{cases}.\]
\begin{comment}
    \[  
\begin{array}{ccccccccccc}
 & & & & & {\rm I\underline{h}}^{0,0} & & & & &\\
 && & & {\rm I\underline{h}}^{1,0} & & {\rm I\underline{h}}^{0,1} & &&&\\
  &&& {\rm I\underline{h}}^{2,0} & & {\rm I\underline{h}}^{1,1} & & {\rm I\underline{h}}^{0,2} &&&\\
  && {\rm I\underline{h}}^{3,0} & & {\rm I\underline{h}}^{2,1} & & {\rm I\underline{h}}^{1,2} & & {\rm I\underline{h}}^{0,3} & &\\
  & {\rm I\underline{h}}^{4,0} & & {\rm I\underline{h}}^{3,1} & & {\rm I\underline{h}}^{2,2}& & {\rm I\underline{h}}^{1,3} & & {\rm I\underline{h}}^{0,4} &\\
  {\rm I\underline{h}}^{5,0} & & {\rm I\underline{h}}^{4,1} & & {\rm I\underline{h}}^{3,2} & & {\rm I\underline{h}}^{2,3} & & {\rm I\underline{h}}^{1,4} & & {\rm I\underline{h}}^{0,5}
\end{array}
\]
are given by:
\[
\begin{array}{ccccccccccc}
 & & & & & 1 & & & & &\\
 && & & 0 & & 0 & &&&\\
  &&& 0 & & 2 & & 0 &&&\\
  && 0 & & 0 & & 0 & & 0 & &\\
  & 0 & & 0 & & 2& & 0 & & 0 &\\
  0 & & 0 & & 0 & & 0 & & 0 & & 0
\end{array}
\]
\end{comment}
{\it Intersection cohomology and Hodge-Du Bois numbers.} Using \theoremref{sing-coh2} (or \exampleref{explicit2coh}), we see that $H^j(\Sigma)$ is pure of weight $j$, and are given by 
\begin{center}
        \begin{tabular}{c|c|c|c|c|c|c|c|c|c|c|c}
        \hline
           $j$ & 0 & 1 & 2 & 3 & 4 & 5 & 6 & 7 & 8 & 9 & 10 \\
           \hline
           $H^j(\Sigma)$ & $\Q^H$ & $0$ & $\Q^H(-1)$ & $0$ & $\Q^H(-2)$  & $0$ & $\Q^H(-3)^{\oplus 2}$ & $0$ & $\Q^H(-4)^{\oplus 2}$ & $0$ & $\Q^H(-5)$ \\
           \hline
        \end{tabular}
    \end{center}
    Consequently, the Hodge-Du Bois numbers $\underline{h}^{p,q}(\Sigma):=\dim {\rm Gr}^p_FH^{p+q}(\Sigma)$ are given by
    \[
    \underline{h}^{p,q}(\Sigma)=\begin{cases}
        1 & (p,q)\in\left\{(0,0), (1,1), (2,2), (5,5)\right\}\\
        2 & (p,q)\in \left\{(3,3),(4,4)\right\}\\
        0 & \textrm{otherwise}
    \end{cases}.
    \]
    \begin{comment}
    \[
\begin{array}{ccccccccccc}
 & & & & & \underline{h}^{0,0} & & & & &\\
 && & & \underline{h}^{1,0} & & \underline{h}^{0,1} & &&&\\
  &&& \underline{h}^{2,0} & & \underline{h}^{1,1} & & \underline{h}^{0,2} &&&\\
  && \underline{h}^{3,0} & & \underline{h}^{2,1} & & \underline{h}^{1,2} & & \underline{h}^{0,3} & &\\
  &\underline{h}^{4,0} & & \underline{h}^{3,1} & & \underline{h}^{2,2}& & \underline{h}^{1,3} & & \underline{h}^{0,4} &\\
  \underline{h}^{5,0} & & \underline{h}^{4,1} & & \underline{h}^{3,2} & &\underline{h}^{2,3} & & \underline{h}^{1,4} & & \underline{h}^{0,5}\\
  & \underline{h}^{4,0} & & \underline{h}^{3,1} & & \underline{h}^{2,2}& & \underline{h}^{1,3} & & \underline{h}^{0,4} &\\
  && \underline{h}^{3,0} & & \underline{h}^{2,1} & & \underline{h}^{1,2} & & \underline{h}^{0,3} & &\\
  &&& \underline{h}^{2,0} & & \underline{h}^{1,1} & & \underline{h}^{0,2} &&&\\
  && & & \underline{h}^{1,0} & & \underline{h}^{0,1} & &&&\\
  & & & & & \underline{h}^{0,0} & & & & &\\
\end{array}
\]
are given by:
\[
\begin{array}{ccccccccccc}
 & & & & & 1 & & & & &\\
 && & & 0 & & 0 & &&&\\
  &&& 0 & & 1 & & 0 &&&\\
  && 0 & & 0 & & 0 & & 0 & &\\
  & 0 & & 0 & & 1& & 0 & & 0 &\\
  0 & & 0 & & 0 & & 0 & & 0 & & 0\\
  & 0 & & 0 & & 2& & 0 & & 0 &\\
  && 0 & & 0 & & 0 & & 0 & &\\
  &&& 0 & & 2 & & 0 &&&\\
  && & & 0 & & 0 & &&&\\
  & & & & & 1 & & & & &
\end{array}
\]
\end{comment}
{\it Defect in $\Q$-factoriality.} Using \corollaryref{corFactorial}, we see that the $\Q$-factoriality defect and its local analytic analogue are given by $$\sigma(\Sigma)=1,\textrm{ and }\sigma^{\rm an}(\Sigma;y)=1\textrm{ for any }y\in\P^2.$$
\end{eg}

\color{black}
\section{Cohomology and Singularities of Higher Secant Varieties of Curves} \label{sec-HigherSec}

Throughout this subsection, $C\subset\P^N$ is a smooth projective curve embedded by the complete linear series of a $(2k-1)$-very ample line bundle $L$. The aim is to study its $k$th secant variety $\sigma_k$. We resume the notation from Subsection \ref{subhs}.

\subsection{Local cohomological defect and related invariants} We return to the situation of a general theory of mixed sheaves $\cM(-)$ with coefficient field $A\subseteq \R$ and ground field $k$ (embeddable into $\C$). %We now let $L$ be a $(2k-1)$-very ample line bundle on a smooth projective curve $C$, and we use the same notation as in \ref{sec-Prelim} above.

For notational convenience, in this subsection we omit push-forwards by closed embeddings. Moreover, we let $p$ denote any of the natural morphisms whose target is a secant variety (so $p\colon Z^k_m \to \sigma_m$ and $p \colon C^{(a)} \times B^b \to \sigma_b$). Equivalently, $p$ will always denote the natural morphisms to $\P(H^0(C,L))$.

We proceed to prove \theoremref{thm-ConstantSheafHigherSecants} whose statement we recall:

\begin{thm}[$=$ \theoremref{thm-ConstantSheafHigherSecants}]\label{thm-ConstantSheafHigherSecantsn} Let $L$ be $(2k-1)$-very ample and assume $\sigma_k\neq\P^N$ for some $k\geq 2$. Then:
\begin{enumerate}
    \item $\Q_{\sigma_k}[2k-1]$ is perverse (equivalently ${\rm lcdef}(\sigma_k)=0$).
    \item We have isomorphisms
\[ {\rm Gr}^W_{2k-1-\ell} \Q^H_{\sigma_k}[2k-1] = {\rm Sym}^{\ell}(H^1(C)) \boxtimes {\rm IC}_{\sigma_{k-\ell}}\]
and dually
\[ {\rm Gr}^W_{N+q_k+\ell} \cH^{q_k}_{\sigma_k}(\cO_{\P^N}) \cong {\rm Sym}^{\ell}(H^1(C)) \boxtimes {\rm IC}_{\sigma_{k-\ell}}(-q_k-\ell).\]
\end{enumerate}
\end{thm}

Let us begin with the observation (due to \cite{Brogan}) that the morphism $p\colon B^k \to \sigma_k$ is semi-small. This will also give an explicit representative of $p_* A^{\cM}_{B^c}[2c-1]$ for all $c$, thanks to \cite{DCMSemismall}*{Rmks 3.2.2, 3.2.3}. Note that the result in \emph{loc. cit.} is stated for perverse sheaves, but the compatibilities in the theory of $A$-mixed sheaves ensure that it holds in our setting.

\begin{lem} \label{lem-semismall} For all $c\geq 1$, the map $B^c \to \sigma_c$ is semi-small, and we have a canonical isomorphism
\[ p_* A^{\cM}_{B^c}[2c-1] = \bigoplus_{\ell \leq c} H^{2(c-\ell)}_{\cM}(C^{(c-\ell)}) \boxtimes {\rm IC}_{\sigma_\ell}^{\cM}.\]
\end{lem}
\begin{proof} It suffices to prove the claim for the underlying $A$-objects.

For the map $p \colon B^c \to \sigma_c$, note that
\[ \{y \in \sigma_c \mid \dim p^{-1}(y) \geq \ell\} = \sigma_{c-\ell},\]
and moreover that, when restricted over $U^{c-\ell} = \sigma_{c-\ell} \setminus \sigma_{c-\ell-1}$, the map is the projection from a product with fiber $C^{(\ell)}$. 

The defect of semi-smallness is then
\[ \max_{\ell \geq 0} 2\ell - \dim B^c +\dim \sigma_{c-\ell} = 0,\]
and so the morphism is semi-small and every stratum $U^{c-\ell}$ is relevant. Hence, we get an isomorphism
\[ p_* A_{B^c}[2c-1] \cong \bigoplus_{\ell \leq c} {\rm IC}_{\sigma_{c-\ell}}(\cL_\ell),\]
where $\cL_\ell$ is the local system on $U^{c-\ell}$ given by
\[ y \mapsto H^{\dim B^c - \dim \sigma_{c-\ell}}(p^{-1}(y)).\]

As $p^{-1}(y) = C^{(\ell)}$ is irreducible, the monodromy action is trivial, and this proves the claim because $\dim B^c - \dim \sigma_{c-\ell} = 2c-1 - (2(c-\ell)-1) = 2\ell$.
\end{proof}

We have the exact triangle
\[ A_{\sigma_k}^\cM \to p_* A_{B^k}^{\cM} \to  \cC^k \xrightarrow[]{+1},\]
where the cone is supported on $\sigma_{k-1}$. By restricting to $\sigma_{k-1}$, we get
\[ A_{\sigma_{k-1}}^\cM \to p_* A_{Z^k_{k-1}}^{\cM} \to \cC^k \xrightarrow[]{+1}.\]

The first triangle shows that we have isomorphisms
\[ \cH^{j-1}\cC^k \cong \cH^j A^{\cM}_{\sigma_k} \text{ for } j < 2k-1\]
\[ {\rm Gr}^W_{\ell} \cH^{2k-2}\cC^k \cong {\rm Gr}^W_\ell \cH^{2k-1} A^{\cM}_{\sigma_k} \text{ for } \ell < 2k-1,\]
and so to understand $A_{\sigma_k}^\cM$, it suffices to study $\cC^k$.

Using the resolution of singularities $\alpha \colon C^{(k-m)} \times B^m \to Z^k_m$, we have triangles
\[ A_{Z^k_m}^\cM \to \alpha_* A^{\cM}_{C^{(k-m)} \times B^m} \to  \Sigma^{k,m}_{m-1}\xrightarrow[]{+1},\]
where now the cone is supported on $Z^k_{m-1}$. If we restrict to $Z^k_{m-1}$, we get a triangle
\[ A_{Z^k_{m-1}}^\cM \to \alpha_* A^{\cM}_{C^{(k-m)} \times Z^m_{m-1}} \to\Sigma^{k,m}_{m-1} \xrightarrow[]{+1},\]
where we also use $\alpha$ to denote the natural morphism $C^{(k-m)} \times Z^m_{m-1} \to Z^k_{m-1}$.

More generally (and to explain the above notation), we look at triangles for any $a>b>c\geq 1$:
\[ A^{\cM}_{Z^a_c} \to \alpha_* A^{\cM}_{C^{(a-b)}\times Z^b_c} \to \Sigma^{a,b}_c \xrightarrow[]{+1},\]
which is induced by the natural morphism $\alpha \colon C^{(a-b)} \times Z^b_c \to Z^a_c$.

We have the commutative diagram
\[ \begin{tikzcd} C\times B^{k-1} \ar[r] \ar[d] & B^{k-1} \ar[d] \\ Z^{k}_{k-1} \ar[r] & \sigma_{k-1}\end{tikzcd},\]
which gives the commutative diagram
\[ \begin{tikzcd} p_*A^{\cM}_{C\times B^{k-1}} & p_*A^{\cM}_{B^{k-1}} \ar[l] \\ p_* A^{\cM}_{Z^{k}_{k-1}} \ar[u] & A^{\cM}_{\sigma_{k-1}} \ar[l] \ar[u] \end{tikzcd}.\]

We let $T_k = {\rm cone}(A^{\cM}_{\sigma_{k-1}} \to p_* A^{\cM}_{C \times B^{k-1}})$. By the octahedral axiom applied to the composition of these morphisms (in either direction along the commutative square), we have two exact triangles (using the notation above):
\begin{equation} \label{eq-Triangle1Octahedral} H^0_{\cM}(C) \boxtimes  \cC^{k-1} \to T_k \to \tau^{>0} \kappa_* A^\cM_C\boxtimes p_* A^{\cM}_{B^{k-1}} \xrightarrow[]{+1},\end{equation}
\begin{equation} \label{eq-Triangle2Octahedral} \cC^k \to T_k \to p_* \Sigma^{k,k-1}_{k-2} \xrightarrow[]{+1}.\end{equation}

Triangle \eqref{eq-Triangle1Octahedral} is split. Indeed, by construction we have vanishing $\cH^\ell ( H^0_{\cM}(C) \boxtimes  \cC^{k-1}) = H^0_{\cM}(C) \boxtimes \cH^{\ell} \cC^{k-1} = 0$ for $\ell > 2k-3$ and isomorphisms \[\cH^\ell(\tau^{>0} \kappa_* A^\cM_C\boxtimes p_* A^{\cM}_{B^{k-1}} ) = \begin{cases} H^1_{\cM}(C) \boxtimes \cH^{2k-3} p_* A^{\cM}_{B^{k-1}} & \ell = 2k-2 \\ H^2_{\cM}(C)\boxtimes \cH^{2k-3} p_* A^{\cM}_{B^{k-1}} & \ell = 2k-1 \\ 0 & \text{otherwise} \end{cases},\]
so that the morphism
\[ \tau^{>0} \kappa_* A^{\cM}_C \boxtimes p_* A^{\cM}_{B^{k-1}}[-1] \to H^0_{\cM}(C)\boxtimes \cC^{k-1}\]
is 0.

In particular, we have isomorphisms
\[ \cH^\ell T_k \cong \begin{cases}  H^0_{\cM}(C) \boxtimes  \cH^{\ell} \cC^{k-1}  & \ell \leq 2k-3 \\ H^1_{\cM}(C) \boxtimes \cH^{2k-3} p_* A^{\cM}_{B^{k-1}} & \ell = 2k-2 \\ H^2_{\cM}(C)\boxtimes \cH^{2k-3} p_* A^{\cM}_{B^{k-1}} & \ell = 2k-1 \\ 0 & \text{otherwise}\end{cases}.\]

From Triangle \eqref{eq-Triangle2Octahedral}, we get the long exact sequence
\[ \dots \to \cH^{2k-4} \cC^k\to H_{\cM}^0(C)\boxtimes \cH^{2k-4} \cC^{k-1} \xrightarrow[]{\alpha} \cH^{2k-4}p_* \Sigma^{k,k-1}_{k-2}\]
\[ \to \cH^{2k-3} \cC^k \to H_{\cM}^0(C)\boxtimes \cH^{2k-3} \cC^{k-1} \xrightarrow[]{\beta} \cH^{2k-3}p_* \Sigma^{k,k-1}_{k-2}\]
\[\xrightarrow[]{\tau} \cH^{2k-2} \cC^k \to H_{\cM}^1(C)  \boxtimes\cH^{2k-3} p_* A^{\cM}_{B^{k-1}} \xrightarrow[]{\gamma} \cH^{2k-2} p_* \Sigma^{k,k-1}_{k-2}\]
\[ \to \cH^{2k-1} \cC^k \xrightarrow[]{\chi} H_{\cM}^2(C)\boxtimes \cH^{2k-3} p_* A^{\cM}_{B^{k-1}}  \to \dots.\]

\begin{rmk} Our computation below (which is independent of the current discussion) will show the following vanishing on the cohomology of $p_* \Sigma^{k,k-1}_{k-2}$:
\begin{equation}\label{itm-cohSigma} \cH^\ell p_* \Sigma^{k,k-1}_{k-2} = 0 \text{ for all } \ell \notin \{2k-4,2k-3,2k-2\}.\end{equation}

There is another way to see a slightly weaker vanishing statement. Indeed, we have the triangle in $D^b(\cM(\sigma_{k-1}))$:
\[ p_* A^{\cM}_{Z^k_{k-1}} \to p_* A^{\cM}_{C\times B^{k-1}} \to p_* \Sigma^{k,k-1}_{k-2}\xrightarrow[]{+1}.\]

We can write $p_* A^{\cM}_{C\times B^{k-1}} = \kappa_* A^{\cM}_C \boxtimes p_* A^{\cM}_{B^{k-1}}$, and so by \lemmaref{lem-semismall} we know that the central object in the triangle only has non-zero cohomology in degrees $2k-3,2k-2$ and $2k-1$.

Using the stratification of the map $p \colon Z^k_{k-1} \to \sigma_{k-1}$, it is easy to see that the semi-smallness defect is exactly equal to $1$. Moreover, as $Z^k_{k-1}$ is a hypersurface in $B^{k-1}$, we know that $A^{\cM}_{Z^k_{k-1}}$ has a unique non-vanishing cohomology in degree $2k-2 = \dim Z^k_{k-1}$. Thus, the only possibly non-zero cohomology objects of $p_* A^{\cM}_{Z^k_{k-1}}$ lie in degrees $2k-3,2k-2$ and $2k-1$. By the long exact sequence in cohomology, we conclude that
\[ \cH^\ell p_* \Sigma^{k,k-1}_{k-2} = 0 \text{ for all } \ell \notin \{2k-4,2k-3,2k-2,2k-1\}.\]
\end{rmk} 

We can now make some reductions concerning the long exact sequence:
\begin{lem} \label{lem-chiIso} The morphism $\chi$ is an isomorphism.
\end{lem}
\begin{proof} By the vanishing \eqref{itm-cohSigma}, we see that $\chi$ is a surjection between pure objects of weight $2k-1$. The domain, by definition, is
\[ \cH^{2k-1} \cC^k =(\cH^{2k-1} p_* A_{B^{k}}^{\cM})/{\rm IC}_{\sigma_k}^{\cM} = \bigoplus_{\ell <k} H^{2(k-\ell)}_{\cM}(C^{(k-\ell)}) \boxtimes {\rm IC}_{\sigma_\ell}^{\cM}.\]

The target is
\[  \bigoplus_{\ell \leq k-1} (H^2_{\cM}(C)\otimes H^{2(k-1-\ell)}_{\cM}(C^{(k-1-\ell)})) \boxtimes {\rm IC}_{\sigma_\ell}^{\cM},\]
and so the domain and target have isomorphic simple components, which by surjectivity shows that the map is an isomorphism.
\end{proof}

The vanishing \eqref{itm-cohSigma} also implies that we have isomorphisms $\cH^{\ell} \cC^k \to H_{\cM}^0(C)\boxtimes  \cH^{\ell} \cC^{k-1} = 0$ for all $\ell < 2k-4$, where the vanishing holds by induction. So the following long exact sequence captures everything we need:
\[ 0 \to \cH^{2k-4} \cC^k\to H_{\cM}^0(C)\boxtimes \cH^{2k-4} \cC^{k-1} \xrightarrow[]{\alpha} \cH^{2k-4}p_* \Sigma^{k,k-1}_{k-2}\]
\[ \to \cH^{2k-3} \cC^k \to H_{\cM}^0(C)\boxtimes \cH^{2k-3} \cC^{k-1} \xrightarrow[]{\beta} \cH^{2k-3}p_* \Sigma^{k,k-1}_{k-2}\]
\[\xrightarrow[]{\tau} \cH^{2k-2} \cC^k \to H_{\cM}^1(C)  \boxtimes\cH^{2k-3} p_* A^{\cM}_{B^{k-1}} \xrightarrow[]{\gamma} \cH^{2k-2} p_* \Sigma^{k,k-1}_{k-2}\to 0.\]

Note that $H^0_{\cM}(C) \boxtimes  \cH^{2k-3} \cC^{k-1}$ is pure of weight $2k-3$ and $\cH^1_{\cM}(C)\boxtimes \cH^{2k-3}p_* A^{\cM}_{B^{k-1}}$ is pure of weight $2k-2$. Thus, we conclude
\begin{equation} \label{eq-lowerWeights} {\rm Gr}^W_{w}(\tau) \colon {\rm Gr}^W_w \cH^{2k-3} p_* \Sigma^{k,k-1}_{k-2} \to {\rm Gr}^W_w \cH^{2k-2} \cC^k \text{ is } \begin{cases} \text{ an isomorphism } & w < 2k-3 \\ \text{ surjective } & w = 2k-3\end{cases}\end{equation}

Thus, we see that the proof of \theoremref{thm-ConstantSheafHigherSecants} reduces to the following technical proposition:
\begin{prop} \label{prop-ComputeSigma} In the above notation, we have
\begin{enumerate} \item \label{itm-GrW2k-2} We have an isomorphism
\[ \cH^{2k-2} p_* \Sigma^{k,k-1}_{k-2} \cong \bigoplus_{\ell \leq k-2} H_{\cM}^1(C) \wedge H_{\cM}^2(C) \otimes H_{\cM}^{2(k-\ell-2)}(C^{(k-\ell-2)})\boxtimes {\rm IC}_{\sigma_\ell}^{\cM},\]
\item \label{itm-GrW2k-3} $W_{2k-3} \cH^{2k-3}p_* \Sigma^{k,k-1}_{k-2} = \cH^{2k-3} p_*\Sigma^{k,k-1}_{k-2}$, we have an isomorphism
\[ {\rm Gr}^W_{2k-3} \cH^{2k-3} p_* \Sigma^{k,k-1}_{k-2} \cong {\rm Sym}^2(H^1_{\cM}(C)) \boxtimes {\rm IC}^{\cM}_{\sigma_{k-2}} \oplus\bigoplus_{\ell \leq k-2} H_{\cM}^2(C)\otimes H_{\cM}^{2(k-\ell-2)}(C^{(k-\ell-2)})\boxtimes {\rm IC}_{\sigma_\ell}^{\cM}\]
and for $\ell \geq 3$ we have an isomorphism
\[ {\rm Gr}^W_{2k-1-\ell} \cH^{2k-3} p_* \Sigma^{k,k-1}_{k-2} \cong {\rm Sym}^{\ell}(H_{\cM}^1(C)) \boxtimes {\rm IC}_{\sigma_{k-\ell}}^{\cM}\]
\item \label{itm-betaInj} The map $\beta$ is injective.
\item \label{itm-alphaIso} The map $\alpha$ is an isomorphism.
\end{enumerate}
\end{prop}

\begin{proof}[Proof that \propositionref{prop-ComputeSigma} implies \theoremref{thm-ConstantSheafHigherSecants}] If $\beta$ is injective and $\alpha$ is an isomorphism, then we conclude that $\cH^{2k-4} \cC^k = \cH^{2k-3} \cC^k = 0$, hence $A_{\sigma_k}[2k-1]$ is perverse.

 If $\beta$ is injective, then by \lemmaref{lem-semismall}, the description of ${\rm Gr}^W_{2k-3} \cH^{2k-3} p_*\Sigma^{k,k-1}_{k-2}$ in Item \eqref{itm-GrW2k-3}, and the surjectivity of ${\rm Gr}^W_{2k-3}(\tau)$ in \eqref{eq-lowerWeights} above, we conclude that 
\[ {\rm Gr}^W_{2k-3} A_{\sigma_k}^{\cM}[2k-1] \cong {\rm Gr}^W_{2k-3}\cH^{2k-2} \cC^k \cong {\rm Sym}^2(H^1_{\cM}(C)) \boxtimes {\rm IC}_{\sigma_{k-2}}^{\cM}.\]

For $\ell > 3$, using that ${\rm Gr}^W_{2k-\ell}(\tau)$ is an isomorphism (as in \eqref{eq-lowerWeights}), then by Item \eqref{itm-GrW2k-3} we get an isomorphism
\[ {\rm Gr}^W_{2k-\ell} A_{\sigma_k}^{\cM}[2k-1] \cong {\rm Gr}^W_{2k-\ell} \cH^{2k-2} \cC^k \cong {\rm Sym}^\ell(H^1_{\cM}(C)) \boxtimes {\rm IC}_{\sigma_{k-\ell}}^{\cM}.\]

Finally, as $W_{2k-3} \cH^{2k-3} p_* \Sigma^{k,k-1}_{k-2} = 0$, we see that
\[ {\rm Gr}^W_{2k-2} \cC^k = \ker(\gamma) \cong H^1_{\cM}(C) \boxtimes {\rm IC}_{\sigma_{k-1}}^{\cM}\]
where the isomorphism follows from surjectivity of $\gamma$ and the description of the domain in \lemmaref{lem-semismall} and the description of the target in \eqref{itm-GrW2k-2}.
\end{proof}

To prove Items \eqref{itm-betaInj} and \eqref{itm-alphaIso}, we will prove the corresponding claim for ${\rm Gr}^W_w(-)$ for all $w\in \Z$. The following observation will be helpful, as it means we only need to understand the lower weights for the morphism $\alpha$, which we can understand via the property \eqref{eq-lowerWeights} (for $\sigma_{k-1}$).

\begin{lem} \label{lem-GrWAlphaInj} The map ${\rm Gr}^W_{2k-3}(\alpha)$ is injective.
\end{lem}
\begin{proof} We have an isomorphism
\[ \cH^{2k-4}\cC^k \cong \cH^{2k-3} A^{\cM}_{\sigma_k},\]
which has weights $\leq 2k-4$ by \cite{PPLefschetz}*{}. Thus, we get
\[ \ker({\rm Gr}^W_{2k-3}(\alpha)) = {\rm Gr}^W_{2k-3} \cH^{2k-4}\cC^k = 0.\]
\end{proof}

So to understand the remaining weighted pieces of $\alpha$, we can study the composed morphism
\[ \alpha' \colon H^0_{\cM}(C)\boxtimes p_* \Sigma^{k-1,k-2}_{k-3}[-1] \to H^0_{\cM}(C)\boxtimes  \cC^{k-1} \to p_* \Sigma^{k,k-1}_{k-2}.\]

We proceed to give some inductive understanding of $p_* \Sigma^{a,b}_{c}$ as $a>b>c \geq 1$ vary. Our inductive description is quite notationally heavy, though we will see that for $a=b+1=c+2$ the description simplifies quite a bit.

\begin{rmk} Our arguments concern a $3\times 3$-diagram with exact triangles as rows and columns. It is well-known (\cite{Neeman}) that such constructions do not behave well at the triangulated category level of $D^b(\cA)$. For example, cones are not functorial in that setting.

So below we work in a dg-enhancement of $D^b(\cM(\sigma_k))$, where the $3\times 3$-diagram holds as an easy exercise due to the definition of the mapping cone construction. For example, in the category of chain complexes, our sign convention for the mapping cone is the following: if $f\colon A \to B$ is a morphism, then $C(f) = A[1]\oplus B$ with differential $\begin{pmatrix} -d_A & 0 \\ -f & d_B\end{pmatrix}$. 

If $\begin{tikzcd} A\ar[r,"f"] \ar[d,"p"] & B \ar[d,"q"]\\ C \ar[r,"g"] & D \end{tikzcd}$ is a commutative diagram, then the morphism $C(f) \to C(g)$ is defined by 
\[A[1] \oplus B \xrightarrow[]{\begin{pmatrix} p & 0 \\ 0 & q \end{pmatrix}} C[1] \oplus D.\]

Then $C(C(f) \to C(g))$ is (up to reordering the direct summands) $A[2] \oplus B[1] \oplus C[1] \oplus D$ with signs on the differentials given by $d_A, -d_B,-d_C,d_D,-f,-g,-p,-q$. It is easy to see that $C(C(p) \to C(q))$ is isomorphic to the same object (again, just up to reordering the direct summands), and the corresponding $3\times 3$-diagram is commutative with all rows and columns exact triangles (in the sense that the third term in any row or column is the mapping cone of the morphism between the first two terms).
\end{rmk}

Our starting point is the following $3\times 3$-diagram,
\begin{equation} \label{eq-inductiveStructure} \begin{tikzcd}  p_*A^{\cM}_{Z^a_{c}} \ar[r] \ar[d] & p_* A^{\cM}_{ C^{(a-c)}\times B^c} \ar[r] \ar[d]  &  p_* \Sigma^{a,c}_{c-1} \ar[r,"+1"] \ar[d] & {}\\   p_* A^{\cM}_{C^{(a-b)} \times Z^b_{c}} \ar[r]\ar[d]   & p_* A^{\cM}_{C^{(a-b)} \times C^{(b-c)} \times B^c}  \ar[r]\ar[d]  &  \kappa_* A^{\cM}_{C^{(a-b)}}\boxtimes  p_* \Sigma^{b,c}_{c-1}  \ar[d,"\widetilde{\alpha}"] \ar[r,"+1"]  & {} \\ p_* \Sigma^{a,b}_c \ar[r] & \eta^{a,b}_c \boxtimes p_* A^{\cM}_{B^c}  \ar[r] & S^{a,b}_c \ar[r,"+1"] & {}\end{tikzcd},\end{equation}
and if we take ${\rm Gr}^W_w(-)$ (an exact functor) of this diagram, we get a resulting $3\times 3$-diagram with exact rows and columns. Here $\eta^{a,b}_c = C( \kappa_* A^{\cM}_{C^{(a-c)}} \to \kappa_* A^{\cM}_{C^{(a-b)} \times C^{(b-c)}})$ which is related to the pull-back morphism $H^j(C^{(a-b)}) \to H^j(C^{(a-c)} \times C^{(b-c)})$ on cohomology.

We will see inductively that we can find $V^{a,b}_{w,c,\ell} ,H^{a,b}_{w,c,\ell} \in D^b(\cM_{\rm cons}(w+1-2\ell))$ such that we have equality
\[ {\rm Gr}^W_w p_* \Sigma^{a,b}_{c} = \bigoplus_{\ell \leq c} V^{a,b}_{w,c,\ell} \boxtimes {\rm IC}_{\sigma_{\ell}}^{\cM},\]
\[ {\rm Gr}^W_w S^{a,b}_c = \bigoplus_{\ell < c} H^{a,b}_{w,c,\ell} \boxtimes {\rm IC}_{\sigma_{\ell}}^{\cM}.\]

\begin{eg} \label{eg-baseCase} The base-case for the induction is $a > b > c =1$. Then the diagram is
\[ \begin{tikzcd}  p_*A^{\cM}_{Z^a_{1}} \ar[r,"\cong"] \ar[d] & p_* A^{\cM}_{ C^{(a-1)}\times C} \ar[r] \ar[d]  &  0  \ar[r,"+1"] \ar[d] & {}\\   p_* A^{\cM}_{C^{(a-b)} \times Z^b_1} \ar[r,"\cong"]\ar[d,"\widetilde{\beta}"]   & p_* A^{\cM}_{C^{(a-b)} \times C^{(b-1)} \times C}  \ar[r]\ar[d]  &  0 \ar[d,"\widetilde{\alpha}"] \ar[r,"+1"]  & {} \\ p_* \Sigma^{a,b}_1 \ar[r] & \eta^{a,b}_1 \boxtimes A^{\cM}_{C}  \ar[r] & 0 \ar[r,"+1"] & {}\end{tikzcd},\]
so we see $H^{a,b}_{w,1,\ell} = 0$ for all $\ell,w$ and
\[ {\rm Gr}^W_w p_* \Sigma^{a,b}_c \cong {\rm Gr}^W_{w-1} \eta^{a,b}_1 \boxtimes A^{\cM}_C,\]
which says $V^{a,b}_{w,1,\ell} = \begin{cases} {\rm Gr}^W_{w-1} \eta^{a,b}_1 & \ell =1 \\ 0 & \text{otherwise}\end{cases}$.
\end{eg}

It is more convenient to consider the shifted version of the bottom triangle in diagram \eqref{eq-inductiveStructure}:
\begin{equation} \label{eq-firstTriangle}  \eta^{a,b}_c\boxtimes  p_* A^{\cM}_{B^c}\xrightarrow[]{\widetilde{\chi}} S^{a,b}_c \xrightarrow[]{\widetilde{\beta}} p_* \Sigma^{a,b}_c[1] \xrightarrow[]{+1}.\end{equation}

The other important triangle is the rightmost vertical one in diagram \eqref{eq-inductiveStructure}:
\begin{equation} \label{eq-secondTriangle}  p_* \Sigma^{a,c}_{c-1} \to \kappa_* A^{\cM}_{C^{(a-b)}} \boxtimes p_* \Sigma^{b,c}_{c-1} \xrightarrow[]{\widetilde{\alpha}} S^{a,b}_c \xrightarrow[]{+1}.\end{equation}

We give names to some important morphisms appearing as K\"{u}nneth summands in the (shifted) composed morphism $\widetilde{\beta} \circ \widetilde{\alpha}$:
\begin{equation} \label{eq-defalphaprime0} \alpha'_{a,b,c,H^0} \colon H^0_{\cM}(C^{(a-b)}) \boxtimes  p_* \Sigma^{b,c}_{c-1}[-1] \to p_* \Sigma^{a,b}_c,\end{equation}
\begin{equation} \label{eq-defalphaprime1} \alpha'_{a,b,c,H^1} \colon H^1_{\cM}(C^{(a-b)}) \boxtimes p_* \Sigma^{b,c}_{c-1}[-2] \to p_* \Sigma^{a,b}_c\end{equation}

We see that, using the above notation for $\alpha'$, we have equality $\alpha' = \alpha'_{k,k-1,k-2,H^0}$ in this notation. Indeed, the equality follows from the fact that both morphisms factor through $H^0_{\cM}(C) \boxtimes p_* A^{\cM}_{Z^{k-1}_{k-2}}$, and the rightmost commutative square in the morphism of triangles
\[\begin{tikzcd} H^0_{\cM}(C) \boxtimes A^{\cM}_{\sigma_{k-2}} \ar[r] \ar[d] & H^0_{\cM}(C) \boxtimes p_* A^{\cM}_{Z^{k-1}_{k-2}} \ar[r] \ar[d] & H^0_{\cM}(C) \boxtimes \cC^{k-1} \ar[d] \ar[r,"+1"] & {} \\ p_* A^{\cM}_{Z^k_{k-2}} \ar[r]  & p_* A^{\cM}_{C\times Z^{k-1}_{k-2}} \ar[r]  & p_*\Sigma^{k,k-1}_{k-2}  \ar[r,"+1"] & {} \end{tikzcd}.\]

Moreover, we can now show that the map $\beta$ defined above is injective (as claimed in Item \eqref{itm-betaInj} of \propositionref{prop-ComputeSigma}).

\begin{proof}[Proof of Item \eqref{itm-betaInj}] We saw in \lemmaref{lem-chiIso} above that $\chi \colon \cH^{2k-1} \cC^k \to H^2_{\cM}(C) \boxtimes \cH^{2k-3}p_* A^{\cM}_{B^{k-1}}$ is an isomorphism. Applying this with $k-1$ in place of $k$, we see that we have an isomorphism
\[ \chi \colon \cH^{2k-3} \cC^{k-1} \to H^2_{\cM}(C) \boxtimes \cH^{2k-5} p_* A^{\cM}_{B^{k-2}}.\]

We have the commutative diagram
\[ \begin{tikzcd} H^0_{\cM}(C) \boxtimes \cH^{2k-3} \cC^{k-1}\ar[r,"\chi"] \ar[dd,"\beta"] & H^0_{\cM}(C) \otimes H^2_{\cM}(C) \boxtimes \cH^{2k-5} p_* A^{\cM}_{B^{k-2}} \ar[d,"\rho_1"] \\ & H^2_{\cM}(C^{\times 2}) \boxtimes  \cH^{2k-5} p_* A^{\cM}_{B^{k-2}} \ar[d,"\rho_2"] \\ \cH^{2k-3} p_* \Sigma^{k,k-1}_{k-2} \ar[r,"\rho_3"] & \cH^2(\eta_{k-2}^{k,k-1}) \boxtimes \cH^{2k-5}p_* A^{\cM}_{B^{k-2}} \end{tikzcd}\]
where the morphism $\rho_3$ is obtained by taking $\cH^{2k-3}$ of the morphism
\[ p_* \Sigma^{k,k-1}_{k-2} \to \eta^{k,k-1}_{k-2} \boxtimes p_* A^{\cM}_{B^{k-2}},\] the morphism $\rho_2$ is obtained by taking $\cH^{2k-3}$ of the morphism
\[ p_*A^{\cM}_{C \times C \times B^{k-2}} \to \eta^{k,k-1}_{k-2} \boxtimes p_* A^{\cM}_{B^{k-2}}\]
and the morphism $\rho_1$ is induced by the K\"{u}nneth decomposition for $H^2_{\cM}(C^{\times 2})$.

As $\chi$ is an isomorphism, to prove injectivity of $\beta$ it suffices to prove that the composition of the two right vertical morphisms is injective, but this boils down to injectivity of the map
\[ \pi_1^*(H^0_{\cM}(C)) \otimes \pi_2^*(H^2_{\cM}(C)) \to H^2_{\cM}(C^{\times 2}) \to H^2_{\cM}(C^{\times 2})/H^2_{\cM}(C^{(2)}),\]
which is easily checked for underlying $A$-vector spaces. This completes the proof.
\end{proof}

By semi-simplicity of the category of pure objects of a fixed weight $w$ and by the strict support condition on the individual ${\rm IC}^{\cM}_{\sigma_\ell}$-modules, when we apply ${\rm Gr}^W_w(-)$ to the $3\times 3$-diagram \eqref{eq-inductiveStructure}, we can write the resulting $3\times 3$-diagram as the direct sum over $\ell \leq c$ of the exact functor $(-)\boxtimes {\rm IC}_{\sigma_\ell}$ applied to certain $3\times 3$-diagrams in $D^b(\cM_{\rm cons}(w+1-2\ell))$. 

Now, we give the resulting exact triangles by applying ${\rm Gr}^W_w(-)$ to the triangles \eqref{eq-firstTriangle} and \eqref{eq-secondTriangle}. As our plan is to give explicit complexes to represent our objects of interest, we begin by replacing ${\rm Gr}^W_w \eta^{a,b}_c$ with a specific complex.

The morphism $H^{\bullet}_{\cM}(C^{(a-c)}) \to H^\bullet_{\cM}(C^{(a-b)} \times C^{(b-c)} )$ is injective because the functor ${\rm For}$ is exact and the pull-back morphism on cohomology is injective, hence $\eta^{a,b}_c$ is a pure complex of weight $0$ with $j$th cohomology equal to
\[ H^j_{\cM}(  C^{(a-b)}\times C^{(b-c)})/H^{j}_{\cM}(C^{(a-c)}).\]

In particular, ${\rm Gr}^W_j \eta^{a,b}_c = (H^j_{\cM}( C^{(a-b)}\times C^{(b-c)}  )/H^{j}_{\cM}(C^{(a-c)}))[-j]$ as objects of the derived category. If we denote $\gamma^{a,b}_{c,j}$ as the complex $H^j_{\cM}(C^{(a-c)}) \to H^j_{\cM}(C^{(a-b)}\times C^{(b-c)})$ with terms in degrees $-1,0$, then we will prefer to write
\[ {\rm Gr}^W_j \eta^{a,b}_c = \gamma^{a,b}_{c,j}[-j].\]

By K\"{u}nneth, we can rewrite (using \lemmaref{lem-semismall} for the $A^{\cM}_{B^c}$ term):\small
\[ {\rm Gr}^W_w (\eta^{a,b}_c \boxtimes p_* A^{\cM}_{B^c}) = ({\rm Gr}^W_{w+1-2c} \eta^{a,b}_{c}[1-2c])\boxtimes (p_* A^{\cM}_{B^c}[2c-1])= \bigoplus_{\ell \leq c}  \gamma^{a,b}_{c,w+1-2c}\otimes H^{2(c-\ell)}_{\cM}(C^{(c-\ell)})[-w] \boxtimes {\rm IC}_{\sigma_\ell}^{\cM}.\]\normalsize

Thus, by applying ${\rm Gr}^W_w(-)$ to the triangle \eqref{eq-firstTriangle} and looking at the coefficient of ${\rm IC}_{\sigma_\ell}^{\cM}$, we end up with the exact triangle for all $\ell \in \Z$:
\begin{equation} \label{eq-induct1} \gamma^{a,b}_{c,w+1-2c} \otimes H^{2(c-\ell)}_{\cM}(C^{(c-\ell)})[-w] \to H^{a,b}_{w,c,\ell} \to V^{a,b}_{w,c,\ell}[1] \xrightarrow[]{+1}.\end{equation}

We provide another triangle to understand the complexes $H^{a,b}_{w,c,\ell}$. Again, by the K\"{u}nneth formula, we have
\[  {\rm Gr}^W_w (\kappa_* A^{\cM}_{C^{(a-b)}}\boxtimes  p_* \Sigma^{b,c}_{c-1})  = \bigoplus_{i\in \Z}{\rm Gr}^W_{w-i} \kappa_* A^{\cM}_{C^{(a-b)}} \boxtimes {\rm Gr}^W_i p_* \Sigma^{b,c}_{c-1} , \]
and again, because $\kappa_* A^{\cM}_{C^{(a-b)}}$ is a pure complex of weight $0$, this can be rewritten as
\[ {\rm Gr}^W_w(\kappa_* A^{\cM}_{C^{(a-b)}}\boxtimes   p_* \Sigma^{b,c}_{c-1})  = \bigoplus_{i\in \Z}  H^{w-i}_{\cM}(C^{(a-b)})[i-w]\boxtimes {\rm Gr}^W_i p_* \Sigma^{b,c}_{c-1} .\]

Now, by applying ${\rm Gr}^W_w(-)$ to \eqref{eq-secondTriangle} we have exact triangles for all $\ell$:
\begin{equation} \label{eq-induct2} V^{a,c}_{w,c-1,\ell} \to \bigoplus_{i\in \Z} H^{w-i}_{\cM}(C^{(a-b)})[i-w] \otimes  V^{b,c}_{i,c-1,\ell} \to H^{a,b}_{w,c,\ell} \xrightarrow[]{+1}.\end{equation}

By induction, these triangles show that we can always find the complexes $V^{a,b}_{w,c,\ell}$ (resp. $H^{a,b}_{w,c,\ell}$), and moreover that we have vanishing for $\ell > c$ (resp. $\ell \geq c$), as claimed above. 

We will rewrite our triangles by setting $c = \ell +e$. It will also be beneficial to impose some shifts, as evidenced by the shift of $[i-w]$ in triangle \eqref{eq-induct2}. So we consider the triangles shifted by $[w+e-1]$, and we introduce the notation: 
\[ \widetilde{V}^{a,b}_{w,e,\ell} = V^{a,b}_{w,\ell+e,\ell}[w+e],\]
\[ \widetilde{H}^{a,b}_{w,e,\ell} = H^{a,b}_{w,\ell+e,\ell}[w+e-1].\]

With this notation, and using the mapping cone notation rather than the exact triangle notation, we get equality (from shifting the triangles \eqref{eq-induct1}, \eqref{eq-induct2}):
\begin{equation} \label{eq-induct1'} \widetilde{V}^{a,b}_{w,e,\ell} = C(\gamma^{a,b}_{\ell+e,w+1-2(\ell+e)}\otimes H^{2e}_{\cM}(C^{(e)})[e-1] \to \widetilde{H}^{a,b}_{w,e,\ell}),\end{equation}
\begin{equation} \label{eq-induct2'} \widetilde{H}^{a,b}_{w,e,\ell} = C(\widetilde{V}^{a,\ell+e}_{w,e-1,\ell} \to \bigoplus_{i\in \Z}H^{w-i}_{\cM}(C^{(a-b)})\otimes  \widetilde{V}^{b,\ell+e}_{i,e-1,\ell}).\end{equation}

\begin{eg} Restating the statement in \exampleref{eg-baseCase} in this notation, we have
\[ \widetilde{V}^{a,b}_{w,0,\ell} \cong \gamma^{a,b}_{\ell,w+1-2\ell}.\]
\end{eg}

We proceed now with our inductive construction. For $s \geq 0$ and $p>q>r+s$, we inductively define complexes of length $s+2$, which we denote ${}^s \gamma^{p,q}_{r,j}$, by
\[ {}^0 \gamma^{p,q}_{r,j} = \gamma^{p,q}_{r,j},\]
and
\begin{equation} \label{eq-defineGamma} {}^{s} \gamma^{p,q}_{r,j} = C({}^{s-1} \gamma^{p,r+s}_{r,j} \to \bigoplus_{i \in \Z} H^{j-i}_{\cM}(C^{(p-q)})\otimes{}^{s-1} \gamma^{q,r+s}_{r,i}  ),\end{equation}
where the morphism along which we take the mapping cone is induced by pull-back along the addition map on cohomology, and is made precise in the following discussion.

For ease, we let
\[ H^j_{\cM}(C^{(a_1)} \times \dots \times C^{(a_r)}) = H^j_{\cM}(a_1,\dots, a_r),\]
where the order is important for the following discussion. For repeated values, we will use an exponent for shorthand: for example $H^j_{\cM}(1,1,1,2,2) = H^j(1^3,2^2)$, but $H^j_{\cM}(1,2,1,1,2)$ is $H^j_{\cM}(1,2,1^2,2)$.

There is a canonical morphism $H^j_{\cM}(a_1,\dots, a_r) \to H^j_{\cM}(b_1,\dots, b_{r'})$ whenever the ordered partition $(\underline{b})$ is a refinement of $(\underline{a})$. From our point of view, the refinement property is most easily expressed as the following condition: if $p = \sum_{i=1}^r a_i = \sum_{j=1}^{r'} b_j$, then $(\underline{b})$ is a refinement of $(\underline{a})$ if and only if $\mathfrak S_{b_1} \times \dots \times \mathfrak S_{b_{r'}} \subseteq \mathfrak S_{a_1} \times \dots \times \mathfrak S_{a_r}$ as subgroups of $\mathfrak S_p$. Then the canonical morphism simply corresponds to the inclusion of the $\mathfrak S_{a_1} \times \dots \times \mathfrak S_{a_r}$-invariants of $H^\bullet_{\cM}(C^{\times p})$ in the $\mathfrak S_{b_1}\times \dots \times \mathfrak S_{b_{r'}}$-invariants.

\begin{lem} \label{lem-InductionConstruction} For all $j$ and $p > q > r+s$, the complex ${}^s \gamma^{p,q}_{r,j}$ has $-(s+1)$th term given by
\[ H^j_{\cM}(p-r),\]
it has $-s$th term given by
\[ \bigoplus_{m=1}^s H^{j}_{\cM}(p-r-m,m)\oplus H^{j}_{\cM}(p-q,q-r), \]
and it has $-(s+1-t)$th term for $t\in \{2,\dots, s+1\}$ given by the direct sum
\[ \bigoplus_{m=t}^s \bigoplus_{i_1+ \dots + i_t = m} H^{j}_{\cM}(p-r-m,i_1,i_2,\dots, i_t)\oplus \bigoplus_{m=t-1}^s \bigoplus_{i_1+\dots+ i_{t-1} = m} H^{j}_{\cM}(p-q,q-r-m,i_1,i_2,\dots, i_{t-1})\]
so that the differential is given by the canonical morphisms induced by refinement of ordered partitions of $p-r$, with signs following the conventions of the mapping cone construction.
\end{lem}
\begin{proof} This is shown by induction on $s$. For $s=0$, the claim is true by definition. So now we assume the claim holds for $s-1$ and all $j,p,q,r$. 

For the inductive step, we use the definition of the mapping cone construction. Then the $-(s+1-t)$th term of the mapping cone is given by taking the direct sum of the $-(s+1-t)+1 = -((s-1)+1-t)$ term of ${}^{s-1} \gamma^{p,r+s}_{r,j}$ with the $-(s+1-t) = -((s-1)+1-(t-1))$ term of $\bigoplus_{i\in \Z} H^{j-i}_{\cM}(C^{(p-q)}) \otimes {}^{s-1} \gamma^{q,r+s}_{r,i}$, which by K\"{u}nneth we can understand by simply adjoining $p-q$ to the front of the tuple of each summand.

The claim for $t =0$ is obvious.

For $t= 1$, the $-(s+1-t)$th term of the mapping cone is given by the direct sum of the following three direct sums
\[ \bigoplus_{m=1}^{s-1} H^j_{\cM}(p-r-m,m),\]
\[\bigoplus_{m=0}^{s-1}H^j_{\cM}(p-r-s,s),\]
\[ H^j_{\cM}(p-q,q-r),\]
which is as described in the lemma statement.

For $t>1$, the $-(s+1-t)$th term is given by the direct sum of the following four direct sums:
\[ \bigoplus_{m=t}^{s-1} \bigoplus_{i_1+\dots +i_t = m} H^j_{\cM}(p-r-m,i_1,\dots, i_t),\]
\[\bigoplus_{m=t-1}^{s-1} \bigoplus_{i_1+\dots +i_{t-1} = m} H^j_{\cM}(p-r-s,s-m,i_1,\dots, i_{t-1}),\]
\[ \bigoplus_{m=t-1}^{s-1} \bigoplus_{i_1+\dots +i_{t-1} = m} H^j_{\cM}(p-q,q-r-m,i_1,\dots, i_{t-1}),\]
\[\bigoplus_{m=t-2}^{s-1} \bigoplus_{i_1+\dots +i_{t-2} = m} H^j_{\cM}(p-q,q-r-s,s-m,i_1,\dots, i_{t-2}),\]
where for the last two, we used the K\"{u}nneth isomorphism to rewrite with the first entry as $p-q$. 

This gives the description in the lemma statement. 
\end{proof}

We can use \lemmaref{lem-InductionConstruction} to explicitly identify the morphism
\[\gamma^{a,b}_{\ell+e,w+1-2(\ell+e)}\otimes H^{2e}_{\cM}(C^{(e)})[e-1] \to \widetilde{H}^{a,b}_{w,e,\ell}\]
induced by the definition of the target as a mapping cone \eqref{eq-induct1'} for $e > 0$. Indeed, keeping in mind the shift, we see that this morphism corresponds to the commutative diagram
\[ \begin{tikzcd} H^{w+1-2(\ell+e)}_{\cM}(C^{(a-\ell-e)}) \otimes H^{2e}_{\cM}(C^{(e)}) \ar[r] \ar[d] & H^{w+1-2(\ell+e)}_{\cM}(C^{(a-b)}\times C^{(b-\ell-e)}) \otimes H^{2e}_{\cM}(C^{(e)}) \ar[d] \\ H^{w+1-2\ell}_{\cM}(a-\ell-e,e) \ar[r] & H^{w+1-2\ell}_{\cM}(a-b,b-\ell-e,e)\end{tikzcd},\]
where the vertical morphisms are inclusions of direct summands along the K\"{u}nneth decomposition. We know that the bottom two spaces are direct summands thanks to the explicit description of \lemmaref{lem-InductionConstruction}.

By applying ${\rm Gr}^W_w(-)$ (for some $w\in \Z$) to the morphisms $\alpha'_{a,b,c,H^0}, \alpha'_{a,b,c,H^1}$, we get for each $\ell \leq c-1$ morphisms
\begin{equation}\label{eq-defRho0} \rho^{a,b}_{w,c,\ell,H^0} \colon H^0_{\cM}(C^{(a-b)}) \otimes V^{b,c}_{w,c-1,\ell}[-1] \to H^{a,b}_{w,c,\ell}[-1]\to V^{a,b}_{w,c,\ell},\end{equation}
\begin{equation}\label{eq-defRho1} \rho^{a,b}_{w,c,\ell,H^1} \colon H^1_{\cM}(C^{(a-b)}) \otimes V^{b,c}_{w-1,c-1,\ell}[-2] \to H^{a,b}_{w,c,\ell}[-1]\to V^{a,b}_{w,c,\ell},\end{equation}
though with our current notation, we prefer the shifted versions:
\[ \widetilde{\rho}^{a,b}_{w,e,\ell,H^0} \colon H^0_{\cM}(C^{(a-b)}) \boxtimes \widetilde{V}^{b,c}_{w,e-1,\ell} \to \widetilde{V}^{a,b}_{w,e,\ell},\]
\[ \widetilde{\rho}^{a,b}_{w,e,\ell,H^1} \colon H^1_{\cM}(C^{(a-b)}) \boxtimes \widetilde{V}^{b,c}_{w-1,e-1,\ell} \to \widetilde{V}^{a,b}_{w,e,\ell}.\]

We now rewrite $\widetilde{H}^{a,b}_{w,e,\ell}$ and $\widetilde{V}^{a,b}_{w,e,\ell}$ as iterated mapping cones.

\begin{prop} \label{prop-HVDescribe} We have\footnotesize
\[ \widetilde{H}^{a,b}_{w,e,\ell} = C({}^1 \gamma^{a,b}_{\ell+e-1,w+1-2(\ell+e-1)}\otimes H^{2(e-1)}_{\cM}(C^{(e-1)})[e-2] \to C(\dots \to C({}^{e-1} \gamma^{a,b}_{\ell+1, w+1-2(\ell+1)}\otimes H^{2}_{\cM}(C) \to {}^e \gamma^{a,b}_{\ell,w+1-2\ell}))),\]
\normalsize
which is an $(e-1)$th iterated mapping cone, and
\[ \widetilde{V}^{a,b}_{w,e,\ell} = C({}^0 \gamma^{a,b}_{\ell+e,w+1-2(\ell+e)} \otimes H^{2e}_{\cM}(C^{(e)}) \to \widetilde{H}^{a,b}_{w,e,\ell})\]
\[ = C({}^0 \gamma^{a,b}_{\ell+e,w+1-2(\ell+e)}\otimes H^{2e}_{\cM}(C^{(e)})[e-1] \to C(\dots \to C({}^{e-1} \gamma^{a,b}_{\ell+1, w+1-2(\ell+1)}\otimes H^{2}_{\cM}(C) \to {}^e \gamma^{a,b}_{\ell,w+1-2\ell}))),\]
which is an $e$th iterated mapping cone. 

The morphism $\widetilde{\rho}^{a,b}_{w,e,\ell,H^0}$ (resp. $ \widetilde{\rho}^{a,b}_{w,e,\ell,H^1}$) is given by the composition of the defining morphism $\widetilde{H}^{a,b}_{w,e,\ell} \to \widetilde{V}^{a,b}_{w,e,\ell}$ with the morphism $\widetilde{V}^{b,c}_{w,e-1,\ell} \to \widetilde{H}^{a,b}_{w,e,\ell}$ which is induced by the natural morphisms
\[ H^0_{\cM}(C^{(a-b)}) \otimes {}^{s-1} \gamma^{b,r+s}_{r,j} \to {}^s \gamma^{a,b}_{r,j}\]
\[ (\text{resp. } H^1_{\cM}(C^{(a-b)}) \otimes {}^{s-1} \gamma^{b,r+s}_{r,j-1} \to {}^s \gamma^{a,b}_{r,j})\]
in the definition of the target as a mapping cone \eqref{eq-defineGamma}.
\end{prop}
\begin{proof} This is immediate by definition of the iterated mapping cone: indeed, the $3\times 3$ rule in the dg enhancement says that if we consider a commutative diagram:
\[ \begin{tikzcd} A \ar[r,"f"] \ar[d,"p"] & B \ar[d,"q"] \\  C \ar[r,"g"] & D \end{tikzcd}\]
and if $h\colon C(f) \to C(g)$ is the morphism induced by $p$ and $q$, then there is a canonical isomorphism
\[ C(C(p) \to C(q)) \cong C(h).\]

The final claim follows by definition of the morphism and exactness of $H^0_{\cM}(C^{(a-b)}) \otimes (-)$, so that it commutes with the mapping cone construction $C(-)$ in the obvious way.
\end{proof}

We now specialize to the case $a =k, b = k-1$, where everything becomes clear. By \lemmaref{lem-InductionConstruction}, these complexes and morphisms between them have a very special form:
\begin{prop} The complex ${}^{k-2-r} \gamma^{k,k-1}_{r,j}$ is equal to
\[ H^j_{\cM}(k-r) \to \bigoplus_{i_1+i_2 = k-r} H^j_{\cM}(i_1,i_2) \to \dots \to H^j_{\cM}(1^{k-r}).\]

The morphism
\[ H^0_{\cM}(C) \otimes {}^{k-3-r} \gamma^{k-1,k-2}_{r,j} \to {}^{k-2-r} \gamma^{k,k-1}_{r,j}\]
\[ (\text{resp. }H^1_{\cM}(C)\otimes {}^{k-3-r} \gamma^{k-1,k-2}_{r,j-1} \to {}^{k-2-r} \gamma^{k,k-1}_{r,j})\]
is induced by the canonical morphism of complexes (where the top row is simply the collection of ordered partitions of $k-r$ which begin with $1$):\footnotesize
\[ \begin{tikzcd} 0\ar[r] \ar[d]& H^j_{\cM}(1,k-1) \ar[r] \ar[d] & \bigoplus_{i_1+i_2= k-1-r} H^j_{\cM}(1,i_1,i_2) \ar[r] \ar[d] & \dots \ar[r]& H^j_{\cM}(1,1^{k-1-r}) \ar[d]\\ H^j_{\cM}(k-r) \ar[r] &\bigoplus_{i_1+i_2 = k-r} H^j_{\cM}(i_1,i_2) \ar[r] & \bigoplus_{i_1+i_2 = k-r} H^j_{\cM}(i_1,i_2,i_3) \ar[r] & \dots \ar[r] & H^j_{\cM}(1^{k-r}) \end{tikzcd}\]\normalsize
by applying K\"{u}nneth to the top row (splitting off the first factor) and looking at the $H^0_{\cM}(C)$ part (resp. $H^1_{\cM}(C)$ part).
\end{prop}
\begin{proof} This follows by setting $s = k-r-2$ in the description of \lemmaref{lem-InductionConstruction} above. Indeed, when $p = k, q= k-1$, the second direct sum in that lemma statement gives all partitions whose first term is $1$. The first direct sum gives all of the other partitions.

The claim about the morphisms is clear by construction. Note that the top row is indeed a sub-complex, because any refinement of an ordered partition which begins with $1$ must also begin with $1$.
\end{proof}

The description in the following corollary was explained to us by Burt Totaro, to whom we are grateful.

Recall that we have been using a natural $\mathfrak S_k$-action on $H^{\otimes k}$ for any $H \in \cM_{\rm cons}$, so we can consider the $k$th divided power of $H$ as the largest $\mathfrak S_k$-invariant subspace inside $H^{\otimes k}$. As $A$ has characteristic 0, the usual averaging trick allows us to identify this with the $k$th symmetric power, defined as the largest $\mathfrak S_k$-invariant quotient of $H^{\otimes k}$. Similarly, we can define $k$th exterior powers for such objects, which, after applying the forgetful functor, recover the usual constructions in the category of (super) $A$-vector spaces.

\begin{cor} \label{cor-WedgePower} The complex ${}^{k-2-r} \gamma^{k,k-1}_{r,j}$ is exact except possibly at the right, where the cohomology is equal to ${\rm Gr}^W_j \bigwedge^{k-r} H^\bullet_{\cM}(C)$.

The induced morphism
\[ H^0_{\cM}(C) \otimes {}^{k-3-r} \gamma^{k-1,k-2}_{r,j} \to {}^{k-2-r} \gamma^{k,k-1}_{r,j}\]
\[ (\text{resp. }H^1_{\cM}(C)\otimes {}^{k-3-r} \gamma^{k-1,k-2}_{r,j-1} \to {}^{k-2-r} \gamma^{k,k-1}_{r,j})\]
is quasi-isomorphic to the natural morphism
\[ H^0_{\cM}(C) \otimes {\rm Gr}^W_j \bigwedge^{k-r-1} H^\bullet_{\cM}(C) \to {\rm Gr}^W_j \bigwedge^{k-r} H^\bullet_{\cM}(C)\]
\[ (\text{resp. } H^1_{\cM}(C) \otimes {\rm Gr}^W_{j-1} \bigwedge^{k-r-1} H^\bullet_{\cM}(C) \to {\rm Gr}^W_j \bigwedge^{k-r} H^\bullet_{\cM}(C)).\]
\end{cor}
\begin{proof} By exactness and faithfulness of the functor ${\rm For}$, it suffices to prove exactness for the underlying $A$-objects. The claim about the induced morphisms is obvious by the previous proposition.

This follows from the identification of $H^j(C^{(a)},A)$ with ${\rm Gr}^W_j {\rm Sym}^a H^\bullet(C,A)$ as in Subsection \ref{subsec-ClassicalCohomology} above. Indeed, with this identification, the complex ${\rm For}({}^{k-2-r} \eta^{k,k-1}_{r,j})$ can be written 
\[ {\rm Gr}^W_j\left({\rm Sym}^{k-r}(H^\bullet(C),A) \to \bigoplus_{i_1+i_2 = k-r} {\rm Sym}^{i_1}(H^\bullet(C),A) \otimes {\rm Sym}^{i_2}(H^\bullet(C),A) \to \dots \to H^{\bullet}(C,A)^{\otimes k-r} \right)\]
and by \cite{Totaro1}*{Pg. 6}, the complex to which we apply ${\rm Gr}^W_j(-)$ is a resolution of $\bigwedge^{k-r} H^\bullet(C,A)$ (with the exterior power taken in the category of super vector spaces). By exactness of the functor ${\rm Gr}^W_j(-)$, we conclude.
\end{proof}

By \propositionref{prop-HVDescribe}, the objects $\widetilde{H}$ and $\widetilde{V}$ are described by mapping cones between wedge powers of $H^\bullet_{\cM}(C)$ tensored with $H^{2e}_{\cM}(C^{(e)})$.

The following explicit description is useful:\small
\[ {\rm Gr}^W_{w'} \bigwedge^{k-\ell} H^\bullet_{\cM}(C) = \begin{cases} H^0_{\cM}(C)\wedge {\rm Sym}^{k-\ell-1}(H^1_{\cM}(C)) & w' = k-\ell-1 \\ H^0_{\cM}(C)\wedge H^2_{\cM}(C) \wedge {\rm Sym}^{k-\ell-2}(H^1_{\cM}(C)) \oplus {\rm Sym}^{k-\ell}(H^1_{\cM}(C)) & w' = k-\ell \\ H^2_{\cM}(C) \wedge {\rm Sym}^{k-\ell-1}(H^1_{\cM}(C)) & w' = k-\ell+1 \end{cases}.\]\normalsize

We see then that the morphisms of which we take the mapping cone will be induced by the natural isomorphism $H^{2e}_{\cM}(C^{(e)}) \to H^2_{\cM}(C) \otimes H_{\cM}^{2(e-1)}(C^{(e-1)})$ given by pull-back along the addition map $C \times C^{(e-1)} \to C^{(e)}$. 

To get a sense for the behavior of these maps, we write out explicitly the first three easiest examples.

\begin{eg} \label{eg-k2} Let $\ell = k-2$ and $e=0$. Then we get an isomorphism
\[ \widetilde{V}^{k,k-1}_{w,k-2,k-2} \cong {}^0 \gamma^{k,k-1}_{k-2,w+5-2k} = {\rm Gr}^W_{w+5-2k} \bigwedge^2 H_{\cM}^\bullet(C).\]

For $w = 2k-4$, we get $H_{\cM}^0(C)\wedge H_{\cM}^1(C)$, for $w = 2k-3$, we get $H_{\cM}^0(C)\wedge H_{\cM}^2(C) \oplus {\rm Sym}^2(H_{\cM}^1(C))$ and for $w = 2k-2$, we get $H_{\cM}^1(C)\wedge H_{\cM}^2(C)$. Hence, by shifting by $[-w]$, we get
\[ V^{k,k-1}_{w,k-2,k-2} = \begin{cases} H_{\cM}^0(C)\wedge H_{\cM}^1(C)[4-2k] & w = 2k-4 \\ \left( H_{\cM}^0(C)\wedge H_{\cM}^2(C) \oplus {\rm Sym}^2(H_{\cM}^1(C))\right)[3-2k] & w= 2k-3 \\  H_{\cM}^1(C)\wedge H_{\cM}^2(C)[2-2k] & w =2k-2\end{cases} \]
\end{eg}

\begin{eg} \label{eg-k3} Let $\ell = k-3$ so that $k-2-\ell = e = 1$, and so we only need to run the inductive process one time. We see then that $\widetilde{V}$ is the cone of ${}^0 \gamma \otimes H_{\cM}^2(C) \to {}^1 \gamma$.

We conclude that $\widetilde{V}^{k,k-1}_{w+2\ell-1,1,\ell}$ is the cone of the morphism
\[{\rm Gr}^W_{w-2} \bigwedge^2 H_{\cM}^\bullet(C) \otimes H_{\cM}^2(C) \to {\rm Gr}^W_w \bigwedge^3 H_{\cM}^\bullet(C)\]

For $w=2$, the left space vanishes and we are just left with $H_{\cM}^0(C)\wedge {\rm Sym}^2(H_{\cM}^1(C))$. For $w =3$ we take the cone of the natural morphism
\[ H_{\cM}^0(C) \wedge H_{\cM}^1(C)\otimes H_{\cM}^2(C) \to H_{\cM}^0(C) \wedge H_{\cM}^1(C) \wedge H_{\cM}^2(C) \oplus {\rm Sym}^3(H_{\cM}^1(C)),\]
which is ${\rm Sym}^3(H_{\cM}^1(C))$.

For $w =4$, we take the cone of
\[ H_{\cM}^0(C)\wedge H_{\cM}^2(C) \otimes H_{\cM}^2(C) \oplus {\rm Sym}^2(H_{\cM}^1(C)) \otimes H_{\cM}^2(C) \to H_{\cM}^2(C)\wedge {\rm Sym}^2(H_{\cM}^1(C)),\]
which is surjective with kernel $H_{\cM}^0(C)\wedge H_{\cM}^2(C) \otimes H_{\cM}^2(C)$. Thus, the cone is 
\[\left(H_{\cM}^0(C) \wedge H_{\cM}^2(C)\right) \otimes H_{\cM}^2(C)[1].\]

For $w=5$, the right space vanishes and so the cone is simply
\[\left(H_{\cM}^2(C)\wedge H_{\cM}^1(C)\right)\otimes H_{\cM}^2(C)[1].\]

Replacing $w$ by $w-2\ell+1$ and shifting by $[-w-e] = [-w-1]$, we conclude
\[ V^{k,k-1}_{w,k-2,k-3} = \begin{cases} {\rm Sym}^2(H_{\cM}^1(C)) [4-2k] & w = 2k-5 \\ {\rm Sym}^3(H_{\cM}^1(C))[3-2k] & w =2k-4 \\ \left(H_{\cM}^0(C)\wedge H_{\cM}^2(C)\right)\otimes H_{\cM}^2(C)[3-2k] & w = 2k-3 \\ \left(H_{\cM}^1(C)\wedge H_{\cM}^2(C)\right)\otimes H_{\cM}^2(C)[2-2k] & w = 2k-2\end{cases}.\]
\end{eg}

\begin{eg} \label{eg-k4} Let $\ell = k-4$, so that $k-2-\ell = e = 2$. Thus, we must run our induction two times. We use short hand: we begin with ${}^0 \gamma$, then run the induction to get the first $\widetilde{H}$ being ${}^1 \gamma$.

Then the first $\widetilde{V}$ is the cone of ${}^0 \gamma \to {}^1\gamma$. To compute the next $\widetilde{H}$, we apply the inductive description to this cone. Alternatively, we can apply the inductive description to each ${}^i \gamma$ involved and then take the cone of the resulting morphism. Thus, we see that the next $\widetilde{H}$ is the cone of ${}^1 \gamma \to {}^2 \gamma$. Finally, we conclude 
\[ \widetilde{V} = {\rm cone}({}^0 \gamma \to {\rm cone}({}^1 \gamma \to {}^2 \gamma)).\]

By the computation above, we see that $\widetilde{V}$ is the cone of the natural morphism\small
\[ {\rm Gr}^W_{w+5-2k} \bigwedge^2 H_{\cM}^\bullet(C)[1]\otimes H_{\cM}^4(C^{(2)}) \to C({\rm Gr}^W_{w+7-2k} \bigwedge^3 H_{\cM}^\bullet(C) \otimes H_{\cM}^2(C) \to {\rm Gr}^W_{w+9-2k} \bigwedge^4 H_{\cM}^\bullet(C)).\]
\normalsize
There are only a few values of $w$ where the terms are non-zero.

$w=2k-6$: then the first two terms vanish and we are just left with $H_{\cM}^0(C) \wedge {\rm Sym}^3(H_{\cM}^1(C))$. 

$w=2k-5$: then the first term vanishes and we are left with the cone of the morphism
\[ H_{\cM}^0(C) \wedge {\rm Sym}^2(H_{\cM}^1(C))\otimes H_{\cM}^2(C) \to \left( H_{\cM}^0(C) \wedge H_{\cM}^2(C) \wedge {\rm Sym}^2(H_{\cM}^1(C))\right)\oplus{\rm Sym}^4(H_{\cM}^1(C)),\]
giving the inclusion of the first summand, hence the cone is isomorphic to ${\rm Sym}^4(H_{\cM}^1(C))$.

$w = 2k-4$: Now three terms are non-zero: we first compute the inner cone:
\[C(\left(H_{\cM}^0(C)\wedge H_{\cM}^1(C) \wedge H_{\cM}^2(C)\right) \otimes H_{\cM}^2(C) \oplus {\rm Sym}^3(H_{\cM}^1(C))\otimes H_{\cM}^2(C) \to H_{\cM}^2(C) \wedge {\rm Sym}^3(H_{\cM}^1(C)))\]\normalsize
which is the surjection, hence this mapping cone is the other direct summand shifted by $[1]$. But then we see that the resulting morphism
\[ H_{\cM}^0(C) \wedge H_{\cM}^1(C)[1] \otimes H_{\cM}^4(C^{(2)}) \to \left(H_{\cM}^0(C)\wedge H_{\cM}^1(C) \wedge H_{\cM}^2(C)\right) \otimes H_{\cM}^2(C)[1] \]
is actually an isomorphism, so the mapping is $0$ here.

For all smaller values of $w$, the very rightmost space vanishes. Thus, the cone which we first take is quasi-isomorphic to the shift by $[1]$ of the second object. Hence, we simply shift the cone of the natural morphism
\[{\rm Gr}^W_{w+5-2k} \bigwedge^2 H_{\cM}^\bullet(C) \otimes H_{\cM}^4(C^{(2)}) \to {\rm Gr}^W_{w+7-2k} \bigwedge^3 H_{\cM}^\bullet(C)\otimes H_{\cM}^2(C)\]
by $[1]$ in each case.

$w=2k-3$: we take the cone of the natural map
\[ H_{\cM}^0(C)\wedge H_{\cM}^2(C) \otimes H_{\cM}^4(C^{(2)}) \oplus {\rm Sym}^2(H_{\cM}^1(C)) \to H_{\cM}^2(C) \wedge {\rm Sym}^2(H_{\cM}^1(C))\otimes H_{\cM}^2(C),\]
which is the natural surjection, hence the cone is the other direct summand (shifted by $[1]$). We conclude that in this case we get $H_{\cM}^0(C)\wedge H_{\cM}^2(C) \otimes H_{\cM}^4(C^{(2)}) [2]$.

$w=2k-2$: Now the rightmost term vanishes, thus, we are left with ${\rm Gr}^W_3 \bigwedge^2 H_{\cM}^\bullet(C) \otimes H_{\cM}^4(C^{(2)})[2] = H_{\cM}^2(C) \wedge H_{\cM}^1(C) \otimes H_{\cM}^4(C^{(2)})[2]$. 

Shifting by $[-w-e] = [-w-2]$, we get
\[ V^{k,k-1}_{w,k-2,k-4} = \begin{cases} {\rm Sym}^3(H_{\cM}^1(C))[4-2k] & w= 2k-6 \\ {\rm Sym}^4(H_{\cM}^1(C))[3-2k] & w= 2k-5 \\ H_{\cM}^0(C)\wedge H_{\cM}^2(C)\otimes H_{\cM}^4(C^{(2)})[3-2k] & w=2k-3 \\ H_{\cM}^1(C)\wedge H_{\cM}^2(C)\otimes H_{\cM}^4(C^{(2)})[2-2k] & w=2k-2 \end{cases}.\]
\end{eg}

We now handle the general case. For ease of notation, we will first compute $\widetilde{V}^{k,k-1}_{w+2\ell-1,k-2,\ell}$ and then substitute $w -2\ell +1$ for $w$ at the end.

Using the explicit description of these objects, we see that at most three of these objects are non-zero for any fixed value of $w$. Indeed, we have that
\[ {}^{k-2+e} \gamma^{k,k-1}_{\ell+e,w-2e} = {\rm Gr}^W_{w-2e} \bigwedge^{k-\ell-e} H_{\cM}^\bullet(C),\]
which is only possibly non-zero for $w -2e \in \{k-\ell-e-1,k-\ell-e,k-\ell-e+1\}$. Hence, for $w$ fixed, the possibly non-zero terms are for $e \in \{w-1 - (k-\ell), w - (k-\ell),w+1-(k-\ell)\}$.

When we take iterated mapping cones as above (working up to quasi-isomorphism), we can ignore the zeroes to the left of the first possibly non-zero term, and the zeros to the right of the rightmost possibly non-zero term amount to a shift by $[1]$, unless there are no such zeroes, meaning $0 \in \{w-1-(k-\ell),w-(k-\ell),w +1-(k-\ell)\}$. 

We ignore the latter case for now, and we also ignore the case $k-\ell-2 \in \{w-1-(k-\ell),w-(k-\ell),w+1-(k-\ell)\}$, as these are the two extreme cases.

Thus, if we assume $(k-\ell)+2 \leq w \leq 2(k-\ell)-4 $ we conclude that $\widetilde{V}^{k,k-1}_{w+2\ell-1,k-2,\ell}[-w+(k-\ell)+1](-k+\ell-1+w)$ is given by\footnotesize
\[ C\left( {\rm Gr}^W_{2(k-\ell)-w-2} \bigwedge^{2(k-\ell)-w-1}H_{\cM}^\bullet(C)[1](-2)\to C({\rm Gr}^W_{2(k-\ell)-w} \bigwedge^{2(k-\ell)-w} H_{\cM}^\bullet(C)(-1)\to {\rm Gr}^W_{2(k-\ell)-w+2} \bigwedge^{2(k-\ell)-w+1}H_{\cM}^\bullet(C))\right),\]
\normalsize
where for these terms we have rewritten $H_{\cM}^{2e}(C^{(e)})$ (by abuse of notation) simply as a Tate twist by $(-e)$. We will see that these terms vanish.

Ignoring the shift and Tate twist, we can write out explicitly what this cone is. Indeed, using the description of these wedge powers, we see that the inner cone is the kernel of that map (shifted by $[1]$), and the kernel of that map is $H_{\cM}^0(C) \wedge H_{\cM}^2(C) \wedge {\rm Sym}^{2(k-\ell)-w-2}(H_{\cM}^1(C))(-1)$.

But then the remaining morphism is, up to a shift, the isomorphism
\[ H_{\cM}^0(C) \wedge {\rm Sym}^{2(k-\ell)-w-2}(H_{\cM}^1(C))(-2) \cong H_{\cM}^0(C) \wedge H_{\cM}^2(C) \wedge {\rm Sym}^{2(k-\ell)-w-2}(H_{\cM}^1(C))(-1),\]
and so we see that the cone is zero.

For those which do include $e=0$, we consider
\[ C({\rm Gr}^W_{w-4} \bigwedge^{k-\ell-2} H_{\cM}^\bullet(C) \otimes H_{\cM}^4(C^{(2)})[1] \to C({\rm Gr}^W_{w-2} \bigwedge^{k-\ell-1} H_{\cM}^\bullet(C)\otimes H_{\cM}^2(C) \to {\rm Gr}^W_{w} \bigwedge^{k-\ell} H_{\cM}^\bullet(C))),\]
and if $w = (k-\ell)+1$, then we see that the same analysis above holds and the cone is $0$.

For $w= (k-\ell)$, the leftmost term vanishes, and so we are left with
\[ C({\rm Gr}^W_{(k-\ell)-2} \bigwedge^{k-\ell-1} H_{\cM}^\bullet(C)\otimes H_{\cM}^2(C) \to {\rm Gr}^W_{(k-\ell)} \bigwedge^{k-\ell} H_{\cM}^\bullet(C)),\]
and here the morphism is injective with cokernel ${\rm Sym}^{k-\ell}(H_{\cM}^1(C))$, which is thus quasi-isomorphic to the mapping cone.

For $w = (k-\ell)-1$, the left two terms vanish, and we are just left with
\[ {\rm Gr}^W_{(k-\ell)-1} \bigwedge^{k-\ell}H_{\cM}^\bullet(C) = H_{\cM}^0(C) \wedge {\rm Sym}^{k-\ell-1}(H_{\cM}^1(C)).\]

For those terms which include $e = k-\ell-2$, we shift by $[4+\ell-k]$ and consider \footnotesize
\[ C({\rm Gr}^W_{w-2(k-\ell-2)} \bigwedge^{2} H_{\cM}^\bullet(C) \otimes H_{\cM}^{2(k-\ell-2)}(C^{(k-\ell-2)})[1] \to\]
\[C({\rm Gr}^W_{w-2(k-\ell-3)} \bigwedge^{3} H_{\cM}^\bullet(C) \otimes H_{\cM}^{2(k-\ell-3)}(C^{(k-\ell-3)}) \to {\rm Gr}^W_{w-2(k-\ell-4)} \bigwedge^{4} H_{\cM}^\bullet(C) \otimes H_{\cM}^{2(k-\ell-4)}(C^{(k-\ell-4)}))).\]
\normalsize

For $w -2(k-\ell-3) = 3$, we see that the same analysis as above holds and we get that the cone vanishes.

For $w = 2(k-\ell)-2$, we see that the rightmost term vanishes, and so the inner cone is simply ${\rm Gr}^W_{4} \bigwedge^3 H_{\cM}^\bullet(C)\otimes H_{\cM}^{2(k-\ell-3)}(C^{(k-\ell-3)})[1]$. The natural morphism
\[ {\rm Gr}^W_2 \bigwedge^2 H_{\cM}^\bullet(C)\otimes H_{\cM}^{2(k-\ell-2)}(C^{(k-\ell-2)}) \to {\rm Gr}^W_4 \bigwedge^3 H_{\cM}^\bullet(C)\otimes H_{\cM}^{2(k-\ell-3)}(C^{(k-\ell-3)})\]
is surjective with kernel equal to $H_{\cM}^0(C)\wedge H_{\cM}^2(C)\otimes H_{\cM}^{2(k-\ell-2)}(C^{(k-\ell-2)})$. Thus, the cone is $H_{\cM}^0(C)\wedge H_{\cM}^2(C)\otimes H_{\cM}^{2(k-\ell-2)}(C^{(k-\ell-2)})[2]$. 

Finally, for $w = 2(k-\ell)-1$, both rightmost terms vanish, and so that cone is simply the leftmost term shifted by $[1]$. As that term already has a shift, we conclude that the cone is
\[ H_{\cM}^1(C)\wedge H_{\cM}^2(C)\otimes H_{\cM}^{2(k-\ell-2)}(C^{(k-\ell-2)})[2].\]

Recall that we must shift by $[k-\ell-4]$ to get the actual values of the cones.

Finally, by replacing $w+2\ell-1$ with $w$ and then shifting by $[-w-(k-\ell-2)]$ (to go from $\widetilde{V}$ to $V$), we get the following description of $V^{k,k-1}_{w,k-2,\ell}$, hence of ${\rm Gr}^W_w p_*\Sigma^{k,k-1}_{k-2}$. Note that this description agrees with those given in \exampleref{eg-k2} (taking the direct sum of the $k+\ell-1 = 2k-3$ terms in this case), \exampleref{eg-k3} and \exampleref{eg-k4}.

\begin{cor} \label{cor-computeVComplex} We have $V^{k,k-1}_{w,k-2,\ell} = 0$ for all $k+\ell \leq w \leq 2k-4$.

We have the following:
\[ V^{k,k-1}_{w,k-2,\ell} = \begin{cases} H_{\cM}^0(C) \wedge {\rm Sym}^{k-\ell-1}(H_{\cM}^1(C))[4-2k] & w = k+\ell-2\\ {\rm Sym}^{k-\ell}(H_{\cM}^1(C))[3 -2k] & w = k+\ell-1 \\ H_{\cM}^0(C)\wedge H_{\cM}^2(C)\otimes H_{\cM}^{2(k-\ell-2)}(C^{(k-\ell-2)})[3-2k] & w = 2k-3 \\ H_{\cM}^1(C)\wedge H_{\cM}^2(C)\otimes H_{\cM}^{2(k-\ell-2)}(C^{(k-\ell-2)})[2-2k] & w=2k-2 \end{cases},\]
and all other terms vanish.

In particular, the vanishing \eqref{itm-cohSigma} and Items \eqref{itm-GrW2k-2} and \eqref{itm-GrW2k-3} of \propositionref{prop-ComputeSigma} hold.
\end{cor}

\begin{cor} The morphism $\alpha$ in Item \eqref{itm-alphaIso} is an isomorphism.
\end{cor}
\begin{proof} As mentioned above, we check that for all $w\in \Z$, the morphism ${\rm Gr}^W_w(\alpha)$ is an isomorphism. By \lemmaref{lem-GrWAlphaInj} and the computation in that case (which shows that the same simple summands appear in the domain and target, hence injectivity implies isomorphism), it suffices to check $w < 2k-3$, and so it suffices to check that the map
\[ {\rm Gr}^W_w(\alpha') \colon H^0_{\cM}(C) \boxtimes {\rm Gr}^W_w \cH^{2k-5} p_* \Sigma^{k-1,k-2}_{k-3}[-1] \to {\rm Gr}^W_w \cH^{2k-4}p_*\Sigma^{k,k-1}_{k-2}\]
is an isomorphism for all $w < 2k-3$. 

Using the identification of the morphism $\widetilde{\rho}$ above and the computation of these modules, we see that if $w=k+\ell-2$, then this morphism is the identity
\[ H^0_{\cM}(C) \otimes {\rm Sym}^{k-\ell-1}(H^1_{\cM}(C)) \boxtimes {\rm IC}_{\sigma_{\ell}}\to H^0_{\cM}(C) \otimes {\rm Sym}^{k-\ell-1}(H^1_{\cM}(C)) \boxtimes {\rm IC}_{\sigma_{\ell}},\]
which proves the claim.
\end{proof}

The above immediately gives a description of the constant Hodge module by taking $\cM(-) = {\rm MHM}(-)$ when $k = \C$ and $A = \Q$. If $k$ is a number field and if we use the theory of mixed sheaves as in \exampleref{eg-SysReal}, then the isomorphism in the theorem gives an isomorphism of Galois equivariant perverse sheaves. This and the Hodge module version will be used below toward a computation of the Hodge structure and Galois representation structure on singular cohomology.

We will need the following statement for the computation of singular cohomology, which explains why we also tracked the $H^1(C)$ K\"{u}nneth component:
\begin{cor} \label{cor-rhoMorphism} If, for any $2\leq \ell \leq k-2$, we apply ${\rm Gr}^W_{2k-1-\ell} \cH^{2k-3}(-)$ to the morphism
\[ H^1_{\cM}(C) \boxtimes p_* \Sigma^{k-1,k-2}_{k-3}[-2] \to p_* \Sigma^{k,k-1}_{k-2}\]
then we get a commutative diagram
\[ \begin{tikzcd} H^1_{\cM}(C)\otimes {\rm Sym}^{\ell-1}(H^1_{\cM}(C)) \boxtimes {\rm IC}_{\sigma_{k-1-(\ell-1)}} \ar[r,equal]\ar[ddd] & H^1_{\cM}(C) \boxtimes {\rm Gr}^W_{2k-2-\ell}\cH^{2k-4} \cC^{k-1} \ar[d,equal] \\ {} & H^1_{\cM}(C) \boxtimes {\rm Gr}^W_{2k-2-\ell}\cH^{2k-4}p_* A^{\cM}_{Z^{k-1}_{k-2}} \ar[d] \\ {} & H^1(C) \boxtimes {\rm Gr}^W_{2k-2-\ell} \cH^{2k-5} p_* \Sigma^{k-1,k-2}_{k-3}\ar[d] \\{\rm Sym}^{\ell}(H^1_{\cM}(C))\boxtimes {\rm IC}_{\sigma_{k-\ell}} \ar[r] & {\rm Gr}^W_{2k-1-\ell} \cH^{2k-3} p_* \Sigma^{k,k-1}_{k-2}\end{tikzcd} \]
where the left vertical morphism is the natural product map.
\end{cor}
\begin{proof} The top horizontal equality is given by \theoremref{thm-ConstantSheafHigherSecants}. The defining triangle
\[ A^{\cM}_{\sigma_{k-2}} \to p_* A^{\cM}_{Z^{k-1}_{k-2}} \to \cC^{k-1} \xrightarrow[]{+1}\]
gives an isomorphism $\cH^{2k-4} p_* A^{\cM}_{Z^{k-1}_{k-2}} \to \cH^{2k-4} \cC^{k-1}$, hence the top right vertical equality.

The last claim follows from the identification in \corollaryref{cor-computeVComplex} and very last result of \corollaryref{cor-WedgePower}, which allows us to identify the bottom right vertical morphism.
\end{proof}

\subsection{Generation level of Hodge filtration on local cohomology} 
In this subsection, assume $k =\C$ and $A =\Q$.

Let $\sigma_k \to \P^N$ be the closed embedding, whose codimension we write as $q_k = N - (2k-1)$. We give a bound on the generation level of $\cH^{q_k}_{\sigma_k}(\cO_{\P^N},F)$ using \theoremref{thm-ConstantSheafHigherSecants}.

\begin{cor}[$=$ \corollaryref{glhs}]\label{glhsn} Assume $L$ is $(2k-1)$-very ample and $\sigma_k\neq\P^N$ where $k\geq 2$. Then we have
\[ {\rm gl}(\cH^{q_k}_{\sigma_k}(\cO_{\P^N},F)) \leq k-1,\]
and hence, for any log resolution $f\colon (\widetilde{P},E) \to \P^N$ of the pair $(\P^N,\sigma_k)$, we have
\[R^{q_k-1+i} f_* \Omega^{N-i}_{\widetilde{P}}(\log E) = 0\text{ for all } i \geq k.\]
\end{cor}

\begin{proof} Using exactness (in a triangulated sense) of the functor ${\rm Gr}^F_{i-N} {\rm DR}(-)$, we can reduce to giving a bound on the generating level of each ${\rm Gr}^W_w \cH^{q_k}_{\sigma_k}(\cO_{\P^N},F)$. We have
\[ \cH^{q_k}_{\sigma_k}(\cO_{\P^N},F) \cong \cH^0 \mathbf D_{\sigma_k}^H (-q_k) = \mathbf D_{\sigma_k}^H(-q_k),\]
where the last equality follows from the fact that $\sigma_k$ is CCI.

We have
\[ {\rm Gr}^W_{2k-1-\ell} \Q_{\sigma_k}^H[2k-1] \cong {\rm Sym}^\ell H^1(C) \boxtimes {\rm IC}_{\sigma_{k-\ell}}^H \]
and so dualizing (and using pure polarizability on the right), we get
\[ {\rm Gr}^W_{\ell+1-2k} \mathbf D(\Q_{\sigma_k}^H[2k-1]) \cong {\rm Sym}^\ell H^1(C) \boxtimes {\rm IC}_{\sigma_{k-\ell}}^H (2k-1-\ell),\]
and so, by Tate twisting by $(1-2k)$ on both sides, we get
\[ {\rm Gr}^W_{2k-1+\ell} \mathbf D_{\sigma_k}^H \cong {\rm Sym}^\ell H^1(C) \boxtimes {\rm IC}_{\sigma_{k-\ell}}^H (-\ell).\]

By Tate twisting by $(-q_k)$, we conclude that
\[ {\rm Gr}^W_{N+q_k+\ell} \cH^{q_k}_{\sigma_k}(\cO_{\P^N}) \cong {\rm Sym}^\ell H^1(C) \boxtimes {\rm IC}_{\sigma_{k-\ell}}^H(-\ell-q_k).\]

By applying $\cH^0(-)$ to Equation \eqref{eq-GrDRBoxtimes}, we have
\[ \cH^0 {\rm Gr}^F_{i-N} {\rm DR}({\rm Gr}^W_{N+q_k+\ell} \cH^{q_k}_{\sigma_k}(\cO_{\P^N})) = \bigoplus_{b \in \Z} {\rm Gr}^F_{i-b}{\rm Sym}^\ell H^1(C) \boxtimes \cH^0 {\rm Gr}^F_{b+\ell+q_k-N} {\rm DR}( {\rm IC}_{\sigma_{k-\ell}}^H)\]

For any fixed $c$, consider the resolution $p \colon B^c \to \sigma_c$. As in the proof of \lemmaref{lem-semismall}, this morphism is stratified over the stratification $C = \sigma_1 \subseteq \sigma_2 \subseteq \dots \subseteq \sigma_c$, where the fiber over $U_a = \sigma_a\setminus \sigma_{a-1}$ is $C^{(c-a)}$. Thus, the maximal dimension occurs at $a = 1$, giving a maximal fiber dimension of $c-1$.

By \lemmaref{lem-genLevelIC}, we see that
\[ \cH^0 {\rm Gr}^F_{b+\ell+q_k-N} {\rm DR}({\rm IC}_{\sigma_{k-\ell}}^H) = 0 \text{ for all } b+\ell+q_k > k-\ell-1 + q_{k-\ell}.\] 

Note that $q_{k-\ell} = q_k + 2\ell$, so the vanishing holds for all $b \geq k$. So we can rewrite the direct sum using this vanishing as:
\[ \cH^0 {\rm Gr}^F_{i-N} {\rm DR}({\rm Gr}^W_{N+q_k+\ell} \cH^{q_k}_{\sigma_k}(\cO_{\P^N})) = \bigoplus_{b < k} {\rm Gr}_F^{b-i}{\rm Sym}^\ell H^1(C) \boxtimes \cH^0 {\rm Gr}^F_{b+\ell+q_k-N}{\rm DR}( {\rm IC}_{\sigma_{k-\ell}}^H),\]
where we have switched to decreasing Hodge filtrations for the Hodge structure. The first factor automatically vanishes for $b-i > 0$, or in other words, for $b > i$. In particular, for any $i \geq k$, the vanishing is automatic.

So we get the uniform bound
\[ {\rm gl}({\rm Gr}^W_{N+q_k+\ell}\cH^{q_k}_{\sigma_k}(\cO_{\P^N},F)) \leq k-1,\]
and we conclude
${\rm gl}(\cH^{q_k}_{\sigma_{k}}(\cO_{\P^N}),F) \leq k-1$.
\end{proof}

\subsection{Hodge-Lyubeznik numbers}
In this subsection, we assume $k =\C$ and $A =\Q$, and we compute the intersection Hodge-Lyubeznik numbers for the $k$th secant variety $\sigma_k$. %of a curve $C$ embedded via a $(2k-1)$-very ample line bundle $L$.

%As $\sigma_k$ is CCI, the interesting Hodge-Lyubeznik numbers are the intersection Hodge-Lyubeznik numbers.

For a given $x\in \sigma_{k-1} = \sigma_{k,{\rm nRS}}$, the value of the intersection Hodge-Lyubeznik number depends on the value $1\leq a\leq k-1$ such that $x\in U_a = \sigma_a \setminus \sigma_{a-1}$.

We rewrite the result of \lemmaref{lem-semismall} for notational convenience:
\begin{equation} \label{eq-semismallTate} p_* \Q^H_{B^k} = \bigoplus_{\ell \leq k} {\rm IC}_{\sigma_{\ell}}^H(\ell-k)[1-2k].\end{equation}

Recall that if $x\in U_a$ then $p^{-1}(x) \cong C^{(k-a)} \subseteq Z^k_{a} \subseteq B^k$. So if we apply $i_x^*$ to both sides of the isomorphism \eqref{eq-semismallTate}, we get
\[ \kappa_* \Q^H_{C^{(k-a)}} \cong \bigoplus_{\ell \leq k} i_x^* {\rm IC}^H_{\sigma_{\ell}}(\ell-k)[1-2k].\]

Using that $x\in U_a$, hence not in $\sigma_{a-1}$, we can rewrite the isomorphism as
\[ \kappa_* \Q^H_{C^{(k-a)}} \cong \bigoplus_{a \leq \ell \leq k} i_x^* {\rm IC}^H_{\sigma_{\ell}}(\ell-k)[1-2k].\]

To make the induction easier, we re-index:
\[ \kappa_* \Q^H_{C^{(k-a)}} \cong \bigoplus_{0 \leq \ell \leq k-a} i_x^* {\rm IC}^H_{\sigma_{a+\ell}}(a+\ell-k)[1-2k],\]
and now we can write $k-a = b$, giving
\[ \kappa_* \Q^H_{C^{(b)}} \cong \bigoplus_{0 \leq \ell \leq b} \left(i_x^* {\rm IC}^H_{\sigma_{a+\ell}}[1-2(a+\ell)]\right)(\ell-b)[2(\ell-b)].\]

Finally, using the same formula with $b-1$ in place of $b$ and the fact that we are only interested in the $\ell = b$ part, we write
\[\begin{split} \kappa_* \Q^H_{C^{(b)}} \cong & \bigoplus_{0 \leq \ell < b} \left(i_x^* {\rm IC}^H_{\sigma_{a+\ell}}[1-2(a+\ell)]\right)(\ell-b)[2(\ell-b)] \oplus i_x^*({\rm IC}_{\sigma_{a+b}}^H[1-2(a+b)])\\
 \cong & (\kappa_* \Q^H_{C^{(b-1)}})(-1)[-2] \oplus i_x^*({\rm IC}_{\sigma_{a+b}}^H[1-2(a+b)]),\end{split}\]
which, by applying $\cH^j(-)$, gives an isomorphism for all $0\leq j \leq 2b$:
\begin{equation} \label{eq-CohixIC} \cH^{j +1 - 2(a+b)} i_x^* {\rm IC}_{\sigma_{a+b}}^H \cong H^j(C^{(b)})/H^{j-2}(C^{(b-1)})(-1).\end{equation}

\begin{thm}[$=$ \theoremref{thm-HLHigherSecants}]\label{thm-HLHigherSecantsn} Assume $C$ is a smooth projective curve of genus $g > 0$ and $L$ is a $(2k-1)$-very ample line bundle such that $\sigma_k\neq\P^N$ where $k\geq 2$. Let $x \in U_a = \sigma_a \setminus \sigma_{a-1}$ for some $a\leq k-1$. Then the intersection Hodge-Lyubeznik number ${\rm I}\lambda^r_{p,q}(\cO_{\sigma_k,x})$ is only possibly non-zero for $r \in [k+a-1,2k-1]$, in which case we have
\[ {\rm I}\lambda^r_{p,q}(\cO_{\sigma_k,x}) = \begin{cases} \binom{g}{-p} \binom{g}{-q} & -p-q = 2k-1-r, p \in [-(2k-1-r),0] \\ 0 & \text{otherwise}\end{cases}.\]
\end{thm}

\begin{proof} We let $b = k-a$ in the isomorphism \eqref{eq-CohixIC}, which gives
\[\cH^{j +1 - 2k} i_x^* {\rm IC}_{\sigma_{k}}^H \cong H^j(C^{(k-a)})/H^{j-2}(C^{(k-a-1)})(-1).\]

By applying the formula \eqref{eq-cohSymmPower}, for all $0 \leq j \leq 2(k-a)$, we get
\[\begin{split} H^j(C^{(k-a)}) \cong &\bigoplus_{\max\{0,j-(k-a)\} \leq i \leq \lfloor\frac{j}{2}\rfloor} \bigwedge^{j-2i} H^1(C)(-i)\\
 = & \begin{cases} \bigoplus_{0 \leq i \leq \lfloor \frac{j}{2}\rfloor} \bigwedge^{j-2i} H^1(C)(-i) & j-(k-a) \leq 0 \\\bigoplus_{j-(k-a) \leq i \leq \lfloor \frac{j}{2}\rfloor} \bigwedge^{j-2i} H^1(C)(-i) & \text{otherwise}\end{cases}\end{split}\]
and similarly for all $0 \leq j-2 \leq 2(k-a-1)$, we have
\[\begin{split} H^{j-2}(C^{(k-a-1)})(-1) \cong & \bigoplus_{\max\{0,j-2-(k-a-1)\} \leq i \leq \lfloor \frac{j-2}{2}\rfloor} \bigwedge^{(j-2)-2i} H^1(C)(-i-1)\\
\cong & \bigoplus_{\max\{0,j-(k-a)-1\} \leq i \leq \lfloor \frac{j}{2}\rfloor-1} \bigwedge^{j-2(i+1)} H^1(C)(-(i+1))\end{split}\]
and so if we re-index, this can be written
\[ \cong \bigoplus_{\max\{0,j-(k-a)-1\}+1 \leq i \leq \lfloor \frac{j}{2}\rfloor} \bigwedge^{j-2i} H^1(C)(-i)\]
\[ = \begin{cases} \bigoplus_{1 \leq i \leq \lfloor \frac{j}{2}\rfloor} \bigwedge^{j-2i} H^1(C)(-i) & j-(k-a) \leq 0 \\\bigoplus_{j-(k-a) \leq i \leq \lfloor \frac{j}{2}\rfloor} \bigwedge^{j-2i} H^1(C)(-i) & \text{otherwise}\end{cases}.\]

So we conclude that we have
\[ H^j(C^{(k-a)})/H^{j-2}(C^{(k-a-1)})(-1) \cong \begin{cases} \bigwedge^{j} H^1(C) & j \leq (k-a) \\ 0 & \text{otherwise}\end{cases}.\]

We conclude that
\[ \cH^{j+1-2k} i_x^* {\rm IC}_{\sigma_k}^H \cong \begin{cases} \bigwedge^{j} H^1(C) & 0 \leq j \leq (k-a) \\ 0 & \text{otherwise}\end{cases}.\]

It is trivial to compute the dimension of the Hodge pieces:
\[ \dim_{\C} {\rm Gr}^F_{p} {\rm Gr}^W_{-p-q} \bigwedge^j H^1(C) = \begin{cases} 0 & -p-q \neq j, p \notin [-j,0] \\ \binom{g}{-p} \binom{g}{-q} & \text{ otherwise}\end{cases},\]
which concludes the proof of the theorem by taking $-r = 2k-1-j$, so that $j = 2k-1-r$.
\end{proof}

%\section{Results on Cohomology}

\subsection{Defect in $\Q$-factoriality}
%We return now to the setting of higher secant varieties of smooth projective curves. We will provide some partial results on the cohomology (which are sufficient to understand $\Q$-factoriality), and conclude with a conjecture for the singular cohomology in general.

%Let $\sigma_k$ be the $k$th secant variety for $C \subseteq \P^N$ embedded by the $(2k-1)$-very ample line bundle $L$. We assume $\sigma_k \neq \P^N$.
We start with the following general fact by \cite{PPFactorial}*{Lem. 3.6}. Note that normality is not required, and we simply need the bound on the codimension of the singular locus.

\begin{lem} \label{lem-topSing} Let $X$ be a projective $d$-equidimensional variety with ${\rm codim}_X(X_{\rm sing}) >1$. Then ${\rm IH}^{2d-1}_{\cM}(X) \cong H^{2d-1}_{\cM}(X)$.
\end{lem}

Recall that $B^k$ is a $\P^{k-1}$-bundle over $C^{(k)}$, so its cohomology can be written
\[ H^j_{\cM}(B^k) = \bigoplus_{i \geq 0} H^{j-2i}_{\cM}(C^{(k)})\xi^i \]
where $\xi = c_1(\cO(1))$ is the relative hyperplane class.

\begin{lem} \label{lem-Injective} For all $j \leq 2k-3$, the restriction map
\[ H^j_{\cM}(B^k) \to H^j_{\cM}(Z^k_{k-1}) \text{ is injective.}\]

In general, for any $j$, the composition
\[ \bigoplus_{i = 0}^{\min\{ k-2, \lfloor \frac{j}{2}\rfloor\}} H^{j-2i}_{\cM}(C^{(k)})\xi^i\to H^j_{\cM}(B^k) \to H^j_{\cM}(Z^k_{k-1})\]
is injective.
\end{lem}
\begin{proof} The map $H^j_{\cM}(B^k) \to H^j_{\cM}(Z^k_{k-1})$ sits in the composition
\[ H^j_{\cM}(B^k) \to H^j_{\cM}(Z^k_{k-1}) \to H^j_{\cM}(C\times B^{k-1}),\]
where the composed map is $\alpha^*$. As $B^k$ is a $\P^{k-1}$-bundle over $C^{(k)}$ (resp. $C\times B^{k-1}$ is a $\P^{k-2}$-bundle over $C\times C^{(k-1)}$), we can write $\alpha^*$ as 
\[ \bigoplus_{i = 0}^{\min\{ k-1, \lfloor \frac{j}{2}\rfloor\}} H^{j-2i}_{\cM}(C^{(k)})\xi^i \to \bigoplus_{i = 0}^{\min\{ k-2,\lfloor \frac{j}{2}\rfloor\}} H^{j-2i}_{\cM}(C\times C^{(k-1)})\zeta^i,\]
and by \cite{Brogan}*{Lem. 2.7}, we know that $\alpha^*(\xi) = \zeta$. The map
\[ H^{j-2i}_{\cM}(C^{(k)}) \to H^{j-2i}_{\cM}(C\times C^{(k-1)})\]
is induced by the addition map $C\times C^{(k-1)} \to C^{(k)}$. This is injective: indeed, it is the inclusion of the $S_k$-invariants into the $({\rm id}\times S_{k-1})$-invariants of $H^\bullet_{\cM}(C^{\times k})$.

As $\alpha^*(\xi) = \zeta$, the pullback $\alpha^*$ respects the direct sum decomposition except for the term with $i = k-1$. In particular, the composition
\[ \bigoplus_{i = 0}^{\min\{ k-2, \lfloor \frac{j}{2}\rfloor\}} H^{j-2i}_{\cM}(C^{(k)})\xi^i\to\bigoplus_{i = 0}^{\min\{ k-1, \lfloor \frac{j}{2}\rfloor\}} H^{j-2i}_{\cM}(C^{(k)})\xi^i \to \bigoplus_{i = 0}^{\min\{ k-2,\lfloor \frac{j}{2}\rfloor\}} H^{j-2i}_{\cM}(C\times C^{(k-1)})\zeta^i\]
is injective. So the claim follows.
\end{proof}

This has the following implication (which is also implied by \theoremref{thm-SingCoh}):
\begin{lem} \label{lem-h1zero} For all $k\geq 2$, we have
\[ H^1_{\cM}(\sigma_k) = 0.\]
\end{lem}
\begin{proof} By \cite{Brogan}*{Cor. 2.16} (or the computation following this lemma), the claim holds for $k=2$.

By the discriminant square
\[ \begin{tikzcd} Z^k_{k-1} \ar[r] \ar[d] & B^k \ar[d] \\ \sigma_{k-1} \ar[r] & \sigma_k \end{tikzcd}\]
we have the long exact sequence (see \cite{PetersSteenbrink}*{Cor. A.14})
\[ \dots \to H^j_{\cM}(\sigma_k) \to H^j_{\cM}(\sigma_{k-1}) \oplus H^j_{\cM}(B^k) \to H^j_{\cM}(Z^k_{k-1}) \to \dots,\]
hence 
\[ 0 \to H^1_{\cM}(\sigma_k) \to H^1_{\cM}(\sigma_{k-1}) \oplus H^1_{\cM}(B^k) \to H^1_{\cM}(Z^k_{k-1})\]
is exact.

By induction, $H^1_{\cM}(\sigma_{k-1}) = 0$. So the claim is that the natural map $H^1_{\cM}(B^k) \to H^1_{\cM}(Z^k_{k-1})$ is injective, but this follows immediately by \lemmaref{lem-Injective} above.
\end{proof}

We now work over $\C$.

\begin{cor}[$=$ \corollaryref{charp1}]\label{charp1n} Let $L$ be $(2k-1)$-very ample and assume $\sigma_k\neq\P^N$ for some $k\geq 2$. Then
the following are equivalent:
\begin{itemize} 
\item[(1)] $C \cong \P^1$,
\item[(2)] $\sigma_\ell$ is a rational homology manifold for some (equivalently for all) $2 \leq \ell \leq k$.
\end{itemize}
\textcolor{black}{In addition, assume $\sigma_k$ is normal. Then any of the above is equivalent to any of the following:}
\begin{itemize}
\item[(3)] $\sigma_\ell$ has finite $\Q$-factoriality defect for some (equivalently for all) $2\leq \ell \leq k$,
\item[(4)] $\sigma_\ell$ is $\Q$-factorial for some (equivalently for all) $2\leq \ell \leq k$,
\item[(5)] \textcolor{black}{$\sigma_\ell$ has finite local analytic $\Q$-factoriality defect at every point $y\in\sigma_l$ for some (equivalently for all) $2\leq \ell \leq k$,}
\item[(6)] \textcolor{black}{$\sigma_\ell$ is locally analytically $\Q$-factorial for some (equivalently for all) $2\leq \ell \leq k$.}
\end{itemize}
\end{cor}

\begin{proof}
    The equivalence of (1) and (2) is immediate from \theoremref{thm-ConstantSheafHigherSecants}. We first prove the equivalence of (3), (4) and (1).

    By \cite{PPFactorial}*{Thm. A}, we know that $\sigma_k$ has finite $\Q$-factoriality defect if and only if $h^1(\sigma_k) = h^{2\dim \sigma_k -1}(\sigma_k)$. By \lemmaref{lem-h1zero}, we see that it has finite $\Q$-factoriality defect if and only if $h^{2\dim \sigma_k -1}(\sigma_k) = 0$.

By \lemmaref{lem-topSing}, this is equivalent to the vanishing ${\rm IH}^{2\dim \sigma_k -1}(\sigma_k) = 0$, which by Poincar\'{e} duality is equivalent to ${\rm IH}^1(\sigma_k) =0$. Finally, we have by \cite{Brogan}*{Thm. 2.12} that ${\rm IH}^1(\sigma_k) \cong H^1(C)$.

Thus, we see that the $\Q$-factoriality defect is finite if and only if $C\cong \P^1$. We have already seen that this implies $\sigma_k$ is a rational homology manifold (the same implication is shown in \cite{Brogan}), and so we are done.

We now prove the equivalence of (5), (6) and (1). Assume $\sigma^{\rm an}(\sigma_l)$ is finite at every $y$, then $\sigma(\sigma_l)<\infty$ hence $C\cong\P^1$ by the previous parts. On the other hand, if $C\cong\P^1$, then $\sigma_l$ has rational singularities by \cite{ENP}, whence \cite{PPFactorial}*{Thm. B} and \theoremref{thm-HLHigherSecants}, $\sigma_l$ is locally analytically $\Q$-factorial.
\end{proof}

\begin{rmk}
    In the above, we used the fact that if a normal projective variety $Z$ has finite $\Q$-factoriality defect at every point $z\in Z$. then its $\Q$-factoriality defect $\sigma(Z)$ is finite. One way to see this as follows: assume $\sigma^{\rm an}(Z;z)<\infty$ at every $z\in Z$. Then $R^1f_*\cO_{\Tilde{Z}}=0$ where $f:\Tilde{Z}\to Z$ is a resolution by \cite{PPFactorial}*{Thm. B}. Then, by \cite{AH}*{Thm. 0.5}, we have $H^1(Z)\cong{\rm IH}^1(Z)\cong {\rm IH}^{2\dim Z-1}$ which is isomorphic to $H^{2\dim Z-1}(Z)$ by \lemmaref{lem-topSing}. Consequently $\sigma(Z)<\infty$ by \cite{PPFactorial}*{Thm. A}.
\end{rmk}

\color{black}

\begin{comment}
\begin{cor} For all $k\geq 2$, the variety $\sigma_k$ is $\Q$-factorial if and only if it has finite $\Q$-factoriality defect if and only if $C\cong \P^1$.
\end{cor}
\begin{proof} By \cite{PPFactorial}*{Thm. A}, we know that $\sigma_k$ has finite $\Q$-factoriality defect if and only if $h^1(\sigma_k) = h^{2\dim \sigma_k -1}(\sigma_k)$. By \lemmaref{lem-h1zero}, we see that it has finite $\Q$-factoriality defect if and only if $h^{2\dim \sigma_k -1}(\sigma_k) = 0$.

By \lemmaref{lem-topSing}, this is equivalent to the vanishing ${\rm IH}^{2\dim \sigma_k -1}(\sigma_k) = 0$, which by Poincar\'{e} duality is equivalent to ${\rm IH}^1(\sigma_k) =0$. Finally, we have by \cite{Brogan}*{Thm. 2.12} that ${\rm IH}^1(\sigma_k) \cong H^1(C)$.

Thus, we see that the $\Q$-factoriality defect is finite if and only if $C\cong \P^1$. We have already seen that this implies $\sigma_k$ is a rational homology manifold (the same implication is shown in \cite{Brogan}), and so we are done.
\end{proof}
\end{comment}

\subsection{Singular cohomology}
We conclude by giving a formula for the singular cohomology in an arbitrary theory of mixed sheaves, so our result applies either to Hodge structures or to ${\rm Gal}(\overline{k}/k)$-representations. Recall the theorem statement:

\begin{thm} Assume $L$ is $(2k-1)$-very ample where $k\geq 2$ and $\sigma_k\neq\P^N$. We have the following for $0\leq j\leq 4k-2$:
\begin{itemize}
    \item If $w\leq k$, then 
    \[
    {\rm Gr}_w^WH_{\cM}^j(\sigma_k)=\begin{cases}
        A^{\cM}(-\frac{w}{2}) & j=w\textrm{ is even}\\
        {\rm Sym}^k(H^1_{\cM}(C)) & (w,j)=(k,2k-1)\\
        0 & \textrm{otherwise}
    \end{cases}.
    \]
    \item If $w\geq k+1$, then
    \[
    {\rm Gr}_w^WH_{\cM}^j(\sigma_k)=\begin{cases}
    A^{\cM}(-\frac{w}{2}) & j=w\textrm{ is even},\, w\notin [3k-2,4k-3]\\
        A^{\cM}(-\frac{w}{2})\oplus{\bf S}_{1^{4k-w-2}}(2k-j-1) & j=w\textrm{ is even},\, 3k-2\leq w\leq 4k-3\\
        {\bf S}_{1^{4k-w-2}}(2k-j-1) & j=w\textrm{ is odd},\, 3k-2\leq w\leq 4k-3\\
        {\bf S}_{j-w+1,1^{2w+4k-3j-3}}(2k-j-1) & w\in [\frac{3}{2}j-2k+\frac{3}{2},2j-3k+2],\,j> w\\
        0 & \textrm{otherwise}
    \end{cases}.
    \]
\end{itemize}
\end{thm}

In the above, $\mathbf S_{\lambda}$ denotes the Schur functor associated to the partition $\lambda$ applied to $H^1_{\cM}(C)$. By purity of $H^1_{\cM}(C)$, it is pure of weight $|\lambda|$.

We take the following approach: apply the weight spectral sequence for the projective morphism $\kappa \colon \sigma_k \to {\rm Spec}(k)$ and the object $A^{\cM}_{\sigma_k}[2k-1]$. 

\begin{rmk} \label{rmk-weightd1} The morphisms in the weight spectral sequence are determined in the following way: for an object $M \in \cM(X)$ and a proper morphism $f\colon X \to Y$, consider
\[ 0 \to {\rm Gr}^W_{w-1} M \to (W_w/ W_{w-2})M \to {\rm Gr}^W_w M \to 0,\]
where we use the shorthand $(W_k/W_\ell)(-) = W_k(-)/W_{\ell}(-)$, which is an exact functor. If we apply $f_*$, we get the exact triangle
\[ f_* {\rm Gr}^W_{w-1} M \to f_*((W_w/W_{w-2})M) \to f_* {\rm Gr}^W_{w}(M) \xrightarrow[]{+1},\]
and the morphisms in the weight spectral sequence are the morphisms in cohomology associated to the induced map $f_* {\rm Gr}^W_w(M) \to f_* {\rm Gr}^W_{w-1}(M)[1]$.
\end{rmk}

By applying \theoremref{thm-ConstantSheafHigherSecants} (working in $\cM(-)$), we have that
\[ {\rm Gr}^W_{2k-1-\ell} A_{\sigma_{k}}^{\cM}[2k-1] = {\rm Sym}^\ell(H^1_{\cM}(C)) \boxtimes {\rm IC}_{\sigma_{k-\ell}}^{\cM}.\]

Moreover, by (the obvious analogue in $\cM(-$) of) the weight spectral sequence \cite{SaitoMHM}*{Prop. 2.15} for the constant map $\kappa \colon \sigma_k \to {\rm Spec}(k)$, we have a spectral sequence
\[E_1^{-i,i+j} = \cH^j \kappa_*({\rm Gr}^W_i A_{\sigma_k}^{\cM}[2k-1])\]
that degenerates at $E_2$ (since $E_1^{-i,i+j}$ is pure of weight $i+j$) and satisfies
\[ E_2^{-i,i+j} = {\rm Gr}^W_{i+j} H^{\dim \sigma_k +j}_{\cM}(\sigma_k).\]

We re-index to see that the $\ell$th cohomology of the complex
\begin{equation} \label{cx-WSS} 0 \to {\rm IH}^p_{\cM}(\sigma_k) \xrightarrow[]{d_1} {\rm Sym}^1(H^1_{\cM}(C)) \otimes {\rm IH}^{p-1}_{\cM}(\sigma_{k-1}) \xrightarrow[]{d_1} \dots \xrightarrow[]{d_1} {\rm Sym}^{k-1}(H^1_{\cM}(C)) \otimes H^{p+1-k}_{\cM}(C) \to 0\end{equation}
is isomorphic to ${\rm Gr}^W_p H^{p+\ell}_{\cM}(\sigma_k)$.

\begin{rmk} \label{rmk-BroganIH} We will use Brogan's description of ${\rm IH}^j(\sigma_k)$ as $\ker(\pi_{H^2,*} \alpha^* \colon H^j(B^k) \to H^2(C)\otimes H^{j-2}(B^{k-1}))$, where $\alpha \colon C\times B^{k-1} \to B$ is the natural morphism and $\pi_{H^2,*} \colon H^j(C\times B^{k-1}) \to H^2(C)\otimes H^{j-2}(B^{k-1})$ is the K\"{u}nneth projection. We use this notation because we will also be interested in the morphism
\[ \pi_{H^1,*} \colon H^j(C\times B^{k-1}) \to H^{1}(C) \otimes H^{j-1}(B^{k-1}).\]
\end{rmk}

Our strategy to determine the $d_1$ morphisms in the weight spectral sequence is the following. By Remark \ref{rmk-weightd1}, we want to first give a useful interpretation of the extension between consecutive ${\rm Gr}^W_w(-)$ pieces of $A^{\cM}_{\sigma_k}[2k-1]$. We will do this for $w = 2k-1$ by hand, and then show that for $w < 2k-1$ the extension is induced by $H^1(C)\boxtimes -$ applied to the corresponding extension for $\sigma_{k-1}$, which will allow us to use induction on $k$.

For notational convenience, we let
\[ \eta_w^{(k)} \colon {\rm Gr}^W_w A^{\cM}_{\sigma_k}[2k-1] \to {\rm Gr}^W_{w-1} A^{\cM}_{\sigma_k}[2k] \]
denote the extension class induced by $(W_w/W_{w-2})A^{\cM}_{\sigma_k}[2k-1]$.

\begin{prop} Let $L$ be $(2k-1)$-very ample on $C$ and assume $\sigma_k \neq \P^N$. Then we have a commutative diagram
\[ \begin{tikzcd} {\rm Gr}^W_{2k-1} A^{\cM}_{\sigma_k}[2k-1] \ar[r,"\eta^{(k)}_{2k-1}"] \ar[d] & {\rm Gr}^W_{2k-2} A^{\cM}_{\sigma_k}[2k] \ar[d] \\ p_* A^{\cM}_{B^k}[2k-1] \ar[r] & H^1_{\cM}(C)\boxtimes p_* A^{\cM}_{B^{k-1}}[2k-2]\end{tikzcd},\]
where the bottom map is the canonical one induced by $\alpha$ and the K\"{u}nneth projection.

Moreover, for $w < 2k-1$, we have a commutative diagram
\[ \begin{tikzcd} H^1_{\cM}(C) \boxtimes {\rm Gr}^W_{w-1} A^{\cM}_{\sigma_k}[2k-3] \ar[r,"\eta^{(k-1)}_{w-1}"] \ar[d] & H^1_{\cM}(C) \boxtimes {\rm Gr}^W_{w-2} A^{\cM}_{\sigma_k}[2k-2] \ar[d] \\  {\rm Gr}^W_{w} A^{\cM}_{\sigma_k}[2k-1] \ar[r,"\eta^{(k)}_{w}"] & {\rm Gr}^W_{w-1} A^{\cM}_{\sigma_k}[2k]\end{tikzcd},\]
where, under the identification of all four terms in \theoremref{thm-ConstantSheafHigherSecants}, the vertical morphisms are induced by the natural, surjective multiplication maps
\[ H^1_{\cM}(C) \otimes {\rm Sym}^a(H^1_{\cM}(C)) \to {\rm Sym}^{a+1}(H^1_{\cM}(C)).\]
\end{prop}
\begin{proof} We begin with the first claim. 

We mod out by $W_{2k-1-2}(-)$ in the triangles in the definition of $\cC^k$, yielding:
\begin{equation} \label{eq-tri1ext} A_{\sigma_k}^{\cM}/W_{2k-3} A_{\sigma_k}^{\cM}\to p_* A_{B^k}^{\cM} \to \cC^k/W_{2k-3}\cC^k \xrightarrow[]{+1},\end{equation}
\[ p_* A_{Z^k_{k-1}}^{\cM}/W_{2k-3} p_* A_{Z^k_{k-1}}^{\cM} \cong \cC^k/W_{2k-3}\cC^k.\]

By construction, if we write $p_* A_{B^k}^{\cM} = {\rm IC}_{\sigma_k}^{\cM} \oplus \bigoplus_{\ell < k} {\rm IC}_{\sigma_{\ell}}^{\cM}(\ell-k)$, then the map
\[ p_* A_{B^k}^{\cM} \to \cH^{2k-1} \cC^k = \cH^{2k-1} \cC^k/W_{2k-3}\cC^k\]
is the natural projection $p_* A_{B^k}^{\cM} \to \bigoplus_{\ell < k} {\rm IC}_{\sigma_{\ell}}^{\cM}(\ell-k)$. We will then use the following basic lemma:
\begin{lem} Let $A,B,C \in \cA$ lie in an abelian category with $p = (f,0) \colon A \to B\oplus C$ and let $D = C(p)$. Assume $f\colon A \to B$ is surjective. Then the induced map $C \to \cH^0 D$ is an isomorphism and we have a canonical morphism of exact triangles in $D^b(\cA)$:
\[ \begin{tikzcd} A\ar[r,"f"] \ar[d,"{\rm id}"] & B \ar[r] \ar[d] & (\cH^{-1} D)[1] \ar[r,"+1"]  \ar[d] & {}\\ A\ar[r,"p"] & B \oplus C \ar[r,"g"] & D \ar[r,"+1"] & {}\end{tikzcd}\]
where the vertical morphisms are the canonical morphisms.
\end{lem}
\begin{proof} This is immediate from the mapping cone construction.
\end{proof}

We apply this with the natural morphism
\[ \cH^{2k-2}(\cC^k/W_{2k-3}\cC^k)[2-2k] \to \cC^k/W_{2k-3}\cC^k,\]
which gives a morphism of triangles
\begin{equation} \label{eq-tri2ext} \begin{tikzcd} A_{\sigma_k}^{\cM}/W_{2k-3} A_{\sigma_k}^{\cM} \ar[r] \ar[d,equal] & {\rm IC}_{\sigma_k}^{\cM}[1-2k] \ar[r] \ar[d] & \cH^{2k-2}(\cC^k/W_{2k-3}\cC^k) [2-2k] \ar[d] \ar[r,"+1"] & {} \\  A_{\sigma_k}^{\cM}/W_{2k-3} A_{\sigma_k}^{\cM}\ar[r] & p_* A_{B^k}^{\cM} \ar[r] & \cC^k/W_{2k-3}\cC^k \ar[r,"+1"] & {} \end{tikzcd}.\end{equation}

By construction, the top triangle is the extension we are after. 

The natural map $p_* A_{B^k}^{\cM} \to \cC^k$ factors through $p_* A_{Z^k_{k-1}}^{\cM}$ via the adjunction for $i\colon \sigma_{k-1} \to \sigma_k$, and so the bottom triangle of the diagram \eqref{eq-tri2ext} is canonically isomorphic to
\[ A_{\sigma_k}^{\cM}/W_{2k-3} A_{\sigma_k}^{\cM}\to p_* A_{B^k}^{\cM}\to p_* A_{Z^k_{k-1}}^{\cM}/W_{2k-3} p_* A_{Z^k_{k-1}}^{\cM}  \xrightarrow[]{+1}.\]

We have the composition $C\times B^{k-1} \to Z^k_{k-1} \to B^k$ giving the morphisms
\[ p_*A_{B^k}^{\cM} \to p_* A_{Z^k_{k-1}}^{\cM} \to p_* A_{C\times B^{k-1}}^{\cM}\]
and hence
\[ p_*A_{B^k}^{\cM} \to p_* A_{Z^k_{k-1}}^{\cM} /W_{2k-3}p_* A_{Z^k_{k-1}}^{\cM} \to p_* A_{C\times B^{k-1}}^{\cM}/W_{2k-3}p_* A_{C\times B^{k-1}}^{\cM}.\]

In the proof of \theoremref{thm-ConstantSheafHigherSecants}, we show that ${\rm Gr}^W_{2k-2} \cH^{2k-2} \cC^k$ is naturally a sub-quotient of $H^1_{\cM}(C) \boxtimes \cH^{2k-3} p_* A^{\cM}_{B^{k-1}}$ via the natural map. In other words, we have a commutative diagram
\[\begin{tikzcd} {\rm IC}_{\sigma_k}^{\cM}[1-2k] \ar[r] \ar[d] & ({\rm Gr}^W_{2k-2} \cH^{2k-2} \cC^k)[2-2k] \ar[r] \ar[d] & H^1_{\cM}(C) \boxtimes \cH^{2k-3} p_* A^{\cM}_{B^{k-1}} \ar[d] \\ p_* A_{B^k}^{\cM} \ar[r] & p_* A^{\cM}_{Z^{k}_{k-1}}/W_{2k-3} p_* A^{\cM}_{Z^k_{k-1}} \ar[r] & H^1_{\cM}(C)\boxtimes p_* A^{\cM}_{B^{k-1}} [2k-2] \end{tikzcd},\]
which proves the first claim.

We now handle the extension between $w=2k-2$ and $w=2k-3$. We will apply the above identification of the extension inductively, for $\sigma_{k-1}$. Recall that the extension for lower weights of $A^{\cM}_{\sigma_k}$ goes through the triangle
\[ \cC^k \to T \to p_*\Sigma^{k,k-1}_{k-2} \xrightarrow[]{+1},\]
and we apply to this triangle the exact functor $W_{2k-2}/W_{2k-4}$. 

As $p_*\Sigma^{k,k-1}_{k-2}$ is supported on $\sigma_{k-2}$, we know that the defining morphism $p_* A^{\cM}_{C\times B^{k-1}} \to p_* \Sigma^{k,k-1}_{k-2}$ factors (via the adjunction for $i\colon \sigma_{k-2} \to \sigma_k$) through $p_* A^{\cM}_{C\times Z^{k-1}_{k-2}}$. Moreover, $T$ naturally receives a morphism from $p_* A^{\cM}_{C\times B^{k-1}}$ in such a way that the following gives a commutative diagram of exact triangles
\[ \begin{tikzcd} \cK \ar[r] \ar[d] & p_* A^{\cM}_{C\times B^{k-1}} \ar[r] \ar[d] & p_* A^{\cM}_{C\times Z^{k-1}_{k-2}} \ar[d] \\ \cC^k \ar[r] & T \ar[r] & p_*\Sigma^{k,k-1}_{k-2}\end{tikzcd}.\]

Now, we can apply the K\"{u}nneth decomposition to the map $p_* A^{\cM}_{C\times B^{k-1}} \to p_* A^{\cM}_{C\times Z^{k-1}_{k-2}}$ to get a direct summand of
\[ H^1(C)\boxtimes p_* A^{\cM}_{B^{k-1}}[-1] \to H^1(C)\boxtimes p_* A^{\cM}_{Z^{k-1}_{k-2}}.\]

From the previous step, this part encodes the extension between $w=2k-3$ and $w=2k-4$ for $A^{\cM}_{\sigma_{k-1}}$, tensored with $H^1_{\cM}(C)$. By composing and looking at the long exact sequence in cohomology, we get a morphism of exact sequences
\[ \begin{tikzcd}   H^1(C)\boxtimes {\rm Gr}^W_{2k-4} A^{\cM}_{\sigma_{k-1}} \ar[r] \ar[d] & H^1(C)\boxtimes (W_{2k-3}/ W_{2k-5})A^{\cM}_{\sigma_{k-1}} \ar[r] \ar[d] & H^1(C) \boxtimes {\rm Gr}^W_{2k-3}A^{\cM}_{\sigma_{k-1}} \ar[d]   \\    {\rm coker}(\beta) \ar[r,"\tau"] & (W_{2k-2}/W_{2k-4}) \cH^{2k-2} \cC^k \ar[r] & \ker(\gamma)\end{tikzcd}.\]

The map on the right is the identity on $H^1(C)\boxtimes {\rm IC}_{\sigma_{k-1}}^{\cM}$. 

For the map on the left, we apply \corollaryref{cor-rhoMorphism}. 

Next, we study the extension between $w=2k-3$ and $w=2k-4$. As above, we have the morphism
\[ p_* A^{\cM}_{C\times Z^{k-1}_{k-2}} \to p_* \Sigma^{k,k-1}_{k-2},\]
and so applying $\cH^{2k-3}$, we get a morphism in $\cM(\sigma_k)$:
\[ H^1(C)\boxtimes \cH^{2k-4} p_* A^{\cM}_{Z^{k-1}_{k-2}} \to \cH^{2k-3} p_* \Sigma^{k,k-1}_{k-2},\]
and the result follows from looking at the morphism of short exact sequences obtained by the usual short exact sequence
\[ 0 \to {\rm Gr}^W_{2k-4}(-) \to (W_{2k-3}/W_{2k-5})(-) \to {\rm Gr}^W_{2k-3}(-) \to 0\]
and by applying \corollaryref{cor-rhoMorphism} to identify the outer vertical morphisms.

Finally, for $w < 2k-4$, we make use of the morphism
\[ H^1(C)\boxtimes p_* \Sigma^{k-1,k-2}_{k-3}[-2] \to p_* \Sigma^{k,k-1}_{k-2}.\]

If we take $\cH^{2k-3}$, we get a morphism in $\cM(\sigma_k)$:
\[ H^1(C)\boxtimes \cH^{2k-5}p_* \Sigma^{k-1,k-2}_{k-3} \to \cH^{2k-3} p_* \Sigma^{k,k-1}_{k-2},\]
and the claim follows by once again looking at the morphism of short exact sequences associated to 
\[ 0 \to {\rm Gr}^W_{w-1}(-) \to (W_w/W_{w-2})(-) \to {\rm Gr}^W_w(-) \to 0\]
and identifying the outer two vertical morphisms by \corollaryref{cor-rhoMorphism}.
\end{proof}

By Remark \ref{rmk-weightd1}, we can easily determine the morphisms in the weight spectral sequence.

\begin{cor} We have the commutative diagrams
\[ \begin{tikzcd} {\rm IH}^p_{\cM}(\sigma_{k}) 
\ar[r,"d_1"] \ar[d] & H^1_{\cM}(C) \otimes {\rm IH}^{p-1}_{\cM}(\sigma_{k-1}) \ar[d] \\ H^j_{\cM}(B^k) \ar[r,"\pi_{H^1,*} \alpha^*"] & H^1_{\cM}(C)\otimes H^{j-1}_{\cM}(B^{k-1})\end{tikzcd},\]
where the horizontal maps come from identifying the intersection cohomology with $\ker(\pi_{H^2,*}\alpha^*)$.

Moreover, for $\ell > 0$, we have a commutative diagram
\[ \begin{tikzcd} H^1_{\cM}(C) \otimes {\rm Sym}^{\ell-1}(H^1_{\cM}(C)) \otimes {\rm IH}^{p-\ell}_{\cM}(\sigma_{k-\ell}) \ar[r,"d_1"] \ar[d] & H^1_{\cM}(C) \otimes {\rm Sym}^{\ell}(H^1_{\cM}(C)) \otimes {\rm IH}^{p-\ell-1}(\sigma_{k-\ell-1}) \ar[d] \\  {\rm Sym}^\ell(H^1_{\cM}(C)) \otimes {\rm IH}^{p-\ell}_{\cM}(\sigma_{k-\ell}) \ar[r,"d_1"] & {\rm Sym}^{\ell+1}(H^1_{\cM}(C)) \otimes {\rm IH}^{p-\ell-1}_{\cM}(\sigma_{k-\ell-1})\end{tikzcd},\]
where the vertical morphisms are the natural surjections.
\end{cor}

It is clear now that we must analyze the composition $\pi_{H^1,*} \alpha^* \colon H^j(B^k) \to H^1(C) \otimes H^{j-1}(B^{k-1})$ at the level of underlying $A$-vector spaces (we omit the symbol $\cM$ below for this reason). This can be done with linear algebra. Indeed, recall that $B^k$ is a projective bundle over $C^{(k)}$, and we can rewrite
\[ \alpha^* \colon H^j(B^k) = \bigoplus_{i=0}^{k-1} H^{j-2i}(C^{(k)}) \xi^i \to H^j(C\times B^{k-1}) = \bigoplus_{i=0}^{k-2} H^{j-2i}(C\times B^{k-1}),\]
where \cite{Brogan} shows that $\alpha^*(\xi) = \xi$ and hence for $i < k-1$, the induced morphism
\[ H^{j-2i}(C^{(k)}) \xi^i \to H^{j-2i}(C\times C^{(k-1)}) \xi^i\]
is given by pull-back along the addition map $C\times C^{(k-1)} \to C^{(k)}$. The case $i=k-1$ requires special attention.

\begin{rmk} \label{rmk-KoszulMaps} For $V$ an $A$-vector space, there are natural ${\rm GL}(V)$-equivariant morphisms 
\[ d_{\rm Kosz}\colon {\rm Sym}^a V \otimes \bigwedge^b V \to {\rm Sym}^{a+1} V \otimes \bigwedge^{b-1} V\]
\[ w_1\otimes w_a \otimes (v_1 \wedge \dots \wedge v_b) \mapsto \sum_{\ell=1}^b (-1)^{\ell-1} w_1\otimes w_a \otimes v_\ell \otimes (v_1\wedge \dots \wedge \widehat{v}_\ell \wedge \dots \wedge v_b),\]
which we call ``Koszul differentials'' (following \cite{Eisenbud} the complex formed by their composition gives a ``Koszul strand''). 

By definition, the following diagram commutes
\[ \begin{tikzcd} V\otimes {\rm Sym}^a V \otimes \bigwedge^b V \ar[r,"d_{\rm Kosz}"] \ar[d] & V\otimes {\rm Sym}^{a+1} V \otimes \bigwedge^{b-1} V \ar[d] \\ {\rm Sym}^{a+1} V \otimes \bigwedge^b V \ar[r,"d_{\rm Kosz}"] & {\rm Sym}^{a+2} V \otimes \bigwedge^{b-1} V\end{tikzcd}.\]

In terms of irreducible ${\rm GL}(V)$-representations, we can apply the Pieri rule to write
\[ d_{\rm Kosz} \colon \mathbf S_{(a,1^b)} \oplus \mathbf S_{(a+1,1^{b-1})} \to \mathbf S_{(a+1,1^{b-1})} \oplus \mathbf S_{(a+2,1^{b-2})}\]
so that $\ker(d_{\rm Kosz}) = \mathbf S_{(a,1^b)}$ and ${\rm coker}(d_{\rm Kosz}) \cong \mathbf S_{(a+2,1^{b-2})}$.
\end{rmk}

Up to the middle degree, we can identify the left-most differential in the weight spectral sequence as a Koszul differential:

\begin{prop} \label{prop-KoszulDiff} For $p\leq 2k-1$, using the isomorphism ${\rm IH}^p(\sigma_k) \cong \bigoplus_{\max\{p-k,0\} \leq 2i} \bigwedge^{p-2i} H^1_{\cM}(C)$, the morphism
\[ {\rm IH}^p(\sigma_k) \xrightarrow[]{d_1} H^1(C)\otimes {\rm IH}^{p-1}(\sigma_{k-1})\]
splits into morphisms
\[ \bigwedge^{p-2i} H^1_{\cM}(C) \xrightarrow[]{d_1} H^1(C)\otimes \bigwedge^{p-1-2i} H^1_{\cM}(C),\]
which agree with the \emph{Koszul morphism} (Remark \ref{rmk-KoszulMaps}).
\end{prop}
\begin{proof} This identification requires some $\mathfrak S_k$-equivariant multi-linear algebra, though is elementary.

Recall that we identify $H^\bullet(C^{(k)})$ with the $\mathfrak S_k$-invariants inside $H^\bullet(C^{\times k})$, where the action is signed. Specifically, for $\sigma \in \mathfrak S_k$, we have
\[ \sigma \cdot v_1\otimes \dots \otimes v_k = (-1)^{\sum_{i<j,\sigma(i)>\sigma(j)} \deg(v_i)\deg(v_j)} v_{\sigma^{-1}(1)} \otimes \dots \otimes v_{\sigma^{-1}(k)}.\]

We describe a basis for the $\mathfrak S_{k}$-invariants. For this, define a symmetrizing functor $\Sigma_k$ by
\[ \Sigma_k(v) = \sum_{\sigma \in\mathfrak S_k} \sigma \cdot v.\]

For a basis of $H^\bullet(C)$, we use $1\in H^0(C), c_1(L) \in H^2(C)$ and $\gamma_1,\dots, \gamma_{2g} \in H^1(C)$. Let $\eta_{p,q,J} = 1\otimes \dots \otimes 1 \otimes c_1(L) \otimes \dots \otimes c_1(L) \otimes \gamma_{j_1}\otimes \dots \otimes \gamma_{j_b}$, where $|J| = b = k-p-q$. Then define
\[ \eta_{p,q,J}^{(k)} = \Sigma_k(\eta_{p,q,J}),\]
and the collection of all such symmetric vectors as $p,q,J$ vary (such that $p+q+|J| = k$) gives a basis for $H^\bullet(C^{(k)})$. 

Similarly, we define the symmetrization operator $\Sigma_{k-1}(-)$ for the action of $\mathfrak S_{k-1}$ on $H^\bullet(C^{\times k-1})$. We view $H^\bullet(C\times C^{(k-1)})$ inside $H^\bullet(C^{\times k})$ as the vectors which are ${\rm id}\times \mathfrak S_{k-1}$-invariant. A basis for this subspace is given by
\[ 1 \otimes \eta_{p,q,J}^{(k-1)}, \, c_1(L)\otimes \eta_{p,q,J}^{(k-1)}, \, \gamma_i \otimes \eta_{p,q,J}^{(k-1)},\]
where $i \in \{1,\dots ,2g\}$ and $p,q,J$ vary such that $p+q+|J| = k-1$.

The map $\alpha^* \colon H^\bullet(C^{(k)}) \to H^\bullet(C\times C^{(k-1)})$ is identified with the natural inclusion of $\mathfrak S_k$-invariants into ${\rm id}\times \mathfrak S_{k-1}$-invariants. To understand the compositions $\pi_{H^\ell,*} \alpha^*$ for $\ell=1,2$, we simply need to start with a $\mathfrak S_k$-invariant vector, re-express it in the basis of ${\rm id}\times \mathfrak S_{k-1}$-vectors given above, and then project to the corresponding sub-basis. 

There is a simple rule for this operation on basis vectors $\eta_{p,q,J}^{(k)}$, whose verification we leave to the reader: we have
\[ \eta_{p,q,J}^{(k)} = p 1\otimes \eta_{p-1,q,J}^{(k-1)} + q c_1(L) \otimes \eta_{p,q-1,J}^{(k-1)} + \sum_{\ell=1}^b (-1)^{\ell-1} \gamma_{j_b} \otimes \eta_{p,q,J\setminus \{j_b\}}^{(k-1)}.\]

This allows us to easily determine $\ker(\pi_{H^2,*} \alpha^*)$: indeed, $\sum_{p,q,J} c_{p,q,J} \eta_{p,q,J}^{(k)}$ lies in the kernel if and only if $qc_{p,q-1,J} = 0$ for all $q>0$, if and only if $c_{p,q,J} = 0$ for all $q>0$. 

Moreover, the morphism $\pi_{H^1,*} \alpha^*$ sends a vector $\sum_{p,J} c_{p,0,J} \eta_{p,q,J}^{(k)}$ in that kernel to
\[\sum_{p,J} \sum_{\ell=1}^b (-1)^{\ell-1} \gamma_{j_b} \otimes \eta_{p,0,J\setminus \{j_b\}}^{(k-1)},\]
which verifies that its image lies in $H^1(C)\otimes \ker(\pi_{H^2,*}\alpha^*)$.

Moreover, using the identification $\eta_{p,0,J}^{(k)} = \gamma_{j_1}\wedge \dots \wedge \gamma_{j_b}$, we see that this map agrees with the Koszul differential as in Remark \ref{rmk-KoszulMaps}.

Now, for $p\leq 2k-3$, we have
\[ \bigoplus_{i=0}^{k-2} H^{p-2i}(C^{(k)})\xi^i \xrightarrow[]{\pi_{H^2,*}\alpha^*} \bigoplus_{i=0}^{k-2} H^2(C)\otimes H^{p-2i-2}( C^{(k-1)})\xi^i\]
is the direct sum of these maps, proving the claim in this case. Indeed, $\sum_{i=0}^{k-2} \eta_i \xi^i$ lies in $\ker(\pi_{H^2,*} \alpha^*)$ if and only if $\eta_i \in \ker(\pi_{H^2,*} \alpha^*)$ for all $i$. Moreover, we see that for any value of $p$ (even greater than $2k-1$), we have containment
\[ \bigoplus_{i=0}^{k-2} \left(\ker(\pi_{H^2,*}\alpha^*) \subseteq H^{j-2i}(C^{(k)})\right)\xi^i \subseteq \ker(\pi_{H^2,*}\alpha^*).\]

For $p=2k-2, 2k-1$, we have to consider the terms for $i=k-1$. We have (by the projective bundle formula for singular cohomology)
\[ \alpha^*(\xi^{k-1}) = \alpha^*(\xi)^{k-1} = (-c_1(E_{k-1})\xi^{k-2} - \dots - c_{k-1}(E_{k-1})),\]
and so we see that $\sum_{i=0}^{k-1} \eta_i \xi^i \in \ker(\pi_{H^2,*}\alpha^*)$ if and only if
\[ \pi_{H^2,*}(\alpha^*(\eta_i)) = \pi_{H^2,*}(1\otimes c_{k-1-i}(E_{k-1})) \cup \eta_{k-1}) \text{ for all } i.\]

For $p=2k-2$ (resp. $p=2k-1$) we have $\eta_{k-1} \in H^0(C^{(k)})$ (resp. $\eta_{k-1} \in H^1(C^{(k)})$). In particular, it is a multiple of $\eta_{0,0,\emptyset}^{(k)}$ (resp. a linear combination of elements of the form $\eta_{0,0,\{j\}}^{(k)}$). As cupping with $1\otimes c_{k-1-i}(E_{k-1})$ doesn't change the very first entry in a vector, we have $\pi_{H^2,*}(1\otimes c_{k-1-i}(E_{k-1}) \cup \eta_{k-1}) =0$ for all $i$. Thus, in these cases, we have equality
\[ \bigoplus_{i=0}^{k-1} \left(\ker(\pi_{H^2,*}\alpha^*) \subseteq H^{j-2i}(C^{(k)})\right)\xi^i=\ker(\pi_{H^2,*}\alpha^*),\]
and the description of the morphism in the proposition statement follows.
\end{proof}

For higher degrees, we use Hard Lefschetz \theoremref{thm-HardLefschetz}. As $\sigma_{k-1} \subseteq \sigma_k \subseteq \P^N$, we can use a very ample class $L_k$ on $\sigma_k$ which restricts to one, $L_{k-1}$, on $\sigma_{k-1}$. Hence, for any $\ell$, the square
\[ \begin{tikzcd} {\rm IH}^{2k-1-\ell}(\sigma_k)(-\ell) \ar[r,"d_1"] \ar[d,"c_1(L_k)^{\ell}"] & H^1(C)\otimes {\rm IH}^{2k-3-(\ell-1)}(\sigma_{k-1})(-\ell)  \ar[d,"{\rm id}\otimes c_1(L_{k-1})^{\ell}"]\\{\rm IH}^{2k-1+\ell}(\sigma_k) \ar[r,"d_1"] & H^1(C) \otimes {\rm IH}^{2k-3+(\ell-1)+2}(\sigma_{k-1}) \end{tikzcd},\]
where the left vertical morphism is an isomorphism, but the right vertical one is not. Indeed, we see that the right vertical morphism is surjective, with kernel (by definition) the primitive classes in ${\rm IH}^{2k-3-(\ell-1)}(\sigma_{k-1})$. 

\begin{cor} \label{cor-KoszulDiffHL} For $p = 2k-1+\ell > 2k-1$, the morphism
\[ {\rm IH}^p(\sigma_k) \xrightarrow[]{d_1} H^1(C)\otimes {\rm IH}^{p-1}(\sigma_{k-1})\]
is identified via Hard Lefschetz with the following: we have
\[ {\rm IH}^p(\sigma_k) \cong {\rm IH}^{2k-1-\ell}(\sigma_k)(-\ell) \cong \bigoplus_{\max\{k-1-\ell,0\} \leq 2i} \bigwedge^{2k-1-\ell-2i} H^1(C)(-i-\ell)\]
and using a similar decomposition for ${\rm IH}^{p-1}(\sigma_{k-1})$, we have that $d_1$ decomposes in the direct sum, and
\[ d_{1}\vert_{\bigwedge^{2k-1-\ell-2i} H^1(C)} = \begin{cases} 0 & i=0 \text{ and } \ell \geq k-1 \\ d_{\rm Kosz} & \text{otherwise}\end{cases}.\]
\end{cor}
\begin{proof} By the discussion preceding the corollary statement, we need only understand the primitive pieces of ${\rm IH}^{2k-3-(\ell-1)}(\sigma_{k-1})$. By the Hard Lefschetz theorem, the primitive pieces are precisely the difference between ${\rm IH}^{2k-3-(\ell-1)}(\sigma_{k-1})$ and ${\rm IH}^{2k-3-(\ell-1)-2}(\sigma_{k-1})(-1)$. We have
\[ {\rm IH}^{2k-3-(\ell-1)}(\sigma_{k-1}) = \bigoplus_{\max\{2k-3-(\ell-1)-(k-1),0\}\leq 2i} \bigwedge^{2k-3-(\ell-1)-2i} H^1(C)(-i),\]
\[ {\rm IH}^{2k-3-(\ell-1)-2}(\sigma_{k-1}) = \bigoplus_{\max\{2k-3-(\ell-1)-(k-1)-2,0\}\leq 2i} \bigwedge^{2k-3-(\ell-1)-2-2i} H^1(C)(-i).\]

We see, then, that there is no difference unless $i=0$ is possible in the first direct sum, meaning $\ell \geq k-1$. Hence, we see that
\[ {\rm IH}^{2k-3-(\ell-1)}_{\rm prim}(\sigma_{k-1}) = \begin{cases} \bigwedge^{2k-3-(\ell-1)} H^1(C) & \ell \geq k-1 \\ 0 & \ell < k-1\end{cases}.\]
\end{proof}

\begin{rmk} \label{rmk-restateKDHL} Another way to phrase the result of the corollary is the following: for $p=2k-1+\ell$, using the decomposition
\[ {\rm IH}^{p}(\sigma_k) \cong \bigoplus_{\max\{k-1-\ell,0\}\leq 2i} \bigwedge^{2k-1-\ell-2i} H^1(C)(-i-\ell),\]
the morphism $d_1$ is $0$ on any term with $i=0$ and it is the Koszul differential $d_{\rm Kosz}$ otherwise.
\end{rmk}

We can now complete the proof.

\begin{thm}[$=$ \theoremref{thm-SingCoh}]\label{thm-SingCohn} Assume $L$ is $(2k-1)$-very ample where $k\geq 2$ and $\sigma_k\neq\P^N$. We have the following for $0\leq j\leq 4k-2$:
\begin{itemize}
    \item If $w\leq k$, then 
    \[
    {\rm Gr}_w^WH_{\cM}^j(\sigma_k)=\begin{cases}
        A^{\cM}(-\frac{w}{2}) & j=w\textrm{ is even}\\
        {\rm Sym}^k(H^1_{\cM}(C)) & (w,j)=(k,2k-1)\\
        0 & \textrm{otherwise}
    \end{cases}.
    \]
    \item If $w\geq k+1$, then
    \[
    {\rm Gr}_w^WH_{\cM}^j(\sigma_k)=\begin{cases}
    A^{\cM}(-\frac{w}{2}) & j=w\textrm{ is even},\, w\notin [3k-2,4k-3]\\
        A^{\cM}(-\frac{w}{2})\oplus{\bf S}_{1^{4k-w-2}}(2k-j-1) & j=w\textrm{ is even},\, 3k-2\leq w\leq 4k-3\\
        {\bf S}_{1^{4k-w-2}}(2k-j-1) & j=w\textrm{ is odd},\, 3k-2\leq w\leq 4k-3\\
        {\bf S}_{j-w+1,1^{2w+4k-3j-3}}(2k-j-1) & w\in [\frac{3}{2}j-2k+\frac{3}{2},2j-3k+2],\,j> w\\
        0 & \textrm{otherwise}
    \end{cases}.
    \]
\end{itemize}
\end{thm}

\begin{proof} Our goal is to analyze, for all $0\leq w\leq 2(2k-1)$, the complex\small
\[
\begin{tikzcd}
    {\rm IH}^w_{\cM}(\sigma_k) \arrow[r, "d_1"] & H^1_{\cM}(C)\otimes {\rm IH}^{w-1}_{\cM}(\sigma_{k-1}) \arrow[r,"d_1"]& \cdots\arrow[r, "d_1"] & {\rm Sym}^{k-2}(H^1_{\cM}(C)) \otimes {\rm IH}^{w-(k-2)}_{\cM}(\sigma_2) \arrow[d,"d_1"]\\& & & {\rm Sym}^{k-1}(H^1_{\cM}(C))\otimes H^{w-(k-1)}_{\cM}(C).
\end{tikzcd}\]\normalsize

By tracking indices in the weight spectral sequence, we see that the $\ell$th cohomology of this complex is isomorphic to ${\rm Gr}^W_w H^{w+\ell}_{\cM}(\sigma_k)$.

We write $0 \leq \ell = j-w \leq k-1$ and focus on the portion\small
\[\begin{tikzcd}
    {\rm Sym}^{\ell-1}(H^1_{\cM}(C)) \otimes {\rm IH}^{w-\ell+1}_{\cM}(\sigma_{k-\ell+1}) \arrow[r,"d_1"] & {\rm Sym}^{\ell}(H^1_{\cM}(C)) \otimes {\rm IH}^{w-\ell}_{\cM}(\sigma_{k-\ell}) \arrow[d,"d_1"]\\
    & {\rm Sym}^{\ell+1}(H^1_{\cM}(C)) \otimes {\rm IH}^{w-\ell-1}_{\cM}(\sigma_{k-\ell-1}).
\end{tikzcd} \]\normalsize

If $w - \ell \leq 2(k-\ell)-1$, then both morphisms are handled by \propositionref{prop-KoszulDiff}. As any full Koszul strand is exact (which is well-known, or easily seen by decomposing everything by Schur functors) unless it is the single term complex $A$, there are few contributions to cohomology in this case. Indeed, either $\ell = 0$ and we have a direct summand of ${\rm IH}^w_{\cM}(\sigma_k)$ of the form $\bigwedge^0 H^1_{\cM}(C)(-i) = A^{\cM}(-i)$ which is equivalent to $w=2i$ being even, or the complex is of the form
\[ {\rm Sym}^{k-2}(H^1_{\cM}(C)) \otimes {\rm IH}^{2}_{\cM}(\sigma_{2}) \xrightarrow[]{d_1}{\rm Sym}^{k-1}(H^1_{\cM}(C)) \otimes H^1_{\cM}(C) \xrightarrow[]{d_1} 0,\]
the cokernel of which is isomorphic to ${\rm Sym}^k(H^1_{\cM}(C))$ and hence is equal to ${\rm Gr}^W_k H^{2k-1}_{\cM}(\sigma_k)$.

If $w - \ell \geq 2(k-\ell)$, then we must use the description in \corollaryref{cor-KoszulDiffHL}. By Remark \ref{rmk-restateKDHL}, the only way in which this case differs from the previous case is, if we use Hard Lefschetz and decompose into a direct sum of wedge powers, when we take $i=0$ (if such a summand even exists) the direct summand should map by $0$. The previous differential will be the standard Koszul differential onto any non-zero direct summand.

Well, Hard Lefschetz gives that the middle term is identified with
\[ {\rm Sym}^\ell(H^1_{\cM}(C)) \otimes {\rm IH}^{4(k-\ell)-2-(w-\ell)}_{\cM}(\sigma_{k-\ell})(2(k-\ell)+\ell-1-w),\]
though we will rewrite $4(k-\ell)-2-(w-\ell) = 4k-3j-2+2w$. Thus, under the direct sum decomposition, the term corresponding to $i=0$ is $\bigwedge^{4k-3j-2+2w} H^1_{\cM}(C)$, which is only present if $4k-3j-2+2w - (k-\ell) \leq 0$, or equivalently, $w \leq 2j-3k+2$. For this summand, the portion of the complex is then
\[ {\rm Sym}^{\ell-1}(H^1_{\cM}(C)) \otimes \bigwedge^{4k-3j-2+2w+1} H^1_{\cM}(C) \xrightarrow[]{d_{\rm Kosz}} {\rm Sym}^{\ell}(H^1_{\cM}(C)) \otimes \bigwedge^{4k-3j-2+2w} H^1_{\cM}(C) \to 0.\]

If $4k-3j-2+2w =0$, then the complex is exact (it is the rightmost end of a full Koszul strand). Thus, we only need to consider $4k-3j-2+2w > 0$, or equivalently, $w \geq \frac{3}{2}j-2k+\frac{3}{2}$. 

For $\ell > 0$, recall that in Remark \ref{rmk-KoszulMaps} we identified the cokernel of this map to be
\[ \mathbf S_{j-w+1,1^{4k+2w-3j-3}}(2k-j-1),\]
and for $\ell =0$, the cokernel is
\[ \mathbf S_{1^{4k-j-2}}(2k-j-1).\]

As in the previous case, if $\ell =0$ and $w=2i$ is even, then we get a contribution of $A(-\frac{w}{2})$. This proves the theorem.
\end{proof}

We give another proof of \theoremref{thm-SingCoh} for $k=3$. It is important to note that we recover the cohomology by diagram chasing using the resolution of singularities diagram, and not by directly verifying the claims on the maps of the spectral sequence. For larger $k$, this is a rather cumbersome computation. In this proof, some ideas of the general one are presented in a simplified way, such as an analogous computation to that in \propositionref{prop-KoszulDiff}.

\begin{proof}[Alternate Proof for $k=3$] In this case, the spectral sequence gives a three-term complex whose cohomology corresponds to the top, middle, and lower weights of the cohomology of $\sigma_3$, respectively. Moreover, in this case, by semi-simplicity of pure objects it is enough to verify the claim on the top and lower weights, since this immediately implies the claim on the middle weights, and that is what we do next. We carry on the verification without taking into consideration the Tate twists, to simplify notation.

To compute the top weight, we use the following exact sequence
\[ 0 \to \Gr^W_jH^j_{\cM}(\sigma_3) \to \Gr^W_jH^j_{\cM}(\sigma_2) \oplus H^j_{\cM}(B^3) \to \Gr^W_jH^j_{\cM}(Z^3_2) \]
 obtained from the discriminant square of the resolution of $\sigma_3$. In addition, we use that the pullback map
 \[\Gr^W_jH^j_{\cM}(Z^3_2) \to H^j_{\cM}(C\times B^2)\] is injective (see e.g. \cite{PetersSteenbrink}*{Corollary 5.42}), in order to conclude that 
 \[ \Gr^W_jH^j_{\cM}(\sigma_3) \cong \ker\{\Gr^W_jH^j_{\cM}(\sigma_2) \oplus H^j_{\cM}(B^3) \to H^j_{\cM}(C\times B^2) \}\]
given by composition, and note that $H^j_{\cM}(B^3) \to H^j_{\cM}(C\times B^2)$ corresponds to pull-back by $\alpha$. As we described in the proof of \lemmaref{lem-Injective}, this map decomposes as 
\[H^j_{\cM}(C^{(3)}) \oplus H^{j-2}_{\cM}(C^{(3)})\xi \oplus H^{j-4}_{\cM}(C^{(3)}) \xi^2 \to H^j_{\cM}(C\times C^{(2)}) \oplus H^{j-2}_{\cM}(C\times C^{(2)})\zeta, \] and $\alpha^*(\xi) = \zeta$. If we denote $H^j_{\cM}(C\times C^{(2)})/\mathfrak S_3$ the quotient given by the injective map from $H^j_{\cM}(C^{(3)})$ whose is the subspace of $\mathfrak S_3$-invariants (viewed inside $H^\bullet_{\cM}(C^{\times 3})$), we obtain that \small
\begin{equation}\label{eq-TopWeightCoh}
\begin{array}{c}
\Gr^W_jH^j_{\cM}(\sigma_3)\\\vspace{5pt}
\cong \ker\{\Gr^W_jH^j_{\cM}(\sigma_2) \oplus H^{j-4}_{\cM}(C^{(3)}) \xi^2 \to H^j_{\cM}(C\times C^{(2)})/\mathfrak S_3 \oplus (H^{j-2}_{\cM}(C\times C^{(2)})/\mathfrak S_3)\zeta \}.
\end{array}
\end{equation}
 \normalsize Here recall that $\zeta = \pi_2^*(\xi_2)$, with $\pi_2: C\times C^{(2)} \to C^{(2)}$, and $\xi_2$ is the relative hyperplane class for the projective bundle $B^2 = \mathbf{P}(E)$. 

To identify the morphism, we look at underlying $A$-vector spaces.

We denote the generators of $H^1(C,A)$ by $\gamma_1, \ldots, \gamma_{2g}$, so that $\gamma_i \cup \gamma_j = \delta_{j, i+g}$, where $g$ is the genus of the curve. Moreover, by the projective bundle formula, we have \[\xi_2^2 = -c_1(E)\xi_2 - c_2(E),\] with (\cite{BD}) $$c_1(E) = \frac{\deg L -1}{\deg L} (1\otimes c_1(L) + c_1(L) \otimes 1) + \sum{\gamma_i\otimes \gamma_{i+g} - \gamma_{i+g}\otimes \gamma_i}$$ seen as an element of $H^2(C\times C)$; and $$c_2(E) = \frac{\deg L - 1}{2 \deg L} (c_1(L) \otimes c_1(L))$$ seen as an element of $H^4(C\times C,A)$. In particular, this means that in general the image of $\xi^2$ consists of components that are not $\mathfrak S_3$-invariants (except in trivial cases).

For $j=9,10$, this immediately implies that $\Gr^W_jH^j_{\cM}(\sigma_3) \cong H^{j-4}_{\cM}(C^{(3)})$ which verifies the top weight for both the cohomologies. For $j=8$, the codomain of the morphism on the right hand side of the equality \eqref{eq-TopWeightCoh} is zero, so again, $$\Gr^W_8H^8_{\cM}(\sigma_3) \cong H^{4}_{\cM}(C^{(3)}) \cong A^{\cM} \oplus \mathbf S_{1,1},$$ as stated in the statement. Finally, the last term where the cohomology of $\sigma_2$ vanishes is $j=7$, and we have that $H^3_{\cM}(C^{(3)}) \cong \mathbf S_1 \oplus \mathbf S_{1,1,1}$. $\mathbf S_1$ maps to the space corresponding to $H^1_{\cM}(C)\otimes H^4_{\cM}(C^{(2)})$ in the quotient, which is never zero. $\mathbf S_{1,1,1}$ also maps to this subspace, and we can therefore identify $\Gr^W_7H^7_{\cM}(\sigma_3)$ with $\mathbf S_{1,1,1}$ as stated in the conjecture.

Next, for $j=6$, we have that $\Gr^W_6H^6_{\cM}(\sigma_2) = H^4_{\cM}(C^{(2)})\xi_2$ seen inside of $H^6_{\cM}(B^2)$, and the map consists of the pullback given by the projection, therefore, this space maps to the space corresponding to $(H^0_{\cM}(C) \otimes H^4_{\cM}(C^{(2)}))\zeta$. Moreover, $H^2_{\cM}(C^{(3)}) \cong A^{\cM} \oplus \mathbf S_{1,1}$. The image of $\mathbf S_{1,1}$ after cupping with $c_1(E)$ in the quotient of $H^4_{\cM}(C\times C^{(2)})/\mathfrak S_3$ by $(H^0_{\cM}(C) \otimes H^4_{\cM}(C^{(2)}))$, bijects onto the space corresponding to $H^2_{\cM}(C)\otimes \mathbf S_{1,1}\subseteq H^2_{\cM}(C) \otimes H^2_{\cM}(C^{(2)})$. The subspace $A^{\cM}$ also maps into $H^2_{\cM}(C)\otimes \mathbf S_{1,1}$, and for this reason, we can identify $\Gr^W_6H^6_{\cM}(\sigma_3) \cong A^{\cM}$ as stated in the conjecture. 

For $j=5$, $\Gr^W_5H^5_{\cM}(\sigma_2) = H^3_{\cM}(C^{(2)})\xi_2$ seen inside of $H^5_{\cM}(B^2)$, and it maps to the subspace corresponding to $(H^0_{\cM}(C) \otimes H^3_{\cM}(C^{(2)}))\zeta$. We then conclude that $\Gr^W_5H^5_{\cM}(\sigma_3) = 0$, as indicated in the conjecture, by noting that $H^1_{\cM}(C^{(3)})\xi^2$ (via the isomorphism $H^1_{\cM}(C^{(3)}) \cong H^1_{\cM}(C)$) maps to the subspace corresponding to $H^1_{\cM}(C)\otimes H^4_{\cM}(C^{(2)})$ in the quotient in the natural way. 

In the $j=4$ case, it is useful to give a more precise description of \begin{equation*}
    \begin{array}{c}
          \Gr^W_4H^4(\sigma_2,A)\\\vspace{5pt}
          = \{(u,v\xi_2)\in H^4(C^{(2)},A)\oplus H^2(C^{(2)},A)\xi_2\subseteq H^4(B^2,A) : u = v\cup c_1(L) \text{ in } H^4(C\times C,A) \}
    \end{array}
\end{equation*} which is of course isomorphic to $H^2(C^{(2)},A)$. This space maps naturally to each summand via the pullback of the projection as in the cases above. The class $\xi^2$ maps to the same element as the corresponding class $(2c_2(E), c_1(E)\xi_2)$, hence, as predicted in the conjecture, $\Gr^W_4H^4_{\cM}(\sigma_3) \cong A^{\cM}$. 

Finally, using that $\Gr^W_3H^3_{\cM}(\sigma_2) = 0$, the case $j=3$ trivially follows. For $j=2$, $\Gr^W_2H^2_{\cM}(\sigma_2) = H^0_{\cM}(C^{(2)})\xi_2$, which maps to a class that is $\mathfrak S_3$-invariant, hence $\Gr^W_2H^2_{\cM}(\sigma_3)\cong A^{\cM}$. The case $j=1$ was proved in general in \lemmaref{lem-h1zero}, and the $j=0$ case is trivial.

We discuss next the lower weights. The conjecture states that the only nonzero lower weight is given by $\Gr^W_3H^5_{\cM}(\sigma_3) \cong \Sym^3(H^1_{\cM}(C))$. Again, by the discriminant square, the lower weights are described by the long exact sequence
\[ \Gr^W_jH^{j+1}_{\cM}(\sigma_2) \to \Gr^W_jH^{j+1}_{\cM}(Z^3_2) \to \Gr^W_jH^{j+2}_{\cM}(\sigma_3) \to 0.\] We also have by the discriminant square of $Z^3_2$ 
\[H^j_{\cM}(C\times B^2) \oplus H^j_{\cM}(C^{(2)}\times C) \to H^j_{\cM}(C^3) \to \Gr^W_jH^{j+1}_{\cM}(Z^3_2) \to 0.\] By the proof of \corollaryref{cor-WedgePower}, and more precisely, \cite{Totaro1}*{Page 6}, we obtain that for $j\neq 2,3,4,$ $\Gr^W_jH^{j+1}_{\cM}(Z^3_2) = 0$. For $j=2$, we obtain that $\Gr^W_2H^3_{\cM} \cong \Sym^2(H^1_{\cM}(C))$. Using again the discriminant square of $Z^3_2$, we obtain that $$\Gr^W_2H^2_{\cM}(Z^3_2) \cong H^2_{\cM}(C^{(3)}) \oplus H^0_{\cM}(C\times C^{(2)})\zeta,$$ which by looking at the discriminant square of $\sigma_3$, implies that $\Gr^W_2H^3_{\cM}(\sigma_3) = 0$, and then the map from $\Gr^W_2H^3_{\cM}(\sigma_2) \cong \Sym^2(H^1_{\cM}(C))$ is an isomorphism. We conclude that $\Gr^W_2H^4_{\cM}(\sigma_3) = 0$. 

For $j=3$, $\Gr^W_3H^4_{\cM}(Z^3_2)$, $\Gr^W_3\bigwedge^3H^{\bullet}_{\cM}(C) \cong H^0_{\cM}(C)\otimes H^1_{\cM}(C) \otimes H^2_{\cM}(C) \oplus \Sym^3(H^1_{\cM}(C))$, and the image of $H^1_{\cM}(C\times C^{(2)})\zeta$ maps onto the first summand of the quotient. Since $\Gr^W_3H^4_{\cM}(\sigma_2) = 0$, the result follows. 

Finally, for $j=4$, $\Gr^W_4\bigwedge^3H^{\bullet}_{\cM}(C) \cong H^2_{\cM}(C) \times \Sym^2(H^1_{\cM}(C))$. But since the subspace $(H^1_{\cM}(C) \otimes H^1_{\cM}(C^{(2)}))\zeta \subseteq H^2_{\cM}(C\times C^{(2)})\zeta\subseteq H^4_{\cM}(B^2)$ maps onto this subspace, we conclude that $\Gr^W_4H^5_{\cM}(Z^3_2) = 0$, and therefore, $\Gr^W_4H^6_{\cM}(\sigma_3) = 0$. This concludes the proof.
\end{proof}

\section{Open Questions}\label{open}
We collect here some open questions regarding (higher) secant varieties related to the material of the paper. We hope that these will motivate further research on these topics.

\begin{problem}
    Let $C\subset \P^N$ be a smooth projective curve such embedded by the complete linear series of a sufficiently positive line bundle $L$. Let $q_k$ be the codimension of the $k$-th secant variety $\sigma_k\subset\P^N$ and $k\geq 3$. Determine the quantities 
    \begin{equation}\label{*}
        {\rm gl}({\rm IC}_{\sigma_k}^H(-q_k),F)\textrm{ and }{\rm gl}(\cH_{\sigma_k}^{q_k}(\cO_{\P^N}),F).\tag{$*$}
    \end{equation}  
\end{problem}

Recall that \corollaryref{glhs} (and its proof) gives the upper bound $k-1$ of these quantities assuming only $(2k-1)$-very ampleness of $L$. However, it is reasonable to expect that one should be able to compute these invariants under sufficient positivity, as is shown in \corollaryref{cor-GenLevel2Secants}(4)(b) when $k=2$. This brings us to the next problem:

\begin{problem}
    Let $C$ be a smooth projective curve such that its canonical bundle $\omega_C$ is $(2k-1)$-very ample for some $k\geq 2$. Let $q_k$ be the codimension of the $k$-th secant variety $\sigma_k$ of its canonical embedding $C\hookrightarrow\P^N$. Determine the quantities \eqref{*}. 
\end{problem}

A good starting point for the above would be the case $k=2$. In fact, we expect that in the above situation, ${\rm gl}({\rm IC}_{\sigma_2}^H(-q_2),F)\neq 0$; in other words, we suspect that the positivity assumption cannot be removed from \corollaryref{cor-GenLevel2Secants}(4)(b).

Now, suppose $C\subset \P^N$ is a smooth projective curve such embedded by the complete linear series of a $(2k-1)$-very ample line bundle $L$ and $\sigma_k\neq\P^N$ for some $k\geq 2$. It follows from our results that 
\[w(\sigma_k)={\rm HRH}(\sigma_k)=\begin{cases}
    \infty & C\cong\P^1\\
    -1 & \textrm{otherwise}
\end{cases}\] \color{black}
Thus, for $C\cong\P^1$, $\sigma_k$ has pre-$m$-Du Bois singularities if and only if it has pre-$m$-rational singularities. Moreover, for arbitrary $C$, $\sigma_k$ is normal when $L$ is sufficiently positive by \cite{ENP}, whence if $g(C)>0$, $\sigma_k$ never has pre-$m$-rational singularities for any $m\geq 1$ (see \cite{SVV} for definitions). In view of this, we pose the following (see also \cite{secantHigher}*{Ques. 7.7}):

\begin{problem}
    Let $L$ be a sufficiently positive line bundle on a smooth projective curve $C$  and let $\sigma_k$ be its $k$-th secant variety with $k\geq 3$. Determine when $\sigma_k$ has pre-$m$-Du Bois singularities for some $m\geq 1$.
\end{problem}

Next, we turn to a problem on secant variety of lines:

\begin{problem}
    Let $L$ be a 3-very ample line bundle on a smooth projective variety $Y$ and let $\Sigma$ be its secant variety. Determine the local analytic $\Q$-factoriality defect $\sigma^{\rm an}(\Sigma;y)$ at a point $y\in Y$.
\end{problem}

In \corollaryref{corFactorial}, we give an upper bound of the quantity in question. Our proof shows that the above problem is about identifying the morphism
\[H^2_{\rm prim}({\rm Bl}_y(Y)) \to H^2(Y,\cO_Y).\]
Is it the canonical one, given by projecting \[H^2_{\rm prim}({\rm Bl}_y(Y)) \to H^{0,2}_{\rm prim}({\rm Bl}_y(Y)) =H^{0,2}({\rm Bl}_y(Y)) = H^{0,2}(Y)?\]

Finally we pose the obvious problem about higher secant varieties of higher dimensional varieties:

\begin{problem}
    Extend this study to higher secant varieties of smooth projective surfaces, or to $3$-secant varieties in arbitrary dimension under sufficient positivity.
\end{problem}

Secant varieties of the above varieties have been recently analyzed in \cite{secantHigher}. In particular, the authors of {\it loc. cit.} assumes $Y\subset\P^N$ is a smooth projective variety embedded by the complete linear series of a sufficiently positive line bundle $L$, and shows that the results analogous to \cites{Ullery, ChouSong, ENP} continue to hold as long as $\dim Y\leq 2$ or $k\leq 3$ where $\Sigma_k\subset\P^N$ is its $k$th secant variety. It would be very interesting to use {\it loc. cit.} to make progress on the above problem, and for a start, we propose the following

\begin{problem}
    Let $Y\subset\P^N$ be a smooth projective surface embedded by the complete linear series of a sufficiently positive line bundle $L$, and let $\Sigma_k\subset\P^N$ be its $k$th secant variety. What is the value of the invariant ${\rm lcdef}(\Sigma_3)$? 
\end{problem}

In fact, we expect that the values of ${\rm lcdef}(\Sigma_k),c(\Sigma_k), {\rm HRH}(\Sigma_k)$ and $w(\Sigma_k)$ will remain unchanged from $k=2$ when $k\geq 3$ in the above setting.

\begin{comment}
\begin{enumerate}  

\item Compute the defect of analytic $\Q$-factoriality for $\Sigma$. For example, this can be done by identifying the morphism
\[H^2_{\rm prim}({\rm Bl}_y(Y)) \to H^2(Y,\cO_Y).\]

Is it the canonical one, given by projecting \[H^2_{\rm prim}({\rm Bl}_y(Y)) \to H^{0,2}_{\rm prim}({\rm Bl}_y(Y)) =H^{0,2}({\rm Bl}_y(Y)) = H^{0,2}(Y)?\]

\item Extend this study to higher secant varieties of smooth projective surfaces, or to $3$-secant varieties in arbitrary dimension, using the work of \cite{secantHigher}. 

\item Is the bound on the generation level sharp? For example, in the case of a canonically embedded curve $C\hookrightarrow \P^1$...
\end{enumerate}
\end{comment}

\bibliography{bib}
\bibliographystyle{abbrv}

\end{document}